\documentclass[11pt,oneside]{mnthesis}

\usepackage{setspace}
\usepackage{lmodern}
\usepackage[utf8]{inputenc}
\usepackage[T1]{fontenc}
\usepackage[final]{microtype}
\usepackage{blkarray}
\usepackage[style=trad-plain,natbib=true,url=false,isbn=false,eprint=false,doi=false]{biblatex}
\usepackage{mathtools}
\usepackage[pdftex]{hyperref}

\usepackage[final]{listingsutf8}

\lstdefinelanguage{M2}
{morekeywords={loadPackage,transpose,matrix,entries},
morekeywords=[2]{QQ,Ideal,ZZ},
morekeywords=[3]{Degrees},
otherkeywords={<---},
string=[b]{"},
sensitive=true,
morecomment=[l]{--},
}

\usepackage{tikz}
\usepackage{tikz-3dplot}
\usetikzlibrary{positioning}
\usetikzlibrary{fit,arrows,decorations,matrix,calc,patterns,shapes}
\usetikzlibrary{decorations.pathreplacing}
\usetikzlibrary{decorations.pathmorphing}
\usetikzlibrary{decorations.markings}

\usepackage{tikz-3dplot}
\pgfdeclaredecoration{sl}{initial}{
  \state{initial}[width=\pgfdecoratedpathlength-1sp]{
     \pgfmoveto{\pgfpointorigin}
  }
  \state{final}{
     \pgflineto{\pgfpointorigin}
    }
}

\tikzset{parallel/.style={decoration={sl,raise=1mm},decorate}}
\tikzset{position/.style args={#1:#2 from #3}{
        at=(#3.#1), anchor=#1+180, shift=(#1:#2)}}

\tikzstyle{none}=[inner sep=0pt]
\tikzstyle{every node}=[font=\tiny]
\tikzstyle{vertex}=[draw,circle,fill=black,inner sep=1pt]
\tikzstyle{new}=[circle,fill=black,draw=black,line width=0 pt]

\usepackage{comment}
\usepackage[letterpaper,left=1.5in,top=1in,right=1in,bottom=1in,includeheadfoot]{geometry}

\usepackage{amssymb,amsmath,amsthm,graphicx,xspace,csar,mathtools,enumitem,bm}
\usepackage[labelformat=simple]{subcaption}

\usepackage[utf8]{inputenc}

\graphicspath{{figures/}}

\makeatletter
\newtheorem*{rep@theorem}{\rep@title}
\newcommand{\newreptheorem}[2]{%
\newenvironment{rep#1}[1]{%
 \def\rep@title{#2 \ref{##1}}%
 \begin{rep@theorem}}%
 {\end{rep@theorem}}}
\makeatother

\newtheorem{thm}{Theorem}[section]
\newtheorem{prop}[thm]{Proposition}
\newtheorem{lem}[thm]{Lemma}
\newtheorem{conj}[thm]{Conjecture}
\theoremstyle{definition}
\newtheorem{defn}[thm]{Definition}

\newtheorem{quest}[thm]{Question}
\newtheorem{cor}[thm]{Corollary}

\makeatletter
\@addtoreset{thm}{chapter}
\makeatother

\newcommand{\Bn}{\Phi_{B_n}}
\newcommand{\An}{\Phi_{A_{n-1}}}
\newcommand{\Cn}{\Phi_{C_n}}
\newcommand{\Kprt}{K_{P}^{\mathrm{rt}}}
\newcommand{\Kpwt}{K_{P}^{\mathrm{wt}}}
\newcommand{\maj}{\mathrm{maj}}
\newcommand{\Hilb}{\mathrm{Hilb}}

\newcommand{\Des}{\mathrm{Des}}

\newcommand{\Lrt}{L^{\mathrm{rt}}}

\newcommand{\Rprt}{R_P^\mathrm{rt}}
\newcommand{\Rpwt}{R_P^\mathrm{wt}}
\newcommand{\syz}{\mathrm{syz}}

\newcommand{\init}{\mathrm{in}}
\newcommand{\initI}{\mathop{\mathrm{in}_\preceq} \Iwt{P}}
\newcommand{\sgn}{\mathrm{sgn}}

\newcommand{\supp}{\mathop{\mathrm{supp}}}
\newcommand{\ssupp}{\mathop{\mathrm{supp}_{\pm}}}

\newcommand{\Ghat}{\widehat{G}}

\newcommand{\Gbhat}{\widehat{G}_B}
\newcommand{\Gchat}{\widehat{G}_C}
\newcommand{\smvee}{{\scriptscriptstyle\vee}}
\newcommand{\Pc}{{P^{\scriptscriptstyle\vee}}}
\newcommand{\dual}{{\scriptscriptstyle\vee}}
\newcommand{\Rrt}[1]{R_{#1}^\mathrm{rt}}

\newcommand{\m}[1]{\bar{#1}}
\newcommand{\Jconn}{J_{\mathrm{conn}}}
\newcommand{\PLC}[1]{\overline{#1}^{PLC}}
\newcommand{\Lbrt}{L^{\mathrm{rt}}_B}
\newcommand{\Lcrt}{L^{\mathrm{rt}}_C}
\newcommand{\Krt}[1]{K^{\mathrm{rt}}_{#1}}
\newcommand{\Srt}[1]{S^{\mathrm{rt}}_{#1}}
\newcommand{\Irt}[1]{I^{\mathrm{rt}}_{#1}}
\newcommand{\Swt}[1]{S^{\mathrm{wt}}_{#1}}
\newcommand{\Iwt}[1]{I^{\mathrm{wt}}_{#1}}
\newcommand{\Lcwt}{L^\mathrm{cowt}}
\newcommand{\Kwt}[1]{K_{#1}^{\mathrm{wt}}}
\newcommand{\Rwt}[1]{R_{#1}^{\mathrm{wt}}}
\newcommand{\resp}{resp.\ }

\newcommand{\feray}{F\'eray\xspace}

\newcommand{\cZ}{\mathcal{Z}}
\newcommand{\conv}{\mathop{\mathrm{conv}}}

\allowdisplaybreaks
\begin{document}

\phd

\title{\bf Root and weight semigroup rings for signed posets}
\author{Sebastian Alexander Csar}
\campus{University of Minnesota} 
\program{Mathematics} 
\director{Victor Reiner, advisor} 
\submissionyear{2014}
\submissionmonth{August}

\abstract{\begin{spacing}{1}    We consider a pair of semigroups associated to a signed poset, called the root
semigroup and the weight semigroup, and their
semigroup rings, $\Rprt$ and $\Rpwt$, respectively.

Theorem~\ref{thm:toricideal} gives generators for the toric ideal of
        affine semigroup rings associated to signed posets and, more generally,
        oriented signed graphs. These are the subrings of Laurent polynomials
        generated by monomials of the form $t_i^{\pm 1},t_i^{\pm 2},t_i^{\pm
            1}t_j^{\pm 1}$. This result appears to be new and generalizes work
        of~\citet*{BoussicaultFerayLascouxReiner2012},~\citet*{GitlerReyesVillarreal2010}
        and~\citet{Villarreal1995}.
 Theorem~\ref{thm:csspci} shows that strongly planar signed posets
        $P$ have rings $\Rprt$, $\Rrt{\Pc}$ which are complete intersections,
        with Corollary~\ref{cor:computingpsi} showing how to compute $\Psi_P$
        in this case.
    Theorem~\ref{thm:typebpresentation} gives a Gr\"obner basis for the
        toric ideal of $\Rpwt$ in type B, generalizing~\citet*[Proposition
        6.4]{FerayReiner2012}. Theorems~\ref{thm:excludedposets} and~\ref{thm:construction} giving
        two characterizations (via forbidden subposets versus via inductive
        constructions) of the situation where this Gr\"obner basis gives a
        complete intersection presentation for its initial ideal,
        generalizing~\citet*[Theorems 10.5, 10.6]{FerayReiner2012}.
    \end{spacing}
}
\words{}    
\copyrightpage 
\acknowledgements{It is perhaps obligatory to start the acknowledgements by thanking your
advisor. To that end, I thank Vic Reiner for his guidance, patience and
willingness to learn to parse the varied meanings of ``all right''. I am
similarly indebted to Dennis Stanton and Gregg Musiker.

I thank Nathan Williams, Kevin Dilks and Thomas McConville for fielding my
random questions,  Maggie Ewing
and David Morawski for the standing work appointments, Alanna Hoyer-Leitzl for
her thoughtful support, Heidi Goodson, Jeff Cruse, Patrick Campbell and Baiying
Liu for their conversations in the office and Julie Leifeld for deciding we
should be friends.

I am grateful to the FLBs first for letting me join the group and second for
being such wonderful friends. I thank Jenna Ingalls for being my best friend
with all that entails, and my parents and brother for their endless love and support.
}
\dedication{For my grandad.}


\beforepreface 

\figurespage
\tablespage

\afterpreface            

\chapter{Introduction}\label{ch:introduction}

A partially ordered set, or \emph{poset}, is a set $P$ together with a binary
relation $<$ such that
\begin{itemize}
    \item $<$ is antisymmetric: if $x < y$ then $y \not < x$ for all $x,y \in
        P$, and
    \item $<$ is transitive: if $x < y$ and $y < z$ then $x < z$ for all $x,y,z
        \in P$.
\end{itemize}
Consider the poset in Figure~\ref{fig:typeaex}. The underlying set is $P =
\{1,2,3,4,5\}$ and $x < y$ if there is a path from $x$ to $y$ travelling
upwards along each edge in the path.
\begin{figure}[htbp]
    \begin{center}
        \begin{tikzpicture}
            \node[vertex] (1) at (0,0) [label=left:$1$] {};
            \node[vertex] (2) at (2,0) [label=right:$2$] {};
            \node[vertex] (3) at (1,1) [label=left:$3$] {};
            \node[vertex] (4) at (3,1) [label=right:$4$] {};
            \node[vertex] (5) at (2,2) [label=right:$5$] {};
            \draw (1)--(3)--(5)--(4)--(2)--(3);
        \end{tikzpicture}
    \end{center}
    \caption{A poset}
    \label{fig:typeaex}
\end{figure}
For example, $1 < 3$, but $1 \not < 4$.

An element of a poset $y \in P$ is said to \emph{cover} $x \in P$ if $x < y$
and there is no $z \in P$ such that $x < z < y$. One writes $x \lessdot y$ to
emphasize that the relation between $x$ and $y$ is a \emph{covering relation}.

A \emph{linear extension} of a poset is an extension $\prec$ of $<$ to a total
order, i.e.\ a linear order of the elements of $P$ by $\prec$ so that if $x <
y$ then $x \prec y$. The linear extensions of our example are
\[
    \begin{split}
        1 \prec 2 \prec 3 \prec 4 \prec 5 \\
        1 \prec 2 \prec 4 \prec 3 \prec 5 \\
        2 \prec 1 \prec 3 \prec 4 \prec 5 \\
        2 \prec 1 \prec 4 \prec 3 \prec 5 \\
        2 \prec 4 \prec 1 \prec 3 \prec 5
    \end{split}
\]
The set of linear extensions of $P$ is denoted $\cL(P)$. When $P$ is a poset on
$[n] \coloneqq \{1,2,\ldots,n\}$, it is natural to regard the linear extensions
as permutations of $[n]$. When $P$ is the poset in Figure~\ref{fig:typeaex} one
then has $\cL(P) = \{12345,12435,21345,21435,24135\}$.

While finding a linear extension of a finite poset is straightforward---it is what is
known as \emph{topological sorting} in computer science and can be done in
linear time (see, for instance~\cite[\S
22.4]{CormenLeisersonRivestStein2009})---counting the number of linear extensions
proves rather more difficult. Brightwell and Winkler showed
in~\cite{Brightwell1991} that the problem is $\# P$-complete. As a consequence,
computing rational functions which are sums over linear extensions proves
difficult. This leads to two basic questions:
\begin{itemize}
    \item When and how can the linear extensions of a poset be counted without
        listing them?
    \item When and how can such a rational function be computed without listing all the
linear extensions?
\end{itemize}

\section{A root system perspective on posets}\label{sec:rootsystemperspective}

In~\cite{BoussicaultFerayLascouxReiner2012}, Boussicault, \feray, Lascoux and
Reiner used a view of posets as sets of type A roots to explain how a pair of
rational functions which are sums over linear extensions can be evaluated, at
least in certain cases.
They transform a poset $P$ into the collection of type A roots $\{e_i-e_j
\colon i <_P j\}$. Under this scheme, $P$ from Figure~\ref{fig:typeaex}
corresponds to $\{e_1-e_3,e_1-e_5,e_2-e_3,e_2-e_4,e_2-e_5,e_3-e_5,e_4-e_5\}$.
The two rational functions they considered are, for a
poset $P$ on $[n]\coloneqq \{1,2,\ldots,n\}$,
\[
    \Psi_P(\bm{x}) = \sum_{w \in \cL(P)}
    w\left(\dfrac{1}{(x_1-x_2)(x_2-x_3)\cdots (x_{n-1}-x_n)}\right)
\]
and
\[
    \Phi_P(\bm{x}) = \sum_{w \in \cL(P)}
    w\left(\dfrac{1}{x_1(x_1+x_2)\cdots(x_1+\cdots+x_n)}\right),
\]
where the linear extensions are viewed as permutations acting on the indices of
$x_i$.

The function $\Psi_P$ had previously been considered by \citet{Greene1992}, where
he gave an evaluation for \emph{strongly planar} posets, i.e.\ those posets
whose Hasse diagrams remain planar after the addition of minimum and maximum
elements $\hat{0}$ and $\hat{1}$.
\begin{thm}[{Greene,~\cite{Greene1992}}]\label{thm:greene}
Suppose $P$ is a strongly planar poset. Then 
\[
    \Psi_P(\bm{x}) = \dfrac{\prod_\rho (x_{\min(\rho)} -
	x_{\max(\rho)})}{\prod_{i \lessdot j} (x_i-x_j)},
\]
where $\rho$ runs over all the bounded regions enclosed by the Hasse diagram of
$P$ and $i \lessdot j$ runs over all covering relations of the poset.
\end{thm}

In our example, computing $\Psi_P$ as a sum over the linear extensions
gives
\begin{align*}
    \Psi_P(\bm{x}) =& \sum_{w \in \cL(P)}
    w\left(\dfrac{1}{(x_1-x_2)(x_2-x_3)(x_3-x_4)(x_4-x_5)}\right) \\
    =& \dfrac{1}{(x_1-x_2)(x_2-x_3)(x_3-x_4)(x_4-x_5)} +
    \dfrac{1}{(x_1-x_2)(x_2-x_4)(x_4-x_3)(x_3-x_5)} \\
   & + \dfrac{1}{(x_2-x_1)(x_1-x_3)(x_3-x_4)(x_4-x_5)} 
    + \dfrac{1}{(x_2-x_1)(x_1-x_4)(x_4-x_3)(x_3-x_5)} \\
    &+ \dfrac{1}{(x_2-x_4)(x_4-x_1)(x_1-x_3)(x_3-x_5)} \\
    =& \dfrac{x_2-x_5}{(x_1-x_3)(x_2-x_3)(x_2-x_4)(x_3-x_5)(x_4-x_5)}.
\end{align*}
On the other hand, 
\[
\Psi_P(\bm{x}) = \dfrac{\prod_\rho (x_{\min(\rho)} -
	x_{\max(\rho)})}{\prod_{i \lessdot j} (x_i-x_j)}=\dfrac{x_2-x_5}{(x_1-x_3)(x_2-x_3)(x_2-x-4)(x_3-x_5)(x_4-x_5)}.
\]

The function $\Phi_P$ was considered in the case of forests by~\citet*{ChapotonHivertNovelliThibon2008}, who proved the following.
\begin{thm}[{Chapoton, Hivert, Novelli,
        Thibon,~\cite{ChapotonHivertNovelliThibon2008}}]\label{thm:chapotonhivert}
Suppose $P$ is a forest (i.e.\ every element is covered by at most one other
element). Then
\[
    \Phi_P(\bm{x}) = \prod_{i=1}^n \dfrac{1}{\sum_{j \leq_P i} x_j}.
\]
\end{thm}

Boussicault, \feray, Lascoux and Reiner then defined a pair of dual cones, the root cone $\Kprt =
\bR_+P$ and the weight cone $\Kpwt = \bR_+\{\chi_J \colon J \in J(P)\}$, where
$J(P)$ is the set of \emph{order ideals} $J$, subsets $J \subset P$ with the
condition that if $y \in J$ and $x <_P y$, then $x \in J$, and $\chi_J$ is the
characteristic vector of $J$. They
made two important realizations:
\begin{itemize}
    \item $\Psi_P$ and $\Phi_P$ are the Laplace transform valuations of $\Kprt$
        and $\Kpwt$, and
    \item $\Psi_P$ and $\Phi_P$ can be recovered from the Hilbert series of the
        semigroup rings 
        \[
            \Rprt = k[\Kprt\cap \bZ^n]\quad \text{and}\quad \Rpwt= k[\Kpwt \cap \bZ^n],
        \]
        respectively.
\end{itemize}
These two observations enabled Boussicault, \feray, Lascoux and Reiner
in~\cite{BoussicaultFerayLascouxReiner2012} and \feray and Reiner
in~\cite{FerayReiner2012} to use Proposition~\ref{prop:sandsigmaci} to obtain
the following two results.

\begin{thm}[{Boussicault, \feray, Lascoux, Reiner}]\label{thm:bflr}
Suppose $P$ is a strongly planar poset. Then $\Rprt$ is a complete
intersection,
\[
    \Hilb(\Rprt,\bm{x}) = \dfrac{\prod_{\rho}
        (1-x_{\min(\rho)}x_{\max(\rho)}^{-1})}{\prod_{i \lessdot_P j}
        (1-x_ix_j^{-1})}
\]
and
\[
    \Psi_P(\bm{x}) = \dfrac{\prod_{\rho}
        (x_{\min(\rho)}-x_{\max(\rho)})}{\prod_{i \lessdot_P j} (x_i-x_j)}.
\]
\end{thm}

In~\cite{FerayReiner2012}, \feray and Reiner described a class of posets
generalizing forests called
\emph{forests with duplication}, which are precisely those posets such that
$\Rpwt$ is a complete intersection, and compute $\Phi_P$ in this case. 

\begin{thm}[\feray and Reiner]\label{thm:ferayreiner}
A poset $P$ is a forest with duplication if and only if $\Rpwt$ is a complete
intersection, in which case
\[
    \Hilb(\Rpwt,\bm{x}) = \dfrac{\prod_{\{J,K\} \in \Pi(P)}
        (1-\bm{x}^J\bm{x}^k)}{\prod_{J \in \Jconn(P)} (1-\bm{x}^J)}
\]
and
\[
    \Phi_P(\bm{x}) = \dfrac{\prod_{\{J,K\} \in \Pi(P)} \langle
        \bm{x},\chi_{J+K}\rangle}{\prod_{J \in \Jconn(P)}\langle
        \bm{x},\chi_J\rangle},
\]
where $\Jconn(P)$ is the set of connected order ideals of $P$, $\Pi(P)$ is
the set of pairs of connected order ideals which intersect nontrivially ($J\cap
K \ne \emptyset$ and neither $J \subset K$ nor $K \subset J$) and $\langle
\bm{x}, \chi_J \rangle = \sum_{i \in J} x_i$.
\end{thm}

Return again to the poset in Figure~\ref{fig:typeaex}.  
One has that the root cone semigroup ring is presented as
\[
k[U_{13},U_{23},U_{24},U_{35},U_{45}]/(U_{23}U_{35}-U_{24}U_{45})
\]
with the map $U_{ij} \mapsto x_ix_j^{-1} \in k[x_1^{\pm 1},\ldots,x_n^{\pm
    1}]$,
and the weight cone semigroup ring is presented as
\[
k[U_1,U_2,U_{24},U_{123},U_{1234},U_{12345}]/(U_{123}U_{24}-U_2U_{1234})
\]
with the map $U_J \mapsto \prod_{i \in J} x_i \in k[x_1,\ldots,x_n]$.
Since both of the presentation ideals, $(U_{23}U_{35}-U_{24}U_{45})$ and
$(U_{123}U_{24}-U_2U_{1234})$, are principal, both $\Rprt$ and $\Rpwt$ are complete
intersections. Both rings are naturally $\bZ^5$-graded, and this $\bZ^5$-grading
coincides with an $\bN^5$-grading of $\Rpwt$. One then has
\[
    \Hilb(\Rprt,\bm{x}) =
    \dfrac{(1-x_2x_5^{-1})}{(1-x_1x_3^{-1})(1-x_2x_3^{-1})(1-x_2x_4^{-1})(1-x_3x_5^{-1})(1-x_4x_5^{-1})},
\]
so
\[
    \Psi_P(\bm{x}) =
    \dfrac{(x_2-x_5)}{(x_1-x_3)(x_2-x_3)(x_2-x_4)(x_3-x_5)(x_4-x_5)},.
\]
as seen above.
On the other hand,
\[
    \Hilb(\Rpwt,\bm{x}) =
    \dfrac{(1-x_1x_2^2x_3x_4)}{(1-x_1)(1-x_2)(1-x_2x_4)(1-x_1x_2x_3)(1-x_1x_2x_3x_4)(1-x_1x_2x_3x_4x_5)},
\]
so
\[
    \Phi_P(\bm{x}) =
    \dfrac{x_1+2x_2+x_3+x_4}{x_1x_2(x_2+x_4)(x_1+x_2+x_3)(x_1+x_2+x_3+x_4)(x_1+x_2+x_3+x_4+x_5)}.
\]
Computing $\Phi_P(\bm{x})$ as a sum over linear extensions, one has
\begin{multline*}
    \Phi_P(\bm{x})=\\
     \sum_{w \in \cL(P)}
    w\left(\dfrac{1}{x_1(x_1+x_2)(x_1+x_2+x_3)(x_1+x_2+x_3+x_4)(x_1+x_2+x_3+x_4+x_5)}\right)
    \\
    =
    \dfrac{1}{x_1(x_1+x_2)(x_1+x_2+x_3)(x_1+x_2+x_3+x_4)(x_1+x_2+x_3+x_4+x_5)}
    \\
    +
    \dfrac{1}{x_1(x_1+x_2)(x_1+x_2+x_4)(x_1+x_2+x_4+x_3)(x_1+x_2+x_4+x_3+x_5)}
    \\
    +
    \dfrac{1}{x_2(x_2+x_1)(x_2+x_1+x_3)(x_2+x_1+x_3+x_4)(x_2+x_1+x_4+x_3+x_5)}
    \\
    +
    \dfrac{1}{x_2(x_2+x_1)(x_2+x_4+x_1)(x_2+x_1+x_4+x_3)(x_2+x_1+x_4+x_3+x_5)}
    \\
    +
    \dfrac{1}{x_2(x_2+x_4)(x_2+x_4+x_1)(x_2+x_1+x_4+x_3)(x_2+x_4+x_1+x_3+x_5)}\\
    =
    \dfrac{x_1+2x_2+x_3+x_4}{x_1x_2(x_1+x_2+x_3)(x_2+x_4)(x_1+x_2+x_3+x_4)(x_1+x_2+x_3+x_4+x_5)},
\end{multline*}
matching the above computation.

\feray and Reiner also observed that when the $\bN^n$-grading on $\Rpwt$ is
collapsed to an $\bN$-grading, taking $\deg x_i =1$ for all $i$, one has
\[
    \Hilb(\Rpwt,q) = \dfrac{\sum_{w \in \cL(P)}
        q^{\maj(w)}}{(1-q)(1-q^2)\cdots(1-q^n)}.
\]
In particular, when $P$ is a forest, this recovers the $q$-hook formula of
Bj\"orner and Wachs~\cite{BjornerWachs1989}. For the poset in Figure~\ref{fig:typeaex}, one has
\begin{multline*}
    \Hilb(\Rpwt,q) = \dfrac{1-q^5}{(1-q)^2(1-q^2)(1-q^3)(1-q^4)(1-q^5)} \\
    = \dfrac{1+q^3+q+q^4+q^2}{(1-q)(1-q^2)(1-q^3)(1-q^4)(1-q^5)} =
    \dfrac{\sum_{w \in \cL(P)} q^{\maj(w)}}{(1-q)(1-q^2)(1-q^3)(1-q^4)(1-q^5)}.
    \end{multline*}

\section{The signed poset story}\label{subsec:signedposetstory}

The story for signed posets is really quite similar to that for posets and we
will generalize all of the above results to this context. Signed
posets will be defined formally in Definition~\ref{def:signedposet}. They come in pairs, $P \subset \Bn$ and $\Pc \subset \Cn$, since $\Bn$ and
$\Cn$ are dual root systems. Inspecting the definitions of $\Psi_P$ and
$\Phi_P$ in type A, one sees that the denominators of the fraction on which $w$ acts
correspond to a choice of simple roots and the corresponding fundamental
dominant coweights,
respectively. This observation leads one to define
\begin{align*}
    \Psi_P(\bm{x}) &= \sum_{w \in \cL(P)}
    w\left(\dfrac{1}{(x_1-x_2)(x_2-x_3)\cdots(x_{n-1}-x_n)x_n}\right) 
    \text{and} \\
    \Psi^\smvee_{\Pc}(\bm{x}) &= \sum_{w \in \cL(\Pc)}
    w\left(\dfrac{1}{(x_1-x_2)(x_2-x_3)\cdots(x_{n-1}-x_n)2x_n}\right)
\end{align*}
to parallel $\Psi$ in type A and
\begin{align*}
    \Phi_P(\bm{x}) &= \sum_{w \in \cL(P)}
    w\left(\dfrac{1}{x_1(x_1+x_2)\cdots(x_1+\cdots+x_n)}\right) 
    \text{and} \\
    \Phi^\smvee_{\Pc}(\bm{x}) &= \sum_{w \in \cL(\Pc)}
    w\left(\dfrac{1}{x_1(x_1+x_2)+\cdots+(x_1+\cdots+x_{n-1})(\frac{1}{2}x_1+\cdots+\frac{1}{2}x_n)}\right)
\end{align*}
to parallel $\Phi$ in type $A$.

One can likewise define root and weight cones for signed posets (see
Section~\ref{sec:dualcones}) and use the corresponding semigroup rings to
evaluate $\Psi$ and $\Phi$ in some cases where the semigroup ring is a complete intersection.

\begin{figure}[htbp]
    \begin{center}
        \begin{tikzpicture}
            \node[vertex] (4) at (0,0) [label=left:$4$] {};
            \node[vertex] (3) at (-1,1)[label=left:$3$] {};
            \node[vertex] (1) at (1,1) [label=right:$1$] {};
            \node[vertex] (2) at (0,2) [label=left:$2$] {};
            \node[vertex] (-2) at (2,2)[label=right:$-2$] {};
            \node[vertex] (-1) at (1,3)[label=left:$-1$] {};
            \node[vertex] (-3) at (3,3) [label=right:$-3$] {};
            \node[vertex] (-4) at (2,4) [label=left:$-4$] {};
            \node at (0,1) {$\rho$};
            \node at (1,2) {$\sigma,\iota(\sigma)$};
            \node at (2,3) {$\iota(\rho)$};
            \draw (4)--(3)--(2)--(-1)--(-4)--(-3)--(-2)--(1)--(4);
            \draw (1)--(2);
            \draw(-2)--(-1);
        \end{tikzpicture}
    \end{center}
    \caption{A signed poset}
    \label{fig:introsignedex1}
\end{figure}
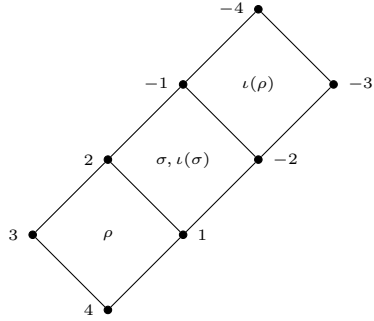

By way of example, consider Figure~\ref{fig:introsignedex1}. As will be
explained in Section~\ref{sec:representingposets}, it is a representation of a
signed poset $\Pc \subset \Cn$, with $\Pc =
\{+e_1-e_2,+e_1+e_2,+e_3-e_2,+e_4-e_3,+e_4-e_1,+e_4-e_2,+e_2+e_4,+e_3+e_4,+e_1+e_3,+e_1+e_4,+2e_1,+2e_4\}$. Note that there is an involutive poset
anti-automorphism, $\iota$, exchanging $i$
and $-i$. The poset is strongly planar (this will mean $\Pc$ is strongly planar)
and the three regions it encloses fall into two orbits under the involution:
$\{\rho,\iota(\rho)\}$ and
$\{\sigma=\iota(\sigma)\}$. Theorem~\ref{thm:csspci} will show that this implies the following complete intersection
presentation for $\Rrt{\Pc}$:
\[
    \Rrt{\Pc} \cong
    k[U_{12},U_{1\m2},U_{\m14},U_{\m23},U_{\m34}]/(U_{\m14}U_{1\m2}-U_{\m34}U_{\m23}),
\]
with the map $U_{ij}\mapsto x_{|i|}^{\sgn(i)}x_{|j|}^{\sgn(j)} \in k[x_1^{\pm
    1},\ldots,x_n^{\pm 1}]$, where $\sgn(i) = 1$ and $\sgn(\m i)=-1$, i.e.
\begin{align*}
U_{12} \mapsto x_1x_2 \\
U_{1\m2} \mapsto x_1x_2^{-1} \\
U_{\m1 4} \mapsto x_1^{-1}x_4 \\
U_{\m23} \mapsto x_2^{-1}x_3 \\
U_{\m34} \mapsto x_3^{-1}x_4.
\end{align*}
Suppose the polynomial ring is $\bZ^4$-graded with
\begin{align*}
    \deg U_{12} & = (1,1,0,0) \\
    \deg U_{1\m2} &= (1,-1,0,0)\\
    \deg U_{\m14} &= (-1,0,0,1) \\
    \deg U_{\m23} &= (0,-1,1,0) \\
    \deg U_{\m34} &= (0,0,-1,1).
\end{align*}
Then
\[
    \Hilb(\Rrt{\Pc},\bm{x}) =
        \dfrac{1-x_2^{-1}x_4}{(1-x_1x_2)(1-x_1x_2^{-1})(1-x_1^{-1}x_4)(1-x_2^{-1}x_3)(1-x_3^{-1}x_4)}.
\]
Looking at the numerator, one sees that it corresponds to
$1-x_{|\max(\rho)|}^{-\sgn(\max(\rho))}x_{|\min(\rho)|}^{\sgn(\min(\rho))}$, and
the terms of the denominator correspond to (pairs of) covering relations in the
poset of Figure~\ref{fig:introsignedex1},
paralleling Theorem~\ref{thm:bflr}. This is Theorem~\ref{thm:csspci}, and
Section~\ref{subsec:computingpsi} will explain the computation of $\Psi$ for certain
signed posets in types B and C.

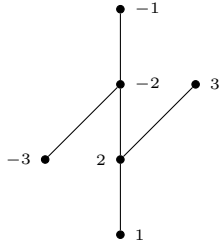
\begin{figure}[htbp]
    \begin{center}
        \begin{tikzpicture}
            \node[vertex] (1) at (0,0) [label=right:$1$] {};
            \node[vertex](2) at (0,1) [label=left:$2$] {};
            \node[vertex] (3) at (1,2) [label=right:$3$] {};
            \node[vertex] (-3) at (-1,1) [label=left:$-3$] {};
            \node[vertex] (-2) at (0,2) [label=right:$-2$] {};
            \node[vertex] (-1) at (0,3) [label=right:$-1$] {};
            \draw (1)--(2)--(-2)--(-1);
            \draw (-3)--(-2);
            \draw (2)--(3);
        \end{tikzpicture}
    \end{center}
    \caption{A signed poset}
    \label{fig:signedintroex2}
\end{figure}
As in type A, the weight cone semigroup will allow one to calculate $\Phi$, but
there is a wrinkle. Consider the poset shown in
Figure~\ref{fig:signedintroex2}. The weight cone semigroup ring $\Rpwt$ is
presented as
\[
    \Rpwt \cong k[U_1,U_{12}U_{123},U_{\m3}]/(U_{123}U_{\m3}-U_{12}),
\]
via the map $U_J \mapsto \bm{x}^J = \prod_{j \in J} x_{|j|}^{\sgn(j)} \in
k[x_1^{\pm 1},\ldots,x_n^{\pm n}]$, i.e.
\begin{align*}
U_1 & \mapsto x_1 \\
U_{12} &\mapsto x_1x_2 \\
U_{123} &\mapsto x_1x_2x_3 \\
U_{\m3} & \mapsto x_3^{-1}.
\end{align*}
When the polynomial ring is graded by $\bZ^3$ with
\begin{align*}
    \deg U_1 &= (1,0,0) \\
    \deg U_{12} &= (1,1,0) \\
    \deg U_{123} &= (1,1,1) \\
    \deg U_{\m3} &= (0,0,-1),
\end{align*}
the Hilbert series is
\[
    \Hilb(\Rpwt,\bm{x}) =
    \dfrac{1-x_1x_2}{(1-x_1)(1-x_1x_2)(1-x_1x_2x_3)(1-x_3^{-1})}.
\]
The ideals of the signed poset are $\{1\},\{1,2\},\{1,2,3\},\{-3\}$ (see
Section~\ref{sec:ideals}), corresponding to each term of the denominator of the
Hilbert series. The numerator corresponds to the one pair of ideals which
intersect nontrivially (to be defined in Section~\ref{sec:wttoricideal}),
$\{1,2,3\}$ and $\{-3\}$, which combine to form $\{1,2\}$. While this parallels
Theorem~\ref{thm:ferayreiner}, it turns out that the ``correct'' thing to do is
to consider an initial ideal, as doing so allows one to migrate through various
specializations of the
grading---the toric ideal from Theorem~\ref{thm:typebpresentation} is not
necessarily homogeneous in gradings other than the one by $\bZ^n$.

Chapter~\ref{ch:background} summarizes necessary background on cones,
semigroups and root systems. Further background information will be introduced
as needed. Chapter~\ref{ch:signedposets} reviews the
definition of signed posets, as well as ideals, $P$-partitions and linear
extensions of signed posets. Section~\ref{sec:representingposets} explains how
signed posets can be represented as posets and oriented signed graphs.
Fischer's representation of a signed poset $P \subset \Bn$ from~\cite{Fischer1993} is
modified for purposes of Chapter~\ref{ch:rootcone}. Section~\ref{sec:dualcones}
defines the root and weight cones of a signed poset and gives dual
characterizations for when they are each pointed, full-dimensional and
simplicial.

Chapters~\ref{ch:rootcone} and~\ref{ch:weightcone} proceed independently of one
another. Chapter~\ref{ch:rootcone} discusses the root cone semigroup.
Section~\ref{sec:signedgraphtoricideals} discusses the toric ideal of the
semigroup associated to a signed graph (Theorem~\ref{thm:toricideal}) and uses
that result to describe generating sets for the toric ideals of the root cone
semigroup, posets/digraphs and graphs. Section~\ref{sec:rprtci} describes a
situation when $\Rprt$ is a complete intersection and
Section~\ref{subsec:computingpsi} computes $\Psi_P$.

Chapter~\ref{ch:weightcone} discusses the weight cone semigroup in type B.
Section~\ref{thm:typebpresentation} gives a presentation for the weight cone
semigroup ring. Section~\ref{sec:computingsums} turns its attention to
computing $\Phi_P$ and
\[
    \sum_{w \in \cL(P)}q^{\maj(w)}.
\]
This is again a question of complete intersections, but not of the weight cone
semigroup ring. Instead, modding out by an initial ideal of the toric ideal
preserves the Hilbert series, but allows a presentation involving an
$\bN$-graded homogeneous
ideal. Signed posets with the property that modding out by the initial ideal of
the toric ideal gives a complete intersection will be
called \emph{initial complete intersections}.
Section~\ref{subsec:characterisinginitci} characterizes these signed posets as
those avoiding certain induced subposets, while
Section~\ref{subsec:constructinginitci} explains how these posets can be
constructed via any sequence of three moves.

Chapter~\ref{ch:unfinished} discusses a few loose ends: understanding two
triangulations of the weight cone, the type C weight cone and 
characterizing when $\Rprt$ is a complete intersection.

\section{Summary of the Main Results}\label{sec:resultsummary}
\begin{itemize}
    \item Theorem~\ref{thm:toricideal} giving generators for the toric ideal of
        affine semigroup rings associated to signed posets and, more generally,
        oriented signed graphs. These are the subrings of Laurent polynomials
        generated by monomials of the form $t_i^{\pm 1},t_i^{\pm 2},t_i^{\pm
            1}t_j^{\pm 1}$. This result appears to be new and generalizes work
        of~\citet*{BoussicaultFerayLascouxReiner2012},~\citet*{GitlerReyesVillarreal2010}
        and~\citet{Villarreal1995}.
    \item Theorem~\ref{thm:csspci} showing that strongly planar signed posets
        $P$ have rings $\Rprt$, $\Rrt{\Pc}$ which are complete intersections,
        with Corollary~\ref{cor:computingpsi} showing how to compute $\Psi_P$
        in this case.
    \item Theorem~\ref{thm:typebpresentation} giving a Gr\"obner basis for the
        toric ideal of $\Rpwt$ in type B, generalizing~\citet*[Proposition
        6.4]{FerayReiner2012}.
    \item Theorems~\ref{thm:excludedposets} and~\ref{thm:construction} giving
        two characterizations (via forbidden subposets versus via inductive
        constructions) of the situation where this Gr\"obner basis gives a
        complete intersection presentation for its initial ideal,
        generalizing~\citet*[Theorems 10.5, 10.6]{FerayReiner2012}.
\end{itemize}

\section{A remark on notation}\label{sec:notation}
\begin{itemize}
    \item To save space and improve readability, negative numbers will
        sometimes be written as $\m1$ rather than $-1$.
    \item The parentheses and brackets will sometimes be dropped from vectors
        and sets in figures, e.g.\ $100$ instead of $(1,0,0)$.
    \item There will be many vertices of graphs that come in pairs $i,-i$ for
        $i \in [n]$ and
        polynomial rings where the variables are indexed by $[n]$. If $v$ is a
        variable, understand $x_v$ to mean the variable $x_{|v|}$.
    \item The characteristic vector of a set $J \subset \{\pm 1, \ldots, \pm
        n\}$ such that $J$ does not contain both $i$ and $-i$ for any $i$ is the vector $\chi_J$ whose coordinates are defined by
        \[
            (\chi_J)_i = \begin{cases}
1 & \text{if }i \in J \\
-1 & \text{if } -i \in J \\
0 & \text{else}
            \end{cases}
        \]
    \item If $J \subset \{\pm1 ,\ldots, \pm n\}$ is such that there is no $i$
        such that both $i, -i \in J$, then 
        \[
            \langle \bm{x},J \rangle = \langle \bm{x},\chi_J\rangle =\sum_i
            (\chi_J)_i x_i.
        \]
\end{itemize}

\chapter{Some Background}\label{ch:background}

This chapter reviews requisite material on polyhedral cones, semigroups and
root systems. Chapter~\ref{ch:signedposets} will fit these ideas together in
the discussion of signed posets. Further background material will be introduced
as needed.

\section{Polyhedral Cones}\label{sec:cones}

Associated to a signed poset will be two polyhedral cones, the root cone and
the weight cone. Consequently, this section reviews some basic facts about
polyhedral cones. One can also refer to~\citet{Fulton1993} for further
information on polyhedral cones.

\begin{defn}\label{def:cone}
A \emph{polyhedral cone} $K \subset \bR^n$ is the intersection of finitely many half-spaces
determined by hyperplanes $H_\alpha = \{x\colon \langle x,\alpha \rangle =
0\}$. The $H_\alpha$ are the \emph{supporting hyperplanes} of $K$.
Alternatively, a cone may be characterized as the positive span of some finite
collection of vectors, $W$, with the positive span being denoted $\bR_+W$.
A set of \emph{extreme rays} of a cone $K$ are vectors comprising a set $W$,
minimal with respect to inclusion, such that $K = \bR_+W$. 
\end{defn}

There are a number of properties that can be used to describe a cone.

\begin{defn}\label{def:coneprops}
The \emph{dimension} of a
cone $K$, denoted $\dim K$ is the dimension of the vector space spanned by its
extreme rays, and $K$ is said to be
\emph{full-dimensional} if $K \subset \bR^n$ and $\dim K = n$. A cone is said to be
\emph{pointed} when it does not contain a line. It is said to be
\emph{rational} with respect to a full-rank lattice $L \subset \bR^n$ when the $\alpha$ determining the
supporting hyperplanes lie in $L$. It is \emph{simplicial} if
the extreme rays are linearly independent. A (simplicial) cone is \emph{unimodular} with
respect to a lattice if the primitive vectors along the extreme rays form a
basis of the lattice.
\end{defn}

For example, consider the cone in Figure~\ref{fig:coneex}.
\begin{figure}[htbp]
\begin{center}
    \tdplotsetmaincoords{70}{100}
    \begin{tikzpicture}[tdplot_main_coords,scale=2]
        \coordinate (O) at (0,0,0);
\draw[thick,->] (0,0,0) -- (2,0,0) node[anchor=north east]{$x$};
\draw[thick,->] (0,0,0) -- (0,2,0) node[anchor=north west]{$y$};
        \coordinate (1) at (1,-1,1);
        \coordinate (2) at (1,1,1);
        \coordinate (3) at (-1,1,1);
        \coordinate (4) at (-1,-1,1);
        \fill[gray!20] (O)--(1)--(4)--cycle;
        \fill[gray!20] (O)--(4)--(3)--cycle;
        \fill[gray!20] (O)--(2)--(3)--cycle;
\draw[thick,->] (0,0,0) -- (0,0,2) node[anchor=south]{$z$};
        \draw[blue, ->] (0,0,0)--(1,-1,1) node[anchor=north
		east,black]{$(1,-1,1)$};
        \draw[blue, ->] (0,0,0)--(1,1,1) node[anchor=
		west,black]{$(1,1,1)$};
        \draw[blue,->] (0,0,0)--(-1,1,1) node[anchor=west,black]{$(-1,1,1)$};
        \draw[blue,->] (0,0,0)--(-1,-1,1)node[anchor=south east,black]{$(-1,-1,1)$};
        \fill[gray!20] (O)--(1)--(2)--cycle;
    \end{tikzpicture}
\end{center}
\caption{Cone spanned by $(1,1,1)$, $(-1,1,1)$, $(-1,-1,1)$ and $(1,-1,1)$}
\label{fig:coneex}
\end{figure}
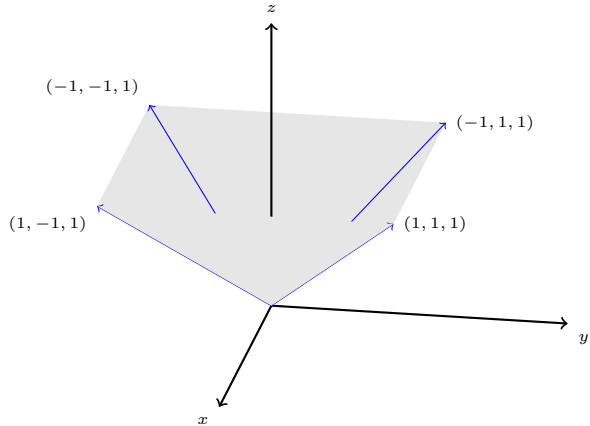
The extreme rays are $(1,1,1)$, $(-1,1,1)$, $(-1,-1,1)$ and $(1,-1,1)$ and the cone is
pointed and rational. Since it is a cone
in $\bR^3$ with four extreme rays, it cannot be simplicial. Likewise, it cannot be
unimodular with respect to the lattice $\bZ^3$.

\begin{defn}\label{def:dualcone}
If $K$ is a cone, its \emph{dual} or \emph{polar cone} is 
\[
K^* = \{x \in \bR^n
\colon \langle x, a\rangle \geq 0\ \forall a \in K\}.
\]
\end{defn}
The dual is sometimes defined
with a $\leq$ rather than a $\geq$, but the appeal of this choice of $\geq$
will become clear in Section~\ref{sec:dualcones} and one can insert a minus
sign where appropriate when the $\leq$ definition is more apt. The following
 facts about cones and their duals are well known.
\begin{itemize}
    \item A cone is the dual of its dual, i.e.\ $K^{**} = K$.
    \item A cone $K$ is full-dimensional if and only if its dual $K^*$ is
        pointed.
    \item A cone $K$ is simplicial if and only if its dual $K^*$ is also
        simplicial.
\end{itemize}

\subsection{Rational Functions and Cones}\label{subsec:rationalfunctionscones}

Recall that the goal is to prove some rational function identities for signed
posets. A key step will be to understand the functions $\Psi$ and $\Phi$ as
valuations of cones.

In~\cite{Barvinok1993}, Barvinok considers an exponential integral and
exponential sum over a polyhedral cone, $K$:
\[
    \int_K e^{-\langle x, u\rangle}\,du \quad \text{and} \quad \sum_{K \cap
        \bZ^n} e^{-\langle x,u\rangle},
\]
where $x\in \bR^n$ and $du$ is Lebesgue measure on $\bR^n$. Each gives a
rational function, in $x_i$ and $X_i=e^{x_i}$, respectively (see
Propositions~\ref{prop:expintrational} and~\ref{prop:expsumrational}).

\begin{prop}[{\cite[Proposition 2.4]{Barvinok1993}}]\label{prop:expintrational}
    Let $K$ be a pointed, full-dimensional polyhedral cone in $\bR^n$. Then,
    for all $x \in \Int K^*$, the integral
    \[
        \int_K e^{-\langle x, u \rangle}\,du
    \]
    exists and determines a function $s(K;x)$, which is rational in $x \in
    \bC^n$. 

    Furthermore, if $K$ is not pointed or if $K$ is not full-dimensional, $s(K;x)
    = 0 $.
\end{prop}

For example, consider the cone in Figure~\ref{fig:coneex} once again. To compute the integral, it
is easiest to split $K$ into two simplicial cones $K_1$ and $K_2$:
\begin{align*}
    K_1 &= \text{span}_{\bR_+}\{(1,-1,1),(-1,-1,1),(-1,1,1)\} \\
    K_2 &= \text{span}_{\bR_+}\{(-1,1,1),(1,-1,1),(1,1,1)\}.
\end{align*}
Since $K_1$ and $K_2$ are both simplicial, every
element of each cone can be written uniquely as a linear combination of the
extreme rays. Then, if one denotes the extreme rays of $K_1$ by $u_1,u_2,u_3$
and the extreme rays of $K_2$ by $v_1,v_2,v_3$, for a fixed $x \in \Int K^*$, one can compute the integral as
follows.
\[
    \int_K e^{-\langle x, u\rangle}\, du = \int_{K_1} e^{-\langle x, u \rangle}\,
	du + \int_{K_2} e^{-\langle x,u \rangle}\, du - \int_{K_1 \cap K_2}e^{-\langle x,u\rangle}\,du.
\]
The last integral is $0$ since $K_1 \cap K_2$ is not full-dimensional.
Then
\begin{align*}
    \int_{K_1} e^{-\langle x, u \rangle}\, du &= \int_0^\infty \int_0^\infty
    \int_0^\infty e^{-\langle x, a_1u_1 + a_2u_2 +
        a_3u_3\rangle}\,da_3da_2da_1 \\
    &= \lim_{b_1,b_2,b_3 \to \infty} \int_0^{b_1}\int_0^{b_2}\int_0^{b_3}
    e^{-\langle x,a_1u_1\rangle}e^{-\langle x,a_2u_2\rangle}e^{-\langle
        x,a_3u_3\rangle}\,da_3da_2da_1 \\
    &= \lim_{b_1,b_2,b_3 \to \infty} \int_0^{b_1}e^{-\langle x,a_1u_1\rangle}
    \int_0^{b_2} e^{-\langle x,a_2u_2\rangle} \int_0^{b_3}e^{-a_3\langle x,
        u_3\rangle}\,da_3da_2da_1 \\
    &=\lim_{b_1,b_2,b_3 \to \infty} \int_0^{b_1}e^{-\langle x,a_1u_1\rangle}
    \int_0^{b_2} e^{-\langle x,a_2u_2\rangle} \left[ \dfrac{-1}{\langle
            x,u_3\rangle} e^{-a_3\langle x, u_3\rangle}\right]_0^{b_3}\,da_2da_1
    \\
    &= \lim_{b_1,b_2,b_3 \to \infty} \left[\dfrac{-1}{\langle x,u_1\rangle}
        e^{-a_1\langle x,u_1\rangle}\right]_0^{b_1}\left[\dfrac{-1}{\langle
            x,u_2\rangle} e^{-a_2\langle
            x,u_2\rangle}\right]_0^{b_2}\left[\dfrac{-1}{\langle
            x,u_3\rangle}e^{-a_3\langle x,u_3\rangle}\right]_0^{b_3}
\end{align*}
Since $x \in \Int K^*$, one knows that $\langle x, u_1\rangle, \langle x,
u_2\rangle, \langle x,u_3\rangle \geq 0$. Therefore, 
\[
    \lim_{b_i \to \infty} e^{-b_i\langle x,u_i\rangle} = 0
\]
for $i = 1,2,3$. Thus, one has
\[
\int_{K_1} e^{-\langle x, u \rangle}\, du = \dfrac{1}{\langle
    x,u_1\rangle}\dfrac{1}{\langle x,u_2\rangle}\dfrac{1}{\langle x,u_3\rangle}
	=
	\dfrac{1}{(-x_1+x_2-x_3)}\dfrac{1}{(x_1+x_2-x_3)}\dfrac{1}{(x_1-x_2-x_3)}.
\]
A similar calculation with $K_2$ gives
\[
    \int_{K_2} e^{-\langle x, u \rangle}\, du =
	\dfrac{1}{(x_1-x_2-x_3)}\dfrac{1}{(-x_1+x_2-x_3)}\dfrac{1}{(-x_1-x_2-x_3)}.
\]
One then has
\begin{multline*}
     \int_K e^{-\langle x, u\rangle}\, du =  
	\dfrac{1}{(-x_1+x_2-x_3)}\dfrac{1}{(x_1+x_2-x_3)}\dfrac{1}{(x_1-x_2-x_3)}\\ +
	\dfrac{1}{(x_1-x_2-x_3)}\dfrac{1}{(-x_1+x_2-x_3)}\dfrac{1}{(-x_1-x_2-x_3)}
	 \\
	= \dfrac{1}{(x_1-x_2-x_3)(-x_1+x_2-x_3)}\left(\dfrac{1}{-x_1-x_2-x_3}+\dfrac{1}{x_1+x_2-x_3}\right)
\end{multline*}

\begin{prop}[{\citet[(2.1)]{Barvinok1993}}]\label{prop:barvinok}
Suppose $K$ is a simplicial cone whose extreme rays are $\{u_1,\ldots,u_n\}$.
Then
\[
s(K;\bm{x}) = |u_1\wedge \cdots \wedge u_n|\prod_{i=1}^n \langle
\bm{x},u_i\rangle^{-1},
\]
where $|u_1\wedge \cdots \wedge u_n|$ is the volume of the parallelopiped
formed by the $u_i$.
\end{prop}

Discussion of the exponential sum will be postponed until the
Section~\ref{sec:semigroups}.

\section{Semigroups}\label{sec:semigroups}

The next important concept is that of the semigroup.

\begin{defn}\label{def:semigroup}
A \emph{semigroup} is a set together with an associative binary operation. 
An \emph{affine semigroup} is a semigroup which is isomorphic to a
finitely-generated subsemigroup
of $\bZ^n$ under addition.
\end{defn}

In general, unlike monoids, semigroups need not have an identity element. However,
the semigroups of interest here will have an identity and the binary operation,
$+$, will be commutative. The semigroups considered in Chapters~\ref{ch:rootcone}
and~\ref{ch:weightcone} are affine semigroups as a consequence of Gordan's
Lemma (see~\cite[Proposition 5.14]{EneHerzog2012}).

\begin{thm}[Gordan's Lemma]\label{thm:gordan}
Suppose $K$ is a rational polyhedral cone in $\bR^n$ and $A$ is a subgroup of
$\bZ^n$. Then $K \cap A$ is an affine semigroup.
\end{thm}

As an example, suppose $K$ is the cone in Figure~\ref{fig:gordanex} whose extreme rays are $(1,0)$ and
$(0,1)$ and suppose $A$ is $\bZ(1,1) \subset \bZ^2$.
\begin{figure}[htbp]
\begin{center}
    \begin{tikzpicture}
        \draw[thin,gray] (-3,0)--(3,0);
        \draw[thin,gray] (0,-3)--(0,3);
        \fill[gray!20] (0,0)--(3,0) arc (0:90:3) -- cycle;        
        \foreach \x in {0,1,2}
        {
            \node[vertex] at (\x,\x) {};
            \node[vertex] at (-\x,-\x) {};
        }
    \end{tikzpicture}
\end{center}
\caption{The cone spanned by $(1,0)$ and $(0,1)$ intersected with $\bZ^2$}
\label{fig:gordanex}
\end{figure}
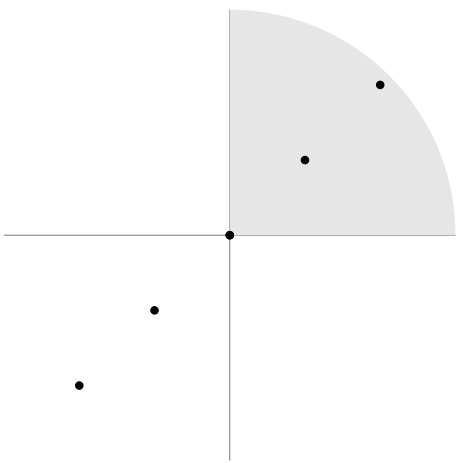
Then $K \cap A=\{(i,i) \colon i \in \bZ_{\geq 0}\}$.

\subsection{Semigroup Rings and Toric Ideals}\label{subsec:semigrouprings}

One can move from the semigroup world to the somewhat more familiar world of
rings by considering semigroup rings.

\begin{defn}\label{def:semigroupring}
Suppose $A \subset \bZ^n$ is an affine semigroup generated by
$\{a_1,\ldots,a_m\}$. Let $L =
k[t_1^{\pm1},\ldots,t_n^{\pm1}]$ be the Laurent polynomial ring. The
\emph{semigroup ring} of $A$ is the subring of $L$ spanned by
$\bm{t}^{a_i}$, $i=1,\ldots,m$, where $\bm{t}^{a_i} =
t_1^{a_{i1}}t_2^{a_{i2}}\cdots
    t_n^{a_{in}}$ when $a_i = (a_{i1},a_{i2},\ldots,a_{in})$. Denote the
    semigroup ring of $A$ by $k[A]$.
\end{defn}

    It is often more convenient to work with a presentation of the semigroup
    ring as a quotient by the toric ideal than to view the semigroup ring as a
    subring of the Laurent polynomial ring. 
	
\begin{defn}\label{def:toricideal}
	Let $S = k[x_1,\ldots,x_m]$ be a
    polynomial ring in $m$ variables and $A$ a semigroup generated by
    $\{a_1,\ldots,a_m\}$. Define a map $\phi \colon S \to L$ by
    \[
        \phi(x_i) = \bm{t}^{a_i}.
    \]
    One then has that $k[A] \cong S/\ker \phi$. The kernel of $\phi$ is the ideal
    known as the \emph{toric ideal} of $A$, denoted $I_A$.
\end{defn}

\begin{defn}\label{def:binomialideal}
    A \emph{binomial} is a polynomial which is a difference of two monomials
    and a \emph{binomial ideal} is an ideal that is generated by binomials.
\end{defn}
    Comprehensive discussion of binomial ideals can be found
    in~\citet{EisenbudSturmfels1994}. 
	
The following definition and notation is useful for understanding the
generators of toric ideals.
\begin{defn}
	If $u \in \bZ^n$, define two vectors, the positive and
	negative supports of $u$ as follows. The \emph{positive support} $u^+$ is given by
    \[
        u^+_i = \begin{cases}
            u_i & \text{if $u_i> 0$} \\
            0 & \text{else}
        \end{cases}
    \]
	Likewise, the \emph{negative support} $u^-$ is
    \[
        u^-_i = \begin{cases}
            u_i & \text{if $u_i < 0$} \\
            0 & \text{else}
        \end{cases}
    \]
    Then to each $u \in \bZ^n$, one can associate a binomial $\bm{x}^{u^+}
    - \bm{x}^{u^-}$.
	\end{defn}

	The following proposition is well known and the
    proof may be found in~\citet[Lemma 4.1]{Sturmfels1996a} or~\citet[Lemma
    5.2]{EneHerzog2012}.

\begin{prop}\label{prop:toricidealsbinomial}
Suppose $A$ is an affine semigroup and $I_A$ its toric ideal. Then $I_A$ is a
binomial ideal generated by $\bm{x}^{u^+}-\bm{x}^{u^-}$ for $u \in \ker
M$ where $M$ is the matrix whose columns are the generators of $A$. 
\end{prop}

As an example, consider the semigroup whose generators are
$(1,-1,1)$, $(-1,-1,1)$, $(-1,1,1)$, $(1,1,1)$. (This is the semigroup obtained by
intersecting the cone of Figure~\ref{fig:coneex} with $\bZ^3$.) The matrix $M$ is then
\[
    \begin{pmatrix*}[r]
    1 & -1 & -1 & 1 \\
	-1 & -1 & 1 & 1 \\
	1 & 1 & 1 & 1 \\
	\end{pmatrix*}
\]
The kernel of $M$ is one-dimensional and spanned by $(1,-1,1,-1)^\top$. Let $S =
k[x_1,x_2,x_3,x_4]$ and $\phi \colon S \to k[t_1^{\pm 1},t_2^{\pm 1},t_3^{\pm 1}]$ be 
defined by
\begin{align*}
		\phi(x_1) &= t_1t_2^{-1}t_3\\
		\phi(x_2) &= t_1^{-1}t_2^{-1}t_3 \\
		\phi(x_3) &= t_1^{-1}t_2t_3 \\
    \phi(x_4) &= t_1t_2t_3 
\end{align*}
The kernel of $\phi$ is then the principal ideal $(x_1x_3-x_2x_4)$. Certainly this
ideal is contained in the kernel. One knows from
Proposition~\ref{prop:toricidealsbinomial} that $\ker \phi =
(\bm{x}^{u^+}-\bm{x}^{u^-})$ for $u = m(1,-1,1,-1)^\top \in \ker M \cap \bN^n$, with $m
\in \bZ_{>0}$, so one
shows that $\bm{x}^{u^+}-\bm{x}^{u^-} \in (x_1x_3-x_2x_4)$. This is
straightforward:
\begin{multline*}
x_1^mx_3^m-x_2^mx_4^m = \\
(x_1x_3-x_2x_4)((x_1x_3)^{m-1}+(x_1x_3)^{m-2}x_2x_4+\cdots +
x_1x_3(x_2x_4)^{m-2}+(x_2x_4)^{m-1}) \\ \in (x_1x_3-x_2x_4),
\end{multline*}
so $(x_1x_3-x_2x_4) = \ker \phi$.

\subsection{A Rational Function}\label{subsec:semigrouprationalfunction}

The exponential sum
\[
    \sum_{K \cap \bZ^n} e^{-\langle x, u\rangle}    
\]
discussed in~\cite{Barvinok1993} and mentioned in
Section~\ref{subsec:rationalfunctionscones} is then a sum over the elements of
an affine semigroup. The following proposition is the analogue of
Proposition~\ref{prop:expintrational}.

\begin{prop}[{\citet[Proposition 4.4]{Barvinok1993}}]\label{prop:expsumrational}
Suppose $K$ is a pointed rational polyhedral cone in $\bR^n$. Then for $x \in
\Int K^*$, the series
\[
    \sum_{K \cap \bZ^n} e^{-\langle x, u \rangle}
\]
converges and determines a function $\sigma(K;x)$ which is rational in $X_i =
e^{x_i}$, $i = 1,\ldots,n$. Furthermore, there exists a representation
\[
    \sigma(K;x) = \dfrac{P(x)}{\prod_{i=1}^m (1-e^{-\langle x, u_i\rangle})},
\]
where $P(x)$ is a Laurent polynomial in $X_i$ and the $u_i$ are the extreme
rays of $K$. If $K$ is not pointed, $\sigma(K;x)
= 0$.
\end{prop}

Moreover, we can understand this sum as the Hilbert series of the semigroup
ring. 

\begin{defn}\label{def:hilbertseries}
Suppose $k$ is a field and $R$ is a finitely-generated $k$-algebra. Suppose
further that $R$ is graded by some index set $I$ equipped with an addition, i.e.\ $R  = \bigoplus_{\alpha
    \in I} R_\alpha$ with $R_\alpha R_\beta \subset R_{\alpha + \beta}$. Its
\emph{Hilbert series} is
\[
    \Hilb(R,x) = \sum_{\alpha \in I} \dim(R_\alpha)x^\alpha = \sum_{r \in R}
    x^{\deg r}.
\]
\end{defn}

A natural grading of the semigroup ring of an affine semigroup $A$ is the $\bZ^n$-grading where
$\deg(\bm{x}^a) = (a_1,\ldots,a_n)$ for $a \in A$. Then one has that
\[
    \Hilb(k[A],\bm{x}) = \sum_{a \in A} x_1^{a_1}x_2^{a_2}\cdots x_n^{a_n}.
\]
Note that we are abusing notation here and using $x_i$ in both the presentation
of the semigroup ring and the Hilbert series, though these are actually
different $x_i$.

Moreover, if one supposes that $A$ arose from Gordan's Lemma as the
intersection of a pointed rational cone $K$ with $\bZ^n$, one has that 
\[
    \Hilb(k[A],\bm{x}) = \sum_{a \in A} x_1^{a_1}x_2^{a_2}\cdots x_n^{a_n}
    = \sum_{a \in K\cap \bZ^n} x_1^{a_1}x_2^{a_2}\cdots x_n^{a_n},
\]
and this last sum is precisely the sum $\sum_{K \cap \bZ^n} e^{-\langle X, u
    \rangle}$ after the change of coordinates $x_i = e^{-X_i}$.

One can compute the Hilbert series of a
graded ring from a minimal finite free resolution courtesy of the following
well-known fact (see, for example,~\citet[Theorem I.11.3]{Stanley1996}).
\begin{prop}\label{prop:hilbfromffr}
Suppose $M$ is a graded $A$-module and 
\[
0 \to F_t \to F_{t-1} \to \cdots \to F_{1} \to F_0 \to M \to 0
\]
 is a finite free resolution of $M$ by graded $A$-modules. Then
 \[
     \Hilb(M,x) = \sum_{i=0}^t (-1)^i \Hilb(F_i,x).
 \]
\end{prop}

Consider the cone $K$ having extreme rays $(1,0,0),(1,1,0),(1,0,1),(1,1,1)$.
In fact, $K$ is the weight cone of a poset on three elements,
so~\cite[Proposition 7.1]{BoussicaultFerayLascouxReiner2012} gives that these
vectors also (minimally) generate the semigroup $K \cap \bZ^n$, which will be
denoted by $A$. Let $S= k[a,b,c,d]$ with $\deg(a) = (1,0,0)$,
$\deg(b)=(1,1,0)$, $\deg(c) = (1,0,1)$, $\deg(d) = (1,1,1)$. The semigroup ring
$k[A]$ is isomorphic to $S/(bc-ad)$ (see~\cite[Proposition 6.4]{FerayReiner2012}).
One then has the following (minimal) finite free resolution for $S/I$:
\[
 0 \to S(-(2,1,1)) \stackrel{\phi}{\to}S \to S/I \to 0
\]
where $\phi$ is the map sending $1$ to $bc-ad$, and $S(-(2,1,1))$ is $S$ with
the grading shifted so that $1 \in S$ has degree $(2,1,1)$. One can then compute the
Hilbert series of $S/I$ (i.e.\ of $k[A]$) as
\begin{multline*}
    \Hilb(S/I,\bm{x}) = \Hilb(S,\bm{x}) -
    \Hilb(S(-(2,1,1)),\bm{x}) \\
    = \dfrac{1}{(1-x_1)(1-x_1x_3)(1-x_1x_2)(1-x_1x_2x_3)} \\
    \shoveright{-\dfrac{x_1^2x_2x_3}{(1-x_1)(1-x_1x_3)(1-x_1x_2)(1-x_1x_2x_3)}} \\
     = \dfrac{1-x_1^2x_2x_3}{(1-x_1)(1-x_1x_3)(1-x_1x_2)(1-x_1x_2x_3)}
\end{multline*}

On the other hand, one can compute the exponential sum directly. Let $K_1$ be
the cone whose extreme rays are $(1,0,0),(1,1,0),(1,0,1)$, $K_2$ the cone whose
extreme rays are $(1,1,0),(1,0,1),(1,1,1)$. Then $K_1 \cap K_2$ is the cone
whose extreme rays are $(1,1,0)$ and $(1,0,1)$. (See
Figure~\ref{fig:semigroupsumex}.)

\begin{figure}[htbp]
\begin{center}
    \tdplotsetmaincoords{70}{100}
    \begin{tikzpicture}[tdplot_main_coords,scale=2]
        \coordinate (O) at (0,0,0);
\draw[thick,->] (0,0,0) -- (2,0,0) node[anchor=north east]{$x$};
\draw[thick,->] (0,0,0) -- (0,2,0) node[anchor=north west]{$y$};
\draw[thick,->] (0,0,0) -- (0,0,2) node[anchor=south]{$z$};
        \coordinate (1) at (1,0,0);
        \coordinate (2) at (1,1,0);
        \coordinate (3) at (1,0,1);
        \coordinate (4) at (1,1,1);
        \fill[gray!20] (O)--(1)--(3)--cycle;
        \fill[gray!20] (O)--(4)--(3)--cycle;
        \fill[gray!20] (O)--(2)--(4)--cycle;
        \fill[gray!20] (O)--(1)--(2)--cycle;
        \draw[blue, ->] (0,0,0)--(1,0,0) node[anchor=east,black]{$(1,0,0)$};
        \draw[blue, ->] (0,0,0)--(1,1,0) node[anchor=
		west,black]{$(1,1,0)$};
        \draw[blue,->] (0,0,0)--(1,0,1) node[anchor=east,black]{$(1,0,1)$};
        \draw[blue,->] (0,0,0)--(1,1,1)node[anchor=west,black]{$(1,1,1)$};
        \draw[pattern=dots,dashed] (0,0,0)--(1,0,1)--(1,1,0)--(0,0,0);
		\node[inner sep=0] (K2) at (0,.5,1) {$K_2$};
		\node (b) at (0,.25,.5) {};
		\node[inner sep=0] (K1) at (1,-.25,.5) {$K_1$};
		\node (a) at (.5,-.075,.25) {};
		\draw[->] (K1) edge[bend left] (a);
		\draw[->] (K2) edge[bend right](b);
\end{tikzpicture}
\end{center}
\caption{Cone spanned by $(1,0,0)$, $(1,1,0)$, $(1,0,1)$ and $(1,1,1)$}
\label{fig:semigroupsumex}
\end{figure}
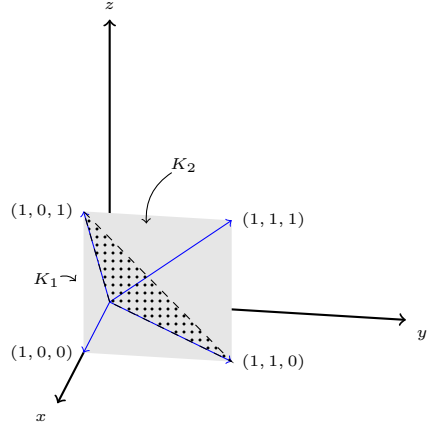
One then has
\[
    \sum_{K \cap \bZ^n} e^{-\langle x, u\rangle} = \sum_{K_1 \cap \bZ^n}
    e^{-\langle x, u\rangle} +\sum_{K_2\cap \bZ^n} e^{-\langle x, u \rangle}
    -\sum_{K_1 \cap K_2 \cap \bZ^n} e^{-\langle x,u \rangle}.
\]
Making the change of variables so that $x_i = e^{-x_i}$, one has that
\begin{align*}
    \sum_{K_1 \cap \bZ^3} e^{-\langle x, u\rangle} &= \sum_{a_1=0}^\infty
    \sum_{a_2=0}^\infty \sum_{a_3=0}^\infty e^{-\langle x,
        a_1(1,0,0)+a_2(1,1,0)+a_3(1,0,1)\rangle} \\
   &=\sum_{a_1=0}^\infty
    \sum_{a_2=0}^\infty \sum_{a_3=0}^\infty
    x_1^{a_1}(x_1x_2)^{a_2}(x_1x_3)^{a_3} \\
    &=\dfrac{1}{(1-x_1x_3)(1-x_1x_2)(1-x_1)}.
\end{align*}
Similarly,
\[
     \sum_{K_2 \cap \bZ^3} e^{-\langle x,u\rangle} =
     \dfrac{1}{(1-x_1x_2)(1-x_1x_3)(1-x_1x_2x_3)}
\]
and
\[
 \sum_{K_1\cap K_2 \cap \bZ^3} e^{-\langle c, x\rangle} =
 \dfrac{1}{(1-x_1x_2)(1-x_1x_3)}.
\]
One then has that
\begin{multline*}
\sum_{K \cap \bZ^n} e^{-\langle x, u\rangle} =
     \dfrac{1}{(1-x_1x_3)(1-x_1x_2)(1-x_1)} +
     \dfrac{1}{(1-x_1x_2)(1-x_1x_3)(1-x_1x_2x_3)} \\ -
     \dfrac{1}{(1-x_1x_2)(1-x_1x_3)} \\ = 
     \dfrac{1-x_1^2x_2x_3}{(1-x_1)(1-x_1x_2)(1-x_1x_3)(1-x_1x_2x_3)},
 \end{multline*}
matching the Hilbert series computation.

\subsection{Complete Intersections}\label{subsec:completeintersections}

Part of the goal of Chapters~\ref{ch:rootcone} and~\ref{ch:weightcone} will be
to identify signed posets for which the computation of the Hilbert series of
(one or the other of) the relevant rings is particularly
straightforward, namely when the rings are complete intersections.

\begin{defn}\label{def:ci}
Suppose $R$ is a ring. A sequence of elements $\theta_1,\ldots,\theta_k \in R$
is said to be a \emph{regular sequence} if $\theta_{i+1}$ is a non-zero divisor
in $R/(\theta_1,\ldots,\theta_i)$ for $i = 0,\ldots,k-1$. A quotient ring $R/I$ is said to
be a \emph{complete intersection} if $I$ is generated by a regular sequence.
\end{defn}

The following proposition simplifies the computation of the Hilbert series of a
complete intersection.

\begin{prop}\label{prop:hilbseriesprincipalidea}
    Suppose $R$ is a graded ring and $\theta$ is a non-zero divisor in $R$ and
    $\deg \theta \ne 0$. Then
    \[
        \Hilb(R/(\theta), \bm{x}) = (1-\bm{x}^{\deg
            \theta})\Hilb(R,\bm{x}).
    \]
\end{prop}
Iterating this relation gives the following.

\begin{cor}\label{cor:cihilbertseries}
    Suppose $R$ is a ring and $(\theta_1,\ldots,\theta_k)$ is a regular
    sequence. Then
    \[
        \Hilb(R/(\theta_1,\ldots,\theta_k),\bm{x}) =
        \Hilb(R,\bm{x})\prod_{i=1}^k (1-\bm{x}^{\deg\theta_i}).
    \]
\end{cor}

Boussicault, \feray, Lascoux and Reiner explain in Section 2.4
of~\cite{BoussicaultFerayLascouxReiner2012} how the valuations $s(K;c)$ from
Section~\ref{subsec:rationalfunctionscones} and $\sigma(K;c)$ from
Section~\ref{subsec:semigrouprationalfunction} are connected via a residue
operation. This will be particularly relevant in the case where the semigroup
ring is a complete intersection as a result of the following proposition.

\begin{prop}\label{prop:sandsigmaci}
Suppose $L$ is a lattice and let $K$ be a pointed $L$-rational cone for which
$R=k[K \cap L]$ is a complete intersection with
\[
R\cong S/I = k[U_1,\ldots,U_k] / (\theta_1,\ldots,\theta_d),
\]
where $\theta_1,\ldots,\theta_d$ are $L$-homogeneous elements of degrees
$\delta_1,\ldots,\delta_d$ forming a regular sequence. Then
\[
    \Hilb(R;\bm{X}) = \sigma(K;\bm{X}) = \dfrac{\prod_{i = 1}^d
        (1-\bm{X}^{\delta_i})}{\prod_{j=1}^d (1-\bm{X}^{u_j})},
\]
and if $K$ is full-dimensional,
\[
s(K;x) = \dfrac{\prod\langle x,\delta_j \rangle}{\prod \langle x,u_j\rangle},
\]
where $\deg U_j = u_j$.
\end{prop}

Section~\ref{sec:rationalfunctions} will show that the two rational functions
$\Psi$ and $\Phi$ on a
signed poset can be understood as the valuation $s(-;x)$ on certain cones.
Proposition~\ref{prop:sandsigmaci} enables one to compute these rational
functions in some cases by establishing that a certain ring is a
complete intersection and finding a regular sequence generating the toric
ideal.

As an example, consider the semigroup generated by 
\[
(1,0,0),(0,0,-1),(1,1,0), (1,-1,-1), (1,1,1).
\]
This will be one of the semigroups
considered in Chapter~\ref{ch:weightcone}. Let $K$ be the cone spanned by the
semigroup generators and $S= k[x_1,x_2,x_3,x_4,x_5]$ with $\deg x_1 = (1,0,0)$,
$\deg x_2 = (0,0,-1)$, $\deg x_3 = (1,1,0)$, $\deg x_4=(1,-1,-1)$ and $\deg
x_5=(1,1,1)$. The toric ideal of the semigroup is, courtesy of Macaulay2, $I = (x_2x_5-x_3,x_4x_5-x_1^2)$,
and one can check that the generators of the toric ideal form a regular
sequence, so the semigroup ring $R \cong S/I$ is a complete intersection.

Then, from Corollary~\ref{cor:cihilbertseries} (or the first half of
Proposition~\ref{prop:sandsigmaci}, which is obtained by repeatedly applying
the corollary), one has
\[
\Hilb(R,\bm{x}) =
\dfrac{(1-x_1x_2)(1-x_1^2)}{(1-x_1)(1-x_3^{-1})(1-x_1x_2)(1-x_1x_2^{-1}x_3^{-1})(1-x_1x_2x_3)}
\]
and
\[
s(K,x) = \dfrac{(x_1+x_2)(2x_1)}{x_1(-x_3)(x_1+x_2)(x_1-x_2-x_3)(x_1+x_2+x_3)}.
\]

\section{Root Systems}\label{sec:rootsystemdefs}

To extend the notion of posets, one must first recall some definitions
regarding root systems.

\begin{defn}\label{def:rootsystem}
	A \emph{(crystallographic) root system} is a set $\Phi \subset \bR^n$ such that
	\begin{enumerate}[label={(\alph*)}]
		\item $\Phi$ spans $\bR^n$;
		\item if $\alpha \in \Phi$, then $-\alpha \in \Phi$ and $\pm \alpha$ are the only multiples of $\alpha$ in $\Phi$;
		\item $\Phi$ is closed under the reflection $\sigma_\alpha$ across the hyperplane perpendicular to $\alpha$ for each $\alpha \in \Phi$;
		\item for all $\alpha,\beta \in \Phi$, one has
			\[
				2\dfrac{\langle \alpha,\beta\rangle}{\langle \alpha,\alpha\rangle} \in \bZ,
			\]
			where $\langle -,-\rangle$ is the standard inner product on $\bR^n$.
	\end{enumerate}
\end{defn}

Condition (d) is known as the \emph{crystallographic condition}. This condition gives the existence of root and weight lattices, which will be of use later.

\begin{defn}\label{def:simpleposnegroots}
		A subset $\Delta$ of a root system $\Phi$ is a choice of \emph{simple
		roots} if $\Delta$ spans $\bR^n$ and partitions the root system into
		those roots lying in their positive integer span---the \emph{positive
		roots}, denoted $\Phi^+$---and those lying in their negative integer
		span---the \emph{negative roots}, $\Phi^-$.
\end{defn}


The discussion of signed posets will focus on two root systems, $\Bn$ and
$\Cn$. 
The roots are
\begin{itemize}
    \item $\{\pm e_i \pm e_j \colon i,j \in [n]\} \cup \{\pm e_i \colon i \in
[n]\}$ for $\Bn$, and
	\item $\{\pm e_i \pm e_j \colon i,j \in [n]\}\cup\{\pm 2e_i \colon i \in
[n]\}$ for $\Cn$.
\end{itemize}
The simple roots will be taken to be
\begin{itemize}
    \item $\{+e_i-e_{i+1} \colon i = 1,\ldots,n-1\}\cup\{+e_n\}$ for $\Bn$, and
    \item $\{+e_i-e_{i+1} \colon i = 1,\ldots,n-1\}\cup\{+2e_n\}$ for $\Cn$.
\end{itemize}
This choice of simple roots gives the following positive roots.
\begin{itemize}
    \item $\{+e_i + e_j \colon i,j \in [n]\} \cup \{+e_i -e _j \colon i < j \}
\cup \{+e_i \colon i \in [n]\}$ for $\Bn$, and
\item $\{+e_i + e_j \colon i,j \in [n]\} \cup \{+e_i-e_j \colon i < j \} \cup
\{+2e_i \colon i \in [n]\}$ for $\Cn$.
\end{itemize}

Subsequent chapters will parallel work of \citet*{BoussicaultFerayLascouxReiner2012} and \citet{FerayReiner2012} addressing the type A case. The $\An$ roots
are 
\[
    \{e_i-e_j \colon i,j \in [n], i\ne j\},
\] and the choice of simple roots
used is $\{e_i-e_{i+1}\colon i=1,\ldots,n-1\}$, giving $\{e_i-e_j \colon i, j
\in [n], i < j\}$ as the positive roots.

\begin{defn}\label{def:weights}
Given a root system $\Phi$, the \emph{(integral) weights} are the $\mu \in
\bR^n$ such that, for each $\alpha \in \Phi$,
\[
   2 \dfrac{\langle \mu,\alpha\rangle}{\langle \alpha,\alpha \rangle} \in \bZ.
\]
If one has made a choice of simple roots, say $\{\alpha_1,\ldots,\alpha_n\}$,
there is then a distinguished set of weights, the \emph{fundamental dominant
    weights} $\mu_1,\ldots,\mu_n$ uniquely defined by the conditions
\[
    2\dfrac{\langle \mu_i, \alpha_j\rangle}{\langle \alpha_j,\alpha_j\rangle} =
    \delta_{ij},
\]
where $\delta_{ij}$ is the Kronecker delta.
\end{defn}

For $\Bn$, the weights are the elements of $\bZ^n$ plus $\left(\pm
    \frac{m}{2},\pm\frac{m}{2},\ldots,\pm\frac{m}{2}\right)$ for $m \in \bZ$.
For the choice of simple roots made above, the fundamental dominant weights are
\[
(1,0,\ldots,0),(1,1,0,\ldots,0),\ldots,(1,\ldots,1,0),\left(\frac{1}{2},\ldots,\frac{1}{2}\right).
\]
For $\Cn$, the weights are the elements of $\bZ^n$ and, with the choice of
simple roots made above, the fundamental dominant weights are
$(1,0,\ldots,0)$,$(1,1,0,\ldots,0),\ldots,(1,\ldots,1,0)$,$(1,\ldots,1)$.

\begin{defn}\label{def:dualrootsystem}
Suppose $\Phi$ is a root system. Its \emph{dual root system}, $\Phi^\dual$, is
the root system
whose roots are
\[
   \alpha^\dual = \dfrac{2 \alpha}{\langle \alpha,\alpha \rangle}
\]
for $\alpha \in \Phi$. The roots
of $\Phi^\dual$ are the \emph{coroots} of $\Phi$ and the weights of
$\Phi^\smvee$ are the
\emph{coweights} of $\Phi$. 
\end{defn}

One always has that $\Phi^{\dual\dual} = \Phi$. 
Note that $\Bn$ and $\Cn$ are dual root systems. 
The roots, weights and coweights form lattices, called the \emph{root},
\emph{weight} and \emph{coweight} lattices. As an example, consider the root
and weight/coweight lattices for $\Phi_{B_2}$ and $\Phi_{C_2}$ in
Figure~\ref{fig:rootweightlattices}.

\begin{figure}[htbp]
    \begin{center}
		\begin{tabular}{cc}
    \begin{subfigure}[b]{.4\linewidth}
        \begin{center}
        \begin{tikzpicture}[scale=.5]
             \draw[step=1.0,black] (-3,-3) grid (3,3);
             \foreach \x in {1,0,-1,2,3,-2,-3}
             {
             \foreach \y in {1,0,-1,2,3,-2,-3}
             {
                 \node[vertex] at (\x,\y) {};
             }
         }
         \end{tikzpicture}
     \end{center}
         \caption{The $\Phi_{B_2}$ root lattice}
    \end{subfigure} &
    \begin{subfigure}[b]{.4\linewidth}
        \begin{center}
            \begin{tikzpicture}[scale=.5]
                 \draw[step=1.0,black] (-3,-3) grid (3,3);
                 \foreach \x in {1,-1}
                 {
                     \foreach \y in {1,-1}
                     {
                         \node[vertex] at (\x,\y) {};
                         \node[vertex] at (\x+2,\y) {};
                         \node[vertex] at (\x,\y+2) {};
                         \node[vertex] at (\x+2,\y+2) {};
                         \node[vertex] at (\x-2,\y) {};
                         \node[vertex] at (\x,\y-2) {};
                         \node[vertex] at (\x-2,\y-2) {};
                     }
                 }
                 \foreach \x in {0,2,-2}
                 {
                     \foreach \y in {0,2,-2}
                     {
                         \node[vertex] at (\x,\y) {};
                     }
                 }
                 \node[vertex] at (-3,3) {};
                 \node[vertex] at (3,-3) {};
            \end{tikzpicture}
        \end{center}
            \caption{The $\Phi_{C_2}$ root lattice}        
    \end{subfigure}\\
    \begin{subfigure}[b]{.4\textwidth}
        \begin{center}
            \begin{tikzpicture}[scale=.5]
                \draw[step=1.0,black] (-3,-3) grid (3,3);
                \foreach \x in {0,1,2,3}
                {
                    \foreach \y in {0,1,2,3}
                    {
                        \node[vertex] at (\x,\y) {};
                        \node[vertex] at (-\x,\y) {};
                        \node[vertex] at (\x,-\y) {};
                        \node[vertex] at (-\x,-\y) {};
                    }
                }
            \end{tikzpicture}
        \end{center}
        \caption{The $\Phi_{B_2}$ coweight/$\Phi_{C_2}$ weight lattice}
    \end{subfigure} &
    \begin{subfigure}[b]{.4\textwidth}
        \begin{center}
            \begin{tikzpicture}[scale=.5]
                \draw[step=1.0,black] (-3,-3) grid (3,3);
                \foreach \x in {0,1,2,3}
                {
                    \foreach \y in {0,1,2,3}
                    {
                        \node[vertex] at (\x,\y) {};
                        \node[vertex] at (-\x,\y) {};
                        \node[vertex] at (\x,-\y) {};
                        \node[vertex] at (-\x,-\y) {};
                    }
                }
                \foreach \z in {1,3,5}
                {
                    \foreach \w in {1,3,5}
                    {
                    \node[vertex] at (.5*\z,.5*\w) {};
                    \node[vertex] at (-.5*\z,.5*\w) {};
                    \node[vertex] at (.5*\z,-.5*\w) {};
                    \node[vertex] at (-.5*\z,-.5*\w){};
                }
                }
            \end{tikzpicture}
        \end{center}
        \caption{The $\Phi_{B_2}$ weight/$\Phi_{C_2}$ coweight lattice}
    \end{subfigure}
	\end{tabular}
\end{center}
    \caption{Root and Weight Lattices}
    \label{fig:rootweightlattices}
\end{figure}
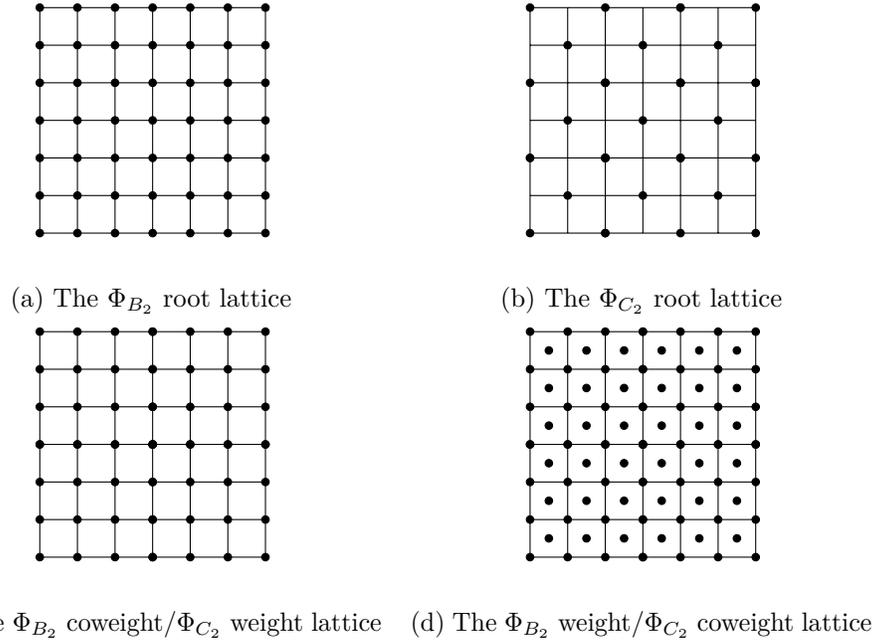

\subsection{The Weyl Group and Signed Permutations}\label{subsec:weylgroup}

\begin{defn}\label{def:weylgroup}
The reflections $\sigma_\alpha$ across the hyperplanes $H_\alpha$ perpendicular to $\alpha
\in \Phi$ form a group called the \emph{Weyl group} and denoted by $W$.
\end{defn}

If one
removes the hyperplanes $H_\alpha$, one divides $\bR^n$ into connected
components which are open simplicial cones known as
\emph{Weyl chambers}. The Weyl group $W$ acts simply transitively on the Weyl
chambers.

The Weyl group of $\An$ is the familiar symmetric group $S_n$. Since $\Bn$ and $\Cn$ are dual root systems they share the same reflections and
thus the same Weyl group. This is the group of \emph{signed permutations},
known as the \emph{hyperoctahedral group}. 

\begin{defn}\label{def:signedperm}
A signed permutation of $[n]$ is a permutation $w$ of $\{\pm 1,\ldots,\pm
n\}$ such that $w(-i) = -w(i)$. 
\end{defn}
Consequently, signed permutations may be
written in a two-line notation specifying the images of $\{1,\ldots,n\}$. For
example,
\begin{equation}\label{eq:signedpermex}
    w = \begin{pmatrix} 1 & 2 & 3 & 4 \\
        2 & -3 & 4 & -1 
    \end{pmatrix}
\end{equation}
is the signed permutation where $w(1)=2$, $w(2) = -3$, $w(3)=4$, $w(4)=-1$ and
$w(-i)=-w(i)$ for $i = 1,\ldots,4$.

The signed permutation analogues of a well-known permutation statistic will prove useful in
Chapter~\ref{ch:weightcone}. Name our choice of simple roots as $\alpha_i = e_i-e_{i+1}$ for $i =
1,\ldots,n-1$ and $\alpha_n = e_n$.
\begin{defn}\label{def:signedmaj}
A signed permutation $w$ is said to have a
\emph{descent} at $i$ if $w(\alpha_i) \in \Phi^-$. The set of descents of $w$
is denoted $\Des(w)$. The \emph{major index} of $w$ is
\[
    \maj(w) = \sum_{i \in \Des(w)} i
\]
\end{defn}
For example, the signed permutation $w$ shown in \eqref{eq:signedpermex} has descents at $2$
and $4$, so $\maj(w) = 6$.

\chapter{Signed Posets}\label{ch:signedposets}

This chapter will discuss background on signed posets and introduce the root
cone and weight cone.

Signed posets were introduced by Reiner in~\cite{Reiner1990}
and~\cite{Reiner1993}. Before defining signed posets, it is useful to define
one piece of notation which will reappear later. 

\begin{defn}
		Suppose $\Phi$ is a root
system and $S \subset \Phi$. Then the \emph{positive linear closure},
$\overline{S}^{PLC}$, is the set of $\alpha
\in \Phi$ which are non-negative linear combinations of elements of $S$, i.e.\
$\PLC{S}=\bR_+S\cap \Phi$.
\end{defn}

\begin{defn}\label{def:signedposet}
If $\Phi$ is a root system, a \emph{$\Phi$-poset} is a subset $P \subset \Phi$ such that
\begin{enumerate}[label=(\alph{*})]
    \item if $\alpha \in P$, then $-\alpha \notin P$
    \item $P = \PLC{P}$.
\end{enumerate}
\end{defn}

A poset $P$ can be understood as the $\An$-poset $\{e_i-e_j \colon i <_P j\}$.
This view of posets reveals condition (a) in the definition as the analogue of
antisymmetry and condition (b) as the analogue of transitivity.

The $\Phi$-posets are the \emph{parsets} of~\cite{Reiner1990}. Since $\Bn$ and
$\Cn$ are dual root systems, $P \subset \Bn$ is a $\Bn$-poset if and only if $\Pc =
\{\alpha^\vee \colon \alpha \in P\} \subset \Cn$ is a $\Cn$-poset. Signed
posets were defined to be the $\Bn$-posets, but in light of this duality, both
$\Bn$- and $\Cn$-posets will be called \emph{signed posets}, with $\Bn$ or
$\Cn$ being specified when necessary.

One can check that $P = \{ +e_1-e_2,+e_1-e_3,+e_2+e_3,+e_1+e_3,+e_1\} \subset
\Phi_{B_3}$ is a signed poset. In this case, $\Pc = \{
+e_1-e_2,+e_1-e_3,+e_2+e_3,+e_1+e_3,+2e_1\} \subset \Phi_{C_3}$. On the other hand, $P \smallsetminus \{+e_1\}$
is \emph{not} a signed poset as it is not closed under positive linear
combinations remaining in $\Phi_{B_3}$ since $\frac{1}{2}((e_1-e_3)+(e_1+e_3)) =
e_1$.

\begin{defn}\label{def:isomorphic}
Two signed posets $P$ and $P'$ are \emph{isomorphic} if there is a signed permutation $w$ such that $wP=P'$.
\end{defn}

Note that, strictly speaking, the definition of $\Phi$-poset allows a poset $P
\subset \Phi \subset \bR^m$ where $P$ is supported on $\{e_i\colon i \in A\}$
for some $A \subsetneq [m]$. In this case, $P$ is isomorphic to some $P'
\subset \Phi' \subset \bR^n$ for $|A|=n$.

\section{Representing Signed Posets}\label{sec:representingposets}

There are two representations of signed posets, one as an oriented signed graph
and the other as a poset on $\pm[n]$, that prove useful in different contexts. 
First, though, one needs to define oriented signed graphs.

\begin{defn}\label{def:signedgraph}
A
\emph{signed graph} $\Sigma$ is a pair $(\Gamma,\sigma)$ where $\Gamma$ is a
graph with vertex set $V$ and edge set $E$ and $\sigma$ is a map $\sigma \colon
E \to \{\pm\}$, assigning a sign to each edge. 
\end{defn}

\begin{defn}\label{def:orientedsignedgraph}
An \emph{oriented signed 
    graph} is a signed graph $\Sigma$ together with a \emph{bidirection},
$\tau$, assigning signs to the incidences of $\Gamma$ in such a way as to be
compatible with $\sigma$, i.e.\ $\tau \colon I(\Gamma) \to \{\pm\}$ such that
\[
    \sigma(e) = -\tau(v,e)\tau(w,e)
\]
when $e$ is an edge between vertices $v$ and $w$ ($v$ and $w$ need not be
distinct).
\end{defn}

As an example, consider the signed graph $\Sigma$ and an orientation  of $\Sigma$
by the bidirection $\tau$ in Figure~\ref{fig:signedgraph}. Notice that the bidirected edges of
$\Sigma$ correspond to elements of $\Cn$: the edge $(u,v)$ corresponds to
$\tau(u)e_u+\tau(v)e_v$.

\begin{figure}[htbp]
    \begin{center}
    \begin{subfigure}[b]{.4\textwidth}
        \begin{center}
            \begin{tikzpicture}[every loop/.style={in=135,out=90,looseness=3}]
                \matrix (m) [matrix of math nodes,row sep=3em, column sep=4em,
                nodes={circle,draw}]{
                    |(1)| 1 & |(2)| 2 \\
                    & |(3)| 3 \\
                }; 
                \draw (1)--(2) node[midway,above] {$+$};
                \draw (1.325)--(3.125) node[midway,right] {$-$};
                \draw (1.305)--(3.145) node[midway,left] {$+$};
                \draw (2)--(3) node[midway,right] {$-$};
                \path (1) edge [loop above] node{$+$} ();
            \end{tikzpicture}
            \subcaption{The signed graph $\Sigma$}
        \end{center}
    \end{subfigure}
    \begin{subfigure}[b]{.4\textwidth}
        \begin{center}
            \begin{tikzpicture}[every loop/.style={in=135,out=90,looseness=3}]
                \matrix (m) [matrix of math nodes,row sep=3em, column sep=4em,
                nodes={circle,draw}]{
                    |(1)| 1 & |(2)| 2 \\
                    & |(3)| 3 \\
                }; 
                 \draw (1)--(2) node[very near start,above] {$+$} node[very
                 near end,above] {$-$};
                \draw (1.325)--(3.125) node[very near start,right] {$+$}  
                node[very near end,above] {$+$};
                \draw (1.305)--(3.145) node[very near start,left] {$+$}
                node[very near end,left] {$-$};
                \draw (2)--(3) node[very near start, right] {$+$}
                node[very near end,right] {$+$};
                \path (1) edge [loop above] node[very near start,right]{$+$}
				node[very near end,left]{$+$} ();
            \end{tikzpicture}
            \subcaption{$\Sigma$ oriented by $\tau$}\label{fig:bidirected}
        \end{center}
    \end{subfigure}
\end{center}
    \caption{A signed graph and a bidirection}
    \label{fig:signedgraph}
\end{figure}
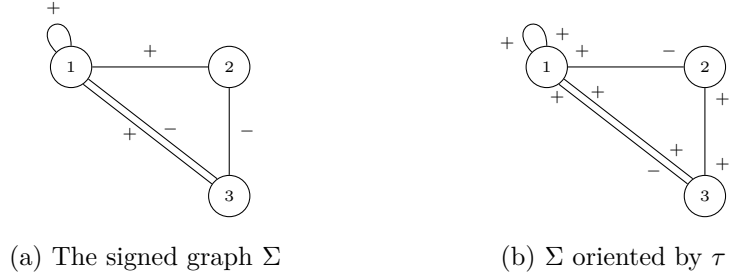

To construct the Hasse diagram of a signed poset, one first constructs an
oriented signed graph as follows. If $P \subset \Bn$ (resp.\ $P \subset \Cn$),
the vertices of the graph are $[n]$. The bidirected edges are as in
Table~\ref{table:bidirectededges}.
\begin{table}[htbp]
\begin{center}
\begin{tabular}{lr@{$\;\in\;$}l}
    \multicolumn{1}{c}{edge} & \multicolumn{2}{c}{poset element}\\
    \begin{tikzpicture}
        \node[circle,draw] (1) at (0,0) {$i$};
        \node[circle,draw] (2) at (2,0) {$j$};
        \draw (1)--(2) node[very near start,above] {$+$} node[very near
        end,above] {$-$};
    \end{tikzpicture} &
    $e_i-e_j$ & $P$ \\
    \begin{tikzpicture}
 \node[circle,draw] (1) at (0,0) {$i$};
        \node[circle,draw] (2) at (2,0) {$j$};
        \draw (1)--(2) node[very near start,above] {$+$} node[very near
        end,above] {$+$};
    \end{tikzpicture} & $e_i + e_j$&$P$ \\
    \begin{tikzpicture}
 \node[circle,draw] (1) at (0,0) {$i$};
        \node[circle,draw] (2) at (2,0) {$j$};
        \draw (1)--(2) node[very near start,above] {$-$} node[very near
        end,above] {$-$};
    \end{tikzpicture} & $-e_i-e_j$&$P$ \\
    \begin{tikzpicture}[every loop/.style={in=135,out=90,looseness=3}]
        \node[circle,draw] (1) at (0,0) {$i$};
        \path (1) edge [loop above] node[very near start, right]{$+$}
		node[very near end, left]{$+$} ();
    \end{tikzpicture} & $e_i\in P$, $2e_i$ & $\Pc$ \\
    \begin{tikzpicture}[every loop/.style={in=135,out=90,looseness=3}]
        \node[circle,draw] (1) at (0,0) {$i$};
        \path (1) edge [loop above] node[very near start,right]{$-$} node[very
		near end,left] {$-$} ();
    \end{tikzpicture} & $-e_i\in P$, $-2e_i$& $\Pc$\\
\end{tabular}
\end{center}
\caption{Association between bidirected edges and elements of a signed poset}
\label{table:bidirectededges}
\end{table}
However, because signed posets are closed under positive linear combinations,
the presence of some elements is implied by others. 

\begin{defn}\label{def:hassediagram}
		Suppose $P \subset \Bn$ (\resp $\Pc \subset \Cn$) is a signed poset.
        Let $\Gamma$ be the oriented signed graph whose edges are obtained from
        Table~\ref{table:bidirectededges}. Let $\Sigma_P$ be the oriented
        signed graph obtained by removing the implied edges from $\Gamma$. Then
        $\Sigma_P$ is the \emph{Hasse diagram} of $P$.
\end{defn}
It is not immediately obvious that the Hasse diagram should be well-defined.
That it is well-defined is explained by~\citet[p.\ 329]{Reiner1993}.


Consider the oriented signed graph in Figure~\ref{fig:bidirected}. It
corresponds to the signed poset $\{+e_1-e_2,+e_1-e_3,+e_2+e_3,+e_1+e_3,+e_1\}$.
The Hasse diagram is shown in Figure~\ref{fig:hasseexample}.
\begin{figure}
    \begin{center}
        \begin{tikzpicture}
            \matrix (m) [matrix of math nodes,row sep=3em, column sep=4em,
            nodes={circle,draw}]{
                |(1)| 1 & |(2)| 2 \\
                & |(3)| 3 \\
            }; 
             \draw (1)--(2) node[very near start,above] {$+$} node[very
             near end,above] {$-$};
            \draw (1)--(3) node[very near start,left] {$+$}
            node[very near end,left] {$-$};
            \draw (2)--(3) node[very near start, right] {$+$}
            node[very near end,right] {$+$};
        \end{tikzpicture}
    \end{center}
    \caption{The Hasse diagram of $\{+e_1-e_2,+e_1-e_3,+e_2+e_3,+e_1+e_3,+e_1\}$}
    \label{fig:hasseexample}
\end{figure}
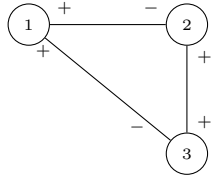

The other representation of a signed poset was introduced by \citet{Fischer1993} and then again independently by \citet*{AndoFujishigeNemoto1996}. 

\begin{defn}\label{def:fischerposets}
Suppose $P \subset \Bn$ and $\Pc \subset \Cn$ are a pair of dual signed posets.
	Define a poset $\Gbhat(P)$ on
$[-n,n] = \{-n,\ldots,1,0,1,\ldots,n\}$ by taking the transitive closure of the
relations determined by Table~\ref{table:gbhat}.
Similarly, define a poset $\Gchat(\Pc)$ on $\pm [n] = \{\pm 1 ,\ldots,\pm n\}$ by
taking the transitive closure of the relations determined by
Table~\ref{table:gchat}.
\begin{table}[htbp]
\[
    \begin{array}{r@{\;<\;}lcr@{\;<\;}lcl}
			i &j & \text{and} & -j & -i & \text{for} & +e_i -e_j \in P \\
        i & -j  & \text{and} &j & -i & \text{for} & +e_i+e_j \in P \\
        -i & j & \text{and} &-j & i & \text{for} & -e_i-e_j \in P \\
        i & 0 & \text{and} & 0 & -i & \text{for} & +e_i \in P \\
        -i & 0 & \text{and} &0 & i & \text{for} & -e_i \in P \\
    \end{array}
\]
\caption{Relations defining $\Gbhat(P)$}
\label{table:gbhat}
\end{table}
\begin{table}[htbp]
\[
    \begin{array}{r@{\;<\;}lcr@{\;<\;}lcl}
			i &j & \text{and} & -j & -i & \text{for} & +e_i -e_j \in P \\
        i & -j  & \text{and} &j & -i & \text{for} & +e_i+e_j \in P \\
        -i & j & \text{and} &-j & i & \text{for} & -e_i-e_j \in P \\
        \pm i &  \mp i &\multicolumn{4}{r}{\text{for}} & \pm 2e_i \in P \\
    \end{array}
\]
    \caption{Relations defining $\Gchat(\Pc)$}
\label{table:gchat}
\end{table}
\end{defn}

Understand $\Ghat(P)$ to mean either $\Gbhat(P)$ or $\Gchat(\Pc)$ as
appropriate. In Chapter~\ref{ch:weightcone} it will be convenient to revert to Fischer's original definition and use
$\Gchat(\Pc)$ to represent a $P \subset \Bn$.
Proposition~\ref{prop:isotropicideals} shows that this
abuse is acceptable. We will call $\Ghat(P)$ the \emph{Fischer poset} of the
signed poset. The Fischer posets for $P^\star
=\{+e_1-e_2,+e_1-e_3,+e_2+e_3,+e_1+e_3,+e_1\}$ are shown in
Figure~\ref{fig:fischerex}. (This poset is being named $P^\star$ since it will
feature as an example in this section and the next.) Both Fischer and Ando, Fujishige and Nemoto defined
$\Gchat(P)$ as being associated to $P \subset \Bn$. However, adding the $0$ vertex in type B and
defining $\Gbhat(P)$ and $\Gchat(P)$ separately
simplifies Chapter~\ref{ch:rootcone}.

\begin{figure}
    \begin{center}
    \begin{subfigure}[b]{.4\textwidth}
        \begin{center}
            \begin{tikzpicture}
                \node[vertex] (1) at (0,0) [label=right:$1$] {};
                \node[vertex] (2) at (-1,1) [label=left:$2$] {};
                \node[vertex] (3) at (1,1) [label=right:$3$] {};
                \node[vertex] (-3) at (-1,2) [label=left:$-3$] {};
                \node[vertex] (-2) at (1,2) [label=right:$-2$] {};
                \node[vertex] (0) at (0,1.5) [label=right:$0$] {};
                \node[vertex] (-1) at (0,3) [label=right:$-1$] {};
                \draw (1)--(2)--(-3)--(-1)--(0)--(1)--(3)--(-2)--(-1);
            \end{tikzpicture}
        \end{center}
        \caption{$\Gbhat(P^\star)$}
    \end{subfigure}
    \begin{subfigure}[b]{.4\textwidth}
        \begin{center}
            \begin{tikzpicture}
                \node[vertex] (1) at (0,0) [label=right:$1$] {};
                \node[vertex] (2) at (-1,1) [label=left:$2$] {};
                \node[vertex] (3) at (1,1) [label=right:$3$] {};
                \node[vertex] (-3) at (-1,2) [label=left:$-3$] {};
                \node[vertex] (-2) at (1,2) [label=right:$-2$] {};
                \node[vertex] (-1) at (0,3) [label=right:$-1$] {};
                \draw (1)--(2)--(-3)--(-1)--(-2)--(3)--(1);
            \end{tikzpicture}
        \end{center}
        \caption{$\Gchat(P^{\star\smvee})$}
    \end{subfigure}
    \caption{$\Gbhat(P)$ and $\Gchat(P)$ for
        $P^\star=\{+e_1-e_2,+e_1-e_3,+e_2+e_3,+e_1+e_3,+e_1\}$}
    \label{fig:fischerex}
\end{center}
\end{figure}

	Both $\Gbhat(P)$ and $\Gchat(\Pc)$ are equipped with involutions, denoted
	$\iota$, sending $i$ to $-i$ and the edge $i \to j$ to $-j \to -i$, making
    $\iota$
    a poset anti-automorphism. In
	$\Gbhat(P)$, this involution fixes $0$, and in $\Gchat(\Pc)$ it fixes edges
	of the form $i \to -i$. That $\Gbhat(P)$ and $\Gchat(\Pc)$
	are equipped with this involution means that they connect to the theory of
	coverings of signed graphs. This will be discussed in
	Section~\ref{sec:signedgraphtoricideals}. Somewhat confusingly, the
	oriented signed graphs which they cover are \emph{not} necessarily the Hasse diagram
	just described in Definition~\ref{def:hassediagram}. The Hasse diagram will be useful in
	Section~\ref{sec:dualcones} and then dispensed with until
    Chapter~\ref{ch:unfinished}, as the Fischer poset
	will prove to be the more convenient notion, lending itself to analogies to 
	the type A case of~\cite{BoussicaultFerayLascouxReiner2012}
	and~\cite{FerayReiner2012}. 

	Note that when $P$ does not contain an element of the form $\pm e_i$, the
    element $0$ is not comparable to any other elements of $\Ghat(P)$.

The following two propositions characterizing $\Gbhat(P)$ and $\Gchat(\Pc)$ are immediate consequences of the definitions.

\begin{prop}\label{prop:gbhatcharacterisation}
A poset on $[-n,n]$ is $\Gbhat(P)$ of a $\Bn$-poset $P$ if and only if the
following two conditions hold:
\begin{itemize}
    \item for all $i,j \in \{-n,\ldots,n\}$, if $i < j$ then $-j < -i$.
	\item for all $i \in \{-n,\ldots,n\}$, if $i < -i$ then $i < 0 < -i$.
\end{itemize}
\end{prop}

\begin{prop}\label{prop:gchatcharacterisation}
A poset on $\pm[n]$ is $\Gchat(\Pc)$ for a $\Cn$-poset $\Pc$ if and only if the
following two conditions hold:
\begin{itemize}
    \item if $i < j$ then $-j < -i$ for all $i,j \in \pm [n]$ and
    \item if $i < -i$ and $j < -j$ then $i < -j$ and $j < -i$ for all $i,j \in
        \pm[n]$.
\end{itemize}
\end{prop}

\begin{defn}
A signed poset $P'$ is an \emph{induced subposet} of another signed poset $P$
if $P' \subset P$ and there is an $A \subsetneq [n]$ such that $P' \subset
P \cap \text{span}_\bR\{e_i \colon i \in A\}$.
\end{defn}
Equally, $P'$ is an induced subposet of $P$ if $\Ghat(P')$ is an induced subposet
of $\Ghat(P)$.

\section{Ideals and $P$-partitions in Signed Posets}\label{sec:ideals}

Recall that an order ideal of a poset $P$ is a subset $I \subset P$ such that if $x \in I$ and $y < x$, then $y \in I$.
On the other hand, when one views a poset $P$ on $\{1,\ldots,n\}$ as a collection of roots in $\An$, the
characteristic vectors of the ideals are precisely those $J \in \{1,0\}^n$ such
that $\langle J,\alpha\rangle \geq 0$ for all $\alpha \in P$. This view
motivates Boussicault, \feray, Lascoux and Reiner's definition of the weight
cone in~\cite{BoussicaultFerayLascouxReiner2012}.

In~\cite{Reiner1993}, Reiner defined an \emph{ideal} of a signed poset $P
\subset \Bn$ to be a $J \in \{1,-1,0\}^n$ such that $\langle J,\alpha \rangle
\geq 0$ for all $\alpha \in P$. 
The two key observations to generalize these 
definitions to $\Phi$-posets are 
\begin{itemize}
		\item the projections of $\{1,0\}^n$ to the hyperplane $\sum_{i=1}^n x_i =0$ (i.e.\ the hyperplane of $\bR^n$ in which $\An$ lives) form the orbit of the $\An$-fundamental coweights under the action of the Weyl group (recall that $\An$ is dual to
itself, so its weights and coweights are the same), and
\item $\{1,-1,0\}^n$ is the orbit of the fundamental coweights of $\Bn$ under the action of the Weyl group.
\end{itemize}

\begin{defn}
		Suppose $P \subset \Phi$ is a $\Phi$-poset. An \emph{ideal} of
$P$ is an $f$ in the orbit of the fundamental weights under the action of the Weyl
 group such that $\langle f, \alpha \rangle \geq 0$ for all $\alpha \in P$. The
 set of ideals of $P$ is denoted $J(P)$, as for posets.
\end{defn}

As an example, consider once again the poset $P^\star$ shown in Figure~\ref{fig:fischerex}. The ideals of $P^\star$ are the elements of the poset shown in Figure~\ref{fig:jpexample}.
One will note that these are the characteristic
vectors of some of the order ideals of $\Ghat(P)$.

\begin{defn}
		Suppose $P \subset \Bn$ (\resp $P \subset \Cn$) is a signed poset. An
        order ideal, $I$, of $\Ghat(P)$ is said to be \emph{isotropic} when the
        following two conditions hold: 
		\begin{itemize}
		\item $0 \notin I$ (when $P \subset \Bn)$
		\item if $i \in I$, then $-i \notin I$ for $i \in \pm[n]$
\end{itemize}
\end{defn}

Fischer and later Ando,
Fujishige and Nemoto showed that the isotropic order ideals correspond to the
ideals of a signed poset. It turns out that it is more convenient to look
only at $\Gchat(\Pc)$ when talking about the ideals. The correspondence between
ideals of a signed poset and isotropic order ideals and the fact that examining $\Gchat(\Pc)$
suffices is encapsulated in the following proposition.  

\begin{prop}\label{prop:isotropicideals}
Suppose $P\subset \Bn$ is a signed poset and $\Pc \subset \Cn$ is the poset
consisting of the corresponding dual roots. Then
\begin{enumerate}[label=(\alph{*})]
    \item the ideals of $P$ are the characteristic vectors of the isotropic
        order ideals of $\Gbhat(P)$,
    \item the ideals of $\Pc$ are the characteristic vectors of the isotropic
        order ideals of $\Gchat(\Pc)$ with the exception that the order ideals
        of size $n$ correspond to one half times their characteristic vectors
        and
    \item the set of isotropic order ideals of $\Gbhat(P)$ is the same as the
        set of isotropic order ideals of $\Gchat(\Pc)$, and the set of
        connected isotropic order ideals of $\Gbhat(P)$ is the same as the set
        of connected isotropic order ideals of $\Gchat(\Pc)$.
\end{enumerate}
\end{prop}

In light of this view of the ideals of a signed poset, the vector of an ideal
and the corresponding (isotropic) set of elements of $\pm[n]$ will be used interchangeably. 

\begin{proof}
    \begin{enumerate}[label=(\alph{*})]
			\item Suppose $J$ is an isotropic order ideal of $\Gbhat(P)$.
                Suppose $\chi_J$ is not an ideal of $f$. Then there is an
                $\alpha \in P$ such that $\langle\chi_J , \alpha \rangle <0$.
                For ease of notation, let $\epsilon = \chi_J$. Note that for
                all $i$, $\epsilon_i \in \{0,1,-1\}$. There are five cases.
				\begin{itemize}
						\item \textbf{Suppose $\alpha = e_i-e_j$.} Then $\epsilon_i - \epsilon_j < 0$. If $\epsilon_i=0$, then $\epsilon_j =1$ so $j \in J$. However, one knows (from $\alpha$) that $i < j$ in $\Gbhat(P)$. If $\epsilon_i = -1$, then $\epsilon_j \geq 0$. Therefore, $-i \in J$. However, $-j < -i$ in $\Gbhat(P)$, so $-j \in J$, contradicting that $J$ was isotropic.
						\item \textbf{Suppose $\alpha = e_i+e_j$.} Then $\epsilon_i +\epsilon_j <0$. Without loss of generality suppose $\epsilon_i = -1$, so $-i \in J$. Therefore, since $j < -i$ in $\Gbhat(P)$, so $j \in J$, meaning $\epsilon_j=1$, a contradiction since then $\epsilon_i+\epsilon_j = 0$.
						\item \textbf{Suppose $\alpha = +e_i$.} Then $\epsilon_i <0$, so $-i \in J$. Since $i < -i$ in $\Gbhat(P)$, one has $i \in J$, contradicting that $J$ was isotropic.
						\item \textbf{Suppose $\alpha = -e_i-e_j$.} The symmetric argument to the case of $\alpha = e_i+e_j$ shows this is impossible.
						\item \textbf{Suppose $\alpha = -e_i$.} The symmetric argument
								to the case of $\alpha = e_i$ shows this is impossible.
				\end{itemize}

				Now suppose $f \in \Lcwt_B$ is an ideal. By construction, the
				set in $\Gbhat(P)$ corresponding to $f$, call it $J$, is
				isotropic. Suppose $J$ is not an order ideal. Then there is $i,j \in [n]$ and $\delta,\epsilon \in \{\pm\}$ such that
				$\delta i \lessdot \epsilon j$ and $\epsilon j \in J$ and
				$\delta i \notin J$. In particular, $j \ne 0$. Then $\delta
				e_i-\epsilon j \in P$ and $\langle \delta e_i-\epsilon e_j,f
				\rangle < 0$, contradicting that $f$ was an ideal.
        \item The argument for type B also works for type C.
        \item Suppose $J$ is an isotropic order ideal in $\Gbhat(P)$. Since $J$
            is isotropic, $0 \notin J$. Consider $J$ as a subset of
            $\Gchat(\Pc)$. Suppose it is not an order ideal. Then there is an $i \in \pm[n]$ and $j \in J$
            such that $i < j$ and $i \notin J$. Without loss of generality, one
            may assume $i, j > 0$. Then, since $i < j$, one has that $e_i - e_j
            \in \Pc$, meaning $e_i - e_j \in P$. However, since $i \notin J$,
            one must have $\langle e_i-e_j, \chi_J\rangle = -1$, contradicting
            that $J$ is an ideal of $P$.

            The argument when $J$ is an isotropic order ideal in $\Gchat(\Pc)$
            is exactly the same, except one must also account for the
            possibility that $i = 0$. However, if $i = 0$, one must have $-j <
            j$ by Proposition~\ref{prop:gbhatcharacterisation}.
    \end{enumerate}
\end{proof}

\begin{defn}\label{def:connecteideal}
An ideal of a signed poset $P \subset \Bn$ (\resp $\Pc \subset \Cn$) will be called
\emph{connected} if it corresponds to a connected isotropic order ideal in
$\Ghat(P)$. Denote the set of connected order ideals by $\Jconn(P)$.
\end{defn}
See Figure~\ref{fig:extensibilityex} for an example of a signed poset where not every ideal is connected.

In~\cite{Reiner1993}, an order was defined on the ideals of a signed poset by
extending componentwise the order $0 < 1,-1$. This order on the ideals
corresponds to ordering the corresponding isotropic order ideals by inclusion.
As in the case of posets, this order gives a meet-semilattice of ideals, denoted, like the
set of ideals, $J(P)$. Figure~\ref{fig:jpexample} shows $J(P)$ for $P 
=\{+e_1-e_2,+e_1-e_3,+e_2+e_3,+e_1+e_3,+e_1\}$.

\begin{figure}
    \begin{center}
        \begin{tikzpicture}[every node/.style={font=\normalsize},scale=2]
            \node (0) at (0,0) {$(0,0,0)$};
            \node (1) at (0,1) {$(1,0,0)$};
            \node (12) at (-1,2) {$(1,1,0)$};
            \node (13) at (1,2) {$(1,0,1)$};
            \node (12m3) at (-2,3) {$(1,1,-1)$};
            \node (123) at (0,3) {$(1,1,1)$};
            \node (1m23) at (2,3) {$(1,-1,1)$};
            \tikzstyle{every node}=[font=\small]
            \draw (0)-- node[midway,left]{$+1$} (1);
            \draw (1)-- node[midway,below left]{$+2$}(12);
            \draw (1)-- node[midway,below right]{$+3$}(13);
            \draw (12)--node[midway,below left] {$-3$}(12m3);
            \draw (12)--node[midway,above left] {$+3$} (123);
            \draw (13)--node[midway,above right]{$+2$} (123);
            \draw (13)--node[midway,below right]{$-2$}(1m23);
        \end{tikzpicture}
    \end{center}
    \caption[{$J(P^\star)$ for $P^\star
        =\{+e_1-e_2,+e_1-e_3,+e_2+e_3,+e_1+e_3,+e_1\}$}]{$J(P^\star)$ for $P^\star 
=\{+e_1-e_2,+e_1-e_3,+e_2+e_3,+e_1+e_3,+e_1\}$ with the edges annotated with
the difference between the ideals}
    \label{fig:jpexample}
\end{figure}
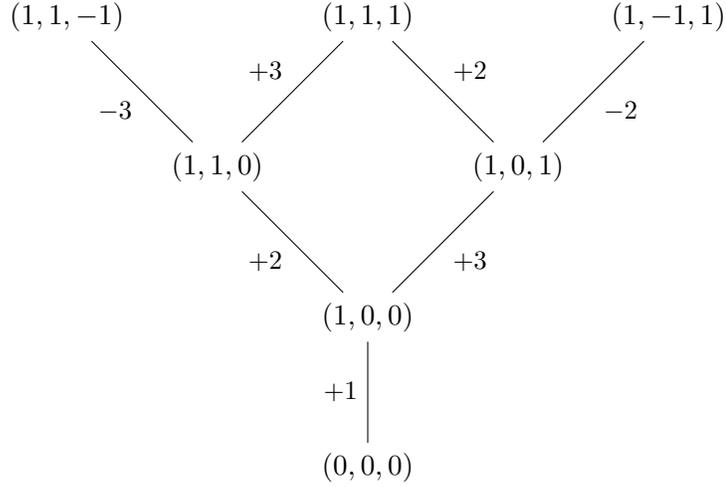

The definition of $P$-partition for a signed poset is the analogue of that for
a poset when a poset is viewed as an $\An$-poset. Recall that, from this
perspective, a $P$-partition is an $f \in \bN^n$ such that $\langle f,\alpha
\rangle \geq 0$ for all $\alpha \in P$. In~\cite{Reiner1993}, Reiner defined a
$P$-partition for a signed poset $P \subset \Bn$ to be an $f \in \bZ^n$ such
that $\langle f, \alpha \rangle \geq 0$ for all $\alpha \in P$.

One makes a similar pair of observations regarding the $P$-partitions as one did regarding ideals:
\begin{itemize}
		\item the projections of $\bN^n$ to the hyperplane $\sum_{i=0}^n x_i = 0$ in $\bR^n$ are the $\An$ coweights, and
		\item the $\bZ^n$ are the $\Bn$ coweights.
\end{itemize}

\begin{defn}\label{def:ppartition}
Suppose $P$ is a $\Phi$-poset. A $P$-partition is an $f$ in the coweight lattice of $\Phi$ such that $\langle f,\alpha \rangle \geq 0$ for all $\alpha \in P$.
The set of $P$-partitions is denoted $\cA(P)$.
\end{defn}

\section{Linear Extensions}
\begin{defn}
A \emph{linear extension} of a $\Phi$-poset $P$ is an element $w$ of the Weyl
group of $\Phi$ such that $P \subset w\Phi^+$. The set of linear extensions is
denoted $\cL(P)$.
\end{defn}

In the case of signed posets, the Weyl group consists of the signed permutations, as explained in
Section~\ref{sec:rootsystemdefs}. 
\begin{defn}
    A signed poset is said to be \emph{naturally
labelled} if the identity signed permutation is a linear extension. 
\end{defn}
Note that
since every signed poset has a linear extension (see the proof of~\cite[Theorem 3.3]{Reiner1993}), every signed poset is
isomorphic to a naturally labeled signed poset. Thus, it is not deceptive that
examples will usually be naturally labelled.

The linear extensions can be read off $J(P)$ by recording the difference between
successive ideals in the maximal chains of $J(P)$.  In the running example of
$P^\star$, these differences were noted in
Figure~\ref{fig:jpexample}. One sees that $P^\star$ is naturally labelled and the linear extensions are
\[
    \begin{pmatrix} 1 & 2& 3 \\
        1 & 2 & 3 \\
    \end{pmatrix} \quad
    \begin{pmatrix} 1 & 2 & 3 \\
        1 & 3 & 2 \\
    \end{pmatrix} \quad
    \begin{pmatrix}
        1 & 2& 3 \\
        1 & 2 & -3 \\
    \end{pmatrix} \quad
    \begin{pmatrix}
        1 & 2 & 3 \\
        1 & -2 & 3 \\
    \end{pmatrix}
\]

In~\cite{Reiner1993}, Reiner proved an analogue of Stanley's Fundamental
Theorem of $P$-partitions for $\Phi$-posets.

\begin{prop}[{\cite[Theorem 3.3]{Reiner1993}}]\label{prop:fundthmppartitions}
    Suppose $P \subset \Bn$ is a signed poset. Then
    \[
        \cA(P) = \bigsqcup_{w \in \cL(P)} \cA(w\Phi^+).
    \]
\end{prop}

\section{Embedding Posets as Signed Posets}\label{sec:embeddingtypea}

In~\cite{Reiner1993}, Reiner defined the embedding of a poset as a signed
poset. Recall that a poset $P$ on $[n]$ can be viewed as a $\An$-poset by
taking $P = \{e_i-e_j\colon i <_P j\}$. The poset $P$ is then embedded in $\Bn$
as the signed poset $P_B=P \cup \{+e_i\colon i \in [n]\} \cup\{+e_i+e_j \colon i,j
\in [n]\}$. Then $P_C=(P \cup \{+e_i\colon i \in [n]\} \cup\{+e_i+e_j \colon i,j
\in [n]\})^\smvee = P \cup \{+2e_i\colon i \in [n]\} \cup\{+e_i+e_j \colon i,j
\in [n]\}$. Call signed posets arising in this manner \emph{type A signed posets}.
The ideals of $P_B$ are precisely the ideals of $P$. Likewise, $P_B$ and $P$ share the same linear extensions (in the sense that any
linear extension of $P_B$ and of $P$ have the same action on $[n]$).

\begin{prop}\label{prop:ilessmitypea}
Suppose $P \subset \Bn$ (resp.\ $\Pc \subset \Cn$) is such that $i < -i$ in $\Ghat(P)$ for all $i \in [n]$. Then $P$ is a type A signed poset.
\end{prop}

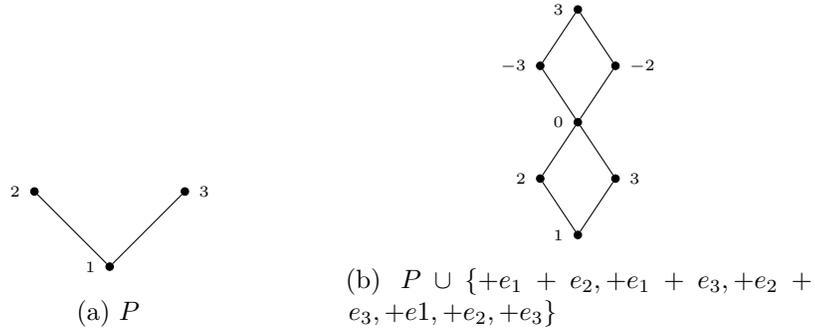
\begin{figure}[htbp]
\begin{center}
		\begin{subfigure}[b]{.4\textwidth}
				\begin{center}
				\begin{tikzpicture}
						\node[vertex] (1) at (0,0) [label=left:$1$] {};
						\node[vertex] (2) at (-1,1) [label=left:$2$] {};
						\node[vertex] (3) at (1,1) [label=right:$3$] {};
						\draw (1)--(2);
						\draw (1)--(3);
				\end{tikzpicture}
				\caption{$P$}
		\end{center}
		\end{subfigure}
		\begin{subfigure}[b]{.4\textwidth}
				\begin{center}
						\begin{tikzpicture}
						\node[vertex] (1) at (0,0) [label=left:$1$] {};
						\node[vertex] (2) at (-.5,.75) [label=left:$2$] {};
						\node[vertex] (3) at (.5,.75) [label=right:$3$] {};
						\node[vertex] (-3) at (-.5,2.25) [label=left:$-3$]{};
						\node[vertex] (-2) at (.5,2.25) [label=right:$-2$] {};
						\node[vertex] (-1) at (0,3) [label=left:$3$] {};
						\node[vertex] (0) at (0,1.5) [label=left:$0$] {};
						\draw (1)--(2);
						\draw (1)--(3);
						\draw (3)--(0)--(-3);
						\draw (2)--(0)--(-2);
						\draw (-2)--(-1);
						\draw (-3)--(-1);
						\end{tikzpicture}
						\caption{$P \cup
                            \{+e_1+e_2,+e_1+e_3,+e_2+e_3,+e1,+e_2,+e_3\}$}
				\end{center}
		\end{subfigure}
		\caption{The poset $P = \{e_1-e_2,e_1-e_3\}$ and its embedding as a
		signed poset}
		\label{fig:embeddingex}
\end{center}
\end{figure}
As an example, consider the posets shown in Figure~\ref{fig:embeddingex}.
$P$ is the intersection of $P \cup \{+e_1+e_2,+e_1+e_3,+e_2+e_3,+e1,+e_2,+e_3\}$ with the hyperplane $\{x \in \bR^n
\colon x_1+\cdots+x_n = 0 \}$.
The ideals of $P$ as a signed poset are precisely the ideals of $P$ as a poset.

\section{The Root and Weight Cones and Their Extreme Rays}\label{sec:dualcones}

Associated to a signed poset $P \subset \Bn$ (resp.\ $\Pc \subset \Cn$) are two
polyhedral cones: the \emph{root cone}, denoted $\Kprt$, and the \emph{weight
    cone}, denoted
$\Kpwt$.

\begin{defn}\label{def:rootandweightcones}
		Suppose $P \subset \Bn$ (\resp $\Pc \subset \Cn$) is a signed poset.
		Its \emph{root cone} is the positive linear span of its elements:
		$\Kprt = \bR_+ P$.

		Its \emph{weight cone} is the dual to the root cone: $\Kpwt = \{f \in \bR^n \colon \langle f,\alpha \rangle \geq 0\ \forall
\alpha \in P\}$.
\end{defn}
Observe that if $P \subset \Bn$ and $\Pc \subset \Cn$ are dual
posets, $P$ and $\Pc$ share the same root and weight cones. The definition of
Hasse diagram gives a characterization of the extreme rays of the root cone.

\begin{prop}\label{prop:kprtextremerays}
Suppose $P$ is a signed poset. Its root cone is an affine polyhedral cone whose extreme rays are given by the
edges of the Hasse diagram.
\end{prop}

Proposition~\ref{prop:kprtextremerays} generalizes the type A result
of~\citet*[Proposition 5.1(i)]{BoussicaultFerayLascouxReiner2012}.

\begin{prop}\label{prop:kprtgeneralisestypea}
Suppose $P \subset \Bn$ (\resp $\Pc \subset \Cn$) is a type A signed poset. Let
$P' = P \cap \An$ be the poset on $[n]$ which $P$ embeds. Then
\[
\Krt{P'} = \Kprt \cap \{ x \in \bR^n \colon x_1+\cdots+x_n = 0\},
\]
and furthermore, extreme rays of $\Krt{P'}$ are extreme rays of $\Kprt$.
\end{prop}

\begin{proof}
		Inspecting the definitions, one sees immediately that $\Krt{P'} \subset
		\Kprt \cap \{ x \in \bR^n \colon x_1+\cdots+x_n = 0\}$. 
		
		The extreme
		rays of $\Krt{P'}$ correspond to the covering relations of the Hasse
		diagram of $P'$. To show that each covering relation $\alpha \in P'$
		is an extreme ray of $\Kprt$, it suffices to check that there cannot
		exist $\beta,\gamma \in P$ such that $\gamma \in \Bn^+$ (\resp $\gamma
		\in \Cn^+$) and $\alpha = \beta + \gamma$. However, since $\alpha = e_i
		- e_j$, this is impossible.
\end{proof}

Characterizing the extreme rays of $\Kpwt$ requires an additional definition
and a lemma regarding ideals. 

\begin{defn}\label{def:extensible}
An ideal $J$ of $P$ is said to be
\emph{extensible} if there is a set $I \subset [n]$ such that $J \cup I$ and $J
\cup -I$ are both ideals of $P$. An ideal which is not extensible is called
\emph{non-extensible}.
\end{defn}

For example, in $P^\star$ (Figure~\ref{fig:fischerex}), the ideal $\{1,2\}$ is extensible  because both
$\{1,2,3\}$ and $\{1,2,-3\}$ are ideals of $P^*$. However, it is not obligatory
that $J \cup I$ and $J \cup -I$ be connected, nor that $J$ be connected. If one
considers instead $P = \{+e_1-e_2,+e_1+e_2,+e_1-e_3,+e_2-e_3,+e_1\}$ (see
Figure~\ref{fig:extensibilityex} for $\Ghat(P)$ and recall that when thinking
about ideals, one can always refer to $\Gchat(\Pc)$), the ideal $\{1,2\}$ is
extensible, as $\{1,2,3\}$ and $\{1,2,-3\}$ are ideals, though $\{1,2,-3\}$ is
not connected. Similarly, $\{1,3\}$ is extensible and disconnected.

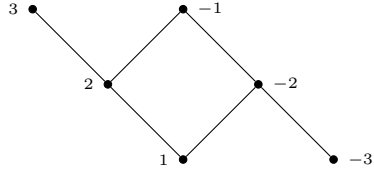
\begin{figure}
    \begin{center}
        \begin{tikzpicture}
            \node[vertex] (1) at (0,0) [label=left:$1$] {};
            \node[vertex] (-3) at (2,0) [label=right:$-3$] {};
            \node[vertex] (2) at (-1,1) [label=left:$2$] {};
            \node[vertex] (-2) at (1,1) [label=right:$-2$] {};
            \node[vertex] (3) at (-2,2) [label=left:$3$] {};
            \node[vertex] (-1) at (0,2) [label=right:$-1$] {};
            \draw (1)--(2)--(3);
            \draw (2)--(-1)--(-2)--(1);
            \draw (-3)--(-2);
        \end{tikzpicture}
    \end{center}
    \caption{$\Ghat(P)$ for $P = \{+e_1-e_2,+e_1+e_2,+e_1-e_3,+e_2-e_3,+e_1\}$}
    \label{fig:extensibilityex}
    \label{fig:extvssimpleext}
\end{figure}

\begin{prop}\label{prop:kpwtextremerays}
Suppose $P$ is a signed poset. Its weight cone $\Kpwt$ is an affine polyhedral
cone whose extreme rays correspond to the connected, non-extensible ideals of
$P$.
\end{prop}

Proposition~\ref{prop:kpwtextremerays} generalizes~\cite[Proposition
5.1(ii)]{BoussicaultFerayLascouxReiner2012}.

\begin{prop}\label{prop:kpwtgeneralisestypea}
Suppose $P \subset \Bn$ (\resp $\Pc \subset \Cn$) is a type A signed poset. Let
$P' = P \cap \An$. Then $\Kpwt = \Kwt{P'}$.
\end{prop}

\begin{proof}
It suffices to show that the ideals of $P'$ are precisely the ideals of $P$.
(If one is considering $\Pc \subset \Cn$, let $P$ be the dual poset.) Certainly
ideals of $P$ are ideals of $P'$ since $P' \subset P$. Suppose $f$ is an ideal
of $P'$.
meaning $\langle f,\alpha \rangle \geq 0$ for all $\alpha \in \Bn^+$. Since $P
= P' \cup \Bn^+$, one then has that $f$ is an ideal of $P$. Thus $\Kpwt =
\Kwt{P'}$.
\end{proof}

Figure~\ref{fig:rootweightconeex} shows the root and weight cones of
$\{+e_1-e_2,+e_1+e_2,+e_1-e_3,+e_2-e_3,+e_1\}$. 

\begin{figure}
    \begin{subfigure}[b]{.5\linewidth}
\begin{center}
    \tdplotsetmaincoords{60}{100}
    \begin{tikzpicture}[tdplot_main_coords,scale=2]
        \coordinate (O) at (0,0,0);
\draw[thick,->] (0,0,0) -- (0,1,0) node[anchor=north west]{$y$};
\draw[thick,<->] (0,0,1)  node[anchor=south]{$z$} -- (0,0,-1);
        \coordinate (1) at (1,-1,0);
        \coordinate (2) at (1,1,0);
        \coordinate (3) at (0,1,-1);
        \fill[gray!20] (O)--(1)--(3)--cycle;
        \fill[gray!20] (O)--(2)--(3)--cycle;
        \draw[blue,->] (O)--(3) node[black,anchor=north west] {$(1,-1,0)$};
        \fill[gray!20] (O)--(1)--(2)--cycle;
        \draw[blue,->] (O)--(1) node[black,anchor=east] {$(0,1,-1)$};
        \draw[blue,->] (O)--(2) node[black,anchor=south west] {$(1,1,0)$};
        \draw[dashed] (1)--(3)--(2)--(1);
        \draw[thick,->] (0,0,0) -- (2,0,0) node[anchor=north east]{$x$};
    \end{tikzpicture}
\end{center}
\caption{$\Kprt$}
\end{subfigure}
\begin{subfigure}[b]{.45\linewidth}
    \begin{center}
    \tdplotsetmaincoords{60}{100}
        \begin{tikzpicture}[tdplot_main_coords,scale=2]
            \coordinate (O) at (0,0,0);
            \draw[thick,<->] (0,-1,0) -- (O) -- (0,1,0) node[anchor=north west]{$y$};
            \draw[thick,<->] (0,0,1)  node[anchor=south]{$z$} -- (O)-- (0,0,-1);
            \draw[thick,->] (0,0,0) -- (2,0,0) node[anchor=north east]{$x$};
            \coordinate (1) at (1,1,1);
            \coordinate (2) at (1,1,-1);
            \coordinate (3) at (1,-1,-1);
            \fill[gray!20] (O)--(2)--(3)--cycle;
            \fill[gray!20] (O)--(1)--(2)--cycle;
            \draw[blue,->] (O)--(2) node[black,anchor=north west] {$(1,1,-1)$};            
            \fill[gray!20] (O)--(1)--(3)--cycle;
            \draw[blue,->] (O)--(1) node[black,anchor=south west] {$(1,1,1)$};
            \draw[blue,->] (O)--(3) node[black,anchor=north east] {$(1,-1,-1)$};
            \draw[dashed] (1)--(3)--(2)--(1);
        \end{tikzpicture}
    \end{center}
    \caption{$\Kpwt$}
\end{subfigure}
\caption{The root cone and weight cone of
    $\{+e_1-e_2,+e_1+e_2,+e_1-e_3,+e_2-e_3,+e_1\}$}
\label{fig:rootweightconeex}
\end{figure}
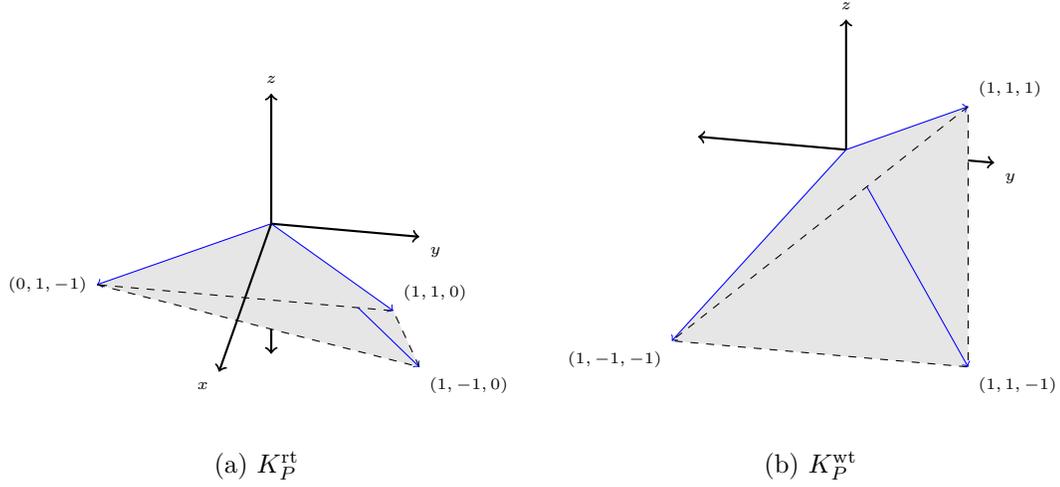

The proof of Proposition~\ref{prop:kpwtextremerays} requires the following
lemma.

\begin{lem}\label{lem:lincombideals}
  Suppose $P \subset \Bn$ is a signed poset and  $f=(f_1,f_2,\ldots,f_n) \in \Kpwt$. Let $c_1\leq c_2 \leq \cdots \leq c_k$ be the \emph{nonzero} $|f_i|$ arranged in increasing order. Then 
  \[
      J_k=\{\sgn(f_i)i \colon i \in [n], |f_i| \geq c_k\}
  \]
  is an ideal.
\end{lem}

As an example, consider the signed poset in
Figure~\ref{fig:extvssimpleext}. One has $c_1=1$, $c_2=2$ and $c_3=3$. Writing
$J_1,J_2,J_3$ in a tower as shown in~\cite{Reiner1993}, one has
\[
\begin{array}{lr}
		J_3 & (1,0,\phantom{-}0) \\
		J_2 & (1,1,\phantom{0}0) \\
		J_1 & +(1,1,-1) \\
\hline
	f	&(3,2,-1)
\end{array}
\]
\begin{defn}\label{def:signedsupport}
The ideal $J_1$ from Lemma~\ref{lem:lincombideals} is the \emph{signed support} of $f$, denoted $\ssupp(f)$.
\end{defn}

\begin{proof}
  Suppose $J_k$ is not an ideal, so it is does not correspond to an order ideal
  in $\widehat{G}(P)$ (it is isotropic by construction). Then there is a $\pm k$ such that $\pm i \in J_k$ with $\pm k < \pm i$. Since $\pm k \notin J_k$, $|f_i| \geq c_k > |f_k|$. One must consider several cases, one for each of the possible combinations of signs of $\pm i$ and $\pm k$.
  \begin{itemize}
	\item \textbf{Suppose $+k < +i$.} Then $e_k - e_i \in P$. Then $\langle f,e_k-e_i\rangle=f_k-f_i \leq |f_k|-f_i<0$, a contradiction.
	\item \textbf{Suppose $-k < +i$.} Then $-e_k-e_i \in P$. Then $\langle f, -e_k-e_i\rangle = -f_k - f_i < 0$, since $0<|f_k|<f_i$, a contradiction.
	\item \textbf{Suppose $+k < -i$.} Then $e_k+e_i \in P$. In this case, $f_i <0$ and $|f_i|>|f_k|$. Then $\langle f, e_k+e_i \rangle = f_k + f_i= f_k - |f_i|<0$, a contradiction.
	\item \textbf{Suppose $ -k < -i$.} Then $e_i-e_k \in P$. Then $\langle f, e_i-e_k \rangle = f_i - f_k \leq -|f_i|+|f_k| < 0$, a contradiction.
  \end{itemize}
  Thus, $J_k$ must be closed under going down in $\widehat{G}(P)$. Recall that,
  it is isotropic by construction, so $J_k$ is an ideal of $P$.
\end{proof}

\begin{proof}[Proof of Proposition~\ref{prop:kpwtextremerays}]
First, one shows that $\Kpwt$ is spanned by the ideals of $P$, then that the
connected ideals suffice to span $\Kpwt$ and, finally, that the connected nonextensible ideals suffice.

Suppose $f=(f_1,\ldots,f_n) \in \Kpwt$ and $c_i$ and $J_i$ are as in Lemma~\ref{lem:lincombideals}.
Then
\[
	f = \sum_{k=1}^n (c_k-c_{k-1})\chi_{J_k},
\]
taking $c_0=0$. The $J_k$ can be decomposed into their connected components,
showing $\Kpwt$ is spanned by the connected ideals of $P$. Since there are
only finitely many ideals, the definition of extensibility means every extensible ideal can be written as a positive linear combination of connected nonextensible ideals. Thus, the nonextensible connected ideals span $\Kpwt$.

To show that the nonextensible connected ideals of $P$ are the extreme rays of
$\Kpwt$, one can show that every such ideal lies in the intersection of $n-1$
(linearly independent) hyperplanes, each supporting $\Kpwt$. Let $J$ be a
connected nonextensible ideal and consider a spanning tree $T$ of its Hasse diagram (as a subposet of $\widehat{G}(P)$). Then $J$ lies in the intersection of the following $n-1$ hyperplanes:
\[
	\begin{array}{ll}
			x_i=0 & \text{for } \pm i \notin J \\
			x_i=x_j & \text{for }+i \lessdot +j \text{ or } -j \lessdot -i \in T\ (+e_i-e_j \in P) \\
			x_i=-x_j & \text{for } +i \lessdot -j \text{ or } -i \lessdot +j \in T\ (\pm(e_i+e_j) \in P)
	\end{array}
\]
The above could, \emph{a priori}, fail to specify $n-1$ hyperplanes in one of
two ways. First, it may specify both $x_i=x_j$ and $x_i=-x_j$ as hyperplanes,
which really would specify only the hyperplane $x_i=0$. However, this situation
cannot arise since $J$ is isotropic. The second possible problem is that a
single hyperplane is specified by two different edges in $T$. Again, this is
not possible since $J$ is isotropic, so for each pair $(i,j)$ there can only be
one edge in $T$ involving $\pm i$ and $\pm j$. Thus the above truly specifies
$n-1$ linearly independent hyperplanes supporting $\Kpwt$ such that $J$ lies in their intersection, meaning $J$ is an extreme ray of $\Kpwt$, as claimed.
\end{proof}

\subsection{Dual Characterizations}\label{subsec:dualcones}

Since the root cone and weight cone are dual to one another, one should have
that a characterization of when $\Kprt$ is pointed gives a characterization of
when $\Kpwt$ is full-dimensional and vice versa. Equally, characterizing when
$\Kprt$ is simplicial gives a characterization of when $\Kpwt$ is simplicial
and vice versa.

As a consequence of the antisymmetry condition in the definition of signed
poset, $\Kprt$ is always pointed, meaning $\Kpwt$ is always full-dimensional.
One can also exhibit $n$ linearly independent ideals of a signed poset
$P$.

\begin{lem}
    Suppose $P \subset \Bn$ (resp.\ $\Pc \subset \Cn$) is a signed poset. For
    each $i \in [n]$, at least one of $I_\leq(i)=\{j \in \widehat{G}(P)
    \colon j \leq i\}$ and $I_\leq(-i)=\{j \in \widehat{G}(P) \colon j \leq -i\}$ is isotropic.
\end{lem}

This lemma introduces some notation that will be used again. In addition to $I_\leq(i)$, let $I_<(i) = \{i \in \Ghat(P)\colon j < i\}$.

\begin{proof}[Proof of Lemma 3.4.4]
    Suppose not and neither $I_\leq(i)$ nor $I_\leq(-i)$ is isotropic. Then
	there are $j$ and $k$, not necessarily distinct, such that $j,-j < i$ and
    $k,-k < -i$. Then, since $j < i$, one has that $-i < -j$ and since $-j <
    -i$ one has that $-i < j$. Transitivity means that $-i < i$. On the other
    hand, since $k < -i$, one has that $i < -k$ and since $-k < -i$ one has
    that $i < k$. Then $i < -i$. It is impossible that both $i < -i$ and $-i <
    i$, so at least one of $I_\leq(i)$ and $I_\leq(-i)$ is isotropic.
\end{proof}

\begin{prop}\label{prop:kpwtfulldim}
    Suppose $P\subset \Bn$ (resp.\ $\Pc \subset \Cn$) is a signed poset. Then $\Kpwt$ is full-dimensional.
\end{prop}

\begin{proof}
    Fix a linear extension, $\prec$, of $\widehat{G}(P)$. For each $i$, pick an ideal $J_i$ as follows:
        \[
            J_i =\begin{cases}
                I_\leq(i) & \text{if } I_\leq(-i) \text{ is not isotropic} \\
                I_\leq(-i) & \text{if } I_\leq(i) \text{ is not isotropic} \\
                I_\leq(i) & \text{if both are isotropic and } i\prec -i \\
                I_\leq(-i) & \text{if both are isotropic and } -i \prec i \\
        \end{cases}
        \]
		The key observation is that $J_i$ contains only elements of $\pm [n]$ that precede $i$ in the $\prec$ order. Writing the $J_i$ as the rows of a matrix (ordered by $\prec$), one has a lower triangular matrix with $\pm 1$ on the diagonal, meaning the $J_i$ are linearly independent. Since $\Kpwt$ contains $n$ linearly independent vectors, it must be full-dimensional.
\end{proof}

\subsection{Matroids}\label{subsec:matroids}

In contrast to Proposition~\ref{prop:kpwtfulldim}, the root cone is not always full-dimensional and the weight
cone is not always pointed. Understanding when the root cone is
full-dimensional, the weight cone is pointed and when both the root cone and
weight cone are simplicial requires some facts from the theory of signed
graphic matroids developed by~\citet{Zaslavsky1982}. This section will review the two equivalent definitions of a matroid that will be useful.

\begin{defn}\label{def:matroidindset}
A \emph{matroid} is a pair $(E,\mathcal{I})$ where $E$ is a finite set called
the \emph{ground set} and $\mathcal{I}$ is a collection of subsets called the
\emph{independent sets} of $E$ such that
\begin{enumerate}[label=(\alph{*})]
		\item $\mathcal{I}$ is nonempty.
		\item If $I \in \mathcal{I}$ and $J \subset I$, then $J \in \mathcal{I}$.
		\item If $I,J \in \mathcal{I}$ and $|I|=|J|+1$, there is an $x \in I \smallsetminus J$ such that $J \cup \{x\} \in \mathcal{I}$.
\end{enumerate}
 Subsets of $E$ which are not independent are said to be \emph{dependent}.
\end{defn}
That the elements of $\mathcal{I}$ are called the independent sets is not a
coincidence---these are the properties defining the collection of sets of linearly independent columns of a matrix. 

Whitney introduced matroids in~\cite{Whitney1935} and gave a number of equivalent definitions. Rather than defining a matroid as a pair $(E,\mathcal{I})$, one can define a matroid as a pair $(E,\mathcal{C})$, where $\mathcal{C}$ is the collection of \emph{circuits}, a collection of subsets of $E$ such that
\begin{enumerate}[label=(\alph{*})]
		\item If $I \in \mathcal{C}$ and $J$ is a proper subset of $I$. Then $J \notin \mathcal{C}$.
		\item If $C_1,C_2 \in \mathcal{C}$, $x \in C_1 \cap C_2$ and $y \in C_1 \smallsetminus C_2$, then there is a $C_3 \in \mathcal{C}$ with $y \in C_3$ and $x \notin C_3$.
\end{enumerate}

One can show that these two definitions are equivalent by taking the circuits to be
the minimal dependent sets. (See~\citet[(6.13)]{Aigner1979}.)

\subsection{When the root cone is full-dimensional and simplicial}\label{subsec:rootconefulldim}
Characterization of when the root cone is
full-dimensional as well as when it is simplicial relies on the notion of
balance in a signed graph. (Signed graphs were discussed in
Section~\ref{sec:representingposets}.) Recall that the Hasse diagram of a
signed poset is an oriented signed graph. Denote the underlying signed graph by
$\Sigma_P$.

\begin{defn}\label{def:balanced}
A cycle in a signed graph is said to be \emph{balanced} if it has an even
number of edges whose sign is $-$. A cycle containing an odd number of edges
with sign $-$ is said to be \emph{unbalanced}. If all cycles in a signed graph
are balanced, the graph itself is said to be balanced.
\end{defn}

Like unsigned graphs, signed graphs have an associated matroid, introduced
by~\citet{Zaslavsky1982}.
\begin{defn}\label{def:signedgraphicmatroid}
    Suppose $\Sigma$ is a signed graph. The \emph{signed graphic matroid}
$\Gamma(\Sigma)$ is the matroid whose circuits are the balanced cycles of
$\Sigma$ and pairs of unbalanced cycles joined by a (possibly empty) path.
\end{defn}

Figure~\ref{fig:matroidex} shows the Hasse diagram of $P =
\{+e_1-e_2,+e_3-e_2,+e_1+e_3,e_3-e_5,e_3-e_4,-e_4-e_5,+e_6-e_4,+e_6-e_5\}$ and
the signed graph it orients, $\Sigma_P$. There is one balanced cycle $3-4-6-5$ and two
pairs of unbalanced cycles joined by a path: $1-2-3$ and $3-4-5$ are joined by
the empty path and $1-2-3$ and $4-5-6$ are joined either by $3-4$ or $3-5$.

\begin{figure}
    \begin{center}
        \begin{subfigure}[b]{.4\textwidth}    
    \begin{center}
        \begin{tikzpicture}[every node/.style={font=\normalsize},scale=1.5,
            incidence/.style={font=\tiny}]
            \node[circle,draw] (1) at (0,0) {$1$};
            \node[circle,draw] (2) at (0,2) {$2$};
            \node[circle,draw] (3) at (1,1) {$3$};
            \node[circle,draw] (4) at (2,0) {$4$};
            \node[circle,draw] (5) at (2,2) {$5$};
            \node[circle,draw] (6) at (3,1) {$6$};
            \draw (1)--node[very near start,incidence,left]{$+$}
           node[very near end, incidence, left]{$-$} (2);
            \draw (1)-- node[very near start,incidence,right]{$+$}
            node[very near end, incidence, below]{$+$}(3);
            \draw (2)-- node[very near start,incidence,above] {$-$} node[very
            near end, incidence, above] {$+$}(3);
            \draw (3)-- node[very near start,incidence,below]{$+$} node[very
            near end,incidence,left]{$-$}(4);
            \draw (3)-- node[very near start, incidence, anchor=south east]
            {$+$} node[very near end,incidence,left] {$-$}(5);
            \draw (4)-- node[very near start,incidence,left] {$-$} node[very
            near end, incidence,left] {$-$}(5);
            \draw (4)-- node[very near start,incidence,above]{$-$} node[very
            near end,incidence,left] {$+$}(6);
            \draw (5)-- node[very near start,incidence,right] {$-$} node[very
            near end, incidence,above] {$+$} (6);
        \end{tikzpicture}
    \end{center}
    \caption{Hasse diagram}
\end{subfigure}
\begin{subfigure}[b]{.4\textwidth}
    \begin{center}
        \begin{tikzpicture}[every node/.style={font=\normalsize},scale=1.5,
            incidence/.style={font=\tiny}]
            \node[circle,draw] (1) at (0,0) {$1$};
            \node[circle,draw] (2) at (0,2) {$2$};
            \node[circle,draw] (3) at (1,1) {$3$};
            \node[circle,draw] (4) at (2,0) {$4$};
            \node[circle,draw] (5) at (2,2) {$5$};
            \node[circle,draw] (6) at (3,1) {$6$};
            \draw (1)-- node[midway,left,incidence] {$+$} (2);
            \draw (1)-- node[midway,right,incidence] {$-$} (3);
            \draw (2)-- node[midway,right,incidence] {$+$} (3);
            \draw (3)-- node[midway,below,incidence] {$+$} (4);
            \draw (3)-- node[midway,left,incidence] {$+$} (5);
            \draw (4)-- node[midway,right,incidence] {$-$} (5);
            \draw (4)-- node[midway,right,incidence] {$+$}(6);
            \draw (5)-- node[midway,right,incidence] {$+$}(6);
        \end{tikzpicture}
    \end{center}
    \caption{Associated signed graph}
\end{subfigure}
\caption{Hasse diagram of $P =
\{+e_1-e_2,+e_3-e_2,+e_1+e_3,e_3-e_5,e_3-e_4,-e_4-e_5,+e_6-e_4,+e_6-e_5\}$ and
associated signed graph}
    \label{fig:matroidex}
\end{center}
\end{figure}
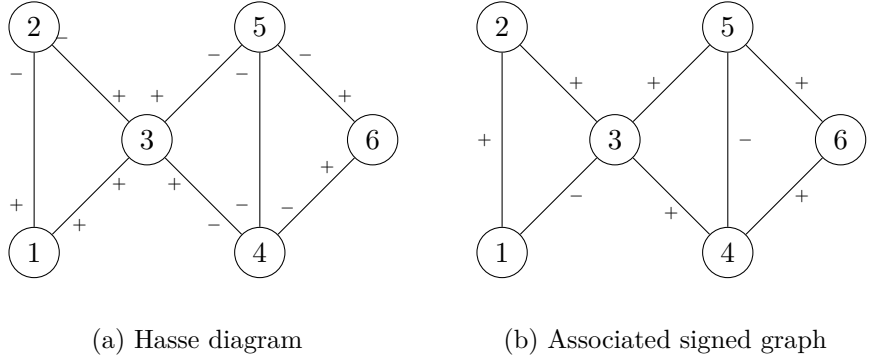

On the other hand, any set of vectors can be used to define a matroid by
taking the independent sets to be the linearly independent subsets, so the
extreme rays of $\Kprt$, i.e.\ the edges of the Hasse diagram of $P$, define a
matroid, call it $\cM_P$. Thus, one has two matroids arising from the edges
of the Hasse diagram of $P$. However, it is a consequence of a result of
Zaslavsky~\cite{Zaslavsky1982} that these two matroids are the same.

Zaslavsky also defines the \emph{incidence matrix} $M(\Sigma)$ of a signed
graph.

\begin{defn}\label{def:incidencematrix}
    Suppose $\Sigma$ is a signed graph. The \emph{incidence matrix} of $\Sigma$
    is the matrix $M(\Sigma)$ whose columns $M_e$ are indexed by the edges of $\Sigma$
and 
\begin{center}
\begin{tabular}{ll}
if $e = (v,w)$ is an edge, & $m_{ve}= \pm 1$ and $m_{we} = -\sigma(e)m_{ve}$,\\
if $e = (v,v)$ is a loop, & $m_{ve} = 0$ if $\sigma(e) = +$ and $m_{ve} = \pm 2$ if
$\sigma(e) = -$, \\
if $e$ and $v$ not incident, & $m_{ve} = 0$.
\end{tabular}
\end{center}
\end{defn}
Note that these $M_e$ are the elements of a signed poset $\Pc
\subset \Cn$ corresponding to the edges of the Hasse diagram. (Dividing
$m_{ve}$ by two for a loop in type B also does not alter the matroid.) For the
poset $P$ in Figure~\ref{fig:matroidex}, one
has
\[
M(\Sigma_P) =
\begin{tikzpicture}[baseline = (M.center),every
    node/.style={font=\normalsize},every left
    delimiter/.style={xshift=1ex},every right delimiter/.style={xshift=-1ex}]
    \matrix (M) [matrix of math nodes,left delimiter={(},right %
    delimiter={)},nodes={anchor=east}]{
 1 & 0 & \phantom{-}1 & 0 & 0 & 0 & 0 & 0 \\
 -1 & -1 & 0 & 0 & 0 & 0 & 0 & 0 \\
 0 & 1 & 1 & 1 & 1 & 0 & 0 & 0 \\
 0 & 0 & 0 & 0 & -1 & -1 & -1 & 0 \\
 0 & 0 & 0 & -1 & 0 & -1 & 0 & -1 \\
 0 & 0 & 0 & 0 & 0 & 0 & 1 & 1 \\
    };
\node[anchor=south east] (cornernode) at (M-1-1.north west) {}; 
\node[rotate=90,anchor=west] at (cornernode -| M-1-1.center) {$+e_1-e_2$};
\node[rotate=90,anchor=west] at (cornernode -| M-1-2.center) {$-e_2+e_3$};
\node[rotate=90,anchor=west,yshift=-1ex] at (cornernode -| M-1-3) {$+e_1+e_3$};
\node[rotate=90,anchor=west] at (cornernode -| M-1-4) {$+e_3-e_5$};
\node[rotate=90,anchor=west] at (cornernode -| M-1-5) {$+e_3-e_4$};
\node[rotate=90,anchor=west] at (cornernode -| M-1-6) {$-e_4-e_5$};
\node[rotate=90,anchor=west] at (cornernode -| M-1-7) {$-e_4+e_6$};
\node[rotate=90,anchor=west] at (cornernode -| M-1-8) {$-e_5+e_6$};
\end{tikzpicture}
\]

That the matroid formed by the columns of $M(\Sigma_P)$ and the signed graphic matroid
of $\Sigma_P$ are the
same is a consequence of the observation that the $M_e$ correspond to elements
of the signed poset and the following result of Zaslavsky.
\begin{thm}[{\cite[Theorem 8B.1]{Zaslavsky1982}}]\label{thm:matroidsequiv}
    Let $\Sigma$ be a signed graph and $f \colon E \to \bR^n$ be the mapping
    $f(e) = M_e$. The matroid structure induced on $E$ by the dependencies
    among the vectors $f(e)$ is precisely that of the signed graphic matroid
    $\Gamma(\Sigma)$.
\end{thm}

In other words, the circuits of $\Gamma(\Sigma)$, i.e.\ balanced cycles and pairs of unbalanced cycles joined by
a path in $\Sigma_P$, correspond linearly dependent subsets of $P$ which are
minimal with respect to inclusion. This allows one to
characterize when the root cone is full-dimensional and simplicial.

\begin{prop}\label{prop:kprtfulldim}
Suppose $P \subset \Bn$ (resp.\ $\Pc \subset \Cn$) is a signed poset. Its root
cone is full-dimensional if and only if each connected component of $\Sigma_P$ contains an unbalanced cycle.
\end{prop}

The proof requires an alternate characterization of balance due to
Harary~\cite{Harary1953}.

\begin{thm}[{\cite[Theorem 3]{Harary1953}}]\label{thm:hararybalance}
A signed graph is balanced if and only if its vertices can be partitioned into
two sets $V_+$ and $V_-$ such that every edge of the graph labeled $+$ joins two
vertices both in either $V_+$ or $V_-$ and every edge of the graph labeled $-$
joints a vertex in $V_+$ to one in $V_-$.
\end{thm}

\begin{proof}[Proof of Proposition~\ref{prop:kprtfulldim}]
   Suppose $\Sigma_P$ contained a component $K$ which is balanced. Then from the theorem of Harary, there is a partition of vertices of $K$ into $V_+(K)$ and $V_-(K)$. Let $f=(f_1,\ldots,f_n) \in \bR^n$ be defined by
    \[
        f_i= \begin{cases}
            1 & v_i \in V_+(K) \\
            -1 & v_i \in V_-(K) \\
            0 & v_i \notin K
        \end{cases}
    \]
    Recall that an edge of $\Sigma$ labeled $+$ corresponded to $e_i-e_j$,
    meaning $f_i=f_j=\pm1$, so $\langle f, e_i-e_j \rangle =0$. Similarly, an
    edge of $\Sigma$ labeled $-$ corresponds to $\pm (e_i+e_j)$, with
    $f_i=-f_j$, so $\langle f, \pm(e_i+e_j)\rangle =0$. Then one has that
    $\langle f, \alpha \rangle=\langle -f ,\alpha \rangle =0$ for all $\alpha
    \in P$. In other words, $\Kprt$ lies in $f^\perp$, so cannot be full-dimensional.
    
     In the other direction, suppose every connected component of $\Sigma_P$
	 contains an unbalanced cycle. To show that $\Kprt$ is full-dimensional it suffices to show there are $k$ edges corresponding to $k$ linearly independent elements of $P$ in each connected component of $\Sigma_P$ with $k$ vertices. Suppose $K$ is a connected component of $\Sigma_P$ with $k$ vertices. 
	 A spanning tree of $K$ has $k-1$ edges corresponding to $k-1$ linearly independent elements of $P$. Adding any other edge of $K$ would create a cycle. 
	 Imagine that an edge $e$ is added to $K$ to create a cycle. The only way this edge could fail to correspond to a $k$th linearly independent element of $P$ is if the edge ``completes'' both the unbalanced cycle and a balanced cycle. However, this would mean the first $k-1$ edges formed a cycle, which is impossible. Thus, $K$ contains edges corresponding to $k$ linearly independent elements of $P$, so $\Sigma_P$ has $n$ such edges, so $\Kprt$ is full-dimensional.
\end{proof}

Note that the $f$ found in the proof of Proposition~\ref{prop:kprtfulldim} is
an ideal such that $-f$ is also an ideal. In other words, the proof also shows that
when $\Kprt$ is not full-dimensional, $\Kpwt$ is not pointed.

\begin{prop}\label{prop:kpwtpointed}
    Suppose $P \subset \Bn$ (resp.\ $\Pc \subset \Cn$) is a signed poset. $\Kpwt$ is pointed if and only if no connected component of $\widehat{G}(P)$ is isotropic.
\end{prop}
\begin{proof}
Suppose $J \subset \widehat{G}(P)$ is a connected component of $\widehat{G}(P)$ which is isotropic. Then $-J$ must also be a connected component of $\widehat{G}(P)$, which is isotropic and thus an ideal. Consequently, $\Kpwt$ is not pointed.

In the other direction, suppose $J$ and $-J$ are both isotropic ideals in
$\widehat{G}(P)$, i.e.\ $\Kpwt$ is not pointed. Suppose $i \in J$ (so $-i \in
-J$) and $i < j$. Then, from the definition of $\widehat{G}(P)$, one must have $-j < -i$, meaning $-j \in -J$, so $j \in J$. Thus $J$ is an order filter. An ideal which is also an order filter must be an entire connected component of the Hasse diagram, meaning $\widehat{G}(P)$ has a connected component which is isotropic, namely $J$.
\end{proof}
  
\begin{prop}\label{prop:isotropiccompimpliesbalanced}
    Suppose $P \subset \Bn$ (\resp $\Pc \subset \Cn$) is a signed poset.
If $\widehat{G}(P)$ contains an isotropic connected component, then $\Sigma_P$ contains a balanced connected component.
\end{prop}

\begin{proof}
 Suppose $J \subset \widehat{G}(P)$ is an isotropic connected component. Then $\Sigma_P$ has a connected component whose vertices are $|i|$ for $i \in J$, call it $K$. Then the vertices of $K$ can be partitioned into $V_+$ and $V_-$ by the following rule:
    \[
        V_+=\{ i \colon i \in J\} \quad\text{and}\quad V_-=\{i \colon -i \in J\}.
    \]
    Checking the definition of $\widehat{G}(P)$, one sees that edges labeled $+$ in $\Sigma_P$ correspond to covering relations between two positive or two negative elements of $\widehat{G}(P)$ and edges labeled $-$ correspond to covering relations between one positive and one negative element in $\widehat{G}(P)$. In other words, the $V_+$ and $V_-$ coming from $J$ satisfy the requirements of Theorem~\ref{thm:hararybalance}, so $K$ is a balanced component of $\Sigma_P$.
\end{proof}

    As an example, consider $P = \{+e_1+e_2,+e_2-e_3,-e_3-e_4,+e_1-e_4\}$, with
    $\Sigma_P$ and $\Ghat(P)$ are shown in Figure~\ref{fig:balanceex}. One sees
    that $\Kprt$ is not full-dimensional as $\Sigma_P$ contains a balanced
    cycle. Additionally, one sees that $\Ghat(P)$ contains connected components
    which are isotropic. The corresponding partition of the vertices of
    $\Sigma_P$ is $\{1,4\}\sqcup\{2,3\}$ and one sees that all elements of $P$
    lie in the hyperplane perpendicular to $(1,-1,-1,1)$, meaning both $(1,-1,-1,1)$
    and $(-1,1,1,-1)$ are ideals, so $\Kpwt$ is not pointed.

    \begin{figure}
        \begin{center}
            \begin{subfigure}[b]{.4\textwidth}
        \begin{center}
            \begin{tikzpicture}[scale=2,every
                node/.style={font=\normalsize},pm/.style={font=\tiny}]
                \node[circle,draw] (1) at (0,0) {$1$};
                \node[circle,draw] (2) at (1,0) {$2$};
                \node[circle,draw] (3) at (1,-1) {$3$};
                \node[circle,draw] (4) at (0,-1) {$4$};
                \draw (1)-- node[midway,above,pm] {$-$} (2) -- node[midway,
                right,pm] {$+$} (3) -- node[midway,below,pm] {$-$} (4) --
                node[midway,left,pm] {$+$} (1);
            \end{tikzpicture}
        \end{center}
        \caption{$\Sigma_P$}
    \end{subfigure}
    \begin{subfigure}[b]{.5\textwidth}
        \begin{center}
            \begin{tikzpicture}[scale=2]
                \node[vertex] (-3) at (0,0) [label=left:$-3$] {};
                \node[vertex] (1) at (1,0) [label=right:$1$] {};
                \node[vertex] (-2) at (0,1) [label=left:$-2$] {};
                \node[vertex] (4) at (1,1) [label=right:$4$] {};
                \node[vertex] (2) at (2,0) [label=left:$2$] {};
                \node[vertex] (-4) at (3,0) [label=right:$-3$] {};
                \node[vertex] (3) at (2,1) [label=left:$3$] {};
                \node[vertex] (-1) at (3,1) [label=right:$-1$] {};
                \draw (-3)--(-2)--(1)--(4)--(-3);
                \draw (2)--(3)--(-4)--(-1)--(2);
            \end{tikzpicture}
        \end{center}
        \caption{$\Ghat(P)$}
    \end{subfigure}
\end{center}
    \caption{$\Sigma_P$ and $\Ghat(P)$ for $P =
        \{+e_1+e_2,+e_2-e_3,-e_3-e_4,+e_1-e_4\}$}
    \label{fig:balanceex}
    \end{figure}
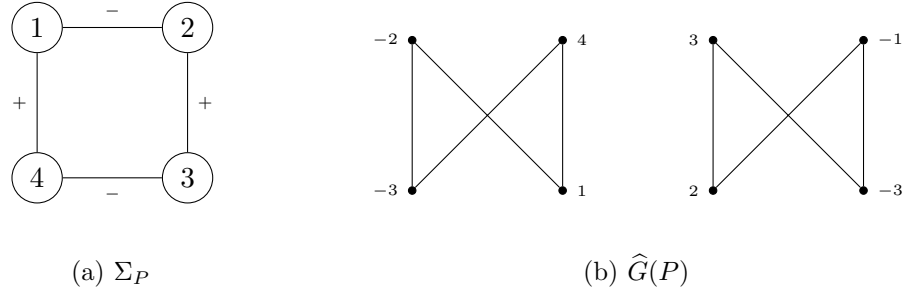

\begin{prop}\label{prop:simplicialcones}
    Suppose $P \subset \Bn$ is a signed poset. The cones $\Kprt$ and $\Kpwt$ are
    simplicial if and only if $\Sigma_P$ does not contain a balanced cycle or
    two unbalanced cycles joined by a path.
\end{prop}

\begin{proof}
From Theorem~\ref{thm:matroidsequiv}, one has that $\Kprt$ is simplicial if and
only if $\Gamma(\Sigma_P)$ contains no dependent sets, as dependent sets in
$\Gamma(\Sigma_P)$ correspond precisely to dependent sets of extreme rays of
$\Kprt$. Therefore, $\Kprt$ is simplicial if and only if $\Gamma(\Sigma_P)$
contains no circuits, i.e.\ no balanced cycles and no pairs of unbalanced
cycles joined by a path. One then has that this characterizes when $\Kpwt$ is
simplicial as well, since $\Kprt$ and $\Kpwt$ are either both simplicial or
both not simplicial as they are dual cones.
\end{proof}

\begin{cor}\label{cor:fulldimpointedsimplicial}
Suppose $P \subset \Bn$ (\resp $\Pc \subset \Cn$) is a signed poset. Its root
cone, $\Kprt$, and weight cone $\Kpwt$ are pointed, full-dimensional and
simplicial if and only if
\begin{itemize}
    \item every connected component of $\Sigma_P$ contains and unbalanced
        cycle, and
    \item $\Sigma_P$ contains a balanced cycle nor two unbalanced cycles joined
        by a path.
\end{itemize}
\end{cor}

\section{Two Rational Functions}\label{sec:rationalfunctions}
Recall from the introduction that associated to each signed poset is a pair of rational functions which are sums
over the linear extensions. For $P \subset \Bn$, define
\begin{align*}
    \Psi_P(\mathbf{x}) &= \sum_{w \in\cL(P)}
    w\left(\dfrac{1}{(x_1-x_2)(x_2-x_3)\cdots(x_{n-1}-x_n)x_n}\right) \quad
    \text{and}  \\
    \Phi_P(\mathbf{x}) &= \sum_{w \in \cL(P)}
    w\left(\dfrac{1}{x_1(x_1+x_2)\cdots(x_1+x_2+\cdots+x_n)}\right),
\end{align*}
and for $\Pc \subset \Cn$, define
\begin{align}
    \Psi^\smvee_\Pc(\bm{x}) &= \sum_{w \in\cL(P)}
    w\left(\dfrac{1}{(x_1-x_2)(x_2-x_3)\cdots(x_{n-1}-x_n)2x_n}\right)=\dfrac{1}{2}\Psi_P(\bm{x})
    \quad \label{eq:psicheck}
    \text{and} \\
    \Phi^\smvee_\Pc(\bm{x}) &= \sum_{w \in \cL(P)}
    w\left(\dfrac{1}{x_1(x_1+x_2)\cdots(x_1+\cdots+x_{n-1})(\frac{1}{2}x_1+\frac{1}{2}x_2+\cdots+\frac{1}{2}x_n)}\right)\label{eq:phicheck}
    \\  &=2\Phi_P(\bm{x}). \nonumber
\end{align}
    These functions are the analogues of the $\Psi$ and $\Phi$ considered
in~\cite{Greene1992},~\cite{BoussicaultFerayLascouxReiner2012}
and~\cite{FerayReiner2012} and discussed in
Section~\ref{subsec:signedposetstory} above. Note that when $P \subset \Bn$ is a type A signed
poset and $P' = P \cap \An$ the corresponding poset, one has $\Phi_P(\bm{x}) =
\Phi_{P'}(\bm{x})$, where $\Phi_P$ is calculated with the type B definition and
$\Phi_{P'}$ is calculated with the type A definition, since $P$ and $P'$ have the same ideals by
Proposition~\ref{prop:kpwtgeneralisestypea} and thus the same linear extensions.

As in~\cite{BoussicaultFerayLascouxReiner2012}, one can understand these functions as the valuation $s(-;\bm{x})$ from
Section~\ref{subsec:rationalfunctionscones}.

\begin{thm}\label{thm:phipsifroms}
Suppose $P \subset \Bn$ (resp.\ $\Pc \subset \Cn$) is a signed poset. Then
\begin{alignat*}{2}
		\Psi_P(\bm{x}) &= s(\Kprt;\bm{x}) &&= \int_{\Kprt} e^{-\langle
		\bm{x},u\rangle}\,du \\
		\Psi^\smvee_{\Pc}(\bm{x}) &= s^\smvee(\Kprt;\bm{x}) &&= \int_{\Kprt}e^{-\langle
		\bm{x},u\rangle}\, \dfrac{1}{2}du \\
		\Phi_P(\bm{x}) &= s(\Kpwt;\bm{x}) &&= \int_{\Kpwt} e^{-\langle
		\bm{x},u\rangle}\,du \\
		\Phi^\smvee_{\Pc}(\bm{x}) &= s^\smvee(\Kpwt;\bm{x}) &&= \int_{\Kpwt}e^{-\langle
		\bm{x},u\rangle}\,2du
\end{alignat*}
where $du$ is Lebesgue measure.
\end{thm}

The measure used for calculating $\Psi^\smvee$ and $\Phi^\smvee$ varies to give
the parallelopiped spanned by the simple roots and fundamental dominant weights,
respectively, volume one.

\begin{proof}
		In all cases, the proof proceeds by induction on the number of pairs $\{i,j\} \subset
		\pm[n]$ such that $i$ and $j$ are incomparable in $\Ghat(P)$. The proof
        for $\Psi_P$ and $\Phi_P$ is given here; the proofs for
        $\Psi^\smvee_{\Pc}$ and $\Phi^\smvee_{P}$ follow 		
        \eqref{eq:psicheck} and \eqref{eq:phicheck}.		
		In the
		base case, suppose all pairs $\{i,j\}$ are comparable, i.e.\ $\Ghat(P)$
		is a chain. Then $\cL(P)$ must consist of a single signed permutation,
		call it $w$. 

		Applying Proposition~\ref{prop:barvinok} in each case, one has
		\begin{multline*}
				s(\Kprt; \bm{x}) = |w(\alpha_1)\wedge \cdots \wedge
				w(\alpha_n)| \prod_{i=1}^n
				\langle \bm{x}, w(\alpha_i) \rangle^{-1}  \\
                \shoveright{=
				w\left(\dfrac{1}{(x_1-x_2)\cdots(x_{n-1}-x_n)x_n}\right)
				= \Psi_P(\bm{x})}\\
            \shoveleft{s(\Kpwt; \bm{x})= |w(\mu_1)\wedge\cdots \wedge
				w(\mu_n)|\prod_{i=1}^n \langle \bm{x},w(\mu_i)\rangle^{-1}} \\=
				w\left(\dfrac{1}{x_1(x_1+x_2)\cdots(x_1+\cdots+x_n)}\right) =
				\Phi_P(\bm{x}) 
		\end{multline*}
		where $|v_1\wedge \cdots \wedge v_n|$ is the volume of the
		parallelopiped they form.

		For the induction step, suppose $\{i,j\}\subset \pm[n]$ are
		incomparable in $\Ghat(P)$. Let 
        \[
            P_{i < j} =
		\PLC{P\cup\{\sgn(i)e_i-\sgn(j)e_j\}} \quad\text{and}\quad P_{j < i} =
		\PLC{P\cup\{\sgn(j)e_j-\sgn(i)e_i\}}.
    \] 
        (If $i=j$, one needs to divide
		$\sgn(i)e_i-\sgn(j)e_j$ and $\sgn(j)e_j-\sgn(i)e_i$ by $2$ in type B.)
		The only way $P_{i<j}$ could fail to be a signed poset is if
		$\sgn(j)e_j-\sgn(i)e_i \in P$, a contradiction since it would mean $i$
		and $j$ were comparable in $\Ghat(P)$. A symmetric argument means
		$P_{j<i}$ is also a signed poset. Next, observe that, by construction, 
		\[
		\cL(P) = \cL(P_{i < j}) \sqcup \cL(P_{j <i}),
		\]
		meaning
		\begin{align*}
				\Psi_P(\bm{x}) &= \Psi_{P_{i < j}}(\bm{x}) + \Psi_{P_{j >
				i}}(\bm{x}) \\
				\Phi_P(\bm{x}) &= \Phi_{P_{i < j}}(\bm{x}) + \Phi_{P_{j >
				i}}(\bm{x}) 
    \end{align*}
		Recall that $s(-;\bm{x})$ is a valuation, so one wants to write $\Kprt$ and $\Kpwt$ as a sum of
		$\Krt{P_{i <j}}$, $\Krt{P_{j < i}}$, $\Kwt{P_{i<j}}$ and $\Kwt{P_{j <
		i}}$ and apply the induction assumption to compute $s(-;\bm{x})$ for
		each of these cones and use that to compute $s(\Kprt;\bm{x})$ and
		$s(\Kpwt;\bm{x})$. To that end, define a set $P_{i =j} =\PLC{P \cup
		\{\sgn(i)e_i-\sgn(j)e_j,\sgn(e_j)e_j-\sgn(i)e_i\}}$. The set $P_{i=j}$ is of
		course, not a signed poset, but the definitions of root cone and weight
		cone still make sense.

		Next, observe that
		\begin{align*}
				\Kprt &= \Krt{P_{i < j}} \cap \Krt{P_{j < i}} \\
				\Krt{P_{i=j}} &= \Krt{P_{i < j}} \cup \Krt{P_{j < i}} \\
				\Kwt{P_{i=j}} &= \Kwt{P_{i<j}} \cap \Kwt{P_{j < i}} \\
				\Kpwt &= \Kwt{P_{i<j}} \cup \Kwt{P_{j < i}}.
		\end{align*}
		One then has by the valuative property that
		\begin{align*}
				s(\Krt{P_{i=j}};\bm{x}) &= s(\Krt{P_{i<j}};\bm{x}) + s(\Krt{P_{j
				< i}};\bm{x}) - s(\Kprt;\bm{x}) \\
				s(\Kpwt;\bm{x}) &= s(\Kwt{P_{i<j}};\bm{x}) + s(\Kwt{P_{j
				< i}};\bm{x})- s(\Kwt{P_{i=j}};\bm{x})
		\end{align*}
		Since $\Krt{P_{i=j}}$ is not pointed and $\Kwt{P_{i=j}}$ is not
		full-dimensional,
		\[
		s(\Krt{P_{i=j}};\bm{x})  =
		s(\Kwt{P_{i=j}};\bm{x})  = 0.
		\]
		Then, applying the induction assumption one has
		\[
		\begin{split}
		s(\Kprt;\bm{x}) = \Psi_{P_{i < j}}(\bm{x}) + \Psi_{P_{j < i}}(\bm{x}) =
		\Psi_P(\bm{x}) \\
		s(\Kpwt;\bm{x}) = \Phi_{P_{i<j}}(\bm{x}) + \Phi_{P_{j < i}}(\bm{x}) =
		\Phi_{P}(\bm{x}) 
    \end{split}
		\]
		completing the proof.
\end{proof}

Together with Proposition~\ref{prop:sandsigmaci}, one now has reason to believe
that
the cone and semigroup perspective of~\cite{BoussicaultFerayLascouxReiner2012}
and~\cite{FerayReiner2012} will bear fruit in the signed poset case.

\begin{cor}
Suppose $P \subset \Bn$ (\resp $\Pc \subset \Cn$) is a signed poset. Both
$\Psi_P$ and $\Phi_P$ (\resp $\Psi^\smvee_\Pc$ and $\Phi^\smvee_\Pc$) vanish
when $\Ghat(P)$ has an isotropic component or, equivalently, $\Sigma_P$ has a
balanced component.
\end{cor}

\begin{cor}\label{cor:circlecount}
Suppose $P \subset \Bn$ (\resp $\Pc \subset \Cn$) is a signed poset such that
$\Kprt$ and $\Kpwt$ are pointed, full-dimensional and simplicial. Then
\begin{align*}
    \Psi_P(\bm{x}) &= \dfrac{1}{2}\Psi^\smvee_\Pc(\bm{x}) \stackrel{(1)}{=} |u_1\wedge \cdots
    \wedge u_n| \prod_{i=1}^n
\langle \bm{x},u_u\rangle^{-1}\stackrel{(3)}{=} 2^k \prod_{i=1}^n
\langle \bm{x},u_u\rangle^{-1} \text{ and } \\
\Phi_P(\bm{x}) &=2\Psi^\smvee_\Pc(\bm{x}) \stackrel{(2)}{=} |J_1\wedge \cdots \wedge J_n|\prod_{i=1}^n
\langle \bm{x},J_i\rangle^{-1} \stackrel{(4)}{=} 2^{n-k} \prod_{i=1}^n
\langle \bm{x},J_i\rangle^{-1},
\end{align*}
where $u_1,\ldots,u_n$ are the extreme rays of $\Kprt$, $J_1,\ldots,J_n$ are
the extreme rays of $\Kpwt$ and $k$ is the number of
(non-loop) cycles in the Hasse diagram of $P$.
\end{cor}

\begin{proof}
Equalities ($1$) and ($2$) are consequences of
Corollary~\ref{cor:fulldimpointedsimplicial}. Equality ($3$) is the result of combining
Theorem~\ref{thm:phipsifroms} with~\cite[Lemma 8A.2]{Zaslavsky1982}.
Recalling that $\Kprt$ and $\Kpwt$ are dual cones and the definition of coroot
(see Definitions~\ref{def:weights} and~\ref{def:dualrootsystem}), one sees that
\[
(u_i)(J_i)^{\top} = 2I_n,
\]
since $\langle u_i,J_k\rangle = 2\delta_{ik}$ (for some indexing of the $u_i$
and $J_k$). Therefore,
\[
    |J_1\wedge \cdots \wedge J_n| = \dfrac{2^n}{|u_1\wedge \cdots \wedge u_n|},
\]
giving equality ($4$).
\end{proof}

Corollary~\ref{cor:circlecount} is an analogue of \cite[Proposition
3.2(1)]{BoussicaultFerayLascouxReiner2012}, which shows
that for a poset, the numerator of $\Psi_P$ was $1$ for a tree and $0$ for a
forest, covering the two cases where $\Kprt$ is simplicial, full-dimensional
and pointed. However, instead of being an indicator of connectedness, it counts
the number of cycles in the Hasse diagram.

%
For example, consider the poset $P$ whose Hasse diagram and Fischer poset are
shown in Figure~\ref{fig:detex}.
\begin{figure}[htbp]
    \begin{center}
        \begin{subfigure}[b]{.4\textwidth}
            \centering
            \begin{tikzpicture}[scale=2]
                \node[circle,draw] (1) at (0,0) {$1$};
                \node[circle,draw] (2) at (60:1) {$2$};
                \node[circle,draw] (3) at (1,0) {$3$};
                \draw (1)-- node[very near start,above left] {$+$} node[very near
                end,left] {$-$}  (2)-- node[very near start, right] {$-$}
                node[very near end,right] {$+$} (3)-- node[very near
                start,below] {$+$} node[very near end,below] {$+$} (1);               
            \end{tikzpicture}
            \caption{Hasse diagram}
        \end{subfigure}
        \begin{subfigure}[b]{.4\textwidth}
            \centering
            \begin{tikzpicture}
                \node[vertex] (1) at (0,0) [label=left:$1$] {};
                \node[vertex] (3) at (2,0) [label=left:$3$] {};
                \node[vertex] (-2) at (4,0) [label=left:$-2$] {};
                \node[vertex] (2) at (1,1) [label=left:$2$] {};
                \node[vertex] (-1) at (3,1) [label=left:$-1$] {};
                \node[vertex] (-3) at (5,1) [label=left:$-3$] {};
                \draw (1)--(2);
                \draw (3)--(2);
                \draw (3)--(-1);
                \draw (1)--(-3);
                \draw (-2)--(-1);
                \draw (-2)--(-3);
            \end{tikzpicture}
            \caption{$\Ghat(P)$}
        \end{subfigure}
        \caption{Hasse diagram and $\Ghat(P)$ for
            $P=\{e_1+e_2,e_1-e_2,e_1+e_3,e_1,e_3)\}$}
        \label{fig:detex}
    \end{center}
\end{figure}
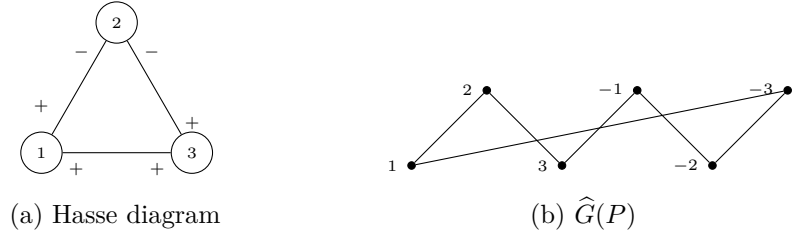
The extreme rays of $\Kprt$ are $(1,-1,0),(0,-1,1),(1,0,1)$ and the extreme rays
of $\Kpwt$ are $(1,1,1),(1,-1,-1),(-1,-1,1)$. Applying
Proposition~\ref{prop:barvinok}, one has that
\[
    \Psi_P(\bm{x}) = \left|\det\begin{pmatrix} 1 & 0 & 1 \\ -1 & -1 & 0 \\ 0 &
            1 & 1 
    \end{pmatrix}\right| \dfrac{1}{(x_1-x_2)(x_3-x_2)(x_1+x_3)} =
    \dfrac{2}{(x_1-x_2)(x_3-x_2)(x_1+x_3)},
\]
and
\begin{multline*}
    \Phi_P(\bm{x}) = \left|\det\begin{pmatrix} 1 & -1 & 1 \\ -1 & -1 & 1 \\ -1
            & 1 & 1
        \end{pmatrix}\right|\dfrac{1}{(x_1+x_2+x_3)(x_1-x_2-x_3)(-x_1-x_2+x_3)}
    \\ =
\dfrac{4}{(x_1+x_2+x_3)(x_1-x_2-x_3)(-x_1-x_2+x_3)}.
\end{multline*}

On the other hand, the second half of Corollary~\ref{cor:circlecount} is a
generalization of Theorem~\ref{thm:chapotonhivert}, which computed $\Phi_P$ for
forests. For example, consider the forest $F$ in Figure~\ref{fig:forest} and
its embedding as a signed poset in Figure~\ref{fig:embeddedforest}.
%
%
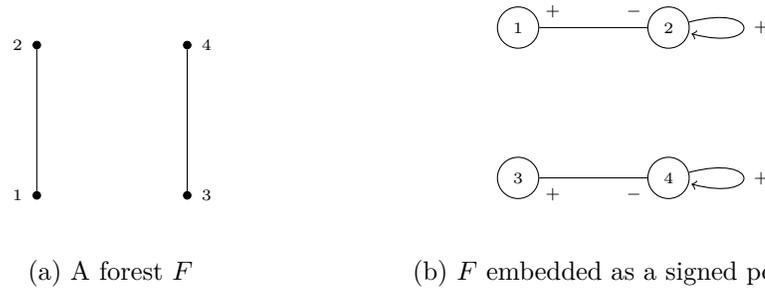
\begin{figure}[htbp]
    \begin{center}
        \begin{subfigure}[b]{.45\textwidth}
            \begin{center}
        \begin{tikzpicture}[scale=2]
            \node[vertex] (1) at (0,0) [label=left:$1$] {};
            \node[vertex] (-1) at (0,1) [label=left:$2$] {};
            \node[vertex] (2) at (1,0) [label=right:$3$] {};
            \node[vertex] (-2) at (1,1) [label=right:$4$] {};
            \draw (1)--(-1);
            \draw (2)--(-2);
        \end{tikzpicture}
    \end{center}
    \caption{A forest $F$}
    \label{fig:forest}
\end{subfigure}
\begin{subfigure}[b]{.45\textwidth}
    \begin{center}
        \begin{tikzpicture}[scale=2]
            \node[circle,draw] (1) at (0,0) {$1$};
            \node[circle,draw] (2) at (1,0) {$2$};
            \node[circle,draw]  (4) at (1,-1) {$4$};
            \node[circle,draw] (3) at (0,-1) {$3$};
            \draw (1)--node [very near start,above] {$+$} node[very near
            end,above]{$-$}(2);
            \draw (4) -- node[very near start,below] {$-$}
            node[very near end,below] {$+$} (3);
            \path (2) edge[loop right] node[midway,right] {$+$} (2);
            \path (4) edge[loop right] node[midway,right] {$+$} (4);
        \end{tikzpicture}
    \end{center}
    \caption{$F$ embedded as a signed poset, $F_B$}
    \label{fig:embeddedforest}
\end{subfigure}
    \end{center}
    \caption{A forest and its embedding as a signed poset}
    \label{fig:forestex}
\end{figure}
Embedding a poset as a signed poset preserves its ideals and linear extensions
and therefore preserves $\Phi$. In particular, this means that $\det(J_i)$,
where $J_i$ runs over the connected non-extensible ideals (i.e.\ all connected
ideals in this case) will be $\pm 1$, so
\[
    \Phi_F(\bm{x}) = \Phi_{F_B}(\bm{x}) =\left|\det\begin{pmatrix} 1 & 0 & 0 &
            0 \\ 1 & 1 & 0 & 0 \\ 0 & 0 & 1 & 0 \\ 0 & 0 & 1 &
            1\end{pmatrix}\right| \dfrac{1}{x_1x_3(x_1+x_2)(x_3+x_4)}=\prod_{j
        \in F}\dfrac{1}{\sum_{i <_F j} x_i}.
\]

\chapter{The Root Cone Semigroup}\label{ch:rootcone}

This chapter considers $\Kprt \cap \Lrt$, the semigroup associated to the root
cone in types $B$ and $C$. After discussion of the generators,
Section~\ref{sec:signedgraphtoricideals} will consider the toric ideal of an
oriented signed graph, but in a more general setting from which one can recover the
signed poset, digraph and poset and graph cases. Section~\ref{sec:rprtci} will
give a sufficient condition for the semigroup ring to be a complete
intersection, enabling the computation of $\Psi_P(\bm{x})$ and
$\Psi^\smvee_{\Pc}(\bm{x})$ via Proposition~\ref{prop:sandsigmaci}.

Recall that $\Ghat(P)$ is equipped with an involution, $\iota$, sending $i$ to
$-i$. This involution will make repeated appearances throughout this chapter.

\begin{prop}\label{prop:rootsemigroupgens}
	The semigroups $\Kprt \cap \Lbrt$ and $\Kprt \cap \Lcrt$ are generated by
    the elements of $P$ and $\Pc$ corresponding to orbits of edges of
    $\Gbhat(P)$ and $\Gchat(\Pc)$, respectively, under the involution.
\end{prop}

\begin{proof}
	In both cases, it suffices to show that every element of $P$ and $\Pc$ lies
in the semigroup generated by the orbits of edges of $\Gbhat(P)$ and
$\Gchat(\Pc)$. First, suppose $\pm e_i \pm e_j \in P$ with $i \ne j$. It is also
an element of $\Pc$. Then, by definition, $\pm i < \mp j$ in both $\Gbhat(P)$
and $\Gchat(\Pc)$. Summing the elements corresponding to edges (see
Table~\ref{table:bidirectededges}) along a chain
from $i$ to $-j$ gives $\pm e_i \pm e_j$ in each case. Next, suppose $\pm e_i \in
P$. Then $\pm 2e_i \in \Pc$. By definition, $\pm i < 0 < \mp i$ in $\Gbhat(P)$. Summing the
elements corresponding to edges along a chain from $\pm i$ to $0$ gives $\pm e_i$. On
the other hand, since $\pm 2e_i \in \Pc$, by definition one has $\pm i < \mp i$ in
$\Gchat(\Pc)$ and adding the elements corresponding to edges along a chain from
$\pm i$ to $\mp i$ gives $\pm 2e_i$.
\end{proof}

This generating set is in fact minimal. All edges of $\Gbhat(P)$ correspond to
edges of the Hasse diagram of $P$, which in turn correspond to the extreme rays
of the root cone. It may be that $\Gchat(\Pc)$ has ``extra'' edges relative to
the Hasse diagram, namely
$\delta i \to -\epsilon j$ and $\epsilon j \to -\delta i$, should the Hasse
diagram contain loops at $i$ and $j$. However, $\delta e_i + \epsilon e_j$
is \emph{not} in the semigroup generated by $\delta2e_i$ and $\epsilon2e_j$, so
the extra edge in the Fischer poset really does correspond to a semigroup
generator.
Proposition~\ref{prop:rootsemigroupgens} generalizes (via the embedding from Section~\ref{sec:embeddingtypea}) the result in type $A$ (Propositions 5.1 and 7.1
of~\cite{BoussicaultFerayLascouxReiner2012}), where $\Kprt \cap \Lrt_A$ was
generated by the roots corresponding to the covering relations of a poset.

\section{The Toric Ideal of an Oriented Signed
    Graph}\label{sec:signedgraphtoricideals}

Recall from Section~\ref{sec:representingposets} that one can view a signed
poset as an oriented signed graph. This section will consider the generators of
the toric ideal associated to an oriented signed graph and then use that
viewpoint to understand the generators of the toric ideals associated to
posets, directed graphs, graphs and signed posets. In other words, we will
describe the toric ideals for \emph{all} affine semigroup rings in $k[t_1^{\pm
1},\ldots,t_n^{\pm 1}]$ generated by monomials of the form $1,t_i^{\pm
1},t_i^{\pm 2},t_i^{\pm 1}t_j^{\pm 1}$.

Suppose $\Sigma$ is a signed graph and $\tau$ a bidirection of $\Sigma$.
(Here $\Sigma$ is allowed to have multiple edges as well as self-loops and
half edges.) The vectors $\tau(i,e)e_i + \tau(j,e) e_j$ as $e=(i,j)$ runs over the
edges of $\Sigma$ (with a half-edge $e=(i,-)$ giving $\tau(i,e)e_i$) generate a semigroup
contained in $\bZ^{|V(\Sigma)|}$. 

\begin{defn}\label{def:signedgraphtoricideal}
		Suppose $\Sigma$ is a signed graph and $\tau$ a bidirection of
		$\Sigma$.
Define a polynomial ring $S_\Sigma = k[U_e]$ where $e$
runs over the (bidirected) edges of $\Sigma$. The \emph{toric ideal} of $\Sigma$ is the kernel of
the map $\phi \colon k[U_e] \to k[t_1^{\pm 1},\ldots, t_n^{\pm 1}]$ defined by
\[
    \phi(U_e) =
    \begin{cases}
        t_i^{\tau(i,e)}t_j^{\tau(j,e)} & e=(i,j), \quad i \ne j \\
        t_i^{\tau(i,e)} & e=(i,-) \text{ a half edge} \\
        t_i^{2\tau(i,e)} & e=(i,i) \text{ a negative loop} \\
        1 & e=(i,i) \text{ a positive loop}
    \end{cases}
\]
The convention will be to put a bar over negative coordinates to save
space. 
\end{defn}

As an example, consider the oriented signed graph in
Figure~\ref{fig:orsignedgraph}.
\begin{figure}[tbp]
    \begin{center}
        \begin{tikzpicture}[every loop/.style={}]
            \node[circle,draw] (1) at (0,0) {$1$};
            \node[circle,draw] (2) at (2,0) {$2$};
            \node[circle,draw] (3) at (-60:2) {$3$};
            \draw (1)--node[very near start,above] {$+$} node[very near
            end,above] {$-$} (2);
            \draw (2)--node[very near start,right] {$+$} node[very near
            end,right] {$-$} (3);
            \draw (1)--node[very near start, left] {$+$} node[very near
            end,left] {$-$} (3);
            \path (1) edge[loop left] node[very near start,below] {$+$}
            node[very near end, above] {$-$} (1);
            \path (3) edge[loop below] node[very near start,right] {$-$}
            node[very near end,left] {$-$} (3);
            \draw (2)--node[very near start,above] {$+$}(3,0);
        \end{tikzpicture}
    \end{center}
    \caption{An oriented signed graph}
    \label{fig:orsignedgraph}
\end{figure}
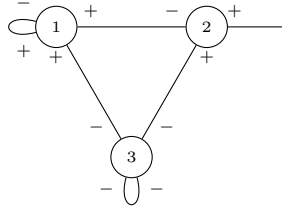
In this case, 
\[
    S_\Sigma = k[U_1,U_2,U_{\m3},U_{1\m2},U_{2\m3},U_{1\m3}]
\]
and $\phi$ is defined by
\begin{alignat*}{1}
U_1 &\mapsto 1 \\
U_2 &\mapsto t_2 \\
U_{\m3} & \mapsto t_3^{-2} \\
U_{1\m2} & \mapsto t_1t_2^{-1} \\
U_{2\m3} & \mapsto t_2t_3^{-1} \\
U_{1\m3} &\mapsto t_1t_3^{-1} \\
\end{alignat*}
The toric ideal $I_\Sigma$ is then
$(U_{1\m2}U_{2\m3}-U_{1\m3},U_2^2U_{\m3}-U_{2\m3}^2,U_1-1)$. (See Figure~\ref{fig:macaulayex} for an example of the computation with Macaulay2~\cite{M2} and 4ti2~\cite{4ti2}.)

\lstset{language=M2}
\begin{figure}
\begin{lstlisting}[breaklines=true,basicstyle=\ttfamily]
i1 : loadPackage "FourTiTwo";
A = transpose matrix {{1,0,0},{0,1,0},{0,0,-1},{1,-1,0},{0,1,-1},{1,0,-1}};
S = QQ[U1,U2,Um3,U1m2,U2m3,U1m3,Degrees=>entries transpose A];
I = toricGroebner(A,S)

              3        6
o2 : Matrix ZZ  <--- ZZ

i4 : using temporary file name /tmp/M2-22805-3
-------------------------------------------------
4ti2 version 1.3.2, Copyright (C) 2006 4ti2 team.
4ti2 comes with ABSOLUTELY NO WARRANTY.
This is free software, and you are welcome
to redistribute it under certain conditions.
For details, see the file COPYING.
-------------------------------------------------
Using 64 bit integers.
4ti2 Total Time:  0.00 secs.

o4 = ideal (- U1*Um3 + U1m3, - U1 + U2*U1m2, - U2*Um3 + U2m3)

o4 : Ideal of S
\end{lstlisting}
\caption{Computing a toric ideal with Macaulay2}
\label{fig:macaulayex}
\end{figure}

Understanding the generators of the toric ideal of an oriented signed graphs
requires the notion of the signed covering of an oriented signed graph.

\begin{defn}\label{def:signedcovering}
Given a signed graph $\Sigma = (\Gamma,\sigma)$ oriented by $\tau$, define a
directed graph $\widetilde{\Sigma}$ whose vertices are $V(\Sigma) \times
\{\pm\}$ and whose edges are determined by the edges of $\Sigma$ according to
Table~\ref{tab:signedcover}. The graph $\widetilde{\Sigma}$ is the \emph{signed covering} of $\Sigma$.
\end{defn}
\begin{table}
    \begin{center}
    \begin{tabular}{ll}
        edge in $\Sigma$ & edge(s) in $\widetilde{\Sigma}$ \\
            \begin{tikzpicture}
        \node[circle,draw] (1) at (0,0) {$i$};
        \node[circle,draw] (2) at (2,0) {$j$};
        \draw (1)--(2) node[very near start,above] {$+$} node[very near
        end,above] {$-$};
    \end{tikzpicture} & $i \to j$, $-j \to -i$ \\
        \begin{tikzpicture}
 \node[circle,draw] (1) at (0,0) {$i$};
        \node[circle,draw] (2) at (2,0) {$j$};
        \draw (1)--(2) node[very near start,above] {$+$} node[very near
        end,above] {$+$};
    \end{tikzpicture} & $i \to -j$, $j \to -i$ \\
     \begin{tikzpicture}
 \node[circle,draw] (1) at (0,0) {$i$};
        \node[circle,draw] (2) at (2,0) {$j$};
        \draw (1)--(2) node[very near start,above] {$-$} node[very near
        end,above] {$-$};
    \end{tikzpicture} & $-i \to j$, $-j \to i$ \\
    \begin{tikzpicture}[every loop/.style={in=135,out=90,looseness=3}]
        \node[circle,draw] (1) at (0,0) {$i$};
        \path (1) edge [loop above] node[very near start,right]{$+$} node[very
        near end,left]{$+$} ();
    \end{tikzpicture} & $i \to -i$ \\
    \begin{tikzpicture}[every loop/.style={in=135,out=90,looseness=3}]
        \node[circle,draw] (1) at (0,0) {$i$};
        \path (1) edge [loop above] node[very near start,right]{$-$} node[very
        near end,left]{$-$} ();
    \end{tikzpicture} & $-i \to i$\\
    \begin{tikzpicture}[every loop/.style={in=135,out=90,looseness=3}]
        \node[circle,draw] (1) at (0,0) {$i$};
        \path (1) edge [loop above] node[very near start,right]{$+$} node[very
        near end,left]{$-$} ();
    \end{tikzpicture} & self-loops at $i$, $-i$ \\
    \begin{tikzpicture}
        \node[circle,draw] (1) at (0,0) {$i$};
        \draw (1) -- node[very near start, above] {$+$} (1.5,0); 
    \end{tikzpicture} & half edges $i \to$, $\to -i$ \\
\begin{tikzpicture}
        \node[circle,draw] (1) at (0,0) {$i$};
        \draw (1) -- node[very near start, above] {$-$} (1.5,0); 
    \end{tikzpicture} & half edges $-i \to$, $\to i$ \\
    \end{tabular}
\end{center}
\caption{Correspondence between edges in $\Sigma$ and $\widetilde{\Sigma}$}
\label{tab:signedcover}
\end{table}
When $\widetilde{\Sigma}$ contains half edges, it will be convenient to imagine
an extra vertex $0$ as the other end of the half edges. With this extra vertex,
when $P \subset \Bn$ is a signed poset, $\Gbhat(P)$ is the signed cover of a
signed graph. $\Gchat(\Pc)$ is also the signed cover of a signed graph, with
or without the dummy $0$, as there are no half edges. (See
Definition~\ref{def:fischerposets} for the definitions of $\Gbhat(P)$ and
$\Gchat(\Pc)$.) Like $\Ghat(P)$, a signed covering $\widetilde{\Sigma}$ is
equipped with an involutive digraph anti-automorphism sending $i$ to $-i$ (and fixing $0$). Figure~\ref{fig:signedcoverex} shows the signed covering of the
oriented signed graph in Figure~\ref{fig:orsignedgraph}.

\begin{figure}[htbp]
    \begin{center}
        \begin{tikzpicture}[every node/.style={font=\normalsize}]
            \node (1) at (0,0) {$1$};
            \node (2) at (60:2) {$2$};
            \node (3) at (2,0) {$3$};
            \begin{scope}[xshift=4cm]
                \node (-3) at (0,0) {$-3$};
                \node (-2) at (60:2) {$-2$};
                \node (-1) at (2,0) {$-1$};
            \end{scope}
            \node (0) at (3,3) {$0$};
            \draw[->] (1)--(2);
            \draw[->] (2)--(3);
            \draw[->] (2)--(0);
            \draw[->] (-2)--(-1);
            \draw[->] (-3)--(-1);
            \draw[->] (-3)--(-2);
            \draw[->] (0)--(-2);
            \draw[->] (1)--(3);
            \path (1) edge[loop left] ();
            \path (-1) edge[loop right] ();
            \draw[->] (-3)--(3);
        \end{tikzpicture}
    \end{center}
    \caption{Signed covering of the oriented signed graph in Figure~\ref{fig:orsignedgraph}}
    \label{fig:signedcoverex}
\end{figure}
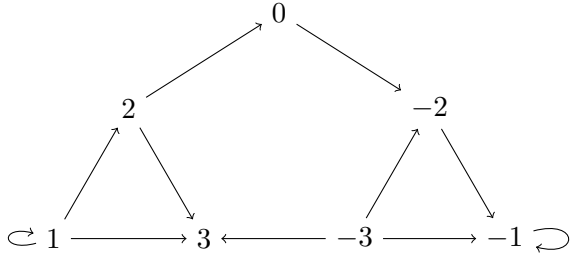

\begin{defn}\label{def:cyclepartition}
		Suppose $\Sigma$ is a signed graph and $\widetilde{\Sigma}$ is its signed covering.
Consider a cycle $C$ in $\widetilde{\Sigma}$ and orient it in some way (i.e.\
choose a direction in which to traverse $C$). This
partitions the edges of $C$ into $W \sqcup A$, where $W$ is the set of edges
such that the orientation is consistent with the direction of the edge in
$\widetilde{\Sigma}$ and $A$ consists of the edges oriented opposite their
direction in $\widetilde{\Sigma}$. Say $C$ is \emph{fixed orientation-wise} by
the involution if $\iota(C) = C$ as edges and $W_C=W_{\iota(C)}$ and $A_C = A_{\iota(C)}$.
\end{defn}

For example, consider the cycle $3 \to 1 \to 2 \to 0 \to -2
\to -1 \to -3 \to 3$ in the graph in
Figure~\ref{fig:signedcoverex}. Then 
\[
W_C = \{(1, 2), (2 , 0), (0, -2), (-2, -1), (-3, 3)\} \quad\text{and}\quad A_C
= \{(3 , 1), (-1,
-3)\}.
\]While this cycle is fixed \emph{edgewise} by the involution, $\iota(C)$ is
oriented in the opposite direction: $-3 \to -1 \to -2 \to 0 \to
2 \to 1 \to 3 \to -3$. Thus $W_{\iota(C)}=A_C$ and $A_{\iota(C)} = W_C$, so
$C$ is \emph{not} fixed orientation-wise. On the other hand, consider the signed
covering in Figure~\ref{fig:anothersignedcovering} (which is actually
$\Gchat(\Pc)$ for $\Pc = \{+e_1-e_2,+e_1+e_2,+2e_1\}$). It has a single cycle
$1 \to 2 \to -1 \to -2 \to 1$, which \emph{is} fixed
orientation-wise by the involution, with $W_C = W_{\iota(C)} =\{(1, 2),( 2, -1)\}$
and $A_C=A_{\iota(C)}=\{(1, -2), (-2 , -1)\}$.

\begin{figure}[htbp]
    \begin{center}
        \begin{tikzpicture}[shorten >=1pt]
            \node[vertex] (1) at (0,-1) [label=left:$1$] {};
            \node[vertex] (2) at (-1,0) [label=left:$2$] {};
            \node[vertex] (-2) at (1,0) [label=right:$-2$] {};
            \node[vertex] (-1) at (0,1) [label=left:$-1$] {};
            \draw[->] (1)--(2);
            \draw[->] (2)--(-1);
            \draw[->] (1)--(-2);
            \draw[->] (-2)--(-1);
        \end{tikzpicture}
    \end{center}
    \caption{Another signed covering}
    \label{fig:anothersignedcovering}
\end{figure}
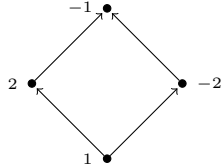

\begin{lem}\label{lem:fixedcycleoppositeedges}
Suppose $\widetilde{\Sigma}$ is the signed covering of an oriented signed
graph $\Sigma$. Suppose $C$ is a cycle in $\widetilde{\Sigma}$ fixed
orientation-wise by the involution. Then if $e \in W_C$ is an edge, $\iota(e) \in
A_C$.
\end{lem}

\begin{proof}
	Suppose $e \in W_C$ and $\iota(e) \in W_C$. There are two cases, when $e$
    involves $0$ and when it doesn't. First, suppose $e$ does not have $0$ as
    an endpoint. Then orienting $C$, one has
	\[
		\cdots \rightarrow i \stackrel{e}{\rightarrow} j \rightarrow \cdots \rightarrow -j \stackrel{\iota(e)}{\rightarrow} -i \rightarrow \cdots
	\]
	Applying the involution gives
	\[
		\cdots \rightarrow -i \rightarrow -j \rightarrow \cdots \rightarrow j\rightarrow i \rightarrow \cdots
	\]
	so $e, \iota(e) \in A_{\iota(C)}$, meaning $C$ was not fixed orientation-wise by the involution.

    On the other hand, suppose $e$ does have $0$ as an endpoint. Without loss of
    generality, one may assume one has
    \[
		\cdots \rightarrow i \stackrel{e}{\rightarrow} 0 \stackrel{\iota(e)}{\rightarrow} -i \rightarrow \cdots
	\]
	Applying the involution gives
	\[
		\cdots \rightarrow -i \rightarrow 0 \rightarrow i \rightarrow \cdots,
	\]
    so $e,\iota(e) \in A_{\iota(C)}$, meaning $C$ was not fixed orientation-wise by the involution.
\end{proof}

If $C$ is a cycle in $\widetilde{\Sigma}$, it gives rise to a relation in the semigroup:
\begin{equation}\label{eq:rtsemigrouprelation}
 \sum_{\alpha \in W(C)} \alpha = \sum_{\alpha \in A(C)} \alpha.
\end{equation}

Recall that $\phi$ was the map from from $S_\Sigma = k[U_e]$ to the Laurent
polynomial ring sending $U_{ij}$ to $t_i^{\tau(i,e)}t_j^{\tau(j,e)}$ for the
edge $e = (i,j)$ (see Definition~\ref{def:signedgraphtoricideal}).

\begin{thm}\label{thm:toricideal}
    Suppose $\Sigma$ is an oriented signed graph.
	The toric ideal $I_\Sigma = \ker \phi$ is generated by \emph{cycle binomials}
	\[
		U(C) = \prod_{e \in W(C)} U_e - \prod_{e \in A(C)} U_e,
	\]
	where $C$ runs over the cycles of $\widetilde{\Sigma}$ not fixed by the
    involution. 
\end{thm}

Consider the oriented signed graph in Figure~\ref{fig:orsignedgraph}, whose
signed covering is in Figure~\ref{fig:signedcoverex}. Recall from the start of
this section that the toric ideal is 
\[
    I_\Sigma =
(U_{1\m2}U_{2\m3}-U_{1\m3},U_2^2U_{\m3}-U_{2\m3}^2,U_1-1).
\]
The signed covering $\widetilde{\Sigma}$ has five
orbits of cycles not fixed orientation-wise by the involution. Their cycle
binomials are:
\[
U_1 - 1 ,\quad U_{1\m2}U_{2\m3}-U_{1\m3},\quad U_{1\m2}^2U_2^2U_{\m3}-U_{1\m3}^2,\quad
U_2^2U_{\m3}-U_{2\m3}^2,\quad
U_{1\m2}U_2^2U_{\m3}-U_{2\m3}U_{1\m3}.
\]
One needs to check that these five binomials generate $I_\Sigma$ and not some
larger ideal. One can use Macaulay2 to find the following relations which show
that the five cycle binomials do generate $I_\Sigma$ by writing the two
``extra'' generators in terms of the other three generators:
\[
    \begin{split}
U_{2\m3}(U_{1\m2}U_{2\m3}-U_{1\m3}) + U_{1\m2}(U_{\m3}U_2^2-U_{2\m3}^2) &=
U_{1\m2}U_2^2U_{\m3}-U_{2\m3}U_{1\m3} \\
(-U_{1\m2}U_{2\m3}-U_{1\m3})(U_{1\m3}-U_{1\m2}U_{2\m3}) +
U_{1\m2}^2(U_{\m3}U_2^2-U_{2\m3}^2) &=
U_{1\m2}^2U_{\m3}U_2^2-U_{1\m3}^2.
    \end{split}
\]

\begin{prop}
	Suppose $C$ is a cycle in $\widetilde{\Sigma}$ fixed orientation-wise by the involution. Then its cycle binomial is zero.
\end{prop}

\begin{proof}
		\emph{A priori}, one has that $\iota(W_C) = A_{\iota(C)}$ and $\iota(A_C) = W_{\iota(C)}$. If a cycle
$C$ is fixed by the involution, one has that $W_C=W_{\iota(C)}$ by definition, but
this means $W_{\iota(C)}=\iota(A_{\iota(C)})$ so the two sums in
(\ref{eq:rtsemigrouprelation}) are over the same set, meaning the cycle
binomial is $0$.
\end{proof}

It will occasionally be convenient to think of \emph{all} cycle binomials as
generating the toric ideal. This is acceptable since the proposition shows the
extra cycle binomials are $0$.

\begin{proof}[Proof of Theorem~\ref{thm:toricideal}]
	Toric ideals are generated by binomials $U^\alpha - U^\beta$ such that
    $\alpha$ and $\beta$ have disjoint support. (See~\citet[Corollary
    4.3]{Sturmfels1996a}.) One must then show that if such a $U^\alpha -
    U^\beta$ is in $I_\Sigma$, it is in the ideal generated by the cycle
    binomials of $\widetilde{\Sigma}$.

Suppose $U^\alpha - U^\beta$ is in $I_\Sigma$ and $U^\alpha$ and $U^\beta$ have
disjoint support. The binomial corresponds to a relation
\[
    \sum a_{\delta i,\epsilon j}(\delta e_i-\epsilon e_j) =\sum b_{\delta i
       \epsilon j}(\delta e_i-\epsilon e_j),
\]
where $\delta,\epsilon \in \{\pm\}$ and, for ease of notation, one takes
$e_0=0$. One builds a digraph on vertex set $\pm[n]$ as follows. For each term $\delta
e_i-\epsilon e_j$ in the left hand sum, take $a_{\delta i,\epsilon j}$ copies
of the directed edge
$\delta i \to \epsilon j$ and $a_{\delta i,\epsilon j}$ copies of $-\epsilon j
\to -\delta i$. For each  term $\delta
e_i-\epsilon e_j$ in the right hand sum, take $b_{\delta i,\epsilon j}$ copies of
$\delta i \leftarrow \epsilon j$ and $b_{\delta i,\epsilon j}$ copies of $-\epsilon j
\leftarrow -\delta i$. In other words, one gets $a_{\delta i,\epsilon j}$ edges
    with orientation coinciding with that of $\widetilde{\Sigma}$ and
    $b_{\delta i,\epsilon j}$ edges with orientation opposite that of
    $\widetilde{\Sigma}$.

    One has constructed a digraph where the in-degree equals the out-degree at
    each vertex, so one has a collection of cycles, possibly with multiplicity.
    By construction, each edge and its image under the involution are oriented
    the same way relative to their direction in $\widetilde{\Sigma}$.
    Consequently, by Lemma~\ref{lem:fixedcycleoppositeedges}, none of these
    cycles can be fixed orientation-wise by the involution. Furthermore, if $C$ is a cycle fixed
    \emph{edgewise} by the involution, one will have an even number of copies
    of $C$.

    Call these pairs of cycles (possibly with multiplicity) $C_1,\ldots,C_k$.
    Induct on $k$. If $k=1$, then $U^\alpha - U^\beta$ was itself a cycle binomial. Suppose that binomials with no common factors corresponding to $k-1$ pairs of cycles lie in the ideal generated by the cycle binomials. Then by induction one has
	\begin{multline*}
			U^\alpha - U^\beta = \prod_{i=1}^k U_{W(C_i)} -\prod_{i=1}^k U_{A(C_i)}  \\
			= (U_{W(C_1)}-U_{A(C_1)})\prod_{i=2}^kU_{W(C_i)} + U_{A(C_1)}\left(\prod_{i=2}^k U_{W(C_i)} -\prod_{i=2}^k U_{A(C_k)}\right),
	\end{multline*}
	where $U_{W(C_i)} = \prod_{e \in W(C_i)} U_e$ and $U_{A(C_i)} =\prod_{e \in A(C_i)}$. Consequently, $U^\alpha -U^\beta$ lies in the ideal generated by the cycle binomials. Thus one has that the toric ideal is generated by the cycle binomials.
\end{proof}

\subsection{The toric ideal of the root cone
    semigroup}\label{subsec:rootconetoricideal}

As noted earlier, examining the definitions of $\Gbhat(P)$ and $\Gchat(\Pc)$
(Definition~\ref{def:fischerposets}), one sees that $\Gbhat(P)$ and $\Gchat(P)$
are both the signed coverings of some oriented signed graph (at least when one
anchors the half edges to a $0$ in the case of $\Gbhat(P)$). Recall from
Proposition~\ref{prop:rootsemigroupgens} that the semigroups $\Kprt \cap \Lbrt$
and $\Kprt \cap \Lcrt$ are generated by the edges of these oriented signed
graphs. Denote the two semigroup rings $k[\Kprt \cap \Lbrt]$ and $k[\Kprt \cap
\Lcrt]$ by $\Rprt$ and $\Rrt{\Pc}$, respectively. In this case,
Theorem~\ref{thm:toricideal} gives the following.

\begin{cor}\label{cor:rttoricideals}
    Suppose $P \subset \Bn$ (resp.\ $\Pc \subset \Cn$) is a signed poset.
Let $\Srt{P} = k[U_{\delta i,\epsilon j}]$ and $\Srt{\Pc}=k[U_{\delta i,
    \epsilon j}]$, where the pairs $(\delta
i,\epsilon j)$ run over orbits under the involution of edges $\delta i \to
\epsilon j$ in $\Gbhat(P)$ and $\Gchat(\Pc)$, respectively. Then the toric
ideals $\Irt{P}$ and $\Irt{\Pc}$ are generated by the cycle binomials
corresponding to cycles in $\Gbhat(P)$ and $\Gchat(\Pc)$, respectively,  not
fixed orientation-wise by the involution and
\[
    \Rprt = \Srt{P} /\Irt{P} \quad \text{and} \quad \Rrt{\Pc} =
    \Srt{\Pc}/\Irt{\Pc}.
\]
\end{cor}

\begin{figure}
    \begin{center}
    \begin{subfigure}[b]{.4\textwidth}
        \begin{center}
        \begin{tikzpicture}
            \node[vertex] (1) at (0,0) [label=left:$1$] {};
            \node[vertex] (4) at (0,1) [label=left:$4$] {};
            \node[vertex] (0) at (0,2) [label=left:$0$] {};
            \node[vertex] (-4) at (0,3) [label=left:$-4$] {};
            \node[vertex] (-1) at (0,4) [label=left:$-1$] {};
            \node[vertex] (2) at (-1,1.3) [label=left:$2$] {};
            \node[vertex] (-3) at (1,1.3) [label=right:$-3$] {};
            \node[vertex] (3) at (-1,2.7) [label=left:$3$] {};
            \node[vertex] (-2) at (1,2.7) [label=right:$-2$] {};
            \draw (1)--(2)--(3)--(-1)--(-4)--(0)--(4)--(1)--(-3)--(-2)--(-1);
        \end{tikzpicture}
    \end{center}
    \caption{$\Gbhat(P)$}
    \end{subfigure}
        \begin{subfigure}[b]{.4\textwidth}
        \begin{center}
        \begin{tikzpicture}
            \node[vertex] (1) at (0,0) [label=left:$1$] {};
            \node[vertex] (4) at (0,1) [label=left:$4$] {};
            \node[vertex] (-4) at (0,3) [label=left:$-4$] {};
            \node[vertex] (-1) at (0,4) [label=left:$-1$] {};
            \node[vertex] (2) at (-1,1.3) [label=left:$2$] {};
            \node[vertex] (-3) at (1,1.3) [label=right:$-3$] {};
            \node[vertex] (3) at (-1,2.7) [label=left:$3$] {};
            \node[vertex] (-2) at (1,2.7) [label=right:$-2$] {};
            \draw (1)--(2)--(3)--(-1)--(-4)--(4)--(1)--(-3)--(-2)--(-1);
        \end{tikzpicture}
    \end{center}
    \caption{$\Gchat(\Pc)$}
    \end{subfigure}
\caption{$\Gbhat(P)$ and $\Gchat(\Pc)$ for $P =
\{+e_1-e_2,+e_1+e_3,+e_1-e_3,+e_1+e_2,+e_1-e_4,+e_1+e_4,+e_1,+e_4\}$}
\label{fig:toricidealex}
\end{center}
\end{figure}

Consider $P =
\{+e_1-e_2,+e_1+e_3,+e_1-e_3,+e_1+e_2,+e_1-e_4,+e_1+e_4,+e_1,+e_4\}$.
Figure~\ref{fig:toricidealex} shows $\Gbhat(P)$ and $\Gchat(\Pc)$. Then 
\[
\Srt{P} = k[U_{12},U_{23},U_{1\bar{3}},U_{14},U_{4\m4}]\quad\text{and} \quad
\Srt{\Pc} =k[U_{12},U_{23},U_{1\m3},U_{14},U_{40}].
\]
Both $\Gbhat(P)$ and $\Gchat(\Pc)$
have a cycle $1 \to 2 \to 3 \to -1 \to -2 \to -3 \to 1$
fixed orientation-wise by the involution, so its cycle binomial is $0$.
Figure~\ref{fig:toricidealcycleex} shows the cycles constructed in the proof of
Theorem~\ref{thm:toricideal}.

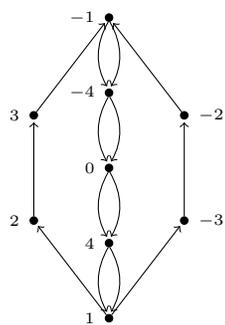
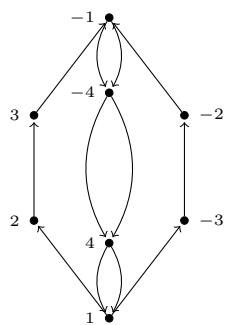
\begin{figure}
    \begin{center}
    \begin{subfigure}[b]{.4\textwidth}
        \begin{center}
        \begin{tikzpicture}[->,shorten >=1pt]
            \node[vertex] (1) at (0,0) [label=left:$1$] {};
            \node[vertex] (4) at (0,1) [label=left:$4$] {};
            \node[vertex] (0) at (0,2) [label=left:$0$] {};
            \node[vertex] (-4) at (0,3) [label=left:$-4$] {};
            \node[vertex] (-1) at (0,4) [label=left:$-1$] {};
            \node[vertex] (2) at (-1,1.3) [label=left:$2$] {};
            \node[vertex] (-3) at (1,1.3) [label=right:$-3$] {};
            \node[vertex] (3) at (-1,2.7) [label=left:$3$] {};
            \node[vertex] (-2) at (1,2.7) [label=right:$-2$] {};
            \draw (1)--(2);
            \draw (2)--(3);
            \draw (3)--(-1);
            \path (-1.260) edge[bend right] (-4.100);
            \path (-4.260) edge[bend right](0.100);
            \path (0.260) edge[bend right] (4.100);
            \path (4.260) edge[bend right] (1.100);
            \path (-1.280) edge[bend left] (-4.80);
            \path (-4.280) edge[bend left](0.80);
            \path (0.280) edge[bend left](4.80);
            \path (4.280) edge[bend left] (1.80);
            \draw (1)--(-3);
            \draw (-3)--(-2);
            \draw (-2)--(-1);
        \end{tikzpicture}
    \end{center}
    \caption{$\Gbhat(P)$}
    \end{subfigure}
        \begin{subfigure}[b]{.4\textwidth}
        \begin{center}
        \begin{tikzpicture}[->,shorten >=1pt]
            \node[vertex] (1) at (0,0) [label=left:$1$] {};
            \node[vertex] (4) at (0,1) [label=left:$4$] {};
            \node[vertex] (-4) at (0,3) [label=left:$-4$] {};
            \node[vertex] (-1) at (0,4) [label=left:$-1$] {};
            \node[vertex] (2) at (-1,1.3) [label=left:$2$] {};
            \node[vertex] (-3) at (1,1.3) [label=right:$-3$] {};
            \node[vertex] (3) at (-1,2.7) [label=left:$3$] {};
            \node[vertex] (-2) at (1,2.7) [label=right:$-2$] {};
            \draw (1)--(2);
            \draw (2)--(3);
            \draw (3)--(-1);
            \path (-1)edge[bend right](-4);
            \path (-4)edge[bend right](4);
            \path (4)edge[bend right](1);
            \path (-1)edge[bend left](-4);
            \path (-4)edge[bend left](4);
            \path (4)edge[bend left](1);
            \draw (1)--(-3);
            \draw (-3)--(-2);
            \draw (-2)--(-1);
        \end{tikzpicture}
    \end{center}
    \caption{$\Gchat(\Pc)$}
    \end{subfigure}
\caption{The directed (multi-)graph from the proof of
    Theorem~\ref{thm:toricideal}}
\label{fig:toricidealcycleex}
\end{center}
\end{figure}

In both $\Gbhat(P)$ and $\Gchat(\Pc)$, there is a single pair $(C,\iota(C))$, so the
toric ideals are principal. One then
has 
\[
    \Irt{P} = (U_{12}U_{23}U_{1\m3} - U_{14}^2U_{40}^2) \quad \text{and} \quad
    \Irt{\Pc} = (U_{12}U_{23}U_{1\m3} - U_{14}^2U_4).
\]

\subsection{Toric ideals of posets and directed
    graphs}\label{subsec:posetgraphtoricideal}

Another application of Theorem~\ref{thm:toricideal} is to posets and directed
graphs. The semigroup associated to a directed graph is generated by $e_i -
e_j$ for each edge $i \to j$. When the directed graph is the Hasse diagram of a
poset, this is the root cone semigroup in type A
of~\cite{BoussicaultFerayLascouxReiner2012}. Using a similar argument to the
proof of Theorem~\ref{thm:toricideal} (but simpler due to having the Hasse
diagram rather than the signed covering), Boussicault, \feray, Lascoux and Reiner showed
in~\cite[Proposition 8.1]{BoussicaultFerayLascouxReiner2012} that, given a
poset, its toric ideal is generated by binomials corresponding to cycles in the
Hasse diagram. Their argument easily generalizes to all directed graphs. Earlier, Gitler, Reyes and Villarreal obtained the same result
for directed graphs~\cite[Proposition 4.3]{GitlerReyesVillarreal2010} using a
result of Sturmfels~\cite[\S5]{Sturmfels1992}.

To see how Theorem~\ref{thm:toricideal} works for directed graphs, first
observe that in a signed graph $\Sigma$ with an orientation coming from a directed
graph $G$, all edges are signed $+$, meaning all cycles are balanced.
Furthermore, the signed covering of $\Sigma$ is $G \sqcup \iota(G)$ and a cycle in
$G$ corresponds to a pair of cycles $(C, \iota(C))$ in $G \sqcup \iota(G)$, and no cycle
in the signed covering is fixed orientation-wise by the involution.
Consequently, the generating set from Theorem~\ref{thm:toricideal} is the same
as the generating set from~\cite{BoussicaultFerayLascouxReiner2012}
and~\cite{GitlerReyesVillarreal2010}.

As an example, consider the digraph (and poset, depending on perspective) shown
in Figure~\ref{subfig:toricposetex}. The associated oriented signed graph is
shown in Figure~\ref{subfig:toricposetosg} and its signed covering in
Figure~\ref{subfig:toricposetsc}.
\begin{figure}[htbp]
    \begin{center}
        \begin{tabular}{cc}
                \begin{subfigure}[b]{.4\textwidth}
                    \begin{center}
                    \begin{tikzpicture}[->,shorten >=1pt]
                        \node[vertex] (1) at (0,0) [label=left:$1$] {};
                        \node[vertex] (2) at (-1,1) [label=left:$2$] {};
                        \node[vertex] (3) at (1,1) [label=right:$3$] {};
                        \node[vertex] (4) at (0,2) [label=left:$4$] {};
                        \node[vertex] (5) at (2,2) [label=right:$5$] {};
                        \node[vertex] (6) at (1,3) [label=right:$6$] {};
                        \draw (1)--(2);
                        \draw (1)--(3);
                        \draw (2)--(4);
                        \draw (3)--(4);
                        \draw (3)--(5);
                        \draw (4)--(6);
                        \draw (5)--(6);
                    \end{tikzpicture}
                \end{center}
                \caption{digraph/poset}
                \label{subfig:toricposetex}
                \end{subfigure}
            &
            \begin{subfigure}[b]{.4\textwidth}
                \begin{center}
                    \begin{tikzpicture}[scale=2]
                        \node[circle,draw] (2) at (0,0) {$2$};
                        \node[circle,draw] (4) at (1,0) {$4$};
                        \node[circle,draw] (6) at (2,0) {$6$};
                        \node[circle,draw] (1) at (0,-1) {$1$};
                        \node[circle,draw] (3) at (1,-1) {$3$};
                        \node[circle,draw] (5) at (2,-1) {$5$};
                        \draw (1)--node[very near start,left] {$+$} node[very
                        near end,left] {$-$} (2);
                        \draw (2)-- node[very near start, above] {$+$}
                        node[very near end, above] {$-$} (4);
                        \draw (4)-- node[very near start, above] {$+$}
                        node[very near end, above] {$-$} (6);
                        \draw (5)--node[very near start, right] {$+$} node[very
                        near end, right] {$-$} (6);
                        \draw (1)--node[very near start, below] {$+$} node[very
                        near end,below] {$-$} (3);
                        \draw (3)--node[very near start,below] {$+$} node[very
                        near end,below] {$-$} (5);
                        \draw (3)--node[very near start,left] {$+$} node[very
                        near end,left] {$-$} (4);
                    \end{tikzpicture}
                \end{center}
                \caption{associated oriented signed graph}
                \label{subfig:toricposetosg}
            \end{subfigure}  \\
            \multicolumn{2}{c}{
            \begin{subfigure}{.9\textwidth}
                \begin{center}
                    \begin{tikzpicture}[->,shorten >=1pt]
                        \node[vertex] (1) at (0,0) [label=left:$1$] {};
                        \node[vertex] (2) at (-1,1) [label=left:$2$] {};
                        \node[vertex] (3) at (1,1) [label=right:$3$] {};
                        \node[vertex] (4) at (0,2) [label=left:$4$] {};
                        \node[vertex] (5) at (2,2) [label=right:$5$] {};
                        \node[vertex] (6) at (1,3) [label=right:$6$] {};
                        \draw (1)--(2);
                        \draw (1)--(3);
                        \draw (2)--(4);
                        \draw (3)--(4);
                        \draw (3)--(5);
                        \draw (4)--(6);
                        \draw (5)--(6);
                        \node[vertex] (-2) at (3,2) [label=left:$-2$] {};
                        \node[vertex] (-1) at (4,3) [label=left:$-1$] {};
                        \node[vertex] (-3) at (5,2) [label=right:$-3$] {};
                        \node[vertex] (-4) at (4,1) [label=left:$-4$] {};
                        \node[vertex] (-5) at (6,1) [label=right:$-5$] {};
                        \node[vertex] (-6) at (5,0) [label=left:$-6$] {};
                        \draw (-6)--(-4);
                        \draw (-6)--(-5);
                        \draw (-4)--(-2);
                        \draw (-4)--(-3);
                        \draw (-5)--(-3);
                        \draw (-2)--(-1);
                        \draw (-3)--(-1);
                    \end{tikzpicture}
                \end{center}
                \caption{signed covering}
                \label{subfig:toricposetsc}
            \end{subfigure}}
        \end{tabular}
    \end{center}
    \caption{A digraph, its associated oriented signed graph and signed
        covering}
    \label{fig:something}
\end{figure}
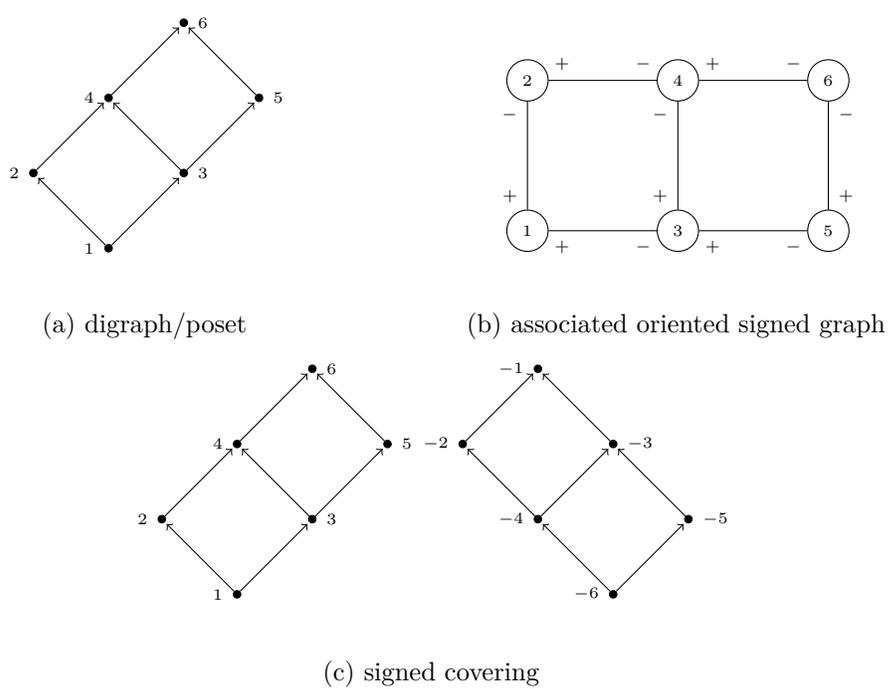
Then $S_G=k[U_{1\m2},U_{1\m3},U_{2\m4},U_{3\m4},U_{3\m5},U_{4\m6},U_{5\m6}]$
and, according to Theorem~\ref{thm:toricideal}, the toric ideal is
\[
I_G =
(U_{1\m2}U_{2\m4}-U_{1\m3}U_{3\m4},U_{3\m4}U_{4\m6}-U_{3\m5}U_{5\m6},U_{1\m2}U_{2\m4}U_{4\m6}-U_{1\m3}U_{3\m5}U_{5\m6}).
\]

\subsection{Toric ideals of graphs}

Suppose $G$ is a (unsigned) graph with no loops or multiple edges. Associated
to $G$ is the semigroup generated by the vectors $e_i+e_j$ where $(i,j)$ runs
over all edges of $G$. This semigroup then gives rise to a toric ideal $I_G$.
Of course, this is the same semigroup as one obtains from an oriented signed
graph $\Sigma$ where $G$ is $\Sigma$'s underlying graph, all edges are signed
$-$ and $\Sigma$ is oriented with all incidences $+$. Consequently,
Theorem~\ref{thm:toricideal} gives a set of generators for $I_G$. This
generating set lies between two other known generating sets.

\begin{defn}\label{def:closedwalk}
A sequence $v_0,\ldots,v_n$ of (not necessarily distinct) vertices in a graph
$G$ is a \emph{closed walk} if $v_0=v_n$ and there is an edge between $v_i$ and
$v_{i+1}$ for $i=0,\ldots,n-1$. It is an \emph{even closed walk} if
$n$ is even. Let $S_G = k[U_e]$ where $e$ runs over the edges of $G$. Associated to an even closed walk $w=\{v_0,\ldots,v_{2k}\}$ is a
binomial 
\[
T_w = U_{v_0v_1}U_{v_2v_3}\cdots U_{v_{2k-2}v_{2k-1}} -
U_{v_1v_2}U_{v_3v_4}\cdots U_{v_{2k-1}v_{2k}}
\]
\end{defn}

\begin{thm}[{Villarreal,~\cite[Proposition
        3.1]{Villarreal1995}}]\label{thm:villarreal}
Suppose $G$ is a graph. Then its toric ideal has this description: 
\[
I_G = (T_w \colon w\text{ an even closed walk in }G).
\]
\end{thm}
This generating set and the generating set from Theorem~\ref{thm:toricideal} do
not necessarily coincide. Consider the graph in Figure~\ref{fig:villarrealex}.
According to Theorem~\ref{thm:villarreal}, the toric ideal is
\[
    (U_{12}U_{13}U_{45}-U_{23}U_{14}U_{15},U_{18}U_{67}-U_{16}U_{78},U_{12}U_{13}U_{45}U_{18}U_{67}-U_{23}U_{14}U_{15}U_{78}U_{16}).
\]
On the other hand, according to Theorem~\ref{thm:toricideal}, it is
\[
    (U_{12}U_{13}U_{45}-U_{23}U_{14}U_{15},U_{18}U_{67}-U_{16}U_{78}).
\]
One has that
\begin{multline*}
U_{12}U_{13}U_{45}(U_{67}U_{18}-U_{16}U_{78}) -
U_{16}U_{78}(U_{23}U_{14}U_{15}-U_{12}U_{13}U_{45}) \\ =
U_{12}U_{13}U_{45}U_{18}U_{67}-U_{23}U_{14}U_{15}U_{78}U_{16},
\end{multline*}
so the two ideals really are equal.

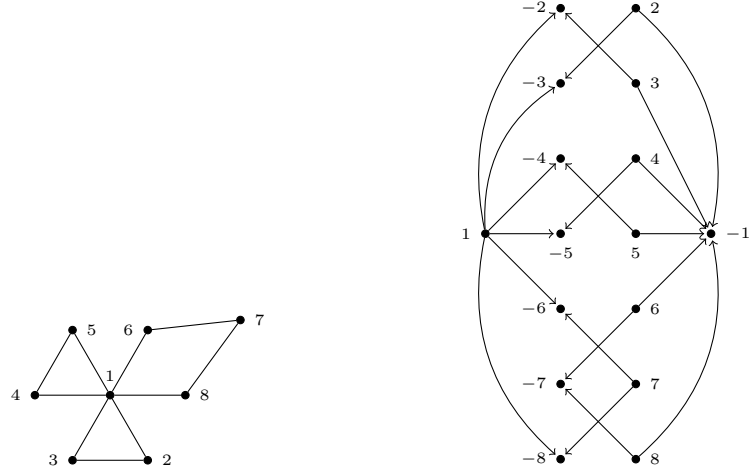
\begin{figure}
    \begin{center}
    \begin{subfigure}[b]{.4\textwidth}
    \begin{center}
        \begin{tikzpicture}
            \node[vertex] (1) at (0,0) [label=above:$1$] {};
            \node[vertex] (6) at (60:1) [label=left:$6$] {};
            \node[vertex] (5) at (120:1) [label=right:$5$] {};
            \node[vertex] (4) at (180:1) [label=left:$4$] {};
            \node[vertex] (3) at (240:1) [label=left:$3$] {};
            \node[vertex] (2) at (300:1) [label=right:$2$] {};
            \node[vertex] (8) at (360:1) [label=right:$8$] {};
            \node[vertex] (7) at (30:2) [label=right:$7$] {};
            \draw (1)--(2)--(3)--(1)--(4)--(5)--(1)--(6)--(7)--(8)--(1);
        \end{tikzpicture}
    \end{center}
    \caption{A graph $G$}
    \end{subfigure}
    \begin{subfigure}[b]{.4\textwidth}
        \begin{center}
            \begin{tikzpicture}[->,shorten >=1pt]
                \node[vertex] (-2) at (0,0) [label=left:$-2$] {};
                \node[vertex] (2) at (1,0) [label=right:$2$] {};
                \node[vertex] (-3) at (0,-1) [label=left:$-3$] {};
                \node[vertex] (3) at (1,-1) [label=right:$3$] {};
                \node[vertex] (-4) at (0,-2) [label=left:$-4$] {};
                \node[vertex] (4) at (1,-2) [label=right:$4$] {};
                \node[vertex] (-5) at (0,-3) [label=below:$-5$] {};
                \node[vertex] (5) at (1,-3) [label=below:$5$] {};
                \node[vertex] (-6) at (0,-4) [label=left:$-6$] {};
                \node[vertex] (6) at (1,-4) [label=right:$6$] {};
                \node[vertex] (-7) at (0,-5) [label=left:$-7$] {};
                \node[vertex] (7) at (1,-5) [label=right:$7$] {};
                \node[vertex] (-8) at(0,-6) [label=left:$-8$] {};
                \node[vertex] (8) at (1,-6) [label=right:$8$] {};
                \node[vertex] (1) at (-1,-3) [label=left:$1$] {};
                \node[vertex] (-1) at (2,-3) [label=right:$-1$] {};
                \path (1) edge[bend left] (-2)
                          edge[bend left] (-3)
                          edge (-4)
                          edge (-5)
                          edge (-6)
                          edge[bend right] (-8);
                \draw (2)--(-3);
                \draw (3)--(-2);
                \draw (5)--(-4);
                \draw (4)--(-5);
                \draw (6)--(-7);
                \draw (7)--(-6);
                \draw (7)--(-8);
                \draw (8)--(-7);
                \path (2) edge[bend left] (-1);
                \draw (3)--(-1);
                \draw (4)--(-1);
                \draw (5)--(-1);
                \draw (6)--(-1);
                \path (8) edge[bend right](-1);
            \end{tikzpicture}
        \end{center}
        \caption{The signed covering of $G$}
    \end{subfigure}
    \caption{A graph and its signed covering} 
    \label{fig:villarrealex}
\end{center}
\end{figure}

Since the signed covering is bipartite, any cycle will correspond to an even
closed walk in $G$. Therefore, the generating set of
Theorem~\ref{thm:toricideal} is contained in the generating set of
Theorem~\ref{thm:villarreal}.

\begin{defn}\label{def:primitivebinomial}
If $I \subset k[x_1,\ldots,x_k]$ is a toric ideal, a \emph{primitive binomial}
is a binomial $\bm{x}^{u^+}-\bm{x}^{u^-} \in I$ such that there does
not exist an $\bm{x}^{v^+}-\bm{x}^{v^-} \in I$ with $ u \ne v$ and
$\bm{x}^{v^+}|\bm{x}^{u^+}$ and $\bm{x}^{v^-}|\bm{x}^{u^-}$.
The set of primitive binomials is called the \emph{Graver basis} and generates
the toric ideal. 
\end{defn}
See~\citet{Sturmfels1996a} for more on the Graver basis.
Ohsugi and Hibi described the Graver basis of the toric ideal of a graph
in~\cite{OhsugiHibi1999}.

\begin{thm}[{Ohsugi and Hibi,~\cite[Lemma
        3.2]{OhsugiHibi1999}}]\label{thm:ohsugihibi}
Let $G$ be a finite connected graph. If $f \in I_G$ is primitive, it is
$f=T_\Gamma$ where $\Gamma$ is one of the following even closed walks:
\begin{enumerate}[label=(\roman{*})]
    \item $\Gamma$ is an even cycle
    \item $\Gamma$ is a pair of odd cycles having exactly one common vertex
    \item $\Gamma = (C_1,\Gamma_1,C_2,\Gamma_2)$ where $C_1$ and $C_2$ are odd
        cycles having no common vertex and where $\Gamma_1$ and $\Gamma_2$ are
        walks of $G$, both of which connect a vertex $v_1$ in $C_1$ to a vertex
        $v_2$ of $C_2$.
\end{enumerate}
\end{thm}

A moment's thought will show that an even closed walk $\Gamma$ corresponding to
a primitive binomial corresponds to a cycle not fixed orientation-wise by the
involution. However, the converse is not true. Consider the graph
in Figure~\ref{fig:ohsugiex}. The cycle $1 \to -2 \to 3 \to -1
\to 2 \to -4 \to 5 \to -6 \to 4 \to -3 \to 1$ in the signed covering is
not fixed orientation-wise by the involution, but does not project to an even
closed walk of any of the forms in Theorem~\ref{thm:ohsugihibi}. The binomial
corresponding to this cycle is
$U_{12}U_{13}U_{24}U_{56}U_{34}-U_{23}U_{12}U_{45}U_{46}U_{13}$, which is
clearly not primitive as the two terms are not relatively prime. Cancelling
$U_{12}U_{13}$ from both terms gives the primitive binomial
$U_{24}U_{56}U_{34}-U_{23}U_{45}U_{46}$, corresponding to the even closed walk
$2 \to 4 \to 6 \to 5 \to 4 \to 3 \to 2$.

\begin{figure}[htbp]
    \begin{center}
    \begin{subfigure}[b]{.4\textwidth}
        \begin{center}
            \begin{tikzpicture}
                \node[circle,draw] (1) at (0,0) {$1$};
                \node[circle,draw] (2) at (60:2) {$2$};
                \node[circle,draw] (3) at (0:2) {$3$};
                \begin{scope}[xshift=2cm]
                    \node[circle,draw] (4) at (60:2) {$4$};
                    \node[circle,draw] (6) at (0:2) {$6$};
                \end{scope}
                \begin{scope}[xshift=4cm]
                    \node[circle,draw] (5) at (60:2) {$5$};
                \end{scope}
                \draw (1)--(2)--(3)--(1);
                \draw (2)--(4)--(3);
                \draw (4)--(5)--(6)--(4);
            \end{tikzpicture}
        \end{center}
        \caption{$G$}
    \end{subfigure}
    \begin{subfigure}[b]{.5\textwidth}
        \begin{center}
            \begin{tikzpicture}[->,shorten >=1pt]
                \node[vertex] (1) at (0,0) [label=left:$1$] {};
                \node[vertex] (-2) at (-1,1) [label=left:$-2$] {};
                \node[vertex] (3) at (-1,2) [label=left:$3$] {};
                \node[vertex] (-1) at (0,3) [label=left:$-1$] {};
                \node[vertex] (2) at (1,2) [label=right:$2$] {};
                \node[vertex] (-3) at (1,1) [label=right:$-3$] {};
                \begin{scope}[xshift=4cm]
                    \node[vertex] (4) at (0,0) [label=right:$4$] {};
                    \node[vertex] (-5) at (-1,1) [label=left:$-5$] {};
                    \node[vertex] (6) at (-1,2) [label=left:$6$] {};
                    \node[vertex] (-4) at (0,3) [label=right:$-4$] {};
                    \node[vertex] (5) at (1,2) [label=right:$5$] {};
                    \node[vertex] (-6) at (1,1) [label=right:$-6$] {};
                \end{scope}
                \draw (1)--(-2);
                \draw (1)--(-3);
                \draw (2)--(-1);
                \draw (2)--(-3);
                \draw (2)--(-4);
                \draw (3)--(-1);
                \draw (3)--(-2);
                \draw (3)--(-4);
                \draw (4)--(-2);
                \draw (4)--(-3);
                \draw (4)--(-5);
                \draw (4)--(-6);
                \draw (5)--(-4);
                \draw (5)--(-6);
                \draw (6)--(-4);
                \draw (6)--(-5);
            \end{tikzpicture}
        \end{center}
        \caption{Signed covering of $G$}
    \end{subfigure}
    \caption{A graph for which the generating sets from
        Theorem~\ref{thm:toricideal} and Theorem~\ref{thm:ohsugihibi} differ}
    \label{fig:ohsugiex}
\end{center}
\end{figure}
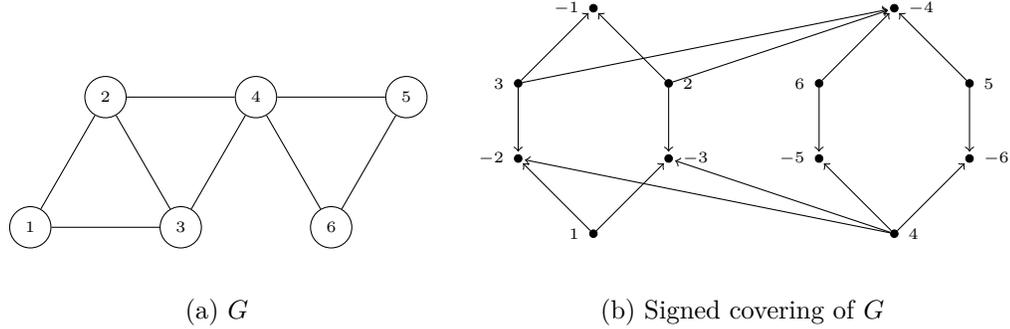

The following proposition summarizes the results of this section, linking the
generating sets of Theorems~\ref{thm:toricideal},~\ref{thm:villarreal},
and~\ref{thm:ohsugihibi}.
\begin{prop}\label{prop:summarytoricidealsgraphs}
    Suppose $G$ is a finite connected graph. There are then three generating
    sets, $S_1 \subset S_2 \subset S_3$
    for the toric ideal $I_G$:
    \begin{itemize}
        \item $S_1 = \{T_w \colon w \text{ an even closed walk in the forms
                of Theorem~\ref{thm:ohsugihibi}}\}$
        \item $S_2 = \{ U(C) \colon C\text{ a cycle in the signed covering}\}$
        \item $S_3 = \{T_w \colon w\text{ an even closed walk}\}$.
    \end{itemize}
    \end{prop}

\section{Some Complete Intersection Rings $\Rprt$}\label{sec:rprtci}
Now that one understands the toric ideal, one can turn one's attention to the
question of computing $\Psi_P$ and $\Psi_P^\smvee$ via complete intersection
presentations.
This section gives factorizations of
\begin{align*}
    \Psi_P(\bm{x}) &= \sum_{w \in \cL(P)}
    w\left(\dfrac{1}{(x_1-x_2)(x_2-x_3)\cdots(x_{n-1}-x_n)x_n}\right) \quad
    \text{and}\quad \\
    \Psi^\smvee_{\Pc}(\bm{x}) &= \sum_{w \in
        \cL(P)}w\left(\dfrac{1}{(x_1-x_2)(x_2-x_3)\cdots(x_{n-1}-x_n)2x_n}\right)
\end{align*}
for a certain class of signed posets using Proposition~\ref{prop:sandsigmaci}.

Recall Theorem~\ref{thm:greene}, where Greene showed that, for a strongly planar poset,
\[
    \Psi_P(\bm{x}) = \dfrac{\prod_{\rho} (x_{\min(\rho)} -
        x_{\max(\rho)})}{\prod_{i \lessdot j}(x_i-x_j)},
\]
where $\rho$ runs over all the regions enclosed by the Hasse diagram of $P$ and $i
\lessdot j$ over the covering relations of the poset. \citet{Boussicault2009} showed that a similar factorization occurs
for posets that are ``gluings of diamonds along chains'' and, in particular,
for
strongly planar posets, recovering Greene's result.
In~\cite{BoussicaultFerayLascouxReiner2012}, Boussicault, \feray, Lascoux and
Reiner gave an algebraic explanation for the disconnecting chains
of~\cite{Boussicault2009} by showing certain posets (including strongly planar
posets) are complete intersections by constructing a regular sequence using the
operation of opening/closing a notch.

The situation for signed posets is quite similar. We will construct a regular
sequence by opening/closing signed notches, culminating in
Theorem~\ref{thm:csspci}.

{
\renewcommand{\thethm}{\ref{thm:csspci}}
\begin{thm}
	Suppose $P\subset \Bn$ (\resp $\Pc \subset \Cn$) is a signed poset such
    that $\Gbhat(P)$ (\resp $\Gchat(\Pc)$) has an embedding in $\bR^2$ which is
    strongly planar. Then
    \begin{itemize}
        \item $\Gbhat(P)$ (\resp $\Gchat(\Pc)$) has an embedding
    that is both centrally symmetric and strongly planar, 
    \item $\Rprt$ (\resp $\Rrt{\Pc}$) is a complete intersection, with $I_P$
        (\resp $I_{\Pc}$) generated by the cycle binomials of the cycles in
        $\Gbhat(P)$ (\resp $\Gchat(P)$) defining the faces of the graph, and
    \item one has
    \begin{align*}
    \Hilb(\Rprt,\bm{x}) &=\dfrac{\prod_{\rho} (1-\bm{x}^\rho)}{\prod_{e} (1-x_a^\delta x_b^{-\epsilon})}
	\quad \text{and} \\
	\Hilb(\Rrt{\Pc},\bm{x}) &=\dfrac{\prod_{\rho} (1-\bm{x}^\rho)}{\prod_{e} (1-x_a^\delta x_b^{-\epsilon})},
\end{align*}
where $\rho$ runs over all regions enclosed by $\Gbhat(P)$ (\resp $\Gchat(P)$)
not fixed by the involution, $e$ runs over the orbits of edges $\delta a \to
\epsilon b$ ($a<b$) in $\Gchat(P)$, $x_0$ is taken to be $1$ and
\[
    \bm{x}^\rho =
    x_{\min(\rho)}^{\sgn(\min(\rho))}x_{\max(\rho)}^{-\sgn(\max(\rho))}.  
\]
\end{itemize}
\end{thm}%
\addtocounter{thm}{-1}%
}

The next several sections will build up to the proof of Theorem~\ref{thm:csspci} in Section~\ref{subsec:stronglyplanar}, where what it means for a signed poset to be strongly planar will be defined.
\begin{itemize}
		\item Section~\ref{subsec:reductions} will reduce to the case of signed
            posets consisting of a single biconnected component (see
            Definition~\ref{def:biconnectedcomp} and Proposition~\ref{prop:biconnectedred}) and whose root cones are full-dimensional (see Propositions~\ref{prop:rtconenotfulldimtypea} and~\ref{prop:samertsemigroup}).
		\item Section~\ref{subsec:signednotches} defines the notion of a signed
				notch (see Definition~\ref{def:signednotch}), explains what it means to
				open/close a signed notch and what impact that
				has on the root cone semigroup ring (see
				Proposition~\ref{prop:closenotchmodout}).
		\item Section~\ref{subsec:stronglyplanar} defines strongly planar
				posets (see Definition~\ref{def:stronglyplanar}), proves a
				number of propositions and lemmas regarding strongly planar
				posets and then proves Theorem~\ref{thm:csspci}.
		\item After the proof, Section~\ref{subsec:computingpsi} will explain
				how Theorem~\ref{thm:csspci} can be used to compute $\Psi_P$
				and $\Psi^\smvee_\Pc$.
\end{itemize}

\subsection{A few reductions}\label{subsec:reductions}
The first reduction will be to the case of signed posets consisting of a single
biconnected component.
\begin{defn}\label{def:biconnectedcomp}
Let $P \subset \Bn$ and $\Pc \subset \Cn$ be signed posets. Let $A \subset P$
(resp.\ $\Pc$) be the elements corresponding to the edges of $\Gbhat(P)$ (resp.\
$\Gchat(\Pc)$). Say two elements of $A$ are \emph{cycle equivalent} if there is
a cycle in $\Gbhat(P)$ (resp.\ $\Gchat(P)$) not fixed orientation-wise by the involution passing
through an edge corresponding to each element. Taking the transitive closure
gives an equivalence relation. Combine equivalence classes lying in the same
orbit of the involution. This partitions the edges of $\Gbhat(P)$ (rep.\
$\Gchat(\Pc)$) into the \emph{biconnected components} of $P$ (resp.\ $\Pc$).
Each biconnected component corresponds to a signed poset, which will also be
called the biconnected components of $P$ (resp.\ $\Pc$).
\end{defn}

Consider  $P =
\{+e_1-e_2,+e_1+e_2,+e_1-e_3,+e_1+e_3,+e_1-e_4,+e_1-e_5,+e_1-e_6,+e_2-e_3,
+e_2-e_4 , +e_2-e_5 , +e_2-e_6 , +e_4-e_6,+e_5-e_6,+e_1\}$.
Figure~\ref{fig:biconnectedex} shows $\Ghat(P)$ with the edges of one biconnected
component solid and the edges of the other dashed.
\begin{figure}
    \begin{center}
    \begin{subfigure}[b]{.4\textwidth}
\begin{center}
   \begin{tikzpicture}
       \node[vertex] (1) at (0,0) [label=left:$1$] {};
	   \node[vertex] (2) at (-1,1) [label=right:$2$] {};
       \node[vertex] (3) at (-1,2) [label=right:$3$] {};
       \node[vertex] (-3) at (1,1) [label=left:$-3$] {};
       \node[vertex] (-2) at (1,2) [label=left:$-2$] {};
	   \node[vertex] (-1) at (0,3) [label=left:$-1$] {};
       \node[vertex] (5) at (-1.5,2) [label=right:$5$] {};
       \node[vertex] (4) at (-2,2) [label=left:$4$] {};
       \node[vertex] (6) at (-2,3) [label=left:$6$] {};
       \node[vertex] (-4) at (1.5,1) [label=below left:$-5$] {};
       \node[vertex] (-5) at (2,1) [label=right:$-4$] {};
       \node[vertex] (-6) at (2,0) [label=right:$-6$] {};
       \node[vertex] (0) at (0,1.5) [label=left:$0$] {};
       \draw[dashed] (1)--(0)--(-1);
       \draw[dashed] (1)--(2)--(3)--(-1)--(-2)--(-3)--(1);
      \draw (2)--(4)--(6)--(5)--(2);
       \draw (-2)--(-4)--(-6)--(-5)--(-2);
   \end{tikzpicture}
\end{center}
   \caption{$\Gbhat(P)$}
\end{subfigure}
    \begin{subfigure}[b]{.4\textwidth}
\begin{center}
   \begin{tikzpicture}
       \node[vertex] (1) at (0,0) [label=left:$1$] {};
	   \node[vertex] (2) at (-1,1) [label=right:$2$] {};
       \node[vertex] (3) at (-1,2) [label=right:$3$] {};
       \node[vertex] (-3) at (1,1) [label=left:$-3$] {};
       \node[vertex] (-2) at (1,2) [label=left:$-2$] {};
	   \node[vertex] (-1) at (0,3) [label=left:$-1$] {};
       \node[vertex] (5) at (-1.5,2) [label=right:$5$] {};
       \node[vertex] (4) at (-2,2) [label=left:$4$] {};
       \node[vertex] (6) at (-2,3) [label=left:$6$] {};
       \node[vertex] (-4) at (1.5,1) [label=below left:$-5$] {};
       \node[vertex] (-5) at (2,1) [label=right:$-4$] {};
       \node[vertex] (-6) at (2,0) [label=right:$-6$] {};
       \draw[dashed] (1)--(2)--(3)--(-1)--(-2)--(-3)--(1);
      \draw (2)--(4)--(6)--(5)--(2);
       \draw (-2)--(-4)--(-6)--(-5)--(-2);
   \end{tikzpicture}
\end{center}
\caption{$\Gchat(\Pc)$}
\end{subfigure}
   \caption{$\Ghat(P)$ for $P =
\{+e_1-e_2,+e_1+e_2,+e_1-e_3,+e_1+e_3,+e_1-e_4,+e_1-e_5,+e_1-e_6,+e_2-e_3,
+e_2-e_4 , +e_2-e_5 , +e_2-e_6 , +e_4-e_6,+e_5-e_6,+e_1\}$ with biconnected
components indicated by dashed and solid lines.}
\label{fig:biconnectedex}
\end{center}
\end{figure}
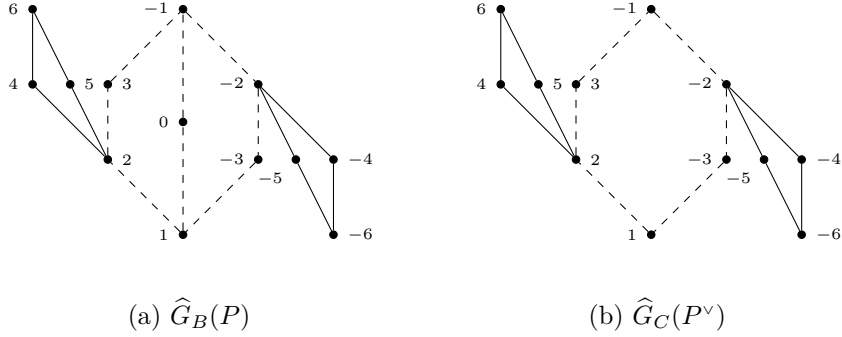

\begin{prop}\label{prop:biconnectedred}
Let $P \subset \Bn$ (\resp $\Pc \subset \Cn$) be a signed poset and $P_1,\ldots,P_k$ (\resp $P^\smvee_1,\ldots,P^\smvee_k$) its biconnected components. Then
\[
\Rprt \cong \Rrt{P_1}\otimes_k \cdots \otimes \Rrt{P_k}\quad\text{and}\quad \Rrt{\Pc} \cong \Rrt{P^\smvee_1}\otimes_k\cdots\otimes_k \Rrt{P^\smvee_k}
\]
and, as a consequence,
\[
	\Hilb(\Rprt,\bm{x})=\prod_{i=1}^k \Hilb(\Rrt{P_i},\bm{x}) \quad \text{and} \quad \Hilb(\Rrt{\Pc},\bm{x})
	=\prod_{i=1}^k \Hilb(\Rrt{P^\smvee_i},\bm{x}).
\]
\end{prop}

\begin{proof}
	Recall that $\Rprt = S_P / I_P$ and $\Rrt{\Pc}=S_{\Pc}/I_{\Pc}$, where
    $I_P$ and $I_{\Pc}$ are generated, respectively, by the cycle binomials of cycles of $\Gbhat(P)$ and $\Gchat(\Pc)$ not fixed by the involution. Every edge in $\Gbhat(P)$ and $\Gchat(\Pc)$ lies in a unique biconnected component and is thus associated to a unique $P_i$ (\resp $P^\smvee_i$). Consequently, $S_P \cong \bigotimes_{\ell =1}^k S_{P_\ell}$ and $S_\Pc \cong \bigotimes_{\ell=1}^k S_{P^\smvee_\ell}$. Furthermore, each cycle not fixed by the involution lies wholly in a single biconnected component. Consequently, $I_P = \bigoplus_{\ell=1}^k I_{P_k}$ and $I_{\Pc} =\bigoplus_{\ell=1}^k I_{P^\smvee_k}$. Then
	\[
	\Rprt \cong S_P/I_P = S/\oplus_{\ell=1}^k I_{P_k} \cong \bigotimes_{\ell =1}^k S_{P_\ell} / I_{P_\ell} \cong \Rrt{P_1}\otimes_k \cdots \otimes_k \Rrt{P_k}
	\]
	and
	\[
	\Rrt{\Pc} \cong S_\Pc/I_{\Pc} = S/\oplus_{\ell=1}^k I_{P^\smvee_k} \cong
    \bigotimes_{\ell =1}^k S_{P^\smvee_\ell} / I_{P^\smvee_\ell} \cong
    \Rrt{P^\smvee_1}\otimes_k \cdots \otimes_k \Rrt{P^\smvee_k},
	\]
    as claimed.
\end{proof}

For the remainder of this chapter, it will be assumed that all signed posets
consist of a single biconnected component. In particular, one may assume that
$\Gbhat(P)$ and $\Gchat(\Pc)$ have no vertices of degree one. 

Furthermore, signed posets for which $\Kprt$ is not full-dimensional need not be
considered here, as they reduce to work of~\citet*{BoussicaultFerayLascouxReiner2012}, as is now explained. 

Suppose $P$ is a signed poset such that $\Kprt$ is not full-dimensional. There
are two possibilities:
\begin{itemize}
		\item $P$ can be relabeled so that $P \subset \Phi_{B_{n-1}}$ (resp.\ $\Pc
				\subset \Phi_{C_{n-1}})$, i.e.\ one should think of $\Kprt \subset
        \bR^{n-1}$ rather than $\Kprt \subset \bR^n$.
    \item The signed graph $\Sigma_P$ underlying the Hasse diagram of $P$ is
        balanced. (See Definition~\ref{def:hassediagram} for
        the definition of the Hasse diagram of a signed poset and
        Definition~\ref{def:balanced} for the notion of a balanced
        signed graph.)
\end{itemize}

While the first case is straightforward, the second requires a little care.

\begin{prop}\label{prop:rtconenotfulldimtypea}
Suppose $P \subset \Bn$ (resp.\ $\Pc \subset \Cn$) is a signed poset such that
$\Sigma_P$ is balanced. Then $P$ (\resp $\Pc$) is isomorphic to a
poset $P' \subset \An \cap \Bn$ (\resp $P' \subset \An \cap \Cn$).
\end{prop}

\begin{proof}
One may assume $\Ghat(P)$ consists of a single biconnected component.
Consequently, $\Sigma_P$ must be connected. Since $\Sigma_P$ is balanced, the
vertices can be partitioned into sets $V^+$ and $V^-$ such that edges labeled
$+$ join a vertex in $V^+$ to one in $V^-$ and edges labeled $-$ join two
vertices in either $V^+$ or $V^-$. (This is Theorem~\ref{thm:hararybalance}.)
Note that since $\Sigma_P$ is balanced, it cannot contain any self-loops.
Therefore, there is no $\pm i < \mp i$ in $\Ghat(P)$. (In particular, should
one be considering $\Gbhat(P)$, $0$ must lie in its own connected component.)
Then $\Ghat(P)$ consists of two isotropic components (plus $0$ in $\Gbhat(P)$).
One can then read off a signed permutation sending $P$ to some $P' \subset
\An$ (it is the permutation flipping signs so that all vertices in each
isotropic component have the same sign), as desired. 
\end{proof}

Next, one checks that, given an $\An$-poset, its root cone semigroup does not
change when one switches among the $\An$, $\Bn$ and $\Cn$ root lattices.

\begin{prop}\label{prop:samertsemigroup}
Suppose $P \subset \An$ is a poset. Then the semigroups $\Kprt \cap \Lrt_A$,
$\Kprt \cap \Lrt_B$ and $\Kprt \cap \Lrt_C$ are equal.
\end{prop}

\begin{proof}
First, begin by observing that $\Kprt \cap \Lrt_A \subset \Kprt \cap \Lrt_B$
and $\Kprt \cap \Lrt_C$. Suppose $\alpha \in \Kprt \cap \Lrt_B$. Then $\alpha$
can be written as a positive integer linear combination of elements of $P$
corresponding to the covering relations in $\Gbhat(P)$. However, these
elements correspond precisely to the covering relations in the Hasse diagram of
$P$, which are a minimal generating set for $\Kprt \cap \Lrt_A$
(see~\cite[Proposition 7.1]{BoussicaultFerayLascouxReiner2012}). Therefore, $\alpha \in \Kprt
\cap \Lrt_A$, so $\Kprt \cap \Lrt_A = \Kprt \cap \Lrt_B$. The same argument
shows $\Kprt \cap \Lrt_A = \Kprt \cap \Lrt_C$.
\end{proof}

As a consequence of the last two propositions, if $P$ is a signed poset and
$\Kprt$ is not full-dimensional, its root cone semigroup ring can be understood
using the results of~\cite{BoussicaultFerayLascouxReiner2012}. If $\Ghat(P)$ is
disconnected, it either has multiple biconnected components or consists of a
single biconnected component made up of two isotropic connected components,
meaning $\Kprt$ is not full-dimensional. Consequently, for the remainder of the
chapter it will be assumed that $\Ghat(P)$ is connected and consists of a
single biconnected component.


\subsection{Signed Notches}\label{subsec:signednotches}

In~\cite{BoussicaultFerayLascouxReiner2012}, a \emph{$\vee$-shaped notch}
(resp.\ a \emph{$\wedge$-shaped notch}) in a poset was defined to be a triple
of elements $(a,b,c)$ such that $a \lessdot b, a \lessdot c$ (resp.\ $a \gtrdot
b,a \gtrdot c$) and $b$ and $c$ are in different connected components of
$P\smallsetminus P_{\leq a}$ (resp.\ $P\smallsetminus P_{\geq a}$). One defines
a signed notch as a pair of notches in this sense.

\begin{defn}\label{def:signednotch}
Given a signed poset $P$, a \emph{signed notch} in $\Ghat(P)$ is a pair
$(a,b,c)$ and $(-a,-b,-c)$ such that both $(a,b,c)$ and $(-a,-b,-c)$ are notches
in $\Ghat(P)$ and, if in $\Gbhat(P)$, neither $b$ nor $c$ is $0$.
\end{defn}

Consider again the signed poset $P$ from 
Figure~\ref{fig:biconnectedex}. Together $(2,3,5)$ and $(-2,-3,-5)$ form a signed notch. 
	
If $(a,b,c)$ is a notch in $\Ghat(P)$, it follows from the properties of the involution that $(-a,-b,-c)$ is also a notch in the type $A$ sense, but facing the other way. Consequently, one may assume that $(a,b,c)$ is a $\vee$-shaped notch and $(-a,-b,-c)$ is a $\wedge$-shaped notch. One may also always assume, without loss of generality, that $\sgn(b)=\sgn(c)$, by relabeling the elements of $\Ghat(P)$ as necessary.

One next wants to show that closing a signed notch in either $\Gbhat(P)$ or
$\Gchat(\Pc)$ gives the $\Gbhat$ or $\Gchat$ associated to some other signed
poset. If $(a,b,c)$ and $(-a,-b,-c)$ form a signed notch in $\Ghat(P)$, the poset $\Ghat(P)/\{b\equiv
c,-b\equiv -c\}$ is the poset obtained by \emph{closing the notch}. It will
often be useful to think of this operation as closing the type $A$ notch
$(a,b,c)$ and then the notch $(-a,-b,-c)$ (or vice versa). To legitimize this
point of view, one needs to show the following:
\begin{itemize}
    \item $(-a,-b,-c)$ remains a notch in $\Ghat(P)/\{b \equiv c\}$
        (Proposition~\ref{prop:closeonenotch})
    \item There is a signed poset $P'$ such that $\Ghat(P') = \Ghat(P)/\{b
        \equiv c,-b\equiv -c\}$.  Such a $P'$ will be said to have been obtained
        from $P$ by \emph{closing the signed notch}. (Propositions~\ref{prop:notchclosingb}
        and~\ref{prop:notchclosingc})
\end{itemize}

The result of closing the signed notch formed by $(2,3,5)$ and $(-2,-3,-5)$ in
$P$ shown
in~Figure \ref{fig:notchcloseex}.

\begin{figure}
    \begin{center}
    \begin{subfigure}[b]{.45\textwidth}
        \begin{center}
      \begin{tikzpicture}
       \node[vertex] (1) at (0,0) [label=left:$1$] {};
	   \node[vertex] (2) at (-1,1) [label=right:$2$] {};
       \node[vertex] (3) at (-1,2) [label=right:$3 \equiv 5$] {};
       \node[vertex] (-3) at (1,1) [label=left:$-3 \equiv -5$] {};
       \node[vertex] (-2) at (1,2) [label=left:$-2$] {};
	   \node[vertex] (-1) at (0,3) [label=left:$-1$] {};
       \node[vertex] (4) at (-1.5,2) [label=left:$4$] {};
       \node[vertex] (6) at (-2,3) [label=left:$6$] {};
       \node[vertex] (-4) at (1.5,1) [label=right:$-4$] {};
       \node[vertex] (-6) at (2,0) [label=right:$-6$] {};
       \node[vertex] (0) at (0,1.5) [label=left:$0$] {};
       \draw (1)--(2)--(3)--(-1)--(-2)--(-3)--(1);
       \draw (2)--(4)--(6)--(3);
       \draw (-2)--(-4)--(-6)--(-3);
       \draw (1)--(0)--(-1);
      \end{tikzpicture}
        \end{center}
        \caption{$\Gbhat(P)/\{3 \equiv 5, -3 \equiv -5\}$}
    \end{subfigure}
    \begin{subfigure}[b]{.45\textwidth}
   \begin{center}
      \begin{tikzpicture}
       \node[vertex] (1) at (0,0) [label=left:$1$] {};
	   \node[vertex] (2) at (-1,1) [label=right:$2$] {};
       \node[vertex] (3) at (-1,2) [label=right:$3 \equiv 5$] {};
       \node[vertex] (-3) at (1,1) [label=left:$-3 \equiv -5$] {};
       \node[vertex] (-2) at (1,2) [label=left:$-2$] {};
	   \node[vertex] (-1) at (0,3) [label=left:$-1$] {};
       \node[vertex] (4) at (-1.5,2) [label=left:$4$] {};
       \node[vertex] (6) at (-2,3) [label=left:$6$] {};
       \node[vertex] (-4) at (1.5,1) [label=right:$-4$] {};
       \node[vertex] (-6) at (2,0) [label=right:$-6$] {};
       \draw (1)--(2)--(3)--(-1)--(-2)--(-3)--(1);
       \draw (2)--(4)--(6)--(3);
       \draw (-2)--(-4)--(-6)--(-3);
      \end{tikzpicture}
   \end{center}
   \caption{$\Gchat(\Pc)/\{3 \equiv 5, -3 \equiv -5\}$}
   \end{subfigure}
   \end{center}
\caption{$\Gbhat(P)$ and $\Gchat(\Pc)$ from Figure~\ref{fig:biconnectedex} after closing the signed notch $(2,3,5)$ and $(-2,-3,-5)$.}
\label{fig:notchcloseex}
\end{figure}

\begin{prop}\label{prop:closeonenotch}
	If $(a,b,c)$ and $(-a,-b,-c)$ form a signed notch in $\Ghat(P)$, $(-a,-b,-c)$ remains a notch in $\Ghat(P)/\{b\equiv c\}$.
\end{prop}

\begin{proof}
One knows from~\cite[Definition 8.5]{BoussicaultFerayLascouxReiner2012} that $\Ghat(P)/\{b \equiv c\}$ is a poset and thus so is $(\Ghat(P)/\{b \equiv c\})/\{-b \equiv -c\}$, as long as $(-a,-b,-c)$ remains a notch in $\Ghat(P)/\{b \equiv c\}$. The only way in which $(-a,-b,-c)$ can fail to be a notch in $\Ghat(P)/\{b \equiv c\}$ is if $-b$ and $-c$ lie in the same connected component of $(\Ghat(P)/\{b \equiv c\})\smallsetminus (\Ghat(P)/\{ b \equiv c\})_{\geq -a}$. In that case, there is a path from $-b$ to $-c$ avoiding $ (\Ghat(P)/\{ b \equiv c\})_{\geq -a}$. However, since $(-a,-b,-c)$ was a notch in $\Ghat(P)$, this path cannot lift to a path in $\Ghat(P)$, so it must pass through the vertex $\{b \equiv c\}$. Therefore at least one of $b$ and $c$ lies in $\Ghat(P)_{\geq -a}$. Without loss of generality, assume $b \in \Ghat(P)_{\geq -a}$. But then $b \equiv c \geq -a$ in $\Ghat(P)/\{b \equiv c\}$, a contradiction. Thus, one must have that $(-a,-b,-c)$ is a notch in $\Ghat(P)/\{b\equiv c\}$.

\end{proof}

With Proposition~\ref{prop:closeonenotch} in hand, it is easier to address
$\Gbhat(P)$ and $\Gchat(\Pc)$ separately, as the next two propositions.

\begin{prop}\label{prop:notchclosingb}
	If $(a,b,c)$ and $(-a,-b,-c)$ form a signed notch in $\Gbhat(P)$, there is
	a signed poset $P'\subset \Bn$ such that $\Gbhat(P') = \Gbhat(P) /\{b
	\equiv c,-b \equiv -c\}$. 
\end{prop}

\begin{proof}
From Proposition~\ref{prop:closeonenotch}, one knows that $\Gbhat(P)/\{b \equiv c\}$ has a notch $(-a,-b,-c)$.
	Closing the second notch $(-a,-b,-c)$ one obtains 
	\begin{multline*}
	\Gbhat(P') = (\Gbhat(P)/\{b\equiv c\})/\{-b\equiv -c\} \\
	=
	(\Gbhat(P)/\{-b\equiv -c\})/\{b\equiv c\} 
	=\Gbhat(P)/\{b \equiv c,-b\equiv -c\},
	\end{multline*}
	as the resulting poset is independent of the order in which the notches are closed. Since $\sgn(b)=\sgn(c)$, this poset has an involution $\pm i \mapsto \mp i$ such that $i \to j$ is sent to $-j \to -i$, since $\Gbhat(P)$ had an involution with this property. 

  To show that $\Gbhat(P') = \Gbhat(P)/\{b \equiv c,-b \equiv -c\}$ is really
associated to a signed poset $P'$, one must show that $\Gbhat(P')$ has the
property that if $i < -i$, then $i < 0 < -i$. Suppose $i < -i$ in $\Gbhat(P')$. It
suffices to consider only the case where $ i \not < -i$ in $\Gbhat(P)$. Then,
without loss of generality, one may assume that $b< -i$ and $i < c$. First,
consider the case where $a \ne 0$. One then has the situation depicted in
Figure~\ref{fig:notchclosingbproof},
  where the solid lines are edges in the Hasse diagram of $\Gbhat(P)$ and the
dashed lines are chains. 
\begin{figure}[htbp]
     \begin{center}
	\begin{tikzpicture}
	  \node[vertex] at (0,-1) (i) [label=right:$i$] {};
	  \node[vertex] at (-1,1) (c) [label=right:$c$] {};
	  \node[vertex] at (-2,0) (a) [label=right:$a$] {};
	  \node[vertex] at (-3,1) (b) [label=left:$b$] {};
	  \node[vertex] at (3,1)  (-b) [label=right:$-b$] {};
	  \node[vertex] at (2,2)  (-a)[label=right:$-a$] {};
	  \node[vertex] at (1,1)  (-c)[label=right:$-c$] {};
	  \node[vertex] at (0,3)  (-i)[label=right:$-i$] {};
	  \draw (a)--(b);
	  \draw (a)--(c);
	  \draw (-b)--(-a);
	  \draw (-c)--(-a);
	  \draw[dashed] (i)--(c);
	  \draw[dashed] (i)--(-b);
	  \draw[dashed] (b)--(-i);
	  \draw[dashed] (-c)--(-i);
	\end{tikzpicture}
  \end{center}
\caption{After closing two notches in the proof of
Proposition~\ref{prop:notchclosingb}}
\label{fig:notchclosingbproof}
\end{figure}
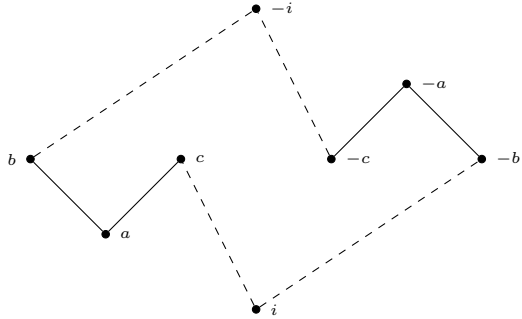
This gives a path from $b$ to $c$, via $-i, -c,-a,-b$ and $i$. Since
$(a,b,c)$ is a notch, this path must intersect $\Gbhat(P)_{\leq a}$ at some $d$.
There are a few cases.
  \begin{enumerate}[label=(\alph*)]
	\item \textbf{Suppose $-c \leq d \leq -i$.} Then $-c <c$ in $\Gbhat(P)$. Therefore, $-c < 0 < c$. Then in $\Gbhat(P')$, one has $b \equiv c < -i$ and $i < -b \equiv -c $, so since $-c < 0 < c$ one has $i  < -b \equiv -c < 0 < b \equiv c < -i$, as desired.
	\item \textbf{Suppose $i \leq d \leq -b$.} Then in $\Gbhat(P)$ one has $i \leq d < a < -i$, so $i < 0 < -i$ and this relation is preserved in $\Gbhat(P')$.
	\item \textbf{Suppose $i \leq d < c$.} Then one has $i \leq d < c <-i$ in $\Gbhat(P)$, so $i < 0 < -i$ and this relation is preserved in $\Gbhat(P')$.
	\item \textbf{Suppose $-a=d < a$.} Then $i < -i$ in $\Gbhat(P)$, so $i < 0 < -i$ and this relation is preserved in $\Gbhat(P')$. 
  \end{enumerate}
  In the other case, suppose $a=0$. Then $-b < 0 < b$. Since $-i >b$, $i < -b$, so $i < -i$ in $\Gbhat(P)$, so $i < 0 < -i$ and this relation is preserved in $\Gbhat(P')$.
\end{proof}

\begin{prop}\label{prop:notchclosingc}
	If $(a,b,c)$ and $(-a,-b,-c)$ form a signed notch in $\Gchat(\Pc)$, there is a signed poset $P'^{\smvee} \subset \Cn$ such that $\Gchat(P'^\smvee) = \Gchat(\Pc) /\{b \equiv c,-b \equiv -c\}$. 
	\end{prop}

\begin{proof}
	From Proposition~\ref{prop:closeonenotch}, one knows that $\Gchat(\Pc)/\{b \equiv c\}$ has a notch $(-a,-b,-c)$.
	Closing the second notch $(-a,-b,-c)$ one obtains 
	\begin{multline*}
	\Gchat(P'^\smvee)=(\Gchat(\Pc)/\{b\equiv c\})/\{-b\equiv -c\}\\
	= (\Gchat(\Pc)/\{-b\equiv -c\})/\{b\equiv c\} =\Gchat(\Pc)/\{b \equiv c,-b\equiv -c\},
	\end{multline*}
	as the resulting poset is independent of the order in which the notches are closed. Since $\sgn(b)=\sgn(c)$, this poset has an involution $\pm i \mapsto \mp i$ such that $i \to j$ is sent to $-j \to -i$, since $\Gchat(\Pc)$ had an involution with this property. 
	
	It remains to show that $\Gchat(P'^\smvee)$ has the property that if $i < -i$ and $j < -j$ (possibly after relabelling) that $i < -j$ and $j < -i$. Closing a notch introduces relations and does not remove any relations. There are then two cases:
	\begin{enumerate}
		\item \textbf{Suppose $i < -i$ in $\Gchat(P'^\smvee)$, but not in $\Gchat(\Pc)$
            and $j < -j$ in both $\Gchat(\Pc)$ and $\Gchat(P'^\smvee)$.} Then,
            without loss of generality, suppose $i < b$ and $c < -i$ in
            $\Gchat(\Pc)$. Then $-b < -i$ and $i < -c$. Then one has the
         situation in $\Gchat(\Pc)$ depicted in Figure~\ref{fig:notchclosingc1}, where the dashed edges are chains and the solid edges are edges in the Hasse diagram of $\Gchat(\Pc)$. 
            \begin{figure}[htbp]
			\begin{center}
				\begin{tikzpicture}[vertex/.style={draw,circle,fill=black,inner sep=1pt}]
					\node[vertex] at (0,-1) (i) [label=below:$i$] {};
					\node[vertex] at (3,1) (-c) [label=right:$-c$] {};
					\node[vertex] at (1,1) (-b) [label=right:$-b$] {};
					\node[vertex] at (-2,1) (a) [label=right:$a$] {};
					\node[vertex] at (-3,2) (b) [label=right:$b$] {};
					\node[vertex] at (-1,2) (c) [label=right:$c$] {};
					\node[vertex] at (2,2) (-a) [label=right:$-a$] {};
					\node[vertex] at (0,3) (-i) [label=above:$-i$] {};
					\draw[dashed] (i) to [bend left] (b);
					\draw (a)--(b);
					\draw (a)--(c);
					\draw[dashed] (c)--(-i);
					\draw[dashed] (-i)--(-b);
					\draw (-b)--(-a);
					\draw (-a)--(-c);
					\draw[dashed] (i) to [bend right](-c);
				\end{tikzpicture}
			\end{center}
            \caption{$\Gchat(\Pc)$ in case 1 of the proof of
                Proposition~\ref{prop:notchclosingc}}
            \label{fig:notchclosingc1}
        \end{figure}
			Since $(a,b,c)$ is a notch in $\Gchat(\Pc)$, one must have that the path from $b$ to $c$ via $i$ must pass through $\Gchat(\Pc)_{\leq a}$ at some $d$. There are a few cases
			\begin{enumerate}[label=(\arabic{enumi}.\alph*)]
				\item \textbf{Suppose $d$ lies in the chain from $b$ to $i$ or
                        the chain from $i$ to $-c$.} Then $i < a$, so $i < -i$ in $\Ghat(P)$, a contradiction.
				\item \textbf{Suppose $d < a$ and $-b \leq d < -i$.} Then, $-b < d < a <
                    b$. Consequently, if $j < -j$ in $\Gchat(\Pc)$ and $b \ne
                    j$, one has that $-b < -j$ and $j < b$ in both
                    $\Gchat(\Pc)$ and $\Gchat(P'^\smvee)$. Then, in
                    $\Gchat(P'^\smvee)$, one has $j < b\equiv c < -i$ and $i < -b\equiv -c < -j$, as required.

				\item \textbf{Suppose $d = -a$ and $-a < a$.} Then in $\Gchat(\Pc)$ one has that $i < -i$, a contradiction.
				\item \textbf{Suppose $d=-c$ and $-c < a$.} Then $\Gchat(\Pc)$ is as in
                    Figure~\ref{fig:notchclosingc2}. Then $i < -i$ in $\Gchat(\Pc)$, a contradiction.
                    \begin{figure}[htbp]
					\begin{center}
					\begin{tikzpicture}[vertex/.style={draw,circle,fill=black,inner sep=1pt}]
						\node[vertex] at (0,0) (-c) [label=left:$-c$] {};
						\node[vertex, below left of=-c] (i) [label=left:$i$] {};
						\node[vertex, above left of=-c] (a) [label=left:$a$] {};
						\node[vertex, above right of=-c](-a) [label=above:$-a$] {};
						\node[vertex,above left of=-a] (c) [label=right:$c$] {};
						\node[vertex,above right of=c](-i) [label=right:$-i$] {};
						\node[vertex,above left of=a] (b) [label=left:$b$]{};
						\node[vertex,below right of=-a] (-b) [label=right:$-b$] {};
						\draw (-c)--(a)--(b);
						\draw (a)--(c);
						\draw (-c)--(-a)--(c);
						\draw (-b)--(-a);
						\draw[dashed] (i)--(-c);
						\draw[dashed] (c)--(-i);
						\draw[dashed] (i)--(b);
						\draw[dashed] (-i)--(-b);
					\end{tikzpicture}
					\end{center}
                    \caption{$\Gchat(\Pc)$ in case 1(d) in the proof of
                        Proposition~\ref{prop:notchclosingc}}
                    \label{fig:notchclosingc2}
                \end{figure}
							\end{enumerate}

		\item Suppose $i < -i$ and $j < -j$ in $\Gchat(P'^\smvee)$, but neither relation exists in $\Gchat(\Pc)$. There are again two cases (up to relabeling the vertices).
			\begin{enumerate}[label=(\arabic{enumi}.\alph*)]
				\item \textbf{Suppose $i < b$, $c<-i$ and $j < c$, $-j >b$.} Then in
                    $\Gchat(P'^\smvee)$, one has $i < b\equiv c < -j$ and $j < b\equiv c < -i$, as required.
				\item \textbf{Suppose $i, j < b$ and $c < -i,-j$.} Without loss of
                    generality, suppose $i < j$. Consequently $-j < -i$. Then
                    in $\Gchat(P'^\smvee)$, one has  $i < j < -j < -i$, giving the required relations.
			\end{enumerate}
	\end{enumerate}

	Thus, this new poset is $\Gchat(P'^\smvee)$ for some signed poset $P'^\smvee$.
\end{proof}

Now that one understands what it means to close a notch in a signed poset, the
natural question is what effect this maneuver has on the semigroup ring.

\begin{prop}\label{prop:closenotchmodout}
	Let $P \subset \Bn$ be a signed poset and $(a,b,c)$ and $(-a,-b,-c)$ form a signed notch in $\Gbhat(P)$. Let $P'$ be the signed poset obtained from $P$ by closing this notch. Then
	\[
	\Rrt{P'} \cong \begin{cases}
			\Rrt{P}/(t_a^\delta t_b^\epsilon - t_a^\delta t_c^\epsilon) & \text{if }a \ne 0 \\
			\Rrt{P}/(t_b^\epsilon - t_c^\epsilon) & \text{if }a =0 \\
	\end{cases}
	\]
	where $\delta =\sgn(a)$ and $\epsilon = -\sgn(b)$.

	Let $\Pc \subset \Cn$ be a signed poset and $(a,b,c)$ and $(-a,-b,-c)$ form a signed notch in $\Gchat(\Pc)$. Let $P'^\smvee$ be the signed poset obtained from $\Pc$ by closing this notch. Then
	\[
		\Rrt{P'^\smvee} \cong \Rrt{\Pc}/(t_a^\delta t_b^\epsilon - t_a^\delta t_c^\epsilon),
	\]
	where $\delta = \sgn(a)$ and $\epsilon = -\sgn(b)$.
\end{prop}

To see how this works, consider once again  $P =
\{+e_1-e_2,+e_1+e_2,+e_1-e_3,+e_1+e_3,+e_1-e_4,+e_1-e_5,+e_1-e_6,+e_2-e_3,
+e_2-e_4 , +e_2-e_5 , +e_2-e_6 , +e_4-e_6,+e_5-e_6, +e_1\}$ shown in
Figure~\ref{fig:biconnectedex} (page~\pageref{fig:biconnectedex}). Closing the notch $(2,3,5)$ and $(-2,-3,-5)$
gives the poset shown in Figure~\ref{fig:notchcloseex}
(page~\pageref{fig:notchcloseex}), $Q =
\{+e_1-e_2,+e_1+e_2,+e_1+e_3,+e_1-e_3,+e_1-e_4,+e_1-e_6,+e_2-e_3,+e_2-e_4,+e_3-e_6,+e_4-e_6\}$
(with $3 \equiv 5$ and $3 \equiv -5$ being renamed $3$ and $-3$, respectively).
In type $B$, one has $\Srt{P} =
k[U_{12},U_{23},U_{\m13},U_{24},U_{25},U_{46},U_{56},U_{10}]$. Define a map
$\psi \colon \Srt{P} \to \Rrt{Q}$ by 
\[
    \begin{array}{llll}
        U_{12} \mapsto t_1t_2^{-1} & U_{23} \mapsto t_2t_3^{-1} &
        U_{\m13}\mapsto t_1t_3 & U_{24} \mapsto t_2t_4^{-1} \\
        U_{25} \mapsto t_2t_3^{-1} & U_{46} \mapsto t_4t_6^{-1} & U_{56}
        \mapsto t_3t_6^{-1} & U_{10} \mapsto t_1
    \end{array}
\]
Then
\begin{align*}
    \Rrt{Q} &\cong \Srt{P}/\ker \psi \\
    &\cong \Srt{P}
    /(U_{24}U_{46}-U_{25}U_{56},U_{12}U_{\m13}U_{25}-U_{10}^2,U_{23}-U_{25}) \\
    &= \Srt{P} /(\Irt{P}+(U_{23}-U_{25})) \\
    &\cong (\Srt{P}/\Irt{P})/(\overline{U}_{23}-\overline{U}_{25}) \\
    & \cong \Rprt / (t_2t_3^{-1}-t_2t_5^{-1})
\end{align*}

\begin{proof}[Proof of Proposition~\ref{prop:closenotchmodout}]
Recall the definitions of $\Srt{P}$, $\Srt{\Pc}$ and of the toric ideals $\Irt{P}$ and
$\Irt{\Pc}$ from Section~\ref{subsec:rootconetoricideal}. Let $P'$ and $P'^\smvee$ be the signed posets obtained by closing notches in $\Gbhat(P)$ and $\Gchat(\Pc)$, respectively. One defines maps from $\Srt{P}$ and $\Srt{\Pc}$ to the semigroup rings $\Rrt{P'}=k[K_{P'}^\mathrm{rt}\cap \Lbrt]$ and $\Rrt{P'^\smvee}=k[K_{P'^\smvee}^\mathrm{rt}\cap \Lcrt]$.

Define a map $\psi\colon \Srt{P} \to \Rrt{P'}$ by
	\[
		U_{\delta i,\epsilon j} \mapsto 
		\begin{cases}
			t_i^\delta t_j^{-\epsilon} & i,j \ne c, 0 \\
			t_i^\delta t_b^{-\epsilon} & j = c, i \ne 0 \\
			t_b^\delta t_j^{-\epsilon} & i = c \\
			t_j^{-\epsilon} & i = 0, j \ne c \\
			t_b^{-\epsilon} & i = 0, j = c
		\end{cases}
		\] 
	  and define a map $\psi^\smvee \colon \Srt{\Pc} \to \Rrt{P'^\smvee}$ by
	 \[
		U_{\delta i,\epsilon j} \mapsto
		\begin{cases}
			t_i^\delta t_j^{-\epsilon} & i,j \ne c \\
			t_b^\delta t_j^{-\epsilon} & i = c \\
			t_i^\delta t_b^{-\epsilon} & j = c
		\end{cases}
			 \]

			 One shows that $\ker \psi = I_P + (U_{ab}-U_{ac})$ and $\ker
             \psi^\smvee = I_{\Pc} +(U_{ab}-U_{ac})$. The argument is the same
             for $P$ and $\Pc$. To simplify notation, in the remainder of the
             proof, write $J$ for the kernel, $I$ for the toric ideal, let the
             variables of the polynomial ring $S$ be $U_{ij}$ and use
             $\Ghat(P)$ to denote $\Gbhat(P)$ or $\Gchat(P)$, as applicable.
             Thus, one needs to show that $J = I + (U_{ab}-U_{ac})$.

	Recall that $I$ is generated by the cycle binomials of cycles in $\Ghat(P)$
    not fixed by the involution. Let $U(C)$ be such a cycle binomial. The
    definition of $\phi$ (\resp $\psi$) ensures that $U(C) \in J$, so $I \subset J$.
	Both $U_{ab}$ and $U_{ac}$ have the same image in $\Rrt{P'}$, so $(U_{ab}-U_{ac}) \subset J$. Thus $I+(U_{ab}-U_{ac}) \subset J$.

	Let $P^+$ be the directed graph that has the same vertices and edges as
    $\Ghat(P')$, but with the edges $-b \equiv -c \to -a$ and $a \to b\equiv c$
    doubled. The argument in the proof of Theorem~\ref{thm:toricideal} shows
    that $J$ is generated by the cycle binomials corresponding to cycles of
    $P^+$ not fixed by the involution. (See Figure~\ref{fig:pplusex} for $P^+$
    for the $\Gbhat(P)$ and $\Gchat(\Pc)$ from Figure~\ref{fig:notchcloseex}.)

    \begin{figure}
        \begin{center}
            \begin{subfigure}[b]{.45\textwidth}
        \begin{center}
      \begin{tikzpicture}
       \node[vertex] (1) at (0,0) [label=left:$1$] {};
	   \node[vertex] (2) at (-1,1) [label=right:$2$] {};
       \node[vertex] (3) at (-1,2) [label=right:$3 \equiv 5$] {};
       \node[vertex] (-3) at (1,1) [label=left:$-3 \equiv -5$] {};
       \node[vertex] (-2) at (1,2) [label=left:$-2$] {};
	   \node[vertex] (-1) at (0,3) [label=left:$-1$] {};
       \node[vertex] (4) at (-1.5,2) [label=left:$4$] {};
       \node[vertex] (6) at (-2,3) [label=left:$6$] {};
       \node[vertex] (-4) at (1.5,1) [label=right:$-4$] {};
       \node[vertex] (-6) at (2,0) [label=right:$-6$] {};
       \node[vertex] (0) at (0,1.5) [label=left:$0$] {};
       \draw (1)--(2);
       \path (2) edge[bend left](3)
             edge[bend right] (3);
       \draw (3)--(-1)--(-2);
       \path (-2) edge[bend left](-3)
       edge[bend right] (-3);
       \draw (-3)--(1);
       \draw (2)--(4)--(6)--(3);
       \draw (-2)--(-4)--(-6)--(-3);
       \draw (1)--(0)--(-1);
      \end{tikzpicture}
        \end{center}
        \caption{$P^+_C$}
    \end{subfigure}
    \begin{subfigure}[b]{.45\textwidth}
   \begin{center}
      \begin{tikzpicture}
       \node[vertex] (1) at (0,0) [label=left:$1$] {};
	   \node[vertex] (2) at (-1,1) [label=right:$2$] {};
       \node[vertex] (3) at (-1,2) [label=right:$3 \equiv 5$] {};
       \node[vertex] (-3) at (1,1) [label=left:$-3 \equiv -5$] {};
       \node[vertex] (-2) at (1,2) [label=left:$-2$] {};
	   \node[vertex] (-1) at (0,3) [label=left:$-1$] {};
       \node[vertex] (4) at (-1.5,2) [label=left:$4$] {};
       \node[vertex] (6) at (-2,3) [label=left:$6$] {};
       \node[vertex] (-4) at (1.5,1) [label=right:$-4$] {};
       \node[vertex] (-6) at (2,0) [label=right:$-6$] {};
       \draw (1)--(2);
       \path (2) edge[bend left] (3)
       edge[bend right] (3);
       \draw (3)--(-1)--(-2);
       \path (-2) edge[bend left] (-3)
       edge [bend right] (-3);
       \draw (-3)--(1);
       \draw (2)--(4)--(6)--(3);
       \draw (-2)--(-4)--(-6)--(-3);
      \end{tikzpicture}
   \end{center}
   \caption{$P^+_C$}
   \end{subfigure}
   \end{center}
\caption{$P^+_B$ and $P^+_C$ for $\Gbhat(P)$ and $\Gchat(\Pc)$ from
    Figure~\ref{fig:notchcloseex}}
\label{fig:pplusex}
    \end{figure}

	One then needs to show that for every cycle of $P^+$ not fixed by the involution, its circuit binomial lies in $I +(U_{ab}-U_{ac})$. There are a number of cases.
	\begin{enumerate}
		\item \textbf{Suppose $C$ is a cycle in $P^+$ not fixed by the
                involution  passing through neither $b \equiv c$ nor $-b \equiv
                -c$.} Then $C$ lifts to a cycle in $\Ghat(P)$, meaning $U(C) \in I_P$.
		\item \textbf{Suppose $C$ is a cycle in $P^+$ not fixed by the
                involution, but passing through at least one of $b \equiv c$
                and $-b \equiv -c$.} One can partition the edges  incident to $b \equiv c$ and $-b \equiv c$ into $E_b \sqcup E_c$ according to whether they lift to edges incident to either $b$ or $-b$ or to either $c$ or $-c$ in $\Ghat(P)$.
			\begin{enumerate}[label=(\arabic{enumi}.\alph*)]
			\item \textbf{Suppose the edges of $C$ incident to $b \equiv c$ or
                    $-b \equiv -c$ all lie in one of $E_b$ and $E_c$.} Then $C$ lifts to a cycle in $\Ghat(P)$, so $U(C) \in I$.
			\item \textbf{Suppose $C$ contains an edge from both $E_b$ and
                    $E_c$.} Assume $C$ has an edge from each of $E_b$ and $E_c$ incident to $b \equiv c$. The other case is symmetric.

				Since $(a,b,c)$ was a notch in $\Ghat(P)$, one has that $b$ and $c$ lie in different connected components of $\Ghat(P) \smallsetminus \Ghat(P)_{\leq a}$, so $C$ must pass through at least one vertex $d \leq a$. Let $\pi_{da}$ be a saturated chain between $d$ and $a$. Let $C_b$ be the cycle in $P^+$ that follows $C$ from $b\equiv c$ to d, then $\pi_{da}$ from $d$ to $a$ and finishes along the edge in $E_b$ between $a$ and $b \equiv c$. Let $C_c$ be the cycle in $P^+$ that follows $C$ from $b \equiv c$ to $d$, then $\pi_{da}$ from $d$ to $a$ and finishes along the edge in $E_c$ from $a$ to $b \equiv c$.
				\begin{enumerate}[label=(\arabic{enumi}.\alph{enumii}.\roman*)]
					\item \textbf{Suppose $C_b$ and $C_c$ both lift to cycles in
                        $\Ghat(P)$ and, with $C_b$ and $C_c$ are oriented so
                        they traverse $\pi_{da}$ in opposite directions.} One has
							\begin{multline*}
		U(C)=U(C_b)\left(\prod_{e \in W(\pi_{dc})} U_e\right) + U(C_c)\left(\prod_{e \in A(\pi_{bd})}
		U_e\right) \\
		+(U_{ab}-U_{ac})\left(\prod_{e \in W(\pi_{dc})} U_e\right)\left( \prod_{e \in
		A(\pi_{bd})} U_e\right)\left( \prod_{e \in W(\pi_{da})}U_e\right),
\end{multline*}
		where $\pi_{dc}$ is the portion of $C_c$ between $d$ and $c$ and $\pi_{bd}$ is the portion of $C_b$ between $d$ and $b$. Note that it is possible one of $C_b$ and $C_c$ is fixed by the involution, in which case its cycle binomial is zero since $W(C)=A(C)$.
		Thus, $U(C) \in I+(U_{ab}-U_{ac})$, as desired.

	\item \textbf{Now suppose at least one of $C_b$ and $C_c$ does not lift to a cycle
        in $\Ghat(P)$.} One can use the previous argument to show both $U(C_b)$ and $U(C_c)$ lie in $I_P+(U_{ab}-U_{ac})$. Without loss of generality, suppose $C_b$ does not lift to a cycle in $\Ghat(P)$. (If neither lifts, one repeats the following argument with each cycle.) Since $C_b$ does not lift to a cycle in $\Ghat(P)$, it must pass through $-b \equiv -c$, along one edge in $E_b$ and one in $E_c$. Since $(-a,-b,-c)$ was a notch in $\Ghat(P)$, there is $d' \geq -a \in P^+$ which $C_b$ passes through. 
		Then choose a saturated chain between $d'$ and $-a$, call it
        $\pi_{da}'$, with the proviso that if $\pi_{da}'$ includes $a$ and $b
        \equiv c$, it includes the edge in $C$. One can then construct $C_{bb}$
        and $C_{bc}$ as in the previous case and the same argument shows that
        $U(C_b) \in I_P + (U_{ab}-U_{ac})$. (However, if $\pi_{da}'$ coincides
        with $C$, it will be the case that one of $C_{bb}$ and $C_{cc}$ is not
        genuinely a cycle, rather one will have a chain from (say) $b$ to $d'$ and then return to $b$ down the same chain, in which case $U(C_{bb})=0$.)
				\end{enumerate}
		\end{enumerate}
	\end{enumerate}
	Thus, $J = I +(U_{ab}-U_{ac})$.

	To finish the proof, one sees that
	\begin{align*}
		\Rrt{P'} & \cong S_P/\ker \psi \\
			&\cong S_P/(I_P +(U_{ab}-U_{ac})) \\
			& \cong (S_P/I_P)/(\bar{U}_{ab}-\bar{U}_{ac}) \\
			& \cong \Rprt/(t_a^{\sgn(a)}t_b^{-\sgn(b)}-t_a^{\sgn(a)}t_c^{-\sgn(c)}),
	\end{align*}
	where $t_a=1$ if $a =0$,
	and
	\begin{align*}
	  \Rrt{P'^\smvee} & \cong S_{\Pc}/\ker \psi^\smvee \\
	  		&\cong S_{\Pc}/(I_{\Pc} +(U_{ab}-U_{ac})) \\
			& \cong (S_{\Pc}/I_{\Pc})/(\bar{U}_{ab}-\bar{U}_{ac}) \\
			& \cong \Rrt{\Pc}/(t_a^{\sgn(a)}t_b^{-\sgn(b)}-t_a^{\sgn(a)}t_c^{-\sgn(c)}).
	\end{align*}
\end{proof}

\subsection{Strongly Planar Signed Posets}\label{subsec:stronglyplanar}

Having looked at the notion of opening notches and its impact on $\Rprt$, one
is ready to show that a certain class of signed posets have $\Rprt$ a complete
intersection, implying that the numerator of $\Psi_P$ factors.

\begin{defn}\label{def:stronglyplanar}
A poset is said to be \emph{strongly planar} if, after the addition of a maximal element
$\hat{1}$ and a minimal element $\hat{0}$, there is an embedding of its
Hasse diagram in $\bR^2$ that is planar and has the property that if $a <_P b$,
the $y$-coordinate of $a$ is smaller than that of $b$. A signed poset
$P \subset \Bn$ (\resp $P \subset \Cn$) will be said to be \emph{strongly
    planar} if
$\Gbhat(P)$ (\resp $\Gchat(\Pc)$) is strongly planar.
\end{defn}
As an example, consider $P
= \{+e_1+e_2,+e_1,+e_2\}$. Figure~\ref{fig:stronglyplanarex} shows $\Gbhat(P)$
and $\Gchat(\Pc)$. One sees that $\Gbhat(P)$ is strongly planar, but $\Gchat(\Pc)$ is not.

\begin{figure}
    \begin{center}
    \begin{subfigure}[b]{.45\textwidth}
        \begin{center}
            \begin{tikzpicture}
                \node[vertex] (0) at (0,0) [label=right:$0$] {};
                \node[vertex] (-1) at (1,1) [label=right:$-1$] {};
                \node[vertex] (-2) at (-1,1) [label=left:$-2$] {};
                \node[vertex] (1) at (-1,-1) [label=left:$1$] {};
                \node[vertex] (2) at (1,-1) [label=right:$2$] {};
                \draw (1)--(0)--(-1);
                \draw (2)--(0)--(-2);
            \end{tikzpicture}
        \end{center}
        \caption{$\Gbhat(P)$}
    \end{subfigure}
    \begin{subfigure}[b]{.45\textwidth}
        \begin{center}
            \begin{tikzpicture}
                \node[vertex] (-1) at (1,1) [label=right:$-1$] {};
                \node[vertex] (-2) at (-1,1) [label=left:$-2$] {};
                \node[vertex] (1) at (-1,-1) [label=left:$1$] {};
                \node[vertex] (2) at (1,-1) [label=right:$2$] {};
                \draw (1)--(-2)--(2)--(-1)--(1);
            \end{tikzpicture}
        \end{center}
        \caption{$\Gchat(\Pc)$}
    \end{subfigure}
\end{center}
    \caption{$\Gbhat(P)$ and $\Gchat(\Pc)$ for $P=\{+e_1+e_2,+e_1,+e_2\}$}
    \label{fig:stronglyplanarex}
\end{figure}

\begin{defn}\label{def:centrallysymmetric}
An embedding of $\Gbhat(P)$ (\resp $\Gchat(\Pc)$) in $\bR^n$ is said to be \emph{centrally
    symmetric} if it is fixed by the map $(x,y) \mapsto (-x,-y)$.
	\end{defn}
The main result of this section will be the following.

\begin{thm}\label{thm:csspci}

	Suppose $P\subset \Bn$ (\resp $\Pc \subset \Cn$) is a signed poset such
    that $\Gbhat(P)$ (\resp $\Gchat(\Pc)$) has an embedding in $\bR^2$ which is
    strongly planar. Then
    \begin{itemize}
        \item $\Gbhat(P)$ (\resp $\Gchat(\Pc)$) has an embedding
    that is both centrally symmetric and strongly planar, 
    \item $\Rprt$ (\resp $\Rrt{\Pc}$) is a complete intersection, with $I_P$
        (\resp $I_{\Pc}$) generated by the cycle binomials of the cycles in
        $\Gbhat(P)$ (\resp $\Gchat(\Pc)$) defining the faces of the graph, and
    \item one has
    \begin{align*}
    \Hilb(\Rprt,\bm{x}) &=\dfrac{\prod_{\rho} (1-\bm{x}^\rho)}{\prod_{e} (1-x_a^\delta x_b^{-\epsilon})}
	\quad \text{and} \\
	\Hilb(\Rrt{\Pc},\bm{x}) &=\dfrac{\prod_{\rho} (1-\bm{x}^\rho)}{\prod_{e} (1-x_a^\delta x_b^{-\epsilon})},
\end{align*}
where $\rho$ runs over all regions enclosed by $\Gbhat(P)$ (\resp $\Gchat(P)$)
not fixed by the involution, $e$ runs over the orbits of edges $\delta a \to
\epsilon b$ ($a<b$) in $\Gchat(P)$, $x_0$ is taken to be $1$ and
\[
    \bm{x}^\rho =
    x_{\min(\rho)}^{\sgn(\min(\rho))}x_{\max(\rho)}^{-\sgn(\max(\rho))}.  
\]
\end{itemize}
\end{thm}

\begin{figure}
    \begin{center}
    \begin{subfigure}[b]{.45\textwidth}
        \begin{center}
            \begin{tikzpicture}
                \node[vertex] (0) at (0,0) [label=left:$0$] {};
                \node[vertex] (-2) at (1,0) [label=right:$-2$] {};
                \node[vertex] (-1) at (0,1) [label=above:$-1$] {};
                \node[vertex] (-6) at (2,1) [label=right:$-3$] {};
                \node[vertex] (-5) at (1,2) [label=right:$-4$] {};
                \node[vertex] (1) at (0,-1) [label=below:$1$] {};
                \node[vertex] (5) at (-1,-2) [label=right:$4$] {};
                \node[vertex] (6) at (-2,-1) [label=left:$3$] {};
                \node[vertex] (2) at (-1,0) [label=left:$2$] {};
				\node at (1,1) {$\rho$};
                \draw (5)--(1)--(2)--(6)--(5);
                \draw (-2)--(-1)--(-5)--(-6)--(-2);
                \draw (1)--(0)--(-1);
				\draw (2)--(-1);
				\draw (1)--(-2);
            \end{tikzpicture}
        \end{center}
        \caption{$\Gbhat(P)$}
    \end{subfigure}
        \begin{subfigure}[b]{.45\textwidth}
        \begin{center}
            \begin{tikzpicture}
                \node[vertex] (-2) at (1,0) [label=right:$-2$] {};
                \node[vertex] (-1) at (0,1) [label=above:$-1$] {};
                \node[vertex] (-6) at (2,1) [label=right:$-3$] {};
                \node[vertex] (-5) at (1,2) [label=right:$-4$] {};
                \node[vertex] (1) at (0,-1) [label=below:$1$] {};
                \node[vertex] (5) at (-1,-2) [label=right:$4$] {};
                \node[vertex] (6) at (-2,-1) [label=left:$3$] {};
                \node[vertex] (2) at (-1,0) [label=left:$2$] {};
                \draw (5)--(1)--(2)--(6)--(5);
				\node at (1,1) {$\rho$};
                \draw (-2)--(-1)--(-5)--(-6)--(-2);
				\draw (2)--(-1);
				\draw (1)--(-2);
            \end{tikzpicture}
        \end{center}
        \caption{$\Gchat(\Pc)$}
    \end{subfigure}
\end{center}
    \caption{A signed poset with $\Gbhat(P)$ and $\Gchat(\Pc)$ both strongly planar}
    \label{fig:hilbex}
\end{figure}

The proof of Theorem~\ref{thm:csspci} is quite involved and requires a number
of propositions. First, the idea of the proof is illustrated using the posets
in Figure~\ref{fig:hilbex}. The proof is by induction on
the number of orbits of regions enclosed by $\Ghat(P)$
under the involution. 
Proposition~\ref{prop:rightregion} will show that
one can find at least one ``rightmost'' region in $\Ghat(P)$, such as the
region marked by $\rho$ in Figure~\ref{fig:hilbex}.
Lemma~\ref{lem:rightregionnotfixed} gives that the region $\rho$ is not fixed by
the involution. 
\begin{figure}[htbp]
    \begin{center}
    \begin{subfigure}[b]{.45\textwidth}
        \begin{center}
            \begin{tikzpicture}
                \node[vertex] (0) at (0,0) [label=left:$0$] {};
                \node[vertex] (-2) at (1,0) [label=right:$-2$] {};
                \node[vertex] (-1) at (0,1) [label=above:$-1$] {};
                \node[vertex] (-1') at (.25,1) [label=right:$-5$] {};
                \node[vertex] (-6) at (2,1) [label=right:$-3$] {};
                \node[vertex] (-5) at (1,2) [label=right:$-4$] {};
                \node[vertex] (1) at (0,-1) [label=below:$1$] {};
                \node[vertex] (1') at (-.25,-1) [label=below:$5$] {};
                \node[vertex] (5) at (-1,-2) [label=right:$4$] {};
                \node[vertex] (6) at (-2,-1) [label=left:$3$] {};
                \node[vertex] (2) at (-1,0) [label=left:$2$] {};
                \draw (5)--(1')--(2)--(6)--(5);
                \draw (1)--(2);
                \draw (-1)--(-2);
                \draw (1)--(-2);
                \draw (-2)--(-1')--(-5)--(-6)--(-2);
                \draw (-1)--(2);
                \draw (1)--(0)--(-1);
            \end{tikzpicture}
        \end{center}
        \caption{$\Gbhat(P')$}
    \end{subfigure}
        \begin{subfigure}[b]{.45\textwidth}
        \begin{center}
            \begin{tikzpicture}
                \node[vertex] (-2) at (1,0) [label=right:$-2$] {};
                \node[vertex] (-1) at (0,1) [label=above:$-1$] {};
                \node[vertex] (-1') at (.25,1) [label=right:$-5$] {};                
                \node[vertex] (-6) at (2,1) [label=right:$-3$] {};
                \node[vertex] (-5) at (1,2) [label=right:$-4$] {};
                \node[vertex] (1) at (0,-1) [label=below:$1$] {};
                \node[vertex] (1') at (-.25,-1) [label=below:$5$] {};                
                \node[vertex] (5) at (-1,-2) [label=right:$4$] {};
                \node[vertex] (6) at (-2,-1) [label=left:$3$] {};
                \node[vertex] (2) at (-1,0) [label=left:$2$] {};
                \draw (5)--(1')--(2)--(6)--(5);
                \draw (1)--(-2);
                \draw (-2)--(-1')--(-5)--(-6)--(-2);
                \draw (-1)--(2);
                \draw (1)--(2);
                \draw (-1)--(-2);
            \end{tikzpicture}
        \end{center}
        \caption{$\Gchat(\Pc')$}
    \end{subfigure}
\end{center}
    \caption{Having opened one notch to obtain $P'$}
    \label{fig:cistep1}
\end{figure}
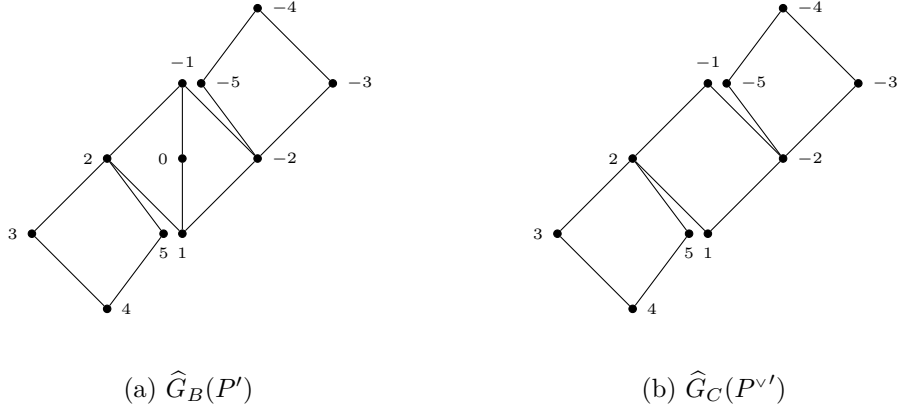
Lemma~\ref{lem:notchopening} guarantees that $\Ghat(P)$ can be obtained by
closing a notch along the left border of $\rho$. Figure~\ref{fig:cistep1} shows the $\Ghat(P')$ obtained by opening a notch along $(-2,-1)$ and $(1,2)$ in $\Ghat(P)$. Proposition~\ref{prop:closenotchmodout} gives that
\[
\Rprt = \Rrt{P'}/(x_1x_2^{-1}-x_5x_2^{-1}) \quad \text{and}\quad 
\Rrt{P'^\smvee} = \Rrt{P'^\smvee}/(x_1x_2^{-1}-x_5x_2^{-1}).
\]
Typically, one thinks of $\Rrt{P'}$ as graded by $\bZ^5$ and $\Rprt$ as
graded by $\bZ^4$. One can alter the grading of $\Rrt{P'}$ so that $\deg
x_1=\deg x_5 = (1,0,0,0)$, so that $\Rrt{P'}$ is also graded by $\bZ^4$
and $\Rprt$ is a quotient of $\Rrt{P'}$ by a homogeneous ideal. Keeping
the altered grading in mind, one then has from
Proposition~\ref{prop:hilbseriesprincipalidea} that
\begin{equation}
\label{eq:hilbseriesexpressions}
\begin{split}
\Hilb(\Rprt,\bm{x}) &= (1-x_1x_2^{-1})\Hilb(\Rrt{P'},\bm{x})|_{x_5=x_1}
\quad \text{and}\quad\\
\Hilb(\Rrt{\Pc},\bm{x}) &= (1-x_1x_2^{-1})\Hilb(\Rrt{P'^\smvee},\bm{x})|_{x_5=x_1}
\end{split}
\end{equation}
$\Gbhat(P')$ and $\Gchat(P'^\smvee)$ each have two biconnected components,
associated to signed posets $P_1$ and $P_2$, shown in
Figure~\ref{fig:biconnected}.
\begin{figure}[htbp]
		\begin{center}
		\begin{tabular}{cc}
				\begin{subfigure}{.4\textwidth}
                    \begin{center}
                        \begin{tikzpicture}
                            \node[vertex] (0) at (0,0) [label=left:$0$] {};
                            \node[vertex] (-2) at (1,0) [label=below:$-2$] {};
                            \node[vertex] (-5) at (0,1) [label=left:$-5$] {};
                            \node[vertex] (-3) at (2,1) [label=right:$-3$] {};
                            \node[vertex] (-4) at (1,2) [label=left:$-4$] {};
                            \node[vertex] (2) at (-1,0) [label=left:$2$] {};
                            \node[vertex] (3) at (-2,-1) [label=left:$3$] {};
                            \node[vertex] (4) at (-1,-2) [label=below:$4$] {};
                            \node[vertex] (5) at (0,-1) [label=right:$5$] {};
                            \draw (2)--(3)--(4)--(5)--(2);
                            \draw (-2)--(-3)--(-4)--(-5)--(-2);
                        \end{tikzpicture}
                    \end{center}    
                    \caption{$\Gbhat(P_1)$}
						\label{subfig:gbhatp1}
				\end{subfigure} &
				\begin{subfigure}{.4\textwidth}
                    \begin{center}
                        \begin{tikzpicture}
                            \node[vertex] (-2) at (1,0) [label=below:$-2$] {};
                            \node[vertex] (-5) at (0,1) [label=left:$-5$] {};
                            \node[vertex] (-3) at (2,1) [label=right:$-3$] {};
                            \node[vertex] (-4) at (1,2) [label=left:$-4$] {};
                            \node[vertex] (2) at (-1,0) [label=left:$2$] {};
                            \node[vertex] (3) at (-2,-1) [label=left:$3$] {};
                            \node[vertex] (4) at (-1,-2) [label=below:$4$] {};
                            \node[vertex] (5) at (0,-1) [label=right:$5$] {};
                            \draw (2)--(3)--(4)--(5)--(2);
                            \draw (-2)--(-3)--(-4)--(-5)--(-2);
                        \end{tikzpicture}
                    \end{center}
						\caption{$\Gchat(P^\smvee_1)$}
						\label{subfig:gchatp1}
				\end{subfigure}\\
				\begin{subfigure}{.4\textwidth}
                    \begin{center}
                        \begin{tikzpicture}
                            \node[vertex] (0) at (0,0) [label=left:$0$] {};
                            \node[vertex] (2) at (-1,0) [label=left:$2$] {};
                            \node[vertex] (-2) at (1,0) [label=right:$-2$] {};
                            \node[vertex] (1) at (0,-1) [label=left:$1$] {};
                            \node[vertex] (-1) at (0,1) [label=left:$-1$] {};
                            \draw (1)--(2)--(-1)--(0)--(1)--(-2)--(-1);
                        \end{tikzpicture}
                    \end{center}
						\caption{$\Gbhat(P_2)$}
						\label{subfig:gbhatp2}
				\end{subfigure} &
				\begin{subfigure}{.4\textwidth}
                    \begin{center}
                        \begin{tikzpicture}
                            \node[vertex] (2) at (-1,0) [label=left:$2$] {};
                            \node[vertex] (-2) at (1,0) [label=right:$-2$] {};
                            \node[vertex] (1) at (0,-1) [label=left:$1$] {};
                            \node[vertex] (-1) at (0,1) [label=left:$-1$] {};
                            \draw (1)--(2)--(-1)--(-2)--(1);
                        \end{tikzpicture}
                    \end{center}
						\caption{$\Gchat(P^\smvee_2)$}
						\label{subfig:gchatp2}
				\end{subfigure}\\
		\end{tabular}
\end{center}
\caption{$P'$ broken into biconnected components}
\label{fig:biconnected}
\end{figure}
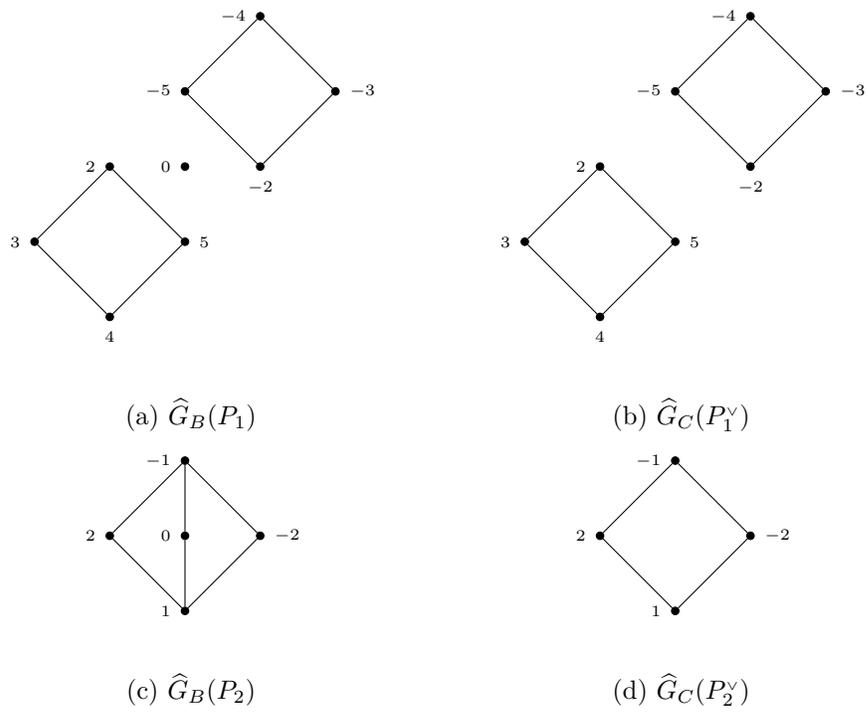
Both $P_1$ and $P^\smvee_1$ have a single biconnected component and principal
toric ideals generated by $U_{\m34}U_{\m23}-U_{4\m5}U_{\m25}$. Then one has
\begin{equation}\label{eq:p1hilbseries}
    \begin{split}
    \Hilb(\Rrt{P_1},\bm{x})&=\dfrac{(1-x_4x_2^{-1})}{(1-x_4x_3^{-1})(1-x_3x_2^{-1})(1-x_4x_5^{-1})(1-x_5x_2^{-1})}
    \quad \text{and} \\
    \Hilb(\Rrt{P^\smvee_1},\bm{x}) &=
    \dfrac{(1-x_4x_2^{-1})}{(1-x_4x_3^{-1})(1-x_3x_2^{-1})(1-x_4x_5^{-1})(1-x_5x_2^{-1})}.
\end{split}
\end{equation}
Looking at Figures~\ref{subfig:gbhatp2} and~\ref{subfig:gchatp2}, one sees that
an additional notch can be opened in $\Gbhat(P_2)$, but not in $\Gchat(P_2^\smvee)$.
Opening the notch in $\Gbhat(P_2)$ results in a signed poset $P_3$ shown in
Figure~\ref{fig:p3}.
\begin{figure}
    \begin{center}
        \begin{tikzpicture}
            \node[vertex] (2) at (-2,0) [label=left:$2$] {};
            \node[vertex] (0) at (0,0) [label=left:$0$] {};
            \node[vertex] (-2) at (2,0) [label=right:$-2$] {};
            \node[vertex] (1) at (-1,-1) [label=left:$1$] {};
            \node[vertex] (3) at (1,-1) [label=right:$3$] {};
            \node[vertex] (-3) at (-1,1) [label=left:$-3$] {};
            \node[vertex] (-1) at (1,1) [label=right:$-1$] {};
            \draw (1)--(2)--(-3)--(0)--(1);
            \draw (3)--(0)--(-1)--(-2)--(3);
        \end{tikzpicture}
    \end{center}
    \caption{$\Gbhat(P_3)$}
    \label{fig:p3}
\end{figure}
Both $\Gbhat(P_3)$ and $\Gchat(P_2^\smvee)$ each have only one biconnected component
and a single cycle. This gives
\begin{align}
    \Hilb(\Rrt{P_2},\bm{x}) &=
    (1-x_1)\Hilb(\Rrt{P_3},\bm{x})|_{x_3=x_1}\nonumber \\
    &=\left.\dfrac{(1-x_1)(1-x_1x_3)}{(1-x_1x_2^{-1})(1-x_2x_3)(1-x_3)(1-x_1)}\right|_{x_3=x_1}
    \label{eq:p2bhilbseries} \\
    &= \dfrac{1-x_1^2}{(1-x_1x_2^{-1})(1-x_1x_2)(1-x_1)}\nonumber
\end{align}
and
\begin{equation}\label{eq:p2chilbseries}
    \Hilb(\Rrt{P_2^\smvee},\bm{x}) =
    \dfrac{(1-x_1^2)}{(1-x_1x_2^{-1})(1-x_1x_2)}
\end{equation}
Since $P_1$ and $P_2$ are the biconnected components of $P'$, from
Proposition~\ref{prop:biconnectedred}, one has that $\Rrt{P'} = \Rrt{P_1}
\otimes \Rrt{P_2}$ and $\Rrt{\Pc'} = \Rrt{P_1^\smvee}\otimes \Rrt{P_2^\smvee}$,
meaning
\begin{multline*}
    \Hilb(\Rrt{P'},\bm{x}) = \Hilb(\Rrt{P_1},\bm{x})\Hilb(\Rrt{P_2},\bm{x})
    \quad\text{and} \\
    \Hilb(\Rrt{\Pc'},\bm{x}) =
    \Hilb(\Rrt{P^\smvee_1},\bm{x})\Hilb(\Rrt{P_2^\smvee},\bm{x}).
\end{multline*}
Combining \eqref{eq:hilbseriesexpressions}, \eqref{eq:p1hilbseries},
\eqref{eq:p2bhilbseries} and \eqref{eq:p2chilbseries} gives
\[
\begin{split}
    \Hilb&(\Rprt,\bm{x})  = (1-x_1x_2^{-1})\Hilb(\Rprt,\bm{x})|_{x_5=x_1} \\
    =& (1-x_1x_2^{-1})\Hilb(\Rrt{P_1},\bm{x})\Hilb(\Rrt{P_2},\bm{x})|_{x_5=x_1}
    \\
    =&
    \left.\dfrac{(1-x_1x_2^{-1})(1-x_4x_2^{-1})(1-x_1^2)}{(1-x_4x_3^{-1})(1-x_3x_2^{-1})(1-x_4x_5^{-1})(1-x_5x_2^{-1})(1-x_1x_2^{-1})(1-x_1x_2)(1-x_1)}\right|_{x_5=x_1}\\
    =&\dfrac{(1-x_4x_2^{-1})(1-x_1^2)}{(1-x_4x_3^{-1})(1-x_3x_2^{-1})(1-x_4x_1^{-1})(1-x_1x_2^{-1})(1-x_2x_4)(1-x_1x_2)(1-x_1)}.
\end{split}
\]
and
\[    
    \begin{split}
\Hilb&(\Rrt{\Pc},\bm{x})=(1-x_1x_2^{-1})\Hilb(\Rrt{\Pc'},\bm{x})|_{x_5=x_1}
    \\
    =&
    (1-x_1x_2^{-1})\Hilb(\Rrt{P^\smvee_1},\bm{x})\Hilb(\Rrt{P_2^\smvee},\bm{x})|_{x_5=x_1}
    \\
    =&\left.\dfrac{(1-x_1x_2^{-1})(1-x_4x_2^{-1})(1-x_1^2)}{(1-x_4x_3^{-1})(1-x_3x_2^{-1})(1-x_4x_5^{-1})(1-x_5x_2^{-1})(1-x_1x_2^{-1})(1-x_1x_2)}\right|_{x_5=x_1}
    \\
    =&\dfrac{(1-x_1x_2^{-1})(1-x_4x_2^{-1})(1-x_1^2)}{(1-x_4x_3^{-1})(1-x_3x_2^{-1})(1-x_4x_1^{-1})(1-x_1x_2^{-1})(1-x_1x_2^{-1})(1-x_1x_2)}
    \\
    =&\dfrac{(1-x_4x_2^{-1})(1-x_1^2)}{(1-x_4x_3^{-1})(1-x_3x_2^{-1})(1-x_4x_1^{-1})(1-x_1x_2^{-1})(1-x_1x_2^{-1})(1-x_1x_2)}.
\end{split}
\]

The edges which correspond to the notches that are opened form a set of
\emph{disconnecting chains} splitting $\Ghat(P)$ into biconnected components.

\begin{defn}\label{def:disconnectingchain}
Suppose $P \subset \Bn$ (\resp $\Pc \subset \Cn$) is a signed poset. A
\emph{disconnecting chain} of $\Ghat(P)$ is a chain $c_1 \lessdot c_2 \lessdot
\dots \lessdot c_k$ such that removing $c_1 \lessdot \cdots \lessdot c_k$ and
$-c_k \lessdot \cdots \lessdot
-c_1$ breaks $\Ghat(P)$ into three connected components.
\end{defn}

The first step in the proof is to locate a region to work with.

\begin{prop}\label{prop:rightregion}
	Suppose $Q$ is a strongly planar poset whose Hasse diagram is connected.
    Then $Q$ has a region $\rho$ such that any vertex other than the maximum
    and minimum element of $\rho$ in the right border of $\rho$ is in the
    border of no other region. Such a region is called a \emph{rightmost
        region}.
\end{prop}

  By considering a strongly planar poset, one has the luxury of assuming a
  given strongly planar embedding, allowing sensible notions of ``left'' and
  ``right''. 
  \begin{defn}\label{def:rightmostcover}
		  Suppose $Q$ is a strongly planar poset. The \emph{rightmost cover} of $x \in Q$ is the $y$ that covers $x$ such that if one traversed a small circle around $x$ counterclockwise (starting from the bottom, say) one passes the edge leading to $y$ last. The \emph{rightmost lower cover} is the $z \lessdot x$ such that the edge leading to $z$ is encountered first when traveling clockwise from the top of the circle.
  \end{defn}

\begin{proof}[Proof of Proposition~\ref{prop:rightregion}]
Without loss of generality, one may assume that every vertex in $Q$ has degree
at least two, since one makes this assumption of $\Ghat(P)$.
Since $Q$ is strongly planar, the poset $\widehat{Q}$ obtained by adding a $\hat{0}$ and $\hat{1}$ to $Q$ is also planar. Construct a saturated chain $\hat{0}\lessdot c_1 \lessdot \cdots \lessdot c_k \lessdot \hat{1}$ such that $c_{i+1}$ is the rightmost cover of $c_i$ for all $i$ and $c_1$ is the rightmost cover of $\hat{0}$. 

  Let $a = c_i$, where $i$ is maximal such that $c_i$ is covered by at least
  two elements. (Since $c_1$ is a minimal element of $Q$, and $Q$ has been
  assumed to have no vertices of degree 1, it has at least two covers, so such
  an $a$ exists.) Construct a chain $a \lessdot b_1 \lessdot \cdots \lessdot
  b_\ell \lessdot \hat{1}$ where $b_1$ is the cover of $a$ that is rightmost but one and $b_{i+1}$ is the rightmost cover of $b_i$.

    \noindent\textbf{Claim: } $c_k$ is the rightmost lower cover of $\hat{1}$ in $\widehat{Q}$.

    Suppose not and $d \lessdot \hat{1}$ is to the right of $c_k$. Then $d$ must lie above some minimal element of $Q$, say $f$. Since $Q$ is strongly planar, a maximal chain from $f$ to $d$ must intersect $c_1 \lessdot \cdots \lessdot c_k$. Let $j$ be maximal such that $c_j$ is in this chain from $f$ to $d$. Since $d \ne c_k$, one must have that $j \ne k$. Since $d$ is to the right of $c_k$, the chain must continue along the rightmost cover of $c_j$. However, this is $c_{j+1}$, contradicting the maximality of $j$.
	
  \noindent\textbf{Claim: } $b_{\ell}=c_k$.

  Suppose not. Then $c_k$ has degree $\geq 2$ and it was maximal in $Q$, so it must cover some $d \ne
  c_{k-1}$. Then $Q$ has a minimal element $a'$ such that $a' < d$. As $Q$ is
  strongly planar, a saturated chain from $a'$ to $d$ must include either some
  $b_i$ or $c_j$, $j < k$. If the chain includes $b_i$, then there must be some
  $m$ where $b_{m+1}$ is not the rightmost cover of $b_m$, a contradiction. If
  $c_j$ is in the chain, but no $b_i$, then, since $d \ne c_k$, there must be
  an $i$ such that $c_n > a$ and $c_n$ has more than two covers, contradicting
  that $a=c_i$ was maximal with this property.

Let $\rho$ be the region enclosed by the $b_i$ and $c_j$. Note that $b_\ell=c_k$ need not be the maximal element of $\rho$. However, by construction, $a$ will be the minimal element of $\rho$.

  \noindent\textbf{Claim: } The $c_j$ with $c_j > a$ are rightmost in $Q$, i.e.\ they are not in the left border of any region.
 
  Suppose $c_j$ is on the right border of $\rho$ (and is not $\max(\rho)$ or
  $\min(\rho)$), that $c_j$ is on the left border of some region $\sigma$ and
  that
  $j$ is minimal such that this is the case. Then, since $c_{j+1}$ is the
  rightmost cover of $c_j$, the edge $(c_j, c_{j+1})$ must be in the left
  border of $\sigma$. Consequently, $c_j$ cannot be the minimal element of
  $\sigma$. Then, by the minimality of $j$, there is some $d \lessdot c_j$ such
  that $d$ is to the right of $c_{j-1}$, as depicted in
  Figure~\ref{fig:rightregionpf}
  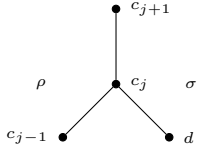
\begin{figure}[htbp]
  \begin{center}
	  \begin{tikzpicture}
		  \node[circle,fill=black,draw,inner sep=1pt] (j) at (0,0) [label=right:$c_j$] {};
		  \node[circle,fill=black,draw,below left of=j,inner sep=1pt] (j1) [label=left:$c_{j-1}$] {};
		  \node[circle,fill=black,draw,below right of=j,inner sep=1pt] (d) [label=right:$d$] {};
		  \node[circle,fill=black,draw,above of=j,inner sep=1pt] (jplus) [label=right:$c_{j+1}$] {};
		  \node[right of=j] {$\sigma$};
		  \node[left of=j] {$\rho$};
		  \draw (j)--(j1);
		  \draw (j)--(d);
		  \draw (j)--(jplus);
	  \end{tikzpicture}
  \end{center}
  \caption{The situation in the last claim of the proof of
      Proposition~\ref{prop:rightregion}}
  \label{fig:rightregionpf}
  \end{figure}
  The region $\sigma$ must have a minimum element, call it $f$. Then $f \geq g$ for some $g$ that is minimal in $Q$. Then, since $Q$ is strongly planar and $c_1$ is the rightmost minimal element of $Q$, a saturated chain from $g$ to $f$ must pass through some $c_i$. But then, since $c_{i+1}$ is the rightmost cover of $c_i$, one must have that $f$ is one of the $c$'s, a contradiction. 
\end{proof}

One can now use the previous result to prove a few lemmas specific to $\Gbhat(P)$ and $\Gchat(\Pc)$.

\begin{lem}\label{lem:onlyoneplusminus}
	Suppose $\Pc \subset \Cn$ is a $\Cn$-signed poset such that $\Gchat(\Pc)$ is strongly planar. Then there is at most one $i$ such that $\pm 2e_i$ corresponds to an edge in $\Gchat(\Pc)$.
\end{lem}

\begin{proof}
	Suppose $\Pc \subset \Cn$ is a signed poset such that $\Gchat(\Pc)$ is
    strongly planar and $i \lessdot -i$ and $j \lessdot -j$. Then one must have
    $i < -j$ and $j < -i$. Then if $\Gchat(\Pc)$ is embedded in the plane (so
    that if $a < b$, then $y_a < y_b$), there must be a path from $i$ to $-j$
    and $j$ to $-i$, but these paths must intersect. For $\Gchat(\Pc)$ to be
    planar, they must intersect at a vertex, say $k$. But then $\Gchat(\Pc)$
    would cease to be the Hasse diagram of a poset, as the relation $i < -i$ would follow by transitivity from $i < k$ and $k < -i$. Thus, such a $\Gchat(\Pc)$ cannot be strongly planar.
\end{proof}

\begin{lem}\label{lem:rightregionnotfixed}
	Suppose $P \subset \Bn$ (\resp $\Pc \subset \Cn$) is a signed poset such
    that $\Gbhat(P)$ (\resp $\Gchat(\Pc)$) is strongly planar. Then, if
    $\Gbhat(P)$ (\resp $\Gchat(P)$) encloses more than one region, a cycle
    defining a rightmost region of $\Ghat(P)$ is not fixed by the involution.
\end{lem}

\begin{proof}
	Suppose $C$ is a cycle in $\Gbhat(P)$ (\resp $\Gchat(\Pc)$) that encloses a rightmost region, $\rho$. Suppose $C$ is fixed by the involution. Let $\sigma$ be a region to the immediate left of $\rho$. The type $B$ and type $C$ cases are slightly different.
	
	First, suppose $P \subset \Bn$. There are two cases.
	\begin{enumerate}
		\item \textbf{Suppose $0$ is not in the left border of $\rho$.} Let $e$ be an edge that lies in the left border of $\rho$ and the right border of $\sigma$. Since $0$ is not in the left border of $\rho$, an edge and its image under the involution cannot both be in the left border. Therefore, since $\rho$ is fixed by the involution, $\iota(e)$ must lie in the right border of $\rho$. But $\iota(e)$ is also in the border of $\iota(\sigma)$, contradicting that $\rho$ is a rightmost region.

		\item \textbf{Suppose $0$ is in the left border of $\rho$.} Then if $j$ is in the left border, $-j$ must also be in the left border, since $\rho$ is fixed by the involution and $\pm j < 0 < \mp j$. Then the left border is fixed by the involution. Since $\rho$ itself is fixed by the involution, this means that the right border must also be fixed and thus symmetric, a contradiction, since $0$ is in the left border, not the right.
	\end{enumerate}

	Next, consider $\Pc \subset \Cn$. There are two cases.
	\begin{enumerate}
		\item \textbf{Suppose there is no edge of the form $(i, -i)$ in the left border
            of $\rho$.} Since $\sigma$ is to the left of $\rho$, there is some
            edge $e$ that lies in the left border of $\rho$ and the right
            border of $\sigma$. Since $\rho$ is fixed by the involution, $\iota(e)$
            lies in the border of $\rho$. Since there is no edge of the form
            $(i, -i)$ in the left border, it is not symmetric and $\iota(e)$ cannot
            be in the left border, so it is in the right border. But,
            $\iota(e)$ is in the border of $\iota(\sigma)$, meaning $\rho$ cannot be a rightmost region, a contradiction.

		\item \textbf{Suppose there is an edge $(i, -i)$ in the left border of
                $\rho$.}
            Then, by the definition of $\Gchat(\Pc)$, if $j$ is in the left
            border of $\rho$, then $-j$ is also in the left border of $\rho$.
            From Lemma~\ref{lem:onlyoneplusminus}, one knows that the right
            border of $\rho$ does not include an edge $(k, -k)$, so there must
            be some $k$ in the right border such that $-k$ is not in the right
            border. (In fact, the absence of such an edge means only one of $k$
            and $-k$ is in the right border for any $k$ that is not maximal or
            minimal in $C$.) Then $\iota(k) = -k$ is a vertex in $\iota(C)$, which cannot be $C$, so $C$ is not fixed by the involution.
	\end{enumerate}
      \end{proof}

\begin{prop}\label{prop:spimpliescssp}
	Suppose $P \subset \Bn$ (\resp $\Pc \subset \Cn$) is a signed poset. Then $\Gbhat(P)$ (\resp $\Gchat(\Pc$) has a strongly planar embedding if and only if it has a centrally symmetric strongly planar embedding.
\end{prop}

\begin{proof}
	One needs to show only that if $\Gbhat(P)$ (\resp $\Gchat(\Pc)$) is
	strongly planar, it has a centrally symmetric strongly planar embedding.
	Instead, prove a slightly stronger statement, namely that if $\Gbhat(P)$
	(\resp $\Gchat(\Pc)$) has a strongly planar embedding, it also has a
	centrally symmetric strongly planar embedding such that the outer border of
	$\Gbhat(P)$ (\resp $\Gchat(\Pc)$) is the same in both the first strongly
	planar embedding and the centrally symmetric strongly planar embedding.

	Induct on the number of regions in $\Gbhat(P)$ (\resp $\Gchat(\Pc)$). There are two base cases.

	First, suppose $\Gbhat(P)$ (\resp $\Gchat(\Pc)$) consists of a single region, bounded by a single cycle. One needs to address type $B$ and type $C$ separately. 
	\begin{itemize}
		\item In type $C$, consider $\Gchat(\Pc)$. Since it is strongly planar,
            it has a single maximum and a single minimum, which must be $i$ and $-i$. 
		
		\textbf{Claim: } There is no $j \ne i$ such that $j < -j$ or $-j < j$ in $\Gchat(\Pc)$.

		Suppose not. Since $\Gchat(P)$ consists of a single cycle, one has that
		either $i < j < -j < -i$, in which case one border of $\Gchat(\Pc)$ is
		fixed by the involution, so $\Gchat(\Pc)$ must have more than one
		region, a contradiction, or $j$ lies in the left border of $\Gchat(P)$ and $-j$ lies in
		the right border of $\Gchat(\Pc)$. In the latter case, one must have
        that the left border of $\Gchat(\Pc)$ is sent to the right border under
        the involution, i.e.\ if $i \lessdot c_1 \lessdot \cdots \lessdot c_k
        \lessdot -i$ is the left border, then $i \lessdot -c_k \lessdot \cdots
        \lessdot -c_1 \lessdot -i$ is the right border. However, it is then
        impossible that $j < -j$ since $\Gchat(\Pc)$ encloses a single region. 
        
        Thus, when $\Gchat(\Pc)$ is strongly planar and encloses a single
        region, $\Gchat(\Pc)$ can be embedded centrally symmetrically by evenly spacing the vertices around a circle centered at the origin with $-i$ at the top and $-i$ at the bottom.

		\item In type $B$, consider $\Gbhat(P)$. Since $i < -i$, one must have
            $i < 0 < -i$, meaning either the left or right border of
            $\Gbhat(P)$ is symmetric about $0$, so it is fixed by the
            involution. But then $\Gbhat(P)$ must have more than one region,
            contradicting the assumption.
\end{itemize}

Next, consider the second base case where $\Gbhat(P)$ (\resp $\Gchat(\Pc)$) consists of two
regions. From Lemma~\ref{lem:rightregionnotfixed}, one knows that the rightmost
region $\rho$ is not fixed by the involution. Therefore, the border between the
two regions consists of a chain fixed by the involution and the two regions are
exchanged by the involution. One then has the scenario depicted in
Figure~\ref{fig:centrallysymregions}.
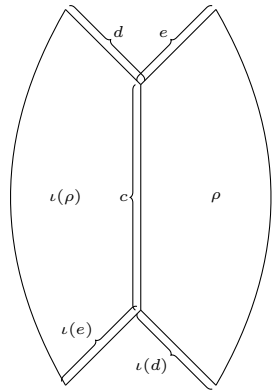
\begin{figure}[htbp]
	\begin{center}
		\begin{tikzpicture}
			\draw (-1,-1)--(0,0);
			\draw (0,0)--(0,3);
			\draw (0,3)--(1,4);
			\draw (1,-1)--(0,0);
			\draw (0,3)--(-1,4);
			\path (-1,-1) edge [bend left](-1,4);
			\path (1,-1) edge [bend right] (1,4);
			\node at (1,1.5) {$\rho$};
			\node at (-1,1.5) {$\iota(\rho)$};
			\draw [decorate,decoration={brace,amplitude=2pt,raise=2pt}] (0,0)--(0,3)
			node [midway,left=1pt] {$c$};
			\draw [decorate,decoration={brace,amplitude=2pt,raise=2pt,mirror}] (0,3) -- (-1,4)
			node[midway,above right] {$d$};
			\draw [decorate,decoration={brace,amplitude=2pt,raise=2pt}] (0,3)--(1,4)
			node[midway,above left] {$e$};
			\draw [decorate,decoration={brace,amplitude=2pt,raise=2pt}] (-1,-1)--(0,0)
			node[midway,above left] {$\iota(e)$};
			\draw [decorate,decoration={brace,amplitude=2pt,raise=2pt}] (1,-1)--(0,0)
			node[midway,below left] {$\iota(d)$};
		\end{tikzpicture}
	\end{center}
    \caption{Two regions exchanged by the involution separated by a chain, $c$, fixed
        by the involution}
    \label{fig:centrallysymregions}
\end{figure}
	One sees that $\Gbhat(P)$ (\resp $\Gchat(\Pc)$) (when translated appropriately) is centrally symmetric and strongly planar.

	Now suppose $\Gbhat(P)$ (\resp $\Gchat(\Pc)$) encloses more than two regions. Then, by Proposition~\ref{prop:rightregion}, it has a rightmost region, call it $\rho$. The right border of $\rho$ is defined by some chain $\min(\rho)\lessdot c_1 \lessdot \cdots \lessdot c_k \lessdot \max(\rho)$. Since $\Gbhat(P)$ (\resp $\Gchat(\Pc)$) is a poset and $\rho$ is enclosed by a cycle, $k\geq 1$. Since $\rho$ is rightmost, by definition each edge in its right border lies in no cycle defining a region other than $\rho$. The same then must be true of the image of the right border of $\rho$ in $\iota(\rho)$. Deleting $c_1,\ldots,c_k,-c_1,\ldots,-c_k$ and the incident edges gives $\Gbhat(P')$ (\resp $\Gchat(P'^\smvee)$) for a signed poset $P'$
	
	
Since $\Ghat(P)$ was strongly planar, $\Ghat(P')$ is strongly planar. By induction, $\Ghat(P')$ has a centrally symmetric strongly planar embedding. The left border of $\rho$ is now part of the right border of $\Ghat(P')$, meaning the chain $c_1 \lessdot \cdots \lessdot c_k$ and be attached to form a new region on the right of $\Ghat(P')$ and similarly for $-c_k \lessdot \cdots \lessdot -c_1$ on the left. Since $\pm \min(\rho)$ and $\pm \max(\rho)$ are positioned so as to be centrally symmetric, it is possible to place the chains in such a way as to preserve central symmetry. Replacing the chains in $\Ghat(P')$ gives a centrally symmetric strongly planar embedding of $\Ghat(P)$.
\end{proof}

   \begin{lem}\label{lem:notchopening}
     Suppose $P \subset \Bn$ (\resp $P \subset \Cn$) is a signed poset such that $\Gbhat(P)$ (\resp $\Gchat(\Pc)$) is centrally symmetric and strongly planar. Suppose $\rho$ is a rightmost region of $\Ghat(P)$ and $c_1 \lessdot \cdots \lessdot c_k$ is the portion of the left border of $\rho$ which is part of the right border of some other region(s), with $k\geq 2$ and $c_k \ne -c_{k-1}$ (or $c_1 \ne -c_2$). Then $\Gbhat(P)$ (\resp $\Gchat(\Pc)$) is obtained from $\Gbhat(P')$ (\resp $\Gchat(P')$) for some signed poset $P'$ by closing a signed notch $(c_{k-1},c_k,c_k')$ and $(-c_{k-1},-c_k,-c_k')$ or a signed notch $(c_2,c_1,c_1')$ and $(-c_2,-c_1,-c_1')$.
     \end{lem}

     \begin{proof}
	     Begin by observing that in the type $B$ case, one of $c_1$ and $c_k$
         is nonzero, and in the type $C$ case, at least one of $c_k \ne
         -c_{k-1}$ and $c_1 \ne -c_2$ holds. Without loss of generality, assume $c_k \ne 0$ in $\Gbhat(P)$ and $c_k \ne -c_{k-1}$ in $\Gchat(\Pc)$. (The argument in the other cases is symmetric.)

	     Construct $\Gbhat(P')$ (\resp $\Gchat(P')$) by replacing $c_k$ by $c_k$ and $c_k'$, with an edge between $c_{k-1}$ and each of $c_k$ and $c_k'$, with the edges incident to $c_k$ in $\Gbhat(P)$ (\resp $\Gchat(\Pc)$) that are not in the border of $\rho$ moved to be incident to $c_k'$. Replace $-c_k$ by $-c_k$ and $-c_k'$ and partition the edges incident to $-c_k$ in $\Gbhat(P)$ (\resp $\Gchat(\Pc)$) in the same way. By construction, $\Gbhat(P')$ (\resp $\Gchat(P'c)$) will be centrally symmetric and strongly planar and have an involution sending $i < j$ to $-j < -i$ (and fixing $0$ in the type $B$ case). Then to show $P'$ is really a signed poset, one needs to check the conditions in Proposition~\ref{prop:gbhatcharacterisation} and Proposition~\ref{prop:gchatcharacterisation}.
      
	First, consider the type $B$ case. Suppose $i < -i$ in $\Gbhat(P')$. Since opening a notch does not introduce relations, one must have that $i < -i$ in $\Gbhat(P)$, so $i < 0 < -i$ in $\Gbhat(P)$. Suppose one does not have $-i > 0$ in $\Gbhat(P')$. There are two cases where $0 < -i$ in $\Gbhat(P)$ could fail to lift to $\Gbhat(P')$.
	\begin{itemize}
		\item \textbf{Suppose $-i > c_k$ and $0 < c_k'$, but $0 \not < -i$.}

		  By construction (keeping in mind that $\Gbhat(P)$ is centrally symmetric and strongly planar), $0$ must be in the right border of $\rho$ and $c_k'=\max(\rho)$. However $0$ being in the right border contradicts central symmetry.

		\item \textbf{Suppose $-i > -c_k$ and $0 < -c_k'$, but $0 \not < -i$.}

		  By construction, one has that $0$ is in the part of the border of
          $\iota(\rho)$ in both $\Gbhat(P)$ and in $\Gbhat(P')$ not shared by any other region. Thus, since $i< 0$ in $\Gbhat(P)$, one must have $i < 0$ in $\Gbhat(P')$, contradicting that $i < -i$ in $\Gbhat(P')$.
	\end{itemize}

       In the type $C$ case, one must check that if $i < -i$ and $j < -j$ in
       $\Gchat(P')$, then $i < -j$ and $j < -i$.
	
	First, observe that if $i< -i$ in $\Gchat(P')$, then $i < -i$ in $\Gchat(\Pc)$. (If $i = c_k'$ in $\Gchat(P')$, then $c_k < -c_k$ in $\Gchat(\Pc)$.) Now suppose that $i < -i$ and $j < -j$ in $\Gchat(P')$. Then $i < -j$ and $j < -i$ in $\Gchat(\Pc)$ (possibly vacuously if $i =j$ in $\Gchat(\Pc)$). Should these relations fail to lift to $\Gchat(P')$, one is in one of the following two cases.
       \begin{enumerate}
		\item \textbf{Suppose $i \leq c_k$ and $-j \geq c_k'$.}

			There are several subcases.
			\begin{enumerate}[label=(\arabic{enumi}.\alph*)]
				\item \textbf{Suppose $i < c_{k-1}$.} Then $i < -j$ and one is done.
				\item \textbf{Suppose $i < c_{k}$, but $i \not < c_{k-1}$.} In
                    other words, $i$ is on the right border of $\rho$ and $c_k
                    = \max(\rho)$. Since $\rho$ is not fixed by the involution
                    and is a right region, $-i$ is not on the right border of
                    $\rho$. (If it were, either $\rho$ would contain an edge $(k
                    , -k)$ and thus not be a right region, or $i = \min(\rho)$, $-i = \max(\rho)$ and $\rho$ would be cut out by a cycle fixed by the involution, since $\max(\rho)$ has degree 2, by construction.) Therefore, one cannot have that $i < -i$ in $\Gchat(\Pc)$, a contradiction.
				\item \textbf{Suppose $ i =c_k$.} Then $c_k < -c_k$, but, by
                    construction of $\Gchat(P'^\smvee)$, elements in the
                    border of $\rho$ other than $c_1 \lessdot \cdots \lessdot
                    c_{k-1}$ are not in the border of any other region in
                    $\Gchat(P')$ and $-c_k$ is in the border of $\iota(\rho)$, a
                    contradiction, a contradiction.

			  \end{enumerate}
			\item \textbf{Suppose $i \leq c_k'$ and $-j \geq c_k$.} In this case, $-j$ must be
					on the outer border of $\rho$, as it is not one of $c_1
					\lessdot \cdots \lessdot c_{k-1}$. Then, if $j < -j$, one has that either $-c_k < c_k$ or $-c_1 < c_1$. In the latter case, it follows that $-c_k < c_k$. Therefore, without loss of generality, suppose $-j = c_k$. Consider two cases.
			\begin{enumerate}[label=(\arabic{enumi}.\alph*)]
				\item \textbf{There is a chain from $-c_k$ to $c_k$ which is
                        not fixed by the involution.} This chain and its image form a cycle fixed by the involution, having $c_k$ as its maximum and $-c_k$ as its minimum. But then $c_k$, which has degree two, must cover two elements in this cycle, meaning $c_k = \max (\rho)$ and this cycle is the one that cuts out $\rho$. This is a contradiction, since $\rho$ is not fixed by the involution.

				\item \textbf{There is a unique chain from $-c_k$ to $c_k$.} By construction of $\Gchat(P')$, this chain passes through $c_{k-1}$ and $-c_{k-1}$ and extends uniquely to a maximal chain from $\min(\iota(\rho))$ to $\max(\rho)$. On the other hand, since $i < -i$, there is a chain from $-i$ to $i$. Since $\Gchat(\Pc)$ is centrally symmetric and strongly planar, this chain must pass through the chain from $\min(i\rho))$ to $\max(\rho)$, meaning $-c_{k-1} < -i$ and $i < c_{k-1}$, so $-c_k < -i$ and $i < c_k$, as desired.			      
			      \end{enumerate}

			      Lastly, one must show that $(c_{k-1},c_k,c_k')$ and $(-c_{k-1},-c_k,-c_k')$ forms a signed notch in $\Ghat(P')$. However, this follows from the fact $c_1 \lessdot \cdots \lessdot c_{k-1}$ and $-c_{k-1} \lessdot \cdots \lessdot -c_1$ are disconnecting chains in $\Gbhat(P')$ (\resp $\Gchat(P'))$.
			  \end{enumerate}
     \end{proof}

     One is finally ready to prove Theorem~\ref{thm:csspci}.

\begin{proof}[Proof of Theorem~\ref{thm:csspci}]
  Begin by realizing one may assume courtesy of Proposition~\ref{prop:spimpliescssp} that $\Ghat(P)$ is embedded in $\bR^n$ so as to
  be centrally symmetric and strongly planar. Next, recall that a centrally
  symmetric strongly planar $\Gbhat(P)$ contains at least two regions. If a
  centrally symmetric strongly planar $\Gchat(\Pc)$ has no region not fixed by
  the involution, central symmetry guarantees it consists of one region, cut
  out by a cycle fixed by the involution. This fixed cycle is the only cycle of
  $\Gchat(\Pc)$, so the toric ideal $\Irt{\Pc}$ is the zero ideal and $\Rrt{\Pc}$ is a complete intersection.

 Thus, one may assume $\Gbhat(P)$ (\resp $\Gchat(\Pc)$) has at least one region
 not fixed by the involution. Induct on the number of two-element orbits of
 regions under the involution of $\Gbhat(P)$ (\resp $\Gchat(P)$). Since $\Gbhat(P)$ (\resp $\Gchat(P)$) is centrally symmetric and strongly planar, it has a rightmost region $\rho$ and $\rho$ is not fixed by the involution.
 
 As a base case, consider the case where $\Gbhat(P)$ (\resp $\Gchat(\Pc)$) has only one orbit. The case where $\Gchat(\Pc)$ consists of a single orbit of one region has already been addressed. Thus, one may assume in both the $\Gbhat(P)$ and $\Gchat(\Pc)$ cases that there are precisely two regions, which are exchanged by the involution. Call them $\rho$ and $\iota(\rho)$, with $\rho$ being the right region. There are two cases.
 \begin{enumerate}
	\item \textbf{The regions $\rho$ and $\iota(\rho)$ share an edge and so
            intersect along a chain.} (See Figure~\ref{fig:thmcase1}.)
        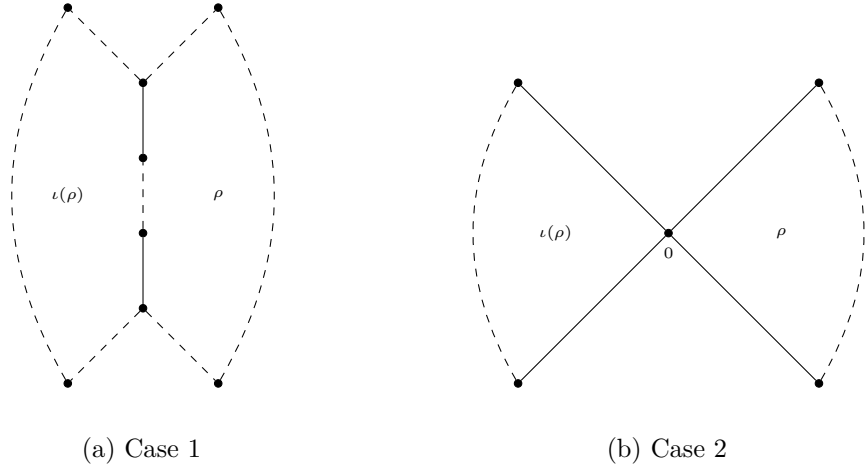
\begin{figure}[htbp]
            \begin{center}
                \begin{subfigure}[b]{.45\textwidth}
                    \begin{center}
                \begin{tikzpicture}
                    \node[vertex] (1) at (0,0){};
                    \node[vertex] (2) at (0,1) {};
                    \node[vertex] (3) at (0,2) {};
                    \node[vertex] (4) at (0,3) {};
                    \node[vertex] (5) at (-1,-1) {};
                    \node[vertex] (6) at (-1,4) {};
                    \node[vertex] (7) at (1,-1) {};
                    \node[vertex] (8) at (1,4) {};
                    \draw (1)--(2);
                    \draw[dashed] (2)--(3);
                    \draw (3)--(4);
                    \draw[dashed] (1)--(5);
                    \draw[dashed] (1)--(7);
                    \draw[dashed] (4)--(6);
                    \draw[dashed] (4)--(8);
                    \draw[dashed] (5)edge[bend left] (6);
                    \draw[dashed] (7) edge[bend right] (8);
                    \node at (1,1.5) {$\rho$};
                    \node at (-1,1.5) {$\iota(\rho)$};
                \end{tikzpicture}
            \end{center}
            \caption{Case 1}\label{fig:thmcase1}
            \end{subfigure}
            \begin{subfigure}[b]{.45\textwidth}
                \begin{center}
                    \begin{tikzpicture}
                        \node[vertex] (0) at (0,0) [label=below:$0$] {};
                        \node[vertex] (1) at (-2,-2) {};
                        \node[vertex] (2) at (-2,2) {};
                        \node[vertex] (3) at (2,2) {};
                        \node[vertex] (4) at (2,-2) {};
                        \draw (1)--(0)--(2);
                        \draw (3)--(0)--(4);
                        \draw[dashed] (1) edge[bend left](2);
                        \draw[dashed] (4) edge[bend right](3);
                        \node at (1.5,0) {$\rho$};
                        \node at (-1.5,0) {$\iota(\rho)$};
                    \end{tikzpicture}
                \end{center}
                \caption{Case 2}\label{fig:thmcase2}
            \end{subfigure}
            \end{center}
            \caption{The first two cases of the proof of Theorem~\ref{thm:csspci}}
        \end{figure}

	  Then $\Gbhat(P)$ (\resp $\Gchat(\Pc)$) has three cycles: the cycle
      defining $\rho$, its image under the involution, which defines $\iota(\rho)$,
      and the cycle consisting of the ``outer border'' of $\Gbhat(P)$ (\resp
      $\Gchat(\Pc)$), i.e.\ those edges surrounding only one of $\rho$ and
      $\iota(\rho)$. This last cycle must be fixed by the involution, as the
      involution exchanges the cycles defining $\rho$ and $\iota(\rho)$ while
      fixing the border between them. Then $\Irt{P}$ (\resp $\Irt{\Pc}$) is generated by the cycle defining $\rho$ and, as the toric ideal is then principal, $\Rrt{P}$ (\resp $\Rrt{\Pc}$) is a complete intersection.

	\item \textbf{The regions $\rho$ and $\iota(\rho)$ do not share an edge and
            instead share a single vertex.} (See Figure~\ref{fig:thmcase2}. Note that this case only arises in $\Gbhat(P)$ (where the shared vertex is $0$).

	  Then the cycle defining $\rho$ and its imagine under the involution,
      which defines $\iota(\rho)$, are the only cycles in $\Gbhat(P)$, so $\Irt{P}$ is a principal ideal and $\Rprt$ is a complete intersection.
      \end{enumerate}

	Suppose $\Ghat(P)$ has $n>1$ orbits of regions and the result holds for posets $P'$ such that $\Ghat(P')$ has $k$ orbits of regions not fixed by the involution. Let $\rho$ be a rightmost region of $\Ghat(P)$. Let $c_1 \lessdot \cdots \lessdot c_k$ be the chain defining the left border of $\rho$. Note that the removal of $c_1 \lessdot \cdots \lessdot c_k$ would disconnect $\Ghat(P)$.

	Let $c_{j_1}\lessdot \cdots \lessdot c_{j_\ell}$ be the portion of the
    border of $\rho$ such that $c_{j_1}$ is the minimal vertex in the border
    that lies in the border of more than one region and $c_{j_\ell}$ is the
    maximal vertex in the border that lies in the border of more than one
    region. (Equivalently, $c_{j_1}$ and $c_{j_\ell}$ are minimal and maximal,
    respectively, elements in the chain having degree $\geq 3$. See
    Figure~\ref{fig:cchain}.) Induct on $\ell$.
    
    \begin{figure}[htbp]
        \begin{center}
            \begin{tikzpicture}
                \node[vertex] (cj1) at (0,0) [label=left:$c_{j_1}$] {};
                \node[vertex] (cjl) at (0,3) [label=left:$c_{j_\ell}$] {};
                \node[vertex] (c1) at (1,-1) [label=right:$c_1$] {};
                \node[vertex] (ck) at (1,4) [label=right:$c_k$] {};
                \node (top) at (-1,4) {};
                \node (bottom) at (-1,-1) {};
                \draw[dashed] (cj1)--(c1);
                \draw[dashed] (cjl)--(ck);
                \draw (cj1)--(bottom);
                \draw (cjl)--(top);
                \draw[dashed] (cj1)--(cjl);
                \node at (1,1.5) {$\rho$};
            \end{tikzpicture}
        \end{center}
        \caption{The chain $c_{j_1}\lessdot \cdots \lessdot
            c_{j_\ell}$ in the proof of Theorem~\ref{thm:csspci}.}
        \label{fig:cchain}
    \end{figure}

	Postpone discussing the base case and suppose $\ell \geq 3$. \\
    \textbf{Claim:} $P$ is obtained by closing a notch in some signed poset
    $P'$ such that $\Gbhat(P')$ (\resp $\Gchat(P'^\smvee)$) has a right region
    whose left border is $c_1\lessdot \cdots \lessdot c_k$ such that
    $c_{j_{\ell-1}}$ is the maximal element along the left border that is in the
        border of more than one region.

    There are three cases.
	\begin{enumerate}
	  \item \textbf{Suppose $c_{j_{\ell}-1} \ne -c_{j_{\ell}}$ and
              $c_{j_\ell-1}\ne 0$.}
			
			By Lemma~\ref{lem:notchopening}, there is a signed poset $P'$ such that $P$ is obtained from $P'$ by closing a signed notch $(c_{j_\ell -1},c_{j_\ell}',c_{j_\ell})$ and $(-c_{j_\ell -1},-c_{j_\ell}',-c_{j_\ell})$.

			By construction, $\Gbhat(P')$ (\resp $\Gchat(P')$) has a right
			region whose left border is $c_1 \lessdot \cdots \lessdot c_k$,
			such that $c_{j_\ell -1}$ is the maximal element along the left
			border that is in the border of more than one region. Then $\Gbhat(P')$ (\resp $\Gchat(P')$) also has a right region whose right border is $-c_k \lessdot \cdots \lessdot -c_1$, such that $-c_{j_\ell-1}$ is the minimal element along the right border that is in the border of more than one region.

			The result then follows by induction.

		      \item \textbf{Suppose $c_{\ell-1} = -c_\ell$. This case can only
                      arise in type $C$. Then since $\ell \geq 3$, from
                      Lemma~\ref{lem:onlyoneplusminus}, one knows that $c_1 \ne
                      -c_2$.}

			      Once again, from Lemma~\ref{lem:notchopening}, there is a
                  signed poset $P'$ such that $\Gchat(\Pc)$ is obtained from $\Gchat(P')$ by closing the signed notch comprised of $(-c_{j_2},-c_{j_1},-c_{j_1}')$ and $(c_{j_2},c_{j_1},c_{j_1}')$. By construction, $\Gchat(P')$ has a right region whose left border is $c_1 \lessdot \cdots \lessdot \cdot c_k$, such that $c_{j_2}\lessdot \cdots \lessdot c_{j_\ell}$ is the shortest saturated chain including all vertices in the left border of this region that lie in the right border of another region. 

			      The result then follows by induction.

			\item \textbf{Suppose $c_{j_\ell}=0$. This case only arises in type
                    $B$. Since $\ell \geq 3$, one knows that $c_{j_1} \ne 0$.}
	
			  Then, from Lemma~\ref{lem:notchopening}, one has a signed poset $P' \subset \Bn$ such that $\Gbhat(P)$ is obtained from $\Gbhat(P')$ by closing a signed notch $(-c_{j_2},-c_{j_1},-c_{j_1}')$ and $(c_{j_2},c_{j_1},c_{j_1}')$ By construction, $\Gbhat(P')$ has a right region whose right left border is $c_1\lessdot \cdots \lessdot \cdot c_k$, such that $c_{j_2}\lessdot \cdots \lessdot c_{j_\ell}$ is the shortest saturated chain including all the vertices in the left border of this region that lie in the right border of another region.

			  The result then follows by induction.
		      \end{enumerate}

	For the base case $\ell =2$, address types $B$ and $C$ separately. In type B, there are two cases.
	\begin{enumerate}
		\item \textbf{Neither $c_1$ nor $c_2$ is $0$.}

			From Lemma~\ref{lem:notchopening}, one has a signed poset $P'$ such that $\Gbhat(P)$ is obtained from $\Gbhat(P')$ by closing a signed notch $(c_1,c_2,c_2')$ and $(-c_1,-c_2,-c_2')$. Then $\Gbhat(P')$ has a rightmost region $\rho$ such that the left border of $\rho$ has precisely one vertex in the border of another region and similarly for the right border of $\iota(\rho)$. Then $\rho$ and $\iota(\rho)$ form a biconnected component in $\Ghat(P')$.

		\item \textbf{One of $c_1$ and $c_2$ is $0$. Without loss of
                generality, assume $c_1=0$.}

		  One applies Lemma~\ref{lem:notchopening} once again, and one has a signed poset $P'$ such that $\Gbhat(P)$ is obtained from $\Gbhat(P')$ by closing a signed notch $(c_1,c_2,c_2')$ and $(-c_1,-c_2,-c_2')$. Then $\Gbhat(P')$ has a rightmost region $\rho$ such that the left border of $\rho$ has precisely one vertex in the border of another region and similarly for the right border of $\iota(\rho)$. Then $\rho$ and $\iota(\rho)$ form a biconnected component in $\Ghat(P')$.
	\end{enumerate}

	In type C, there are two cases.
	\begin{enumerate}
		\item \textbf{Suppose $c_1 \ne -c_2$.}

			In this case, Lemma~\ref{lem:notchopening} applies and there is a signed poset $P'$ such that $\Gchat(\Pc)$ is obtained from $\Gchat(P')$ by closing a signed notch $(c_1,c_2,c_2')$ and $(-c_1,-c_2,-c_2')$. Then $\Gchat(P')$ has a rightmost region $\rho$ such that the left border of $\rho$ has precisely one vertex in the border of another region and similarly for the right border of $\iota(\rho)$. Then $\rho$ and $\iota(\rho)$ form a biconnected component in $\Gchat(P')$.

		\item \textbf{Suppose $c_1=-c_2$.}

			In this case, one has that $(c_1, c_2)$ is the only edge of either
            $\rho$ or $\iota(\rho)$ shared by another region. Since $(c_1,
            c_2)$ is of the form $(i, -i)$, it must be shared by $\rho$ and $\iota(\rho)$ (and, since $\Ghat(P)$ is strongly planar, by no other region). Consequently, $\rho$ and $\iota(\rho)$ must form an entire biconnected component of $\Gchat(\Pc)$.
	\end{enumerate}
	In each of these cases, the rightmost region $\rho$ and its image under the involution form a biconnected component of $\Gbhat(P)$ (\resp $\Gchat(\Pc)$). Let $P_1$ be this biconnected component and $P_2$ correspond to the rest of $\Gbhat(P)$ (\resp $\Gchat(\Pc)$). Then, by the same argument as in the proof of Proposition~\ref{prop:biconnectedred}, one has
	\[
	\Rprt = \Rrt{P_1} \otimes_k \Rrt{P_2} \quad \text{and} \quad \Rrt{\Pc} = \Rrt{P_1} \otimes_k \Rrt{P_2}.
	\]
 One has that $\Ghat(P_2)$ is centrally symmetric and strongly planar since $\Ghat(P)$ was and $\Ghat(P_2)$ has one fewer region not fixed by the involution than $\Ghat(P)$, so the result follows by induction for $\Rrt{P_2}$. One then has
	\[
        \Irt{P} = (U(C_\rho)) \oplus \Irt{P_2} = (U(C_\rho), U(C_{\sigma_1}),\ldots,U(C_{\sigma_{n-1}}))
	\]
	and
	\[
        \Irt{\Pc} = (U(C_\rho)) \oplus \Irt{P_2} = (U(C_\rho),U(C_{\sigma_1}),\ldots,U(C_{\sigma_{n-1}})),
	\]
	where $\sigma_1,\ldots,\sigma_{n-1}$ are the orbits of regions of
	$\Ghat(P_2)$ not fixed by the involution.
Then
	\[
	\Hilb(\Rprt,\bm{x}) = \Hilb(\Rrt{P_1},\bm{x})\Hilb(\Rrt{P_2},\bm{x})
	\quad \text{and}\quad \Hilb(\Rrt{\Pc},\bm{x}) = \Hilb(\Rrt{P_1},\bm{x})\Hilb(\Rrt{P_2},\bm{x}).
	\]
	Since, in each case, $\Irt{P_1}$ is a principal ideal and $\rho$ has a single maximum and a single minimum, one knows that 
	\[
	\Hilb(\Rrt{P_1},\bm{x}) = \dfrac{(1-\bm{x}^\rho)}{\prod_{e \in C_\rho}(1-x_a^\delta x_b^{-\epsilon})},
	\]
	where the product in the denominator runs over edges $e = \delta a \to \epsilon b$ in the cycle defining $\rho$, and taking $x_0=1$ in type $B$. By induction,
	\[
		\Hilb(\Rrt{P_2},\bm{x}) = \dfrac{\prod_{\sigma} (1 -\bm{x}^\rho)}{\prod_{e}(1-x_a^\delta x_b^{-\epsilon})},
	\]
	where $\sigma$ runs over orbits of regions enclosed by $\Ghat(P_2)$ not
    fixed by the involution and $e$ runs over the orbits of edges $\delta a \to
    \epsilon b$ $(a<b)$ in $\Ghat(P_2)$.

	Then
	\[
	\Hilb(\Rprt,\bm{x}) =\dfrac{\prod_\rho (1-\bm{x}^\rho)}{\prod_{e} (1-x_a^\delta x_b^{-\epsilon}),}
	\]
	where $\rho$ runs over all orbits of regions enclosed by $\Ghat(P)$,
	as desired.
\end{proof}

\subsection{Computing $\Psi_P$}\label{subsec:computingpsi}

While the proof of Theorem~\ref{thm:csspci} proceeds as above, using
Proposition~\ref{prop:sandsigmaci} to compute $\Psi_P$ requires having a
regular sequence generating the toric ideal. This regular sequence is implicit
in Theorem~\ref{thm:csspci}---the cycle binomials of the cycles cutting out the
regions of $\Ghat(P)$ form a regular sequence---though one needs to work in the
opposite order to the proof to see the regular sequence, building the poset up
from cycles rather than breaking it apart into cycles.

Recall the posets $P_1$ (Figure~\ref{subfig:gbhatp1} and~\ref{subfig:gchatp1}),
$P_2$ (Figures~\ref{subfig:gbhatp2} and~\ref{subfig:gchatp2}) and $P_3$
(Figure~\ref{fig:p3}) from the previous section. To find a regular sequence
generating $\Irt{P}$, one starts by noting that
\[
    \Rrt{P_3} \cong k[U_1,U_{1\m2},U_{23},U_3]/(U_{1\m2}U_{23}-U_1U_3).
\]
One then closes the notch to obtain $\Gbhat(P_2)$, which has the following
impact on the semigroup ring courtesy of
Proposition~\ref{prop:closenotchmodout}:
\begin{align*}
    \Rrt{P_2} &\cong k[U_1,U_{1\m2}U_{23},U_3]/(U_{1\m2}U_{23}-U_1U_3,U_1-U_3)
    \\
    &\cong k[U_1,U_{1\m2},U_{12}]/(U_{1\m2}U_{12}-U_1^2)
\end{align*}
One sees that $\Gbhat(P')$ has two biconnected components, one coming from $P_1$ and the
other from $P_2$.
\[
    \Rrt{P_1}\cong 
    k[U_{\m34},U_{4\m5},U_{\m23},U_{\m25}]/(U_{\m34}U_{\m23}-U_{4\m5}U_{\m25})
\]
Then
\begin{align*}
    \Rrt{P'} &\cong \Rrt{P_1} \otimes \Rrt{P_2} \\
    &\cong k[U_1,U_{1\m2},U_{12},U_{\m23},U_{\m34},U_{4\m5},U_{\m25}]/(U_{1\m2}U_{12}-U_1^2,U_{\m34}U_{\m23}-U_{4\m5}U_{\m25}).
\end{align*}
Closing the last notch to obtain $\Gbhat(P)$ means
\begin{align*}
    \Rrt{P} &\cong
    k[U_1,U_{1\m2},U_{12},U_{\m23},U_{\m34},U_{4\m5},U_{\m25}]/(U_{1\m2}U_{12}-U_1^2,U_{\m34}U_{\m23}-U_{4\m5}U_{\m25},U_{\m25}-U_{1\m2})
    \\
    &\cong
    k[U_1,U_{1\m2},U_{12},U_{\m23},U_{\m34},U_{\m14}]/(U_{12}U_{1\m2}-U_1^2,U_{1\m2}U_{\m14}-U_{\m34}U_{\m23}).
\end{align*}
Each step in this process results in a regular sequence, so one finishes with a
regular sequence generating $\Irt{P}$. One has
\[
    \Hilb(\Rprt,\bm{x}) =
    \dfrac{(1-x_1^2)(1-x_2^{-1}x_4)}{(1-x_1)(1-x_1x_2)(1-x_1x_2^{-1})(1-x_1^{-1}x_4)(1-x_3^{-1}x_4)(1-x_2^{-1}x_3)}.
\]
Applying Proposition~\ref{prop:sandsigmaci} gives
\begin{equation}\label{eq:psifromhilbB}
    \Psi_P(\bm{x}) =
    \dfrac{2x_1(x_4-x_2)}{x_1(x_1+x_2)(x_1-x_2)(x_4-x_1)(x_4-x_3)(x_3-x_2)}.
\end{equation}
On the other hand, Figure~\ref{fig:idealsforle} shows the poset of order ideals
of $P$, from which the linear extensions can be read off.
\begin{figure}
    \begin{center}
        \begin{tikzpicture}
            \node (0) at (0,0) {$(0,0,0,0)$};
            \node (4) at (0,1) {$(0,0,0,1)$};
            \node (34) at (-1,2) {$(0,0,1,1)$};
            \node (14) at (1,2) {$(1,0,0,1)$};
            \node (134) at (0,3) {$(1,0,1,1)$};
            \node (1m24) at (2,3) {$(1,-1,0,1)$};
            \node (1234) at (-1,4) {$(1,1,1,1)$};
            \node (1m234) at (1,4) {$(1,-1,1,1)$};
            \node (1m2m34) at (3,4) {$(1,-1,-1,1)$};
            \draw (0)--(4)--(34)--(134)--(1234);
            \draw (4)--(14)--(134)--(1m234);
            \draw (14)--(1m24)--(1m2m34);
            \draw (1m24)--(1m234);
        \end{tikzpicture}
    \end{center}
    \caption{$J(P)$ for $P$ in Figure~\ref{fig:hilbex}}
    \label{fig:idealsforle}
\end{figure}
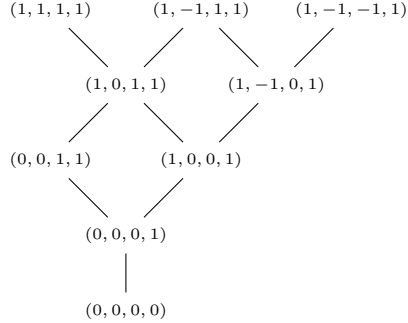
The linear extensions are:
\[
    \begin{split}
    \begin{pmatrix} 1 & 2 & 3 & 4 \\ 4 & 3 & 1 & 2 \end{pmatrix},
    \begin{pmatrix} 1 & 2 & 3 & 4 \\ 4 & 3 & 1 & -2 \end{pmatrix},
    \begin{pmatrix} 1 & 2 & 3 & 4 \\ 4 & 1 & 3 & 2 \end{pmatrix}, \\
    \begin{pmatrix} 1 & 2 & 3 & 4 \\ 4 & 1 & 3 & -2 \end{pmatrix},
    \begin{pmatrix} 1 & 2 & 3 & 4 \\ 4 & 1 & -2 & 3 \end{pmatrix},
    \begin{pmatrix} 1 & 2 & 3 & 4 \\ 4 & 1 & -2 & -3 \end{pmatrix}.
\end{split}
\]
Then
\begin{multline*}
    \Psi_P(\bm{x}) = \sum_{w \in \cL(P)}
    w\left(\dfrac{1}{(x_1-x_2)(x_2-x_3)(x_3-x_4)x_4}\right) =  \\
    \dfrac{1}{(x_4-x_3)(x_3-x_1)(x_1-x_2)x_2} -
    \dfrac{1}{(x_4-x_3)(x_3-x_1)(x_1+x_2)x_2} \\ +
    \dfrac{1}{(x_4-x_1)(x_1-x_3)(x_3-x_2)x_2)} -
    \dfrac{1}{(x_4-x_1)(x_1-x_3)(x_2+x_3)x_2} \\ -
    \dfrac{1}{(x_4-x_1)(x_1+x_2)(x_2+x_3)x_3} -
    \dfrac{1}{(x_4-x_1)(x_1+x_2)(x_3-x_2)x_3} \\
    =
    \dfrac{2(x_4-x_2)}{(x_4-x_3)(x_3-x_2)(x_4-x_1)(x_1-x_2)(x_1+x_2)}\cdot \dfrac{x_1}{x_1},
\end{multline*}
agreeing with \eqref{eq:psifromhilbB}. 

One builds $\Gchat(\Pc)$ (see Figure~\ref{fig:hilbex}) similarly, except one
actually arrives at a principal ideal. 
\[
    \Rrt{P'^\smvee} \cong \Rrt{P_1^\smvee} \otimes \Rrt{P_2^\smvee}.
\]
However, note that $\Gchat(P_2^\smvee)$ (see Figure~\ref{subfig:gchatp2})
consists of a single cycle fixed orientation-wise by the involution, so the
only cycle binomial is $0$. Consequently,
\[
    \Rrt{P'^\smvee} \cong k[U_{1\m2},U_{12},U_{\m34},U_{\m23},U_{\m25},U_{4\m5}] /
    (U_{\m34}U_{\m23}-U_{4\m5}{\m25}),
\]
and closing the notch gives
\begin{align*}
    \Rrt{\Pc} &\cong  k[U_{1\m2},U_{12},U_{\m34},U_{\m23},U_{\m25},U_{4\m5}] /
    (U_{\m34}U_{\m23}-U_{4\m5}{\m25}, U_{\m25} - U_{\m15})\\
    &\cong
    k[U_{1\m2},U_{12},U_{\m34},U_{\m23},U_{\m14}]/(U_{\m34}U_{\m23}-U_{\m14}U_{\m25}).
\end{align*}
Then the Hilbert series is
\[
    \Hilb(\Rrt{\Pc},\bm{x}) =
    \dfrac{1-x_2^{-1}x_4}{(1-x_1^{-1}x_4)(1-x_3^{-1}x_4)(1-x_2^{-1}x_3)(1-x_1x_2^{-1})(1-x_1x_2)}.
\]

Applying Proposition~\ref{prop:sandsigmaci} gives
\[
    \Psi^\smvee_{\Pc}(\bm{x}) =
    \dfrac{x_4-x_2}{(x_4-x_1)(x_4-x_3)(x_3-x_2)(x_1-x_2)(x_1+x_2)}.
\]
On the other hand,
\begin{multline*}
    \Psi^\smvee_{\Pc}(\bm{x}) = \sum_{w \in \cL(P)}
    w\left(\dfrac{1}{(x_1-x_2)(x_2-x_3)(x_3-x_4)2x_4}\right) \\ =
    \dfrac{x_4-x_2}{(x_4-x_3)(x_3-x_2)(x_4-x_1)(x_1-x_2)(x_1+x_2)},
\end{multline*}
agreeing with the computation via the Hilbert series.

With the regular sequence in hand, one can apply
Proposition~\ref{prop:sandsigmaci} to obtain the following corollary.

\begin{cor}\label{cor:computingpsi}
Suppose $P \subset \Bn$ (\resp $\Pc \subset \Cn$) is a strongly planar signed
poset. Then 
\[
    \Psi_P(\bm{x}) = \dfrac{\prod_{\rho} \sgn(\min(\rho))x_{\min(\rho)} -
        \sgn(\max(\rho))x_{\max(\rho)}}{\prod_{(i,j) \in \Sigma_P} \sgn(i)x_i -
        \sgn(j)x_j},
\]
where $\rho$ runs over the regions of $\Ghat(P)$ and $(i,j)$ runs over the
edges of the Hasse diagram.
\end{cor}

\chapter{The Weight Cone Semigroup}\label{ch:weightcone}

Associated to the weight cone, $\Kpwt$, is a semigroup, called the \emph{weight
cone semigroup}, defined by the intersection of $\Kpwt$ with the coweight
lattice. In this chapter, only $\Bn$-signed posets will be considered.
Section~\ref{sec:wtsemigroupgens} discusses the generators of the semigroup,
Section~\ref{sec:wttoricideal} gives a generating set for the toric ideal.
Finally, Section~\ref{sec:computingsums} considers the question of computing
$\Phi_P(\bm{x})$ and
\[
\sum_{w \in \cL(P)} q^{\maj(w)}.
\]
Theorem~\ref{thm:excludedposets} characterizes signed posets which are
so-called initial complete intersections, for which the aforementioned sums can
be readily computed, and Theorem~\ref{thm:construction} explains how the
initial complete intersections are constructed.

Before beginning, recall some definitions and facts about the weight cone:
\begin{itemize}
		\item an ideal $I$ is said to be \emph{extensible} if there is a (nonempty) $J
				\subset \pm [n]$ such that $I \cup J$ and $I \cup -J$ are
				ideals (see Definition~\ref{def:extensible})
		\item the extreme rays of $\Kpwt$ correspond to the connected,
				nonextensible ideals of $P$ (see
				Proposition~\ref{prop:kpwtextremerays})
		\item the elements of the weight cone semigroup are the $P$-partitions
				(see Section~\ref{sec:ideals})
		\item when considering the ideals of a signed poset, it suffices to
				look at $\Gchat(\Pc)$ instead of $\Gbhat(P)$. To that end,
				$\Ghat(P)$ in this chapter denotes $\Gchat(\Pc)$, even though
				signed posets are $P \subset \Bn$. (see
                Definition~\ref{def:ppartition})
		\item the \emph{signed support} of a $P$-partition $f$ is an ideal 
				\[
				\ssupp(f) =
				\{\sgn(f_i)i \colon i \in[n], |f_i|\geq 1\}
				\]
				(see Definition~\ref{def:signedsupport}).
            \item the set of connected ideals is denoted $\Jconn(P)$.
\end{itemize}

\section{Generators of the Semigroup}\label{sec:wtsemigroupgens}

Lemma~\ref{lem:lincombideals} has a consequence that is key to the discussion
of the weight cone semigroup, which is stated now.

\begin{prop}\label{prop:nestedideals}
Suppose $P \subset \Bn$ is a signed poset and $f$ is a $P$-partition. Then
there are ideals $J_1 \supset J_2 \supset \cdots \supset J_k$ such that $f =
\chi_{J_1}+\chi_{J_2}+\cdots \chi_{J_k}$.
\end{prop}
The $J_i$ are precisely those from Lemma~\ref{lem:lincombideals}.

\begin{prop}
	The semigroup $\Kpwt \cap \Lcwt_B$ is generated by the connected ideals of $P$, though not necessarily minimally.
\end{prop}

\begin{proof}
	Suppose $f$ is a $P$-partition. Then there are ideals $I_1 \subseteq \cdots \subseteq I_k$ such that
	\[
		f = \sum_j \chi_{I_j}.
    \]
	Let $I_j^{(i)}$ be the connected components of $I_j$. Then $f$ is in the semigroup generated by the $I_j^{(i)}$, which lies in the semigroup generated by the connected ideals.
\end{proof}

It is simple to see that this is not necessarily a minimal generating set.
Suppose $P=\{e_1+e_2, +e_2\}$. Then $\Jconn(P)=\{\{1\},\{2\},\{-1,2\}\}$, but
$\{1\}$ and $\{-1,2\}$ suffice to generate the semigroup.

\section{The Toric Ideal}\label{sec:wttoricideal}

To discuss the toric ideal of the weight cone semigroup, one needs a few
additional definitions involving ideals of a signed poset. 
\begin{defn}\label{def:support}
Suppose $J$ is an
ideal of a signed poset. The \emph{support} of $J$ is the set $\supp J = \{i \colon \pm
i \in J\}\subset [n]$. Two ideals, $J$ and $K$, \emph{intersect nontrivially} if $\supp J
\cap \supp K \ne \emptyset$ and neither $J \subset K$ nor $K \subset J$. Two
ideals which do not intersect nontrivially will be said to \emph{intersect
trivially}. 
\end{defn}
If $J$ and $K$ are ideals, let $J+K$ denote the $P$-partition $\chi_J +
\chi_K$. In the sequel, it will be pairs of nontrivially intersecting connected
ideals that are paramount. Denote the set of such pairs by $\Pi(P)$.

With this setup one can now give a presentation for the toric ideal of the
weight cone semigroup. The result closely resembles that of \citet*[Theorem 1.2]{FerayReiner2012}, save that one must account for the fact that
the union of two ideals of a $\Bn$-poset is not necessarily an ideal, since the union of two isotropic order ideals of $\Ghat(P)$ is not
necessarily isotropic.

\begin{defn}
Suppose $P \subset \Bn$ is a signed poset. Let
		$\Swt{P}=k[U_{J_1},\ldots,U_{J_k}]$, where the $J_i$ are the connected
		ideals of $P$. Define a map 
		\begin{align*}
				\phi \colon \Swt{P} &\to k[x_1^{\pm 1},\ldots,x_n^{\pm 1}] \\
				U_J &\mapsto \bm{x}^{\chi_J}.
		\end{align*}
		Suppose $J_1, J_2$ are two ideals which intersect nontrivially. Define
  \[
  \syz (U_{J_1},U_{J_2})=U_{J_1}U_{J_2} - \prod U_{J^{(i)}}\prod U_{K^{(j)}},
  \]
  where the $J^{(i)}$ are the connected components of $J_1\cap J_2$ and the $K^{(j)}$ are the connected components of $\ssupp(J_1+J_2)$. 
\end{defn}

\begin{thm}\label{thm:typebpresentation}
Suppose $P \subset \Bn$ is a signed poset.
		One has an exact sequence
  \[
  0 \to \Iwt{P} \to \Swt{P} \to \Rpwt \to 0,
  \]
 with 
  \[
  \Iwt{P}=(\syz(U_{J_1},U_{J_2})),
  \]
  where $\{J_1,J_2\}$ runs over $\Pi(P)$, the set of pairs of nontrivially
  intersecting connected ideals.
\end{thm}

Theorem~\ref{thm:typebpresentation} is the generalization of the type A
result~\cite[Theorem 1.2]{FerayReiner2012}.

Consider the signed poset in Figure~\ref{fig:wttoricidealex}. The connected ideals are
\[
\{1\}, \{-3\}, \{-4\}, \{1,2\}, \{1,4\}, \{-2,-3\}, \{1,2,4\}, \{1,2,3\}, \{1,2,3,4\},
\]
so one has
$\Swt{P}=k[U_1,U_{\m3},U_{\m4},U_{12},U_{14},U_{\m2\m3},U_{124},U_{123},U_{1234}]$.
Then 
\begin{multline*}
		\Iwt{P} =
		(U_{\m3}U_{123}-U_{12},U_{\m3}U_{1234}-U_{124},U_{\m4}U_{14}-U_1,U_{\m4}U_{124}-U_{12},U_{\m4}U_{1234}-U_{123},\\
		U_{12}U_{14}-U_1U_{124},U_{12}U_{\m2\m3}-U_1U_{\m3},
		U_{\m2\m3}U_{124}-U_{14}U_{\m3},U_{124}U_{123}-U_{12}U_{1234})
\end{multline*}

\begin{figure}
\begin{center}
    \begin{tikzpicture}[vertex/.style={draw,circle,fill=black,inner sep=1pt}]
      \node[vertex] at (0,0) (1) [label=right:$1$] {};
      \node[vertex] at (0,1) (2) [label=right:$2$] {};
      \node[vertex] at (-1,1)(4) [label=right:$4$]{};
      \node[vertex] at (0,2) (3) [label=right:$3$]{};
      \node[vertex] at (1,1) (-1) [label=right:$-1$]{};
      \node[vertex] at (1,0) (-4) [label=right:$-4$]{};
      \node[vertex] at (2,0) (-2) [label=right:$-2$]{};
      \node[vertex] at (2,-1) (-3)[label=right:$-3$] {};
      \draw (1)--(4);
      \draw (1)--(2)--(3);
      \draw (1)--(-1);
      \draw (-4) --(-1);
      \draw (-1)--(-2)--(-3);
    \end{tikzpicture}
  \end{center}
  \caption{$P = \{+e_1,+e_1-e_2,+e_1-e_3,+e_1-e_4,+e_2-e_3\}$}
  \label{fig:wttoricidealex}
\end{figure}
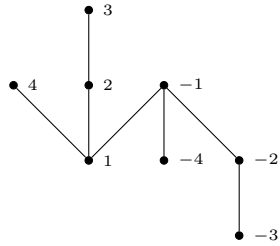

The proof of Theorem~\ref{thm:typebpresentation} requires a number of lemmas and a few facts from the theory of Gr\"obner bases. Section~\ref{subsec:groebner} discusses Gr\"obner bases and Section~\ref{subsec:proofoftypebpresentation} contains the lemmas and proof of the theorem.

\subsection{Term Orders and Gr\"obner Bases}\label{subsec:groebner}

In this section, we pause for a moment to review a few facts about terms orders
and Gr\"obner bases before the proof of Theorem~\ref{thm:typebpresentation} in
the next section. This material is standard and can be found in many places,
including~\cite{Eisenbud1995,EneHerzog2012,Sturmfels1996a}.

\begin{defn}\label{def:termorder}
Suppose $S=k[x_1,\ldots,x_n]$ is a polynomial ring. A \emph{term order}, $<$,
is a total order on the monomials of $S$ such that
\begin{enumerate}[label=(\alph{*})]
		\item there are no infinite descending chains
		\item $1=\bm{x}^0$ is the minimal element
		\item if $\bm{x}^\alpha \leq \bm{x}^\beta$ and
				$\bm{x}^\gamma$ is any monomial, then
				$\bm{x}^\alpha\bm{x}^\gamma \leq \bm{x}^\beta
				\bm{x}^\gamma$.
\end{enumerate}

Suppose $f=a_1\bm{x}^{\alpha_1}+\cdots+a_k\bm{x}^{\alpha_k}$ is a
polynomial in $R$. The \emph{multidegree} of $f$ is $\alpha_{\deg f}$ which is
maximal among $\alpha_1,\ldots,\alpha_k$ with respect to the term order $<$.
The \emph{initial term} of $f$ is then $a_{\deg f}\bm{x}^{\alpha_{\deg f}}$,
denoted $\init_< f$.
\end{defn}

\begin{defn}\label{def:groebnerbasis}
Given an ideal $I \subset R$, the \emph{initial ideal} with respect to $<$ is
\[
   \init_< I = (\init_<f \colon f \in I).
\]
An ideal and its initial ideal are connected by the notion of a Gr\"obner
basis. A set $\cG=\{g_1,\ldots,g_m\} \subset I$ is a \emph{Gr\"obner basis} of
$I$ if 
\[
   \init_< I=(\init_< g_1,\ldots,\init_< g_m).
\]
\end{defn}
\textit{A priori}, it is not clear that an arbitrary ideal in a polynomial ring should have a
finite Gr\"obner basis. However, this is a consequence of Dickson's Lemma (see
Theorem 1.9, Corollary 1.10 and page 20 of~\citet{EneHerzog2012}). The
following facts will underpin the proof of
Theorem~\ref{thm:typebpresentation}.

\begin{thm}[{Macaulay~\cite{Macaulay1927}}]\label{thm:macaulay}
Suppose $S$ is a polynomial ring over the field $K$, there is an ideal $I
\subset S$ and a term order $<$. The monomials of $S$ which do not belong to $\init_< I$ form a $k$-basis
of $S/I$.
\end{thm}

\begin{cor}\label{cor:initsamehilbert}
   Suppose $I$ is a homogeneous ideal in a polynomial ring $S$ and $<$ is a term order.
Then
\[
   \Hilb_{S/\init_< I}(\bm{x}) = \Hilb_{S/I}(\bm{x}).
\]
\end{cor}

\begin{prop}
   If $I$ is an ideal in a polynomial ring $S$ and $\cG$ is a Gr\"obner basis
of $I$, then $\cG$ generates $I$.
\end{prop}

As an aside, Gordan observed in~\cite{Gordan1900} that this last proposition, when combined with the
existence of finite Gr\"obner bases gives the Hilbert Basis Theorem, though it is
common to rely on the Hilbert Basis Theorem to obtain the existence of finite
Gr\"obner bases.

\subsection{The Proof of Theorem~\ref{thm:typebpresentation}}\label{subsec:proofoftypebpresentation}

The basic idea of the proof of Theorem~\ref{thm:typebpresentation} is to show
that the generating set given in the theorem is a Gr\"obner basis for $\Iwt{P}$
in a particular term order. This is accomplished by showing that the rings
$\Swt{P}/\Iwt{P}$ and $\Swt{P}/I^\init_{P}$ share the same Hilbert series.
(See Lemma~\ref{lem:samehilbertseries} for the definition of $I^\init_{P}$.)

\begin{lem}\label{lem:ppartitionsoutoftrivintersectideals}
 Suppose $P \subset \Bn$ is a signed poset. A $P$-partition $f$ can be written uniquely as the sum of trivially intersecting connected ideals.
\end{lem}

Lemma~\ref{lem:ppartitionsoutoftrivintersectideals}
generalizes~\cite[Proposition 2.5(ii)]{FerayReiner2012}, which has the same
statement when $P$ is a poset.

\begin{proof}
		From Proposition~\ref{prop:nestedideals}, any $P$-partition can be written as $f=\sum \chi_{I_k}$ with ideals $I_1
  \subset I_2 \subset \cdots \subset I_\ell$. Thus, the lemma is a matter of transforming this collection of ideals into a set of trivially intersecting connected ideals. Proceed by induction on $|f|=|f_1|+\cdots+|f_n|$. The base case $|f|=0$ is trivial.

  Suppose $|f|\geq 1$. Let $J = \ssupp(f)$. Recall that the signed support is
  an ideal and $J=I_k$. Let $J^{(1)},J^{(2)},\cdots,J^{(c)}$ be the connected
  components of $J$. Suppose $c>1$. Consider $f|_{J^{(i)}}$. One has
  \[
	f|_{J^{(i)}}=\sum_{i=1}^k I_k|_{J^{(i)}},
  \]
   so $f|_{J^{(i)}}$ is a
  $J^{(i)}$-partition. One then has that $|J^{(i)}|< |J|\leq |f|$, so, by
  induction, $f|_{J^{(i)}}$ can be written uniquely as a sum of trivially
  intersecting connected ideals. Consequently, $f$ can be written uniquely as a sum of
  trivially intersecting connected ideals.

  In the other case suppose $c=1$. Then $J=\ssupp(f)$ is connected. Let $\hat{f}=\sum_{i=i}^{k-1} I_i$. Then $f=\chi_J+\hat{f}$ and one has that $\hat{f}$ is a (possibly zero) $P$-partition with $|\hat{f}|<|f|$. Then, by induction, $\hat{f}$ can be written uniquely as a sum of trivially intersecting connected ideals. Further, each of these ideals is contained in $\chi_J=I_k$, so $f$ can be written uniquely as a sum of trivially intersecting connected order ideals.
\end{proof}

\begin{defn}\label{def:Iinit}
Suppose $P \subset \Bn$ is a signed poset. Let 
\[
I_P^\init=(U_{J_1}U_{J_2})
\]
with $\{J_1,J_2\}$ running over $\Pi(P)$.
\end{defn}

\begin{lem}\label{lem:samehilbertseries}
		Suppose $P \subset \Bn$ is a signed poset. Then $\Rpwt$ and $\Swt{P}/I_P^\init$ have the same $\bZ^n$-graded Hilbert series, namely
  \[
  \sum_{f \in \cA(P)} \bm{x}^{f}.
  \]
\end{lem}

\begin{proof}
		Since the $P$-partitions are the elements of the semigroup, $\Rpwt$ has
        $\sum_{f \in \cA(P)} \bm{x}^{f}$ as its $\bZ^n$-graded
  Hilbert series.
  The monomials killed by $I_P^\init$ are precisely those mapped to
  $P$-partitions expressed as a sum of (at least) two nontrivially intersecting
  ideals. By Lemma~\ref{lem:ppartitionsoutoftrivintersectideals}, every
  $P$-partition can be written uniquely as a sum of nontrivially intersecting
  connected order ideals, i.e.\ each $P$-partition corresponds to a unique
  monomial surviving in $\Swt{P}/I_P^\init$.
\end{proof}

The notation $I_P^\init$ is not a coincidence, as Lemma~\ref{lem:initialterm}
will show. One defines a term order for $\Swt{P}=k[U_J]$ that specializes to
the term order given by \feray and Reiner in~\cite{FerayReiner2012}. First,
recall their term order. Place a total order $<$ on the $U_J$ such that $U_J <
U_K$ whenever $|J| < |K|$. (This amounts to choosing a linear extension of the poset of nonempty connected order ideals ordered by inclusion.) Suppose $U^J = U_{J_1}\cdots U_{J_r}$ and $U^K = U_{K_1}\cdots U_{K_s}$ are two monomials and assume without loss of generality that $r< s$ and the $U_{J_i}$ and $U_{K_\ell}$ are ordered by $<$. Find the first $i$ such that $J_i \ne K_i$. If $J_i < K_i$, then $U^J < U^K$, and if $K_i < J_i$, then $U^J < U^K$. If no such $i$ exists, then $U^J$ divides $U^K$, so $U^J < U^K$.

In type A, the initial term of $\syz(U_J,U_K)$ is always $U_JU_K$ with respect
to this term order. However, in type B, this is not always the case. Consider
the poset in Figure~\ref{fig:wttoricidealex}.
$\Swt{P}=k[U_1,U_{\m3},U_{\m4},U_{12},U_{14},U_{\m2\m3},U_{124},U_{123},U_{1234}]$.
Order the ideals as the corresponding variables appear from left to right.
Then, using the term order from~\cite{FerayReiner2012},  the leading term of
$\syz(U_{123},U_{\m3}) = U_{123}U_{\m3}-U_{12}$ is $U_{12}$ and not
$U_{123}U_{\m3}$.

To resolve this issue, define a new term order. 
\begin{defn}
		Let $w = (|J|)_{J \in
J_{\text{conn}}(P)}$ be a weight vector and define a term order as follows.
Consider monomials $U^\alpha$ and $U^\beta$. If $\langle w, \alpha \rangle >
\langle w,\beta \rangle$, then $U^\alpha > U^\beta$. If $\langle w,\alpha\rangle = \langle w, \beta \rangle$, break the tie using the term order from~\cite{FerayReiner2012} described above. Denote this new term order $\preceq$.
\end{defn}

In the example from Figure~\ref{fig:wttoricidealex}, the weight vector is $(1,1,1,2,2,2,3,3,4)$, and $U_{123}U_{\m3}$ has weight four while $U_{12}$ has weight two. Consequently, using the $\preceq$ order, the initial term of $\syz(U_{123},U_{\m3})$ is $U_{123}U_{\m3}$.

\begin{lem}\label{lem:initialterm}
  Suppose $P \subset \Bn$ is a signed poset. Suppose $J_1$ and $J_2$ are connected ideals that intersect nontrivially. Then $\init_\preceq(\syz(U_{J_1},U_{J_2}))=U_{J_1}U_{J_2}$.
\end{lem}

\begin{proof}
  There are two cases to consider. First, suppose $\supp(J_1 + J_2)= \supp(J_1)
  \cup \supp(J_2)$. In other words, no cancellation occurs in $J_1+J_2$. Then
  \[
  |J_1|+|J_2| = |\ssupp(J_1+J_2)| +|J_1 \cap J_2|.
  \]
  In this case, the weight
  vector $w$ has the same inner product with the exponent vectors of both
  monomials of $\syz(U_{J_1},U_{J_2})$. Using the term order $>$
  from~\cite{FerayReiner2012} to break the tie, one will always have that $U_{J_1}U_{J_2} > \prod U_{J^{i}}\prod U_{K^{(j)}}$, since any connected component of $J_1 \cap J_2$ is a (proper) subset of $J_1$ and $J_2$. Therefore, $\init_\preceq(\syz(U_{J_1},U_{J_2}))=U_{J_1}U_{J_2}$. 

  In the other case, cancellation occurs in $J_1+J_2$. In this case,  
  \[
  |J_1|+|J_2| > |\ssupp(J_1+J_2)| +|J_1 \cap J_2|,
  \]
  so $\init_\preceq(\syz(U_{J_1},U_{J_2}))=U_{J_1}U_{J_2}$.
\end{proof}

The proof of Theorem~\ref{thm:typebpresentation} now proceeds much as the proof does for type A given in~\cite{FerayReiner2012}.

\begin{proof}[Proof of Theorem~\ref{thm:typebpresentation}]
  For simplicity, let $K = \ker(\phi \colon S \to k[x_1^{\pm 1},\ldots,x_n^{\pm
  1}])$ and, as before, let $I_P^\init=(U_{J_1}U_{J_2})$, where $\{J_1,J_2\}$
runs over $\Pi(P)$. Observe that by definition, $I_P \subset K$. One then has $\init_\preceq (I_P)\subset \init_\preceq(K)$. On the other hand, recall from Lemma~\ref{lem:initialterm} that $U_{J_1}U_{J_2} = \init_\preceq(\syz(U_{J_1},U_{J_2}))$, so $I_P^\init \subset \init_\preceq(I_P)$. Therefore, one has surjections
  \[
	S/I_P^\init \to S/\init_\preceq(I_P) \to S/\init_\preceq(K).
  \]
  By the definition of $\phi$, the ideals $K$ and $I_P$ are homogeneous in the
  $\bZ^n$-grading, so $S/K$ and $S/I_P$ each share Hilbert series with
  $S/\init_\preceq(K)$ and $S/\init_\preceq(I_P)$, respectively, by
  Corollary~\ref{cor:initsamehilbert}. Furthermore, one knows from
  Lemma~\ref{lem:samehilbertseries} that $S/K$ and $S/I_P^\init$ share the same
  $\bZ^n$-Hilbert series. Hence, one has 
  \[
	\Hilb(S/I_P^\init,\bm{x}) = \Hilb(S/I,\bm{x})=\Hilb(S/K,\bm{x}) = \Hilb(S/\init_\preceq(I),\bm{x}) = \Hilb(S/\init_\preceq(K),\bm{x}).
  \]
  Consequently, the surjections must be isomorphisms and $S/I_P^\init \cong
  S/\init_\preceq(I_P) \cong S/\init_\preceq(K)$, meaning $I_P^\init =
  \init_\preceq(I_P) = \init_\preceq(K)$. One then has that the
  $\syz(U_{J_1},U_{J_2})$ form a Gr\"obner basis for both $I_P$ and $K$, meaning $I_P = K$,
  since an ideal is always generated by a Gr\"obner basis.
\end{proof}

\section{Complete intersections and the sum over linear extensions}\label{sec:computingsums}

In~\cite{Reiner1990}, Reiner proved the following result about the
$P$-partition generating function for signed posets. Recall the definition of
\emph{major index} of a signed permutation $w$ from Definition~\ref{def:signedmaj}:
\[
    \maj(w) = \sum_{ i \in \Des(w)} i.
\]

\begin{prop}\label{prop:qmajppartitions}
    Suppose $P$ is a signed poset. Then
    \[
        \sum_{f \in \cA(P)} q^{|f|} = \dfrac{\sum_{w \in \cL(P)}
            q^{\maj(w)}}{(1-q)(1-q^2)\cdots(1-q^n)}.
    \]
\end{prop}

This parallels the result in type A due to Stanley (see~\cite{Stanley1972}
or~\cite{stanleyvol1}). \feray and Reiner observed (again for type A) that the left hand side is
the Hilbert series of $\Rpwt$ with the grading $\deg \bm{x}^f = |f|$. However, this
is \emph{not} the case when $P$ is a signed poset---$\deg x^f = |f|$ is not a
grading of $\Rpwt$, never mind a specialization of the $\bZ^n$ grading!
Additionally, while Theorem~\ref{thm:typebpresentation} gives a presentation of
$\Rpwt$ as $\Swt{P}/\Iwt{P}$, the toric ideal $\Iwt{P}$ is not homogeneous in
this $\bN$-grading. 

All is not lost, however, as, using the same logic as in
Lemma~\ref{lem:samehilbertseries}, one observes that when $\Swt{P}$ is graded by $\deg
U_J = |J|$, one has
\[
    \Hilb(\Swt{P}/\initI,q) = \sum_{f \in \cA(P)} q^{|f|}.
\]
Additionally, if one grades $\Swt{P}$ by $\deg U_J = 1$, one has
\[
    \Hilb(\Swt{P}/\initI,t) = \sum_{f \in \cA(P)} t^{\nu(f)},
\]
where $\nu(f)$ is the number of ideals used in the unique expression of the
$P$-partition $f$ as a sum of nontrivially intersecting connected order
ideals.

This section is concerned with computing these sums in
the case where $\Swt{P}/\initI{P}$ is a complete intersection.
Section~\ref{subsec:wtreducing} will reduce the case of signed posets which are
not full-dimensional to the work of \feray and Reiner
in~\cite{FerayReiner2012}. Section~\ref{subsec:characterisinginitci} will
define the notion of \emph{initial complete intersection} and characterize in
two ways the
signed posets which are initial complete 
intersections:
\begin{itemize}
    \item they are the signed posets avoiding a certain list of induced
        subposets (Theorem~\ref{thm:excludedposets});
    \item they are the signed posets constructed using certain moves
        (Theorem~\ref{thm:construction}).
\end{itemize}

By shifting one's focus to $S/\initI$ from $S/\Iwt{P}$, one sees immediate
benefit in that one now has a minimal generating set for the ideal of
interest.

\begin{prop}\label{prop:gbinitmingens}
    Suppose $P$ is a signed poset. Then $\initI$ is minimally generated by
    $U_JU_K$ where $\{J,K\}$ runs over all nontrivially intersecting pairs of
    connected order ideals of $P$.
\end{prop}

\begin{proof}
    Since the $\syz(U_J,U_K)$ form a Gr\"obner basis of $I$, their initial
    terms, the $U_JU_K$, certainly generate $\initI$. A generating set of a
    monomial ideal is a minimal generating set precisely when none of the
    monomials divides another. (See~\cite[Proposition 1.1.6]{HerzogHibi2011},
    for instance.) Since the $U_JU_K$ are all distinct and all quadratic, none
    can divide any of the others.
\end{proof}

\subsection{Reducing to connected $\Ghat(P)$}\label{subsec:wtreducing}

As with the root cone, one is able to reduce discussion of the weight cone
semigroup to a more convenient case: that where $\Ghat(P)$ is connected. First,
one uses a similar logic to the biconnected component reduction of
Proposition~\ref{prop:biconnectedred}. 

\begin{defn}\label{def:signedcomponent}
    Suppose $P \subset \Bn$ is a signed poset.
    Let a \emph{signed component} of
$\Ghat(P)$ be
\begin{enumalph}
\item a connected component $P_1$ of $\Ghat(P)$ such that if $i \in P_1$, then
		$-i\in P_1$, or
\item a pair of connected components $P_1,P_2$ such that $P_1=-P_2$.
\end{enumalph}
\end{defn}
Observe that each signed component of $\Ghat(P)$ is the Fischer poset of some
smaller signed poset, call them $P_1,\ldots,P_k$.  One then has the following.

\begin{prop}\label{prop:signedcomponentred}
Let $P$ be a signed poset and let $P_1,\ldots,P_k$ be the signed posets
corresponding to its signed components. Then
\[
\Rpwt \cong \Rwt{P_1} \otimes \cdots \otimes \Rwt{P_k}.
\]
\end{prop}

\begin{proof}
Begin by observing that each connected order ideal of $P$ lies entirely in a
 signed component. Therefore,
\[
\Swt{P} \cong \Swt{P_1} \otimes \cdots \otimes \Swt{P_k},
\]
and
\[
\Iwt{P} \cong \Iwt{P_1}\oplus\cdots\oplus\Iwt{P_k}.
\]
Consequently, one has
\[
\Rpwt = \Rwt{P_1}\otimes \cdots\otimes \Rwt{P_k},
\]
as desired.
\end{proof}

As a result of Proposition~\ref{prop:signedcomponentred} it suffices to
consider signed posets having only one signed component. This assumption will
hold for the remainder of the chapter.

As in Section~\ref{subsec:reductions}, one may assume that each signed poset
under consideration is
not contained in $\text{span} \{e_i \colon i \in A\}$ for any $A \subsetneq [n]$.
We explain here why it suffices to consider signed posets $P$ for which
$\Kprt$ is full-dimensional and $\Kpwt$ is pointed.

Since it was assumed that $\Ghat(P)$ has only one signed component, if $\Kpwt$
is not pointed, one must have that $\Ghat(P)$ consists of two isotropic
connected components, $P_1=-P_2$. Without loss of generality, one may assume
$P_1=[n]$, meaning both $[n]$ and $-[n]$ are both ideals.

\begin{prop}\label{prop:localisation}
Suppose $P$ is a signed poset such that $\Kpwt$ is not pointed. Then there is a
 poset $Q$ on $[n]$ such that
\[
\Rpwt \cong \Rwt{Q}[(x_1x_2\cdots x_n)^{-1}],
\]
the localization of $\Rwt{Q}$ at the multiplicatively closed set $\{(x_1\cdots x_n)^k \colon k \geq 0\}$.
\end{prop}

\begin{proof}
One may assume that $\Ghat(P)$ consists of two connected components $K_1$ and
$K_2$ such that $K_1 = -K_2$ and $[n]$ are the vertices of $K_1$. Let $Q$ be
the poset whose Hasse diagram coincides with $K_1$. One defines a map
\[
\phi \colon \Rpwt \to \Rwt{Q}[(x_1x_2\cdots x_n)^{-1}]
\]
as follows. Given any $f \in \cA(P)$, the $P$-partition $f$ can be written
uniquely as a sum of nontrivially intersecting connected ideals: $f =
J_1+\cdots + J_k$. Without loss of generality, one may assume $J_1,\ldots,J_i
\subset [n]$ and $J_{i+1},\ldots,J_k \subset -[n]$. Then define
\[
\phi(\bm{x}^f) = \dfrac{\bm{x}^{J_1}}{1}\cdots \dfrac{\bm{x}^{J_i}}{1}\cdot
\dfrac{\bm{x}^{([n]\smallsetminus -J_{i+1})}}{x_1\cdots x_n}\cdots
\dfrac{\bm{x}^{([n]\smallsetminus -J_k)}}{x_1\cdots x_n}.
\]
Since $[n]$ and $-[n]$ are both ideals of $P$, for any ideal $J \subset -[n]$,
one has that $[n] \smallsetminus -J$ is also an ideal of $P$, meaning it is an
ideal of $Q$. (It may not be a \emph{connected} ideal of $Q$, but it is an
$Q$-partition, which is all that is needed.)

Certainly $\phi$ is an injection. To complete the proof, one must now show that
$\phi$ is surjective. A prototypical monomial of $\Rwt{Q}[(x_1x_2\cdots
x_n)^{-1}]$ is 
\[
\dfrac{\bm{x}^f}{(x_1\cdots x_n)^k},
\]
with $k$ minimal (i.e.\ having cancelled as many powers of $x_1\cdots x_n$ from
the numerator as possible). Since $f$ is an $Q$-partition, it is also a
$P$-partition. Then $\bm{x}^f(x_1\cdots x_n)^{-k} \in \Rpwt$ and, since $\phi$
is a ring homomorphism, 
\[
\phi(\bm{x}^f(x_1\cdots x-n)^{-k}) = \dfrac{\bm{x}^f}{(x_1\cdots x_n)^k},
\]
meaning $\phi$ is surjective, completing the proof.
\end{proof}

\subsection{Characterizing the initial complete
intersections}\label{subsec:characterisinginitci}

\begin{prop}\label{prop:initciiffuniquenontriv}
    Suppose $P$ is a signed poset. One has that
    \begin{equation}\label{eq:cipresentation}
        0 \to (U_JU_K) \stackrel{\phi}{\to} \Swt{P} \to \Swt{P}/\initI \to 0,
    \end{equation}
    where $\{J,K\}$ runs over the set $\Pi(P)$ of all pairs of nontrivially
    intersecting connected ideals,
    is a complete intersection presentation if and only if no connected order
    ideal of $P$ intersects two or more other connected ideals
    nontrivially.
\end{prop}

Proposition~\ref{prop:initciiffuniquenontriv} is a special case of an easy fact about Stanley-Reisner rings.

\begin{defn}\label{def:stanleyreisner}
Let $V$ be a finite set and suppose $\Delta$ is a \emph{simplicial complex} on
$V$, i.e.\ a collection of subsets of $V$ such that if $F \in \Delta$ and $F'
\subset F$, then $F' \in \Delta$. Let $S$ be a polynomial ring whose variables
are indexed by elements of $V$. The \emph{Stanley-Reisner ideal} of $\Delta$ is the ideal
\[
I_\Delta = \left( \prod_{v \in F} x_v \colon F \notin \Delta\right).
\]
The quotient ring $S/I$ is then called a \emph{Stanley-Reisner ring}.
\end{defn}

\begin{prop}
Suppose $I$ is a Stanley-Reisner ideal. The Stanley-Reisner ring $S/I$ is a
complete intersection if and only if
\[
I = (x_{11}\cdots x_{1n_1},\ldots,x_{m1}\cdots x_{mn_m}),
\]
i.e.\ $I$ is generated by a collection of square-free monomials with pairwise disjoint
support.
\end{prop}

See \citet[Theorem 4.1.1]{Duval1991} for proof of a more general statement.

Proposition~\ref{prop:initciiffuniquenontriv} follows from the observation that
$\Iwt{P}$ is the Stanley-Reisner ideal of the simplicial complex on $\Jconn(P)$
defined by $F \subset \Jconn(P)$ if all ideals in $F$ pairwise intersect
trivially.

This characterization of posets such that $S/\initI$ is a complete intersection
allows one to characterize such posets as those avoiding a list of induced
subposets, similar to that of~\cite[Theorem 10.5]{FerayReiner2012}.
\begin{defn}\label{def:initci}
Suppose $P$ is a signed poset such that
\[
0 \to (U_JU_K) \stackrel{\phi}{\to} \Swt{P} \to \Swt{P}/\initI \to 0
\]
is a complete intersection presentation. Then $P$ is said to be an
\emph{initial complete intersection}.
\end{defn}

\begin{thm}\label{thm:excludedposets}
Suppose $P$ is a signed poset. $P$ is an initial complete intersection
if and only if $P$ does not contain a signed poset isomorphic to any of those shown in Figure~\ref{fig:excludedposets} as an induced subposet.
\end{thm}

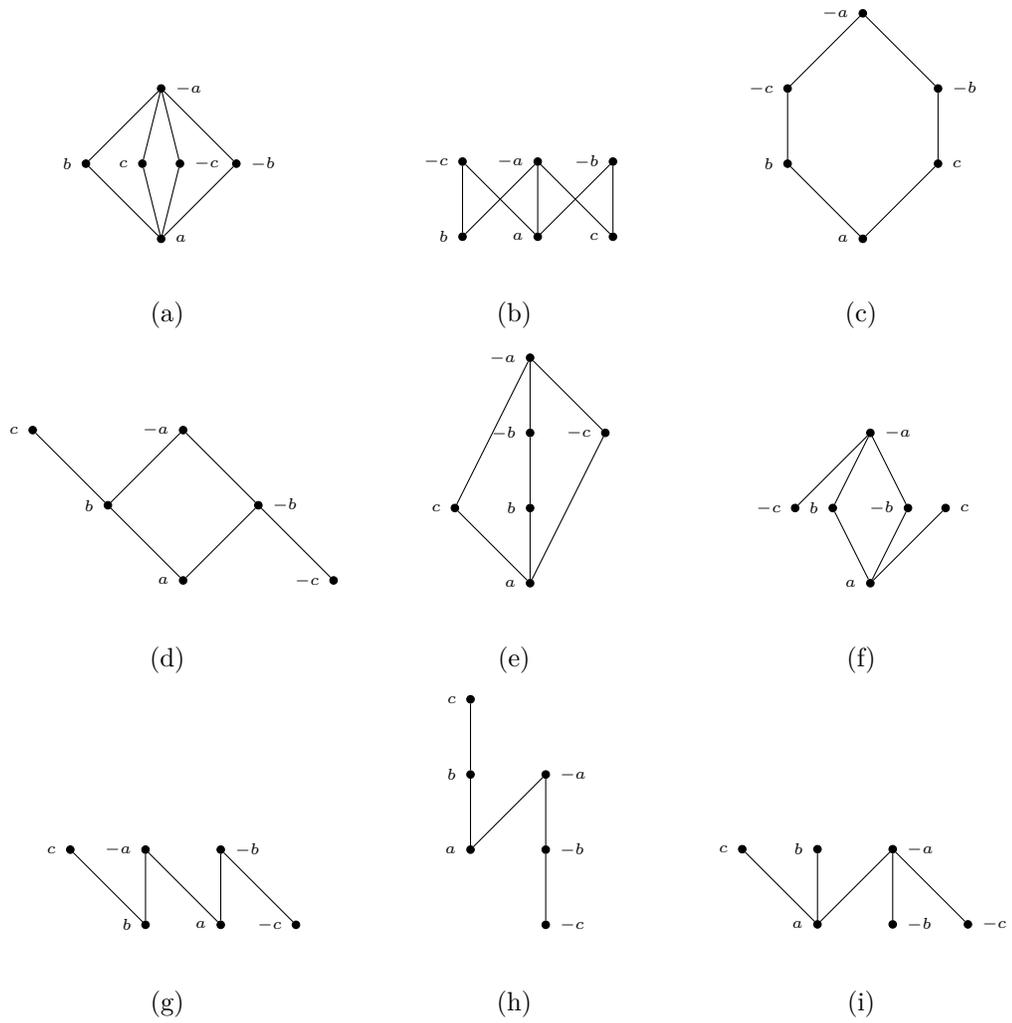
\begin{figure}
\begin{center}
    \begin{tabular}{ccc}
        \begin{subfigure}[b]{.25\textwidth}
            \begin{center}
    \begin{tikzpicture}
    \node[vertex] (1) at (0,0) [label=right:$a$] {}; 
    \node[vertex] (2) at (-1,1) [label=left:$b$] {};
    \node[vertex] (3) at (-.25,1) [label=left:$c$] {};
    \node[vertex] (-3) at (.25,1) [label=right:$-c$] {};
    \node[vertex] (-2) at (1,1) [label=right:$-b$]{};
    \node[vertex] (-1) at (0,2) [label=right:$-a$] {};
    \draw (1)--(2)--(-1);
    \draw (1)--(3)--(-1);
    \draw (1)--(-2)--(-1);
    \draw (1)--(-3)--(-1);
    \end{tikzpicture}
\end{center}
    \caption{}
\end{subfigure}&
\begin{subfigure}[b]{.25\textwidth}
    \begin{center}
    \begin{tikzpicture}
        \node[vertex] (2) at (0,0) [label=left:$b$] {};
        \node[vertex] (1) at (1,0) [label=left:$a$] {};
        \node[vertex] (3) at (2,0) [label=left:$c$] {};
        \node[vertex] (-3) at (0,1) [label=left:$-c$] {};
        \node[vertex] (-1) at (1,1) [label=left:$-a$] {};
        \node[vertex] (-2) at (2,1) [label=left:$-b$] {};
        \draw (2)--(-3);
        \draw (2)--(-1);
        \draw (1)--(-3);
        \draw (1)--(-1);
        \draw (1)--(-2);
        \draw (3)--(-1);
        \draw (3)--(-2);
    \end{tikzpicture}
\end{center}
    \caption{}
\end{subfigure} &
\begin{subfigure}[b]{.25\textwidth}
    \begin{center}
    \begin{tikzpicture}
        \node[vertex] (1) at (0,0) [label=left:$a$] {};
        \node[vertex] (2) at (-1,1) [label=left:$b$] {};
        \node[vertex] (3) at (1,1) [label=right:$c$] {};
        \node[vertex] (-3) at (-1,2) [label=left:$-c$] {};
        \node[vertex] (-2) at (1,2) [label=right:$-b$] {};
        \node[vertex] (-1) at (0,3) [label=left:$-a$] {};
        \draw (1)--(2)--(-3)--(-1);
        \draw (1)--(3)--(-2)--(-1);
    \end{tikzpicture}
\end{center}
\caption{}
\end{subfigure} 
\\
\begin{subfigure}[b]{.3\textwidth}
    \begin{center}
    \begin{tikzpicture}
        \node[vertex] (1) at (0,0) [label=left:$a$] {};
        \node[vertex] (2) at (-1,1) [label=left:$b$] {};
        \node[vertex] (-2) at (1,1) [label=right:$-b$] {};
        \node[vertex] (3) at (-2,2) [label=left:$c$] {};
        \node[vertex] (-1) at (0,2) [label=left:$-a$] {};
        \node[vertex] (-3) at (2,0) [label=left:$-c$] {};
        \draw (1)--(2)--(3);
        \draw (2)--(-1);
        \draw (1)--(-2)--(-1);
        \draw (-3)--(-2);
    \end{tikzpicture}
\end{center}
\caption{}
\end{subfigure} &
\begin{subfigure}[b]{.25\textwidth}
    \begin{center}
    \begin{tikzpicture}
        \node[vertex] (1) at (0,0) [label=left:$a$] {};
        \node[vertex] (3) at (-1,1) [label=left:$c$] {};
        \node[vertex] (2) at (0,1) [label=left:$b$] {};
        \node[vertex] (-3) at (1,2) [label=left:$-c$] {};
        \node[vertex] (-2) at (0,2) [label=left:$-b$] {};
        \node[vertex] (-1) at (0,3) [label=left:$-a$] {};
        \draw (1)--(3)--(-1);
        \draw (1)--(2)--(-2)--(-1);
        \draw (1)--(-3)--(-1);
    \end{tikzpicture}
\end{center}
\caption{}
\end{subfigure}&
\begin{subfigure}[b]{.25\textwidth}
		\begin{center}
    \begin{tikzpicture}
        \node[vertex] (1) at (0,0) [label=left:$a$] {};
        \node[vertex] (-3) at (-1,1) [label=left:$-c$] {};
        \node[vertex] (2) at (-.5,1) [label=left:$b$] {};
        \node[vertex] (-2) at (.5,1) [label=left:$-b$] {};
        \node[vertex] (3) at (1,1) [label=right:$c$] {};
        \node[vertex] (-1) at (0,2) [label=right:$-a$] {};
        \draw (1)--(2)--(-1);
        \draw (1)--(-2)--(-1);
        \draw (-3)--(-1);
        \draw (1)--(3); 
    \end{tikzpicture}
	\end{center}
	\caption{}
	\end{subfigure} \\
	\begin{subfigure}[b]{.25\textwidth}
			\begin{center}
    \begin{tikzpicture}
        \node[vertex] (2) at (0,0) [label=left:$b$] {};
        \node[vertex] (1) at (1,0) [label=left:$a$] {};
        \node[vertex] (-3) at (2,0) [label=left:$-c$] {};
        \node[vertex] (3) at (-1,1) [label=left:$c$] {};
        \node[vertex] (-1) at (0,1) [label=left:$-a$] {};
        \node[vertex] (-2) at (1,1) [label=right:$-b$] {};
        \draw (2)--(3);
        \draw (2)--(-1);
        \draw (1)--(-1);
        \draw (1)--(-2);
        \draw (-3)--(-2);
    \end{tikzpicture} 
\end{center}
\caption{}
\end{subfigure}&
\begin{subfigure}[b]{.25\textwidth}
		\begin{center}
    \begin{tikzpicture}
        \node[vertex] (1) at (0,0) [label=left:$a$] {};
        \node[vertex] (2) at (0,1) [label=left:$b$] {};
        \node[vertex] (3) at (0,2) [label=left:$c$] {};
        \node[vertex] (-1) at (1,1) [label=right:$-a$] {};
        \node[vertex] (-2) at (1,0) [label=right:$-b$] {};
        \node[vertex] (-3) at (1,-1) [label=right:$-c$] {};
        \draw (1)--(2)--(3);
        \draw (1)--(-1);
        \draw (-3)--(-2)--(-1);
    \end{tikzpicture}
\end{center}
\caption{}
\end{subfigure}&
\begin{subfigure}[b]{.3\textwidth}
		\begin{center}
    \begin{tikzpicture}
        \node[vertex] (1) at (0,0) [label=left:$a$] {};
        \node[vertex] (3) at (-1,1) [label=left:$c$] {};
        \node[vertex] (2) at (0,1) [label=left:$b$] {};
        \node[vertex] (-1) at (1,1) [label=right:$-a$] {};
        \node[vertex] (-2) at (1,0) [label=right:$-b$] {};
        \node[vertex] (-3) at (2,0) [label=right:$-c$] {};
        \draw(1)--(2);
        \draw (1)--(3);
        \draw(1)--(-1);
        \draw (-2)--(-1);
        \draw (-3)--(-1);
    \end{tikzpicture}
\end{center}
\caption{}
\end{subfigure}
\end{tabular}
\end{center}
\caption{The excluded posets, part one}
\label{fig:excludedposets}
\end{figure}
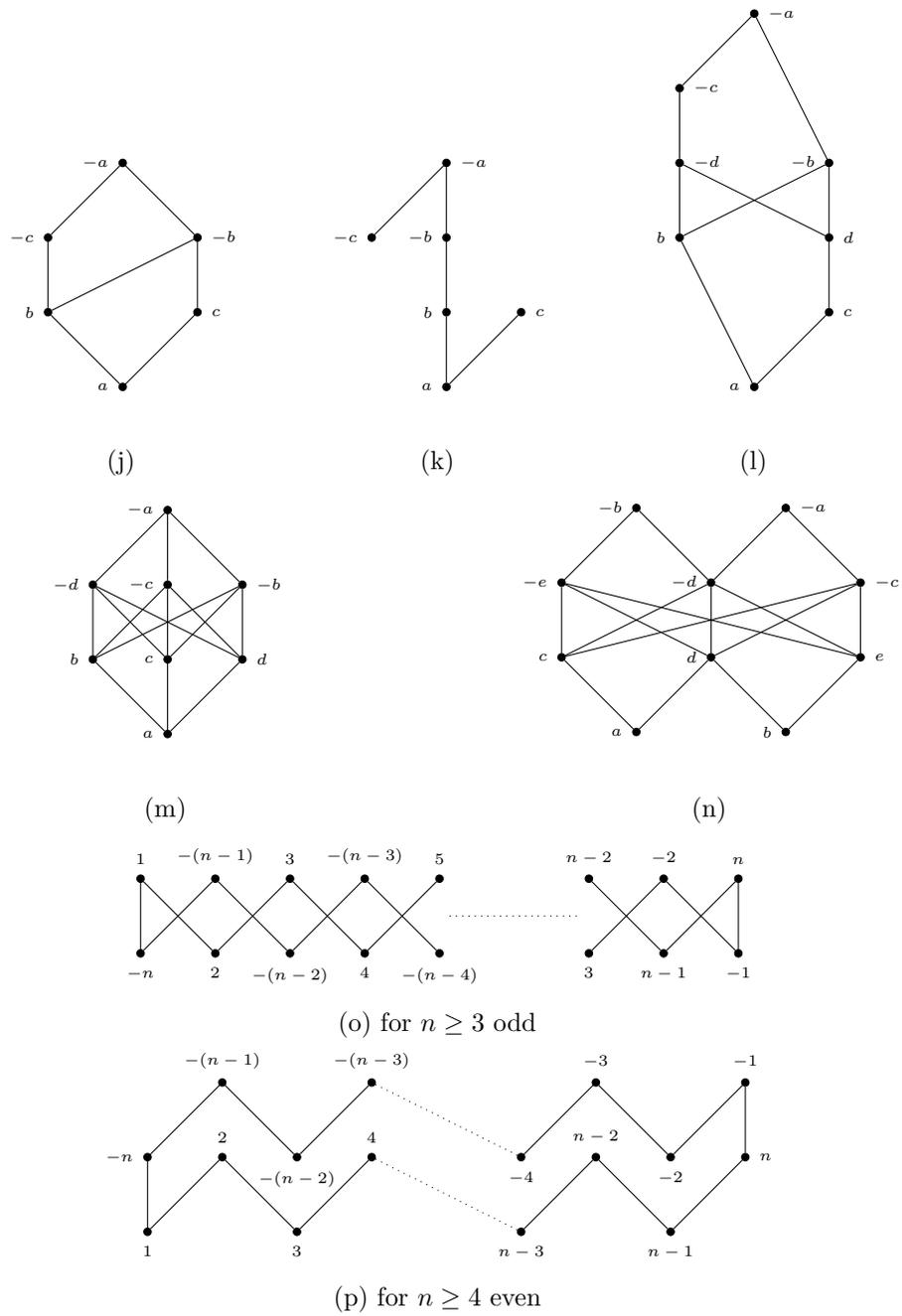
\begin{figure}
	\ContinuedFloat
\begin{center}
    \begin{tabular}{ccc}		
\begin{subfigure}[b]{.25\textwidth}
		\begin{center}
     \begin{tikzpicture}
        \node[vertex] (1) at (0,0) [label=left:$a$] {};
        \node[vertex] (2) at (-1,1) [label=left:$b$] {};
        \node[vertex] (3) at (1,1) [label=right:$c$] {};
        \node[vertex] (-3) at (-1,2) [label=left:$-c$] {};
        \node[vertex] (-2) at (1,2) [label=right:$-b$] {};
        \node[vertex] (-1) at (0,3) [label=left:$-a$] {};
        \draw (1)--(2)--(-3)--(-1);
        \draw (1)--(3)--(-2)--(-1);
        \draw (2)--(-2);
    \end{tikzpicture} 
\end{center}
\caption{}
\end{subfigure}&
\begin{subfigure}[b]{.25\textwidth}
		\begin{center}
    \begin{tikzpicture}
        \node[vertex] (1) at (0,0) [label=left:$a$] {};
        \node[vertex] (2) at (0,1) [label=left:$b$] {};
        \node[vertex] (-2) at (0,2) [label=left:$-b$] {};
        \node[vertex] (-1) at (0,3) [label=right:$-a$] {};
        \node[vertex] (-3) at (-1,2) [label=left:$-c$] {};
        \node[vertex] (3) at (1,1) [label=right:$c$] {};
        \draw (1)--(2)--(-2)--(-1);
        \draw (-3)--(-1);
        \draw (1)--(3);
    \end{tikzpicture}
\end{center}
\caption{}
\end{subfigure}&
\begin{subfigure}[b]{.25\textwidth}
		\begin{center}
    \begin{tikzpicture}
        \node[vertex] (1) at (0,0) [label=left:$a$] {};
        \node[vertex] (2) at (-1,2) [label=left:$b$] {};
        \node[vertex] (-2) at (1,3) [label=left:$-b$] {};
        \node[vertex] (3) at (1,1) [label=right:$c$] {};
        \node[vertex] (4) at (1,2) [label=right:$d$] {};
        \node[vertex] (-4) at (-1,3) [label=right:$-d$] {};
        \node[vertex] (-3) at (-1,4) [label=right:$-c$] {};
        \node[vertex] (-1) at (0,5) [label=right:$-a$] {};
        \draw(1)--(2)--(-2)--(-1);
        \draw (2)--(-4);
        \draw (4)--(-2);
        \draw (1)--(3)--(4)--(-4)--(-3)--(-1);
    \end{tikzpicture}
\end{center}
\caption{}
\end{subfigure}\\
\end{tabular}\\
\begin{tabular}{cc}
\begin{subfigure}[b]{.45\textwidth}
		\begin{center}
    \begin{tikzpicture}
        \node[vertex] (1) at (0,0) [label=left:$a$] {};
        \node[vertex] (2) at (-1,1) [label=left:$b$] {};
        \node[vertex] (3) at (0,1) [label=left:$c$] {};
        \node[vertex] (4) at (1,1) [label=right:$d$] {};
        \node[vertex] (-4) at (-1,2) [label=left:$-d$] {};
        \node[vertex] (-3) at (0,2) [label=left:$-c$] {};
        \node[vertex] (-2) at (1,2) [label=right:$-b$] {};
        \node[vertex] (-1) at (0,3) [label=left:$-a$] {};
        \draw(1)--(2)--(-4)--(-1);
        \draw(1)--(3)--(-3)--(-1);
        \draw(1)--(4)--(-2)--(-1);
        \draw(2)--(-3);
        \draw(2)--(-2);
        \draw(3)--(-4);
        \draw(3)--(-2);
        \draw(4)--(-4);
        \draw(4)--(-3);
    \end{tikzpicture} 
\end{center}
\caption{}
\end{subfigure}&
\begin{subfigure}[b]{.45\textwidth}
		\begin{center}
    \begin{tikzpicture}
        \node[vertex] (1) at (0,0) [label=left:$a$] {};
        \node[vertex] (2) at (2,0) [label=left:$b$] {};
        \node[vertex] (3) at (-1,1) [label=left:$c$] {};
        \node[vertex] (4) at (1,1) [label=left:$d$] {};
        \node[vertex] (5) at (3,1) [label=right:$e$] {};
        \node[vertex] (-5) at (-1,2) [label=left:$-e$] {};
        \node[vertex] (-4) at (1,2) [label=left:$-d$] {};
        \node[vertex] (-3) at (3,2) [label=right:$-c$] {};
        \node[vertex] (-2) at (0,3) [label=left:$-b$] {};
        \node[vertex] (-1) at (2,3) [label=right:$-a$] {};
        \draw (3)--(1)--(4)--(2)--(5)--(-3)--(-1)--(-4)--(-2)--(-5);
        \draw (3)--(-5);
        \draw (3)--(-4);
        \draw (3)--(-3);
        \draw (4)--(-5);
        \draw (4)--(-4);
        \draw (4)--(-3);
        \draw (5)--(-5);
        \draw (5)--(-4);
    \end{tikzpicture}
	\end{center}
	\caption{}\end{subfigure}\\
\end{tabular}
\begin{subfigure}[b]{\textwidth}
    \begin{center}
    \begin{tikzpicture}
        \tikzstyle{every node}=[font=\tiny]
        \node[vertex] (-n) at (0,0) [label=below:$-n$] {};
        \node[vertex] (2) at (1,0) [label=below:$2$] {};
        \node[vertex] (-n2) at (2,0) [label=below:$-(n-2)$] {};
        \node[vertex] (4) at (3,0) [label=below:$4$] {};
        \node[vertex] (-n4) at (4,0) [label=below:$-(n-4)$] {};
        \node[vertex] (1) at (0,1) [label=above:$1$] {};
        \node[vertex] (-n1)at (1,1) [label=above:$-(n-1)$] {};
        \node[vertex] (3) at (2,1) [label=above:$3$] {};
        \node[vertex] (-n3) at (3,1) [label=above:$-(n-3)$] {};
        \node[vertex] (5) at (4,1) [label=above:$5$] {};
        \draw (1)--(-n)--(-n1)--(-n2)--(-n3)--(-n4);
        \draw (1)--(2)--(3)--(4)--(5);

        \node[vertex] (-3) at (6,0) [label=below:$3$] {};
        \node[vertex] (n1) at (7,0) [label=below:$n-1$] {};
        \node[vertex] (-1) at (8,0) [label=below:$-1$] {};
        \node[vertex] (n2) at (6,1) [label=above:$n-2$] {};
        \node[vertex] (-2) at (7,1) [label=above:$-2$] {};
        \node[vertex] (n) at (8,1) [label=above:$n$] {};
        \draw (n2)--(n1)--(n)--(-1)--(-2)--(-3);
        \path (-n4)--(5) node[midway] (A) {};
        \path (-3)--(n2) node[midway] (B) {};
        \draw[dotted] (A)--(B);
    \end{tikzpicture}
    \caption{for $n\geq 3$ odd}
	\label{subfig:oddfamily}
\end{center}
\end{subfigure} \\
\begin{subfigure}[b]{\textwidth}
  \begin{center}
    \begin{tikzpicture}
        \tikzstyle{every node}=[font=\tiny]
        \node[vertex] (1) at (0,0) [label=below:$1$] {};
        \node[vertex] (2) at (1,1) [label=above:$2$] {};
        \node[vertex] (3) at (2,0) [label=below:$3$] {};
        \node[vertex] (4) at (3,1) [label=above:$4$] {};
        \node[vertex] (n1) at (7,0)[label=below:$n-1$] {};
        \node[vertex] (n) at (8,1) [label=right:$n$] {};
        \node[vertex] (-1) at (8,2) [label=above:$-1$] {};
        \node[vertex] (-2) at (7,1) [label=below:$-2$] {};
        \node[vertex] (-3) at (6,2) [label=above:$-3$] {};
        \node[vertex] (-4) at (5,1) [label=below:$-4$] {};
        \node[vertex] (-n2) at (2,1) [label=below:$-(n-2)$] {};
        \node[vertex] (-n1) at (1,2) [label=above:$-(n-1)$] {};
        \node[vertex] (-n) at (0,1) [label=left:$-n$] {};
        \node[vertex] (-n3) at (3,2) [label=above:$-(n-3)$] {};
        \node[vertex] (n2) at (6,1) [label=above:$n-2$] {};
        \node[vertex] (n3) at (5,0) [label=below:$n-3$] {};

        \draw (1)--(-n)--(-n1)--(-n2)--(-n2)--(-n3);
        \draw (1)--(2)--(3)--(4);
        \draw (n3)--(n2)--(n1)--(n)--(-1)--(-2)--(-3)--(-4);
        \draw[dotted] (-n3)--(-4);
        \draw[dotted] (4)--(n3);

    \end{tikzpicture}
\caption{for $n \geq 4$ even}
\label{subfig:evenfamily}
\end{center}
\end{subfigure}
\end{center}
\caption{The excluded posets, part two}
\end{figure}

Key to the proof will be the following lemma.

\begin{lem}\label{lem:inducedbad}
    Suppose $P'$ is a signed poset that contains an induced subposet $P$, such
    that $P$ has a connected order ideal $J$ which intersects at least two
    other connected ideals of $P$ nontrivially. Then $P'$ also has such an ideal.
\end{lem}

\begin{proof}
    It suffices to examine the case where $P$ is a signed poset on $n$ and $P'$ is a
    signed poset on $n+1$. Recall that an ideal is determined by the antichain of its
    maximal elements. Consequently, every connected ideal $J$ of $P$ corresponds
    to an ideal $J'$ in $\Ghat(P')$ having the same determining antichain. 

    \noindent\textbf{Claim:} $J'$ is an isotropic order ideal of $\Ghat(P')$
    and thus an ideal of $P'$.

    Suppose not. There are two cases.
    \begin{itemize}
        \item First suppose $j_1$ and $j_2$ are maximal elements of $J$ (and
            thus of $J'$) with $(n+1) < j_1$ and $-(n+1)< j_2$. Then one must
            have $-j_2 < j_1$ and $-j_1 < j_2$. Since $P$ is an induced
            subposet of $P'$, the same relation holds in $P$, which means $J$
            was not isotropic, a contradiction.
        \item Suppose $j \in J'$ such that $(n+1) < j$ and $-(n+1) <
            j$. Then $-j < j$, which contradicts that $J$ was isotropic.
    \end{itemize}
    Thus, $J'$ must be an ideal of $P'$.

    If $J$ and $K$ intersect nontrivially in $P$ and $J'$ and $K'$ are the
    corresponding ideals of $P'$. Certainly $J' \cap K' \ne \emptyset$. Suppose
    $J' \subset K'$, i.e.\ $J'$ and $K'$ do not intersect nontrivially. Since
    $P$ is an induced subposet of $P'$, that would force $J \subset K$,
    contradicting that $J$ and $K$ intersect nontrivially.

    Consequently, if $P$ has a connected ideal $J$ that intersects connected
    ideals $K_1$ and $K_2$ nontrivially, $J'$ intersects $K_1'$ and $K_2'$
    nontrivially, completing the proof.
\end{proof}

While the above lemma makes the sufficiency of
Theorem~\ref{thm:excludedposets} straightforward, the necessity argument is
made easier by two further lemmas.

\begin{lem}\label{lem:notypeanotinitci}
    Suppose $n \geq 3$ and $P \subset \Bn$ is a signed poset such that no $i <
    -i$. Then $P$ is not an initial complete intersection.
\end{lem}

\begin{proof}
 One may assume $P$ is connected. Begin by observing that a member of the
 infinite families in Figures~\ref{subfig:oddfamily}
 and~\ref{subfig:evenfamily} cannot be an initial complete intersection as a
 consequence of Proposition~\ref{prop:initciiffuniquenontriv}. Since $P$ has no $i < -i$, for $i \in
 \pm[x]$, there must be an
$i \in [n]$ such that there is an isotropic path from $i$ to $-i$. Since $i$
and $-i$ are not comparable, this path must have at least two intermediate
vertices and have a maximum and minimum other than $i$ and $-i$. Then the image
of this path under the involution will form a cycle $C$ which is an induced
subposet of $P$. Taking the smallest induced subposet of $C$ which remains a
cycle gives a member of one of the infinite families in
Figures~\ref{subfig:oddfamily} and~\ref{subfig:evenfamily}, so by Lemma~\ref{lem:inducedbad},
$P$ is not an initial complete intersection.
\end{proof}

\feray and Reiner characterized posets in type A such that $\Rpwt$ is a
complete intersection as being those posets avoiding the intersection of the
three type A posets Figures~\ref{fig:excludedposets}(l),(m),(n) with $\An$.
While $\Swt{P}/\initI$ being a complete
intersection implies $\Swt{P}/\Iwt{P}$ is a complete intersection
(see~\cite[Exercise 3.3]{HerzogHibi2011}), the converse is not, in
general, true. For example, consider the signed poset in
Figure~\ref{fig:cinotinitci}. The semigroup ring $\Rpwt$ has a complete
intersection presentation:
\begin{multline*}
0 \to (U_{\m4}U_{14}-U_1,U_{\m3}U_{134}-U_{14},U_{\m4}U_{134}-U_{13}) \to \\
k[U_1,U_2,U_{\m3},U_{\m4},U_{13},U_{14},U_{134}] \to \Rpwt \to 0.
\end{multline*}
However, $\initI = (U_{\m4}U_{14},U_{\m3}U_{134},U_{\m4}U_{134})$, so
$\Swt{P}/\initI$ is not a complete intersection.

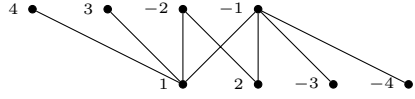
\begin{figure}[htbp]
    \begin{center}
        \begin{tikzpicture}
            \node[vertex] (1) at (0,0) [label=left:$1$] {};
            \node[vertex] (2) at (1,0) [label=left:$2$] {};
            \node[vertex] (-3) at (2,0) [label=left:$-3$] {};
            \node[vertex] (-4) at (3,0) [label=left:$-4$] {};
            \node[vertex] (4) at (-2,1) [label=left:$4$] {};
            \node[vertex] (3) at (-1,1) [label=left:$3$] {};
            \node[vertex] (-2) at (0,1) [label=left:$-2$] {};
            \node[vertex] (-1) at (1,1) [label=left:$-1$] {};
            \draw (1)--(4);
            \draw (1)--(3);
            \draw (1)--(-2)--(2)--(-1)--(1);
            \draw (-3)--(-1);
            \draw (-4)--(-1);
        \end{tikzpicture}
    \end{center}
    \caption{A signed poset $P$ where $\Rpwt$ is a complete intersection, but
        $P$ is not an initial complete intersection.}
    \label{fig:cinotinitci}
\end{figure}

The following lemma is an immediate consequence of Proposition 10.2 and Theorem
10.5 of~\cite{FerayReiner2012}.

\begin{lem}\label{lem:typeainitci}
  Suppose $P \subset \Bn$ is a signed poset such that $i < -i$ for all $i$.
Then $P$ is an initial complete intersection if and only if $P$ does not contain
any of the signed posets in Figures~\ref{fig:excludedposets}(l),(m),(n) as an induced
subposet.
\end{lem}

One is now ready to prove Theorem~\ref{thm:excludedposets}.

\begin{proof}[Proof of Theorem~\ref{thm:excludedposets}]
The necessity follows from checking that all the posets in the list have at
least one connected isotropic order ideal that intersects at least two other
connected ideals nontrivially and applying Lemma~\ref{lem:inducedbad}.

For the sufficiency, Lemmas~\ref{lem:notypeanotinitci} and~\ref{lem:typeainitci}
allow one to immediately reduce to the case where there is $i \in \pm[n]$ such that
$i < -i$ in $P$ and there is some $j$ such that $j$ and $-j$ are incomparable.
There are a number of cases.
\begin{enumerate}[label=(\alph{*})]
    \item \textbf{Suppose $i \in \Ghat(P)$ is such that $i$ and $-i$ are incomparable
and $i$ is not comparable to any $a < -a$.} Since $\Ghat(P)$ is connected, there
must be some path from $i$ to an $a$ such that $a < -a$ such that no
intermediate vertex is less than its negative. Call the vertices on this path
$i, v_1,v_2,\ldots,v_k,a$. Let $m$ be minimal such that $v_m > a$. Then one
must have $v_{m-1} < v_m$. Taking the induced subposet on $v_{m-1},v_m,a$
gives a signed poset isomorphic to Figure~\ref{fig:excludedposets}(g).

\item \textbf{Suppose every $i \in \Ghat(P)$ such that $i$ and $-i$ are incomparable is
comparable to some $a$ such that $a < -a$.}
\begin{enumerate}[label=(\alph{enumi}.\roman*)]
\item \textbf{Suppose $\Ghat(P)$ has a connected isotropic order ideal $J$ containing
two elements $i,j$ which are not comparable to their negatives and there is no
$a$ such that $a < -a$ and $i,j > a$.} Since $J$ is connected, there must be a
path from $i$ to $j$ contained entirely in $J$. If $j < i$, then the induced
subposet on $i,j,a$ gives a signed poset isomorphic to
Figure~\ref{fig:excludedposets}(g).

If $i$ and $j$ are not comparable, one may assume without loss of generality
that the path between them in $J$ passes through only elements which are less
than their negatives. Since $i$ and $j$ are not both comparable to any $a$ with
$a < -a$, the path between them must have at least three intermediate vertices,
$a_1,a_2,a_3$, as shown in Figure~\ref{fig:casebi}.
\begin{figure}[htbp]
\begin{center}
    \begin{tikzpicture}
        \node[vertex] (i) at (0,0) [label=left:$i$] {};
		\node[vertex] (a1) at (1,-1) [label=left:$a_1$] {};
		\node[vertex] (a2) at (2,0) [label=left:$a_2$] {};
		\node[vertex] (a3) at (3,1) [label=right:$a_3$] {};
		\node[vertex] (j) at (5,0) [label=right:$j$] {};
		\draw (i) -- (a1);
		\draw[dotted] (a1)--(a2)--(a3);
		\draw[decorate,decoration={snake},dotted] (a3)--(j);
    \end{tikzpicture}
    \caption{The path described in the second half of case (b)(i) in the proof
        of Theorem~\ref{thm:excludedposets}}
    \label{fig:casebi}
\end{center}
\end{figure}
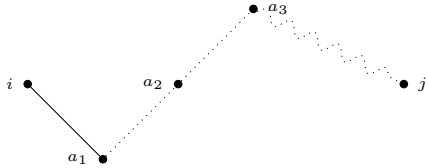
Since $a_1 < -a_1$ and $i$ and $-i$ are incomparable, one knows the first
vertex along the path must be below $i$. Then taking the induced subposet on
$i,a_2,a_3$ gives a signed poset isomorphic to that in
Figure~\ref{fig:excludedposets}(k).

\item \textbf{Assume instead that $J$ contains $i,j$ not comparable to their negatives
and that there is an $a$ with $a < -a$ with $i,j > a$.} Taking the induced
subposet on $i,j,a$ gives one of seven possibilities, which are all forbidden
posets, namely Figures~\ref{fig:excludedposets}(a),(b),(c),(d),(f),(h),(i).

\item \textbf{Suppose $J$ is a connected ideal containing precisely one element not
less than its negative, call it $i$.} One knows that $I_\leq(-i)$ is isotropic,
so $J$ and $I_\leq(-i)$  intersect nontrivially.

\noindent\textbf{Claim:} $i$ is a maximal element of $J$.

Suppose not and $j \in J$ with $j > i$. Since $J$ contains only one element not
less than its negative, i.e.\ $i$, $j < -j$, but the symmetry in $\Ghat(P)$ then
forces $i < -i$, a contradiction.

Having proved the claim, there are two final cases to check.
\begin{enumerate}[label=(\alph{enumi}.\roman{enumii}.\arabic*)]
\item \textbf{Suppose $J \cap K = \emptyset$.} Then one must have $-i \in K$. Since $K
\ne I_\leq(-i)$, there must be $a,b \in K$ with $b,-i > a$ and $a < -a, b<
-b$. However, the induced subposet on $a,b,-i$ is isomorphic to the signed
poset of Figure~\ref{fig:excludedposets}(k).
\item \textbf{Suppose $J \cap K \ne \emptyset$.} Therefore, there must be $j, \ell$ in
$J+K$ with $j < i, \ell$. Since $j, \ell \in J+K$, one must have that $j < -j$
and $\ell < -\ell$. One then has that the induced subposet on $j,k,\ell$ is
isomorphic to one of the signed posets shown in
Figures~\ref{fig:excludedposets}(j) and (k).
\end{enumerate}
\end{enumerate}
\end{enumerate}
\end{proof}

\subsection{Constructing the complete intersection
posets}\label{subsec:constructinginitci}

In~\cite{FerayReiner2012}, \feray and Reiner also characterized the posets for
which $\Rpwt$ is a complete intersection as being the \emph{forests with
duplication}, those posets which could be constructed via operations they
called \emph{disjoint union},
\emph{duplication of a hanger} and \emph{hanging}. An analogous construction exists for
signed posets which are initial complete intersections. Due to the reductions
in Section~\ref{subsec:wtreducing}, one can dispense with the disjoint union
operation in the construction.

Recall that if $P$ is a signed poset, $I_<(a) = \{b \in P \colon b < a\}$ and
$I_{\leq}(a) = \{b \in P \colon b \leq a\}$, with the comparisons being made in
$\Ghat(P)$. Let $P_{<a}$ denote the subposet of
$P$ induced by $\{i \colon \pm i \in I_<(a)\}$ and let $P_{\leq a}$ denote the
subposet of $P$ induced by $\{i \colon \pm i \in I_{\leq}(a)\}$.

\begin{defn}\label{def:hanger}
Suppose $P$ is a signed poset and $i \in \Ghat(P)$ is such that $i < -i$. Then
$i$ is said to be a \emph{hanger} if, for each $a \in I_{<}(i)$, the induced
subset on those vertices $\pm j$ such that $j < i$ (which may be
empty) and for each $b \in G(P) \smallsetminus G(P_{< i})$, any path from $a$ to
$b$ must pass through $i$.
\end{defn}

\begin{defn}\label{def:hanging}
A signed poset $P$ is said to have been \emph{obtained by hanging $P_1$ below $a$ in
$P_2$} if
\begin{itemize}
    \item $a$ is a hanger in $P$,
\item $P_1 = P_{<}(a)$, and
\item $P_2$ is the induced subposet of $P$ on $i$ such that $\pm i \not <a$ in
$\Ghat(P)$.
\end{itemize}
\end{defn}
See Figure~\ref{fig:hangingex} for an example of hanging.
\begin{figure}[htbp]
    \begin{center}
        \begin{tikzpicture}
            \begin{scope}
                \node[vertex] (1) at (0,0) [label=right:$1$] {};
                \node[vertex] (2) at (1,0) [label=right:$2$] {};
                \node[vertex] (-3) at (2,0) [label=right:$-3$] {};
                \node[vertex] (3) at (-1,1) [label=left:$3$] {};
                \node[vertex] (-2) at (0,1) [label=left:$-2$] {};
                \node[vertex] (-1) at (1,1) [label=right:$-1$] {};
                \draw (3)--(1)--(-1)--(-3);
                \draw (1)--(-2)--(2)--(-1);
            \end{scope}
            \begin{scope}[xshift=6cm] 
                \node[vertex] (1) at (0,0) [label=right:$1$] {};
                \node[vertex] (2) at (1,0) [label=right:$2$] {};
                \node[vertex] (-3) at (2,0) [label=right:$-3$] {};
                \node[vertex] (3) at (-1,1) [label=left:$3$] {};
                \node[vertex] (-2) at (0,1) [label=left:$-2$] {};
                \node[vertex] (-1) at (1,1) [label=right:$-1$] {};
                \node[vertex] (-4) at (0,2) [label=left:$-4$] {};
                \node[vertex] (4) at (1,-1) [label=left:$4$] {};
                \draw (3)--(1)--(-1)--(-3);
                \draw (1)--(-2)--(2)--(-1);
                \draw (-2)--(-4);
                \draw (4)--(2);
            \end{scope}
            \draw[->] (2.5,.5)-- node[midway,above] {hanging}(4.5,.5);
        \end{tikzpicture}
    \end{center}
    \caption{Obtaining a new poset by hanging a single element 4 below 2}
    \label{fig:hangingex}
\end{figure}
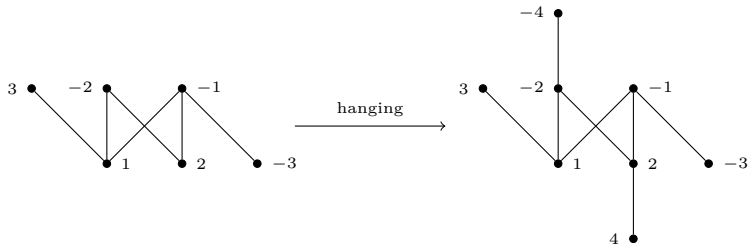
When $P$ is a type A signed poset, this definition of hanging coincides with the
hanging of~\cite{FerayReiner2012} in the sense that if $P'$ is obtained from
$P$ by hanging, $P' \cap \An$ gives the poset that is the result of the same
hanging in $P \cap \An$.

\begin{lem}\label{lem:hangingpreservesci}
Suppose $P$ is obtained by hanging $P_1$ below $a$ in $P_2$ and both $P_1$ and
$P_2$ are initial complete intersections. Then $P$ is also an initial complete
intersection.
\end{lem}

\begin{proof}
    Suppose $P$ is obtained by hanging $P_1$ below $a$ in $P_2$ and both $P_1$
and
$P_2$ are initial complete intersections. The connected
(isotropic) order ideals of $P$ are of three types:
\begin{enumerate}[label=(\Roman{*})]
    \item $J$ for $J \in \Jconn(P_1)$
\item $J\cup P_1$ for $J \in \Jconn(P_2)$ with $a \in J$.
\item $J$ for $J \in \Jconn(P_2)$ with $a \notin J$.
\end{enumerate}
Ideals of type I can only intersect other ideals of type I nontrivially and
can therefore only intersect at most one other ideal nontrivially. A
nontrivially intersecting pair involving ideals of type II or III
corresponds to a nontrivially intersecting pair in $P_2$, so these ideals
can also only be involved in at most nontrivially intersecting pair.
Consequently, one has that $P$ is an initial complete intersection.
\end{proof}

The next move also closely parallels~\cite{FerayReiner2012}. 
\begin{defn}\label{def:duplication}
The signed poset $P'$ is said to have been obtained from $P$ by \emph{duplicating the hanger} $a$
if $\Ghat(P') = \Ghat(P) \cup \{a',-a'\}$ with $i < a'$ or $a' < i$ whenever $i <
a$ or $a < i$ in $\Ghat(P)$.
\end{defn}

Figure~\ref{fig:duplicationex} gives an example of duplicating a hanger.

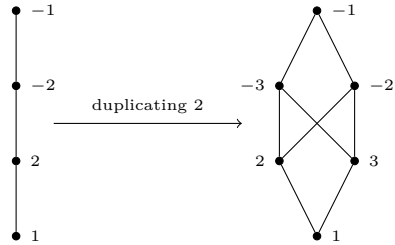
\begin{figure}[htbp]
    \begin{center}
        \begin{tikzpicture}
            \begin{scope}
                \node[vertex] (1) at (0,0) [label=right:$1$] {};
                \node[vertex] (2) at (0,1) [label=right:$2$] {};
                \node[vertex] (-2) at (0,2) [label=right:$-2$] {};
                \node[vertex] (-1) at (0,3) [label=right:$-1$] {};
                \draw (1)--(2)--(-2)--(-1);
            \end{scope}
            \begin{scope}[xshift=4cm]
                \node[vertex] (1) at (0,0) [label=right:$1$] {};
                \node[vertex] (2) at (-.5,1) [label=left:$2$] {};
                \node[vertex] (3) at (.5,1) [label=right:$3$] {};
                \node[vertex] (-2) at (.5,2) [label=right:$-2$] {};
                \node[vertex] (-3) at (-.5,2) [label=left:$-3$] {};
                \node[vertex] (-1) at (0,3) [label=right:$-1$] {};
                \draw (1)--(2)--(-2)--(-1);
                \draw (2)--(-3)--(-1);
                \draw (1)--(3)--(-2);
                \draw (3)--(-3);
            \end{scope}
            \draw[->] (.5,1.5)--node [midway, above] {duplicating $2$}(3,1.5);
        \end{tikzpicture}
        \caption{Obtaining a new poset by duplicating the hanger 2}
        \label{fig:duplicationex}
    \end{center}
\end{figure}

\begin{lem}\label{lem:duplicationpreserveresci}
Suppose $P'$ is obtained from $P$ by duplicating a hanger and $P$ is an initial
complete intersection. Then $P'$ is also an initial complete intersection.
\end{lem}

\begin{proof}
    Begin by noting that if $a$ is a hanger in $P$, it must be that $I_{\leq
a}$ intersects no other ideal nontrivially. Then the connected isotropic order
ideals of $P'$ are
\begin{itemize}
    \item $J$ for $J \in \Jconn(P)$ with $a \notin J$
\item $J\cup\{a'\}$ if $J \in \Jconn(P)$ with $a \in J$
\item $I_\leq(a')$.
\end{itemize}
The only nontrivially intersecting pair in $P'$ that does not
correspond to a nontrivially intersecting pair in $P$ (i.e.\ the only pair
created by the duplication) is $\{I_{\leq}(a),I_\leq(a')\}$, meaning $P'$ must
also have $\Swt{P'}/\init I_{P'}^\mathrm{wt}$ a complete intersection.
\end{proof}

While hanging and duplication are translations of the type A definitions, an
additional move is required to construct all the posets that are initial
complete intersections. 

\begin{defn}\label{def:typebhanging}
	  The signed poset $P'$ is said to have been obtained from $P$ by a
\emph{type B hanging} of $n+1$ from $i$ if $P' = \overline{P\cup \{e_i +\epsilon
    e_{n+1}\}}^{PLC}$,
where $i < -i$ and $\epsilon \in \{\pm 1\}$, and $I_{\leq}(i)$ is a maximal ideal with respect to inclusion.
\end{defn}

Figure~\ref{fig:typebhangingex} gives an example of type B hanging. The
important thing to notice is that the requirement $I_{\leq}(i)$ be a maximal
ideal severely limits the new ideals in $P'$ compared to $P$.

\begin{figure}[htbp]
    \begin{center}
        \begin{tikzpicture}
            \begin{scope}
                \node[vertex] (1) at (0,0) [label=left:$1$] {};
                \node[vertex] (2) at (1,0) [label=right:$2$] {};
                \node[vertex] (-2) at (0,1) [label=left:$-2$] {};
                \node[vertex] (-1) at (1,1) [label=right:$-1$] {};
                \draw (1)--(-1)--(2)--(-2)--(1);
            \end{scope}
            \begin{scope}[xshift=5cm]
                \node[vertex] (1) at (0,0) [label=left:$1$] {};
                \node[vertex] (2) at (1,0) [label=right:$2$] {};
                \node[vertex] (-2) at (0,1) [label=left:$-2$] {};
                \node[vertex] (-1) at (1,1) [label=right:$-1$] {};
                \node[vertex] (3) at (-1,1) [label=left:$3$] {};
                \node[vertex] (-3) at (2,0) [label=right:$-3$] {};
                \draw (1)--(-1)--(2)--(-2)--(1);
                \draw (1)--(3);
                \draw (-3)--(-1);
            \end{scope}
            \draw[->] (1.5,.5) --node[midway,above] {type B hanging} (3.75,.5);
        \end{tikzpicture}
        \caption{Obtaining a signed poset by the type B hanging of 3 above 1}
        \label{fig:typebhangingex}
    \end{center}
\end{figure}
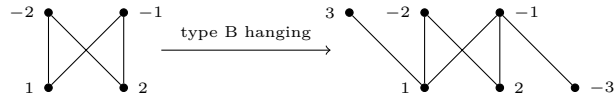

\begin{lem}\label{lem:typebhangingpreservesci}
    Suppose $P'$ is a signed poset obtained from $P$ by a type B hanging. If
$P$ is an initial complete intersection, then $P'$ is also.
\end{lem}

\begin{proof}
    Suppose $P'$ is obtained from $P$ by a type B hanging. Then the connected
	isotropic order ideals are
    \begin{itemize}
        \item $J$ for $J \in \Jconn(P)$
        \item$I_{\leq}(\epsilon(n+1))$
        \item  $\{-\epsilon(n+1)\}$.
    \end{itemize}
    It is clear that $I_\leq(\epsilon(n+1))$ and
$\{-\epsilon(n+1)\}$ intersect nontrivially and each intersect no other ideal
nontrivially. The other ideals cannot be involved in more than one nontrivial
intersection because they were not in $P$. Thus $P'$ must be an initial
complete intersection.
\end{proof}

\begin{thm}\label{thm:construction}
    Up to isomorphism, the signed posets with $\Ghat(P)$ connected which are initial complete
    intersections are precisely those which can be constructed by
hanging, duplication of a hanger and type B hanging from the two posets in
Figure~\ref{fig:baseposets}.
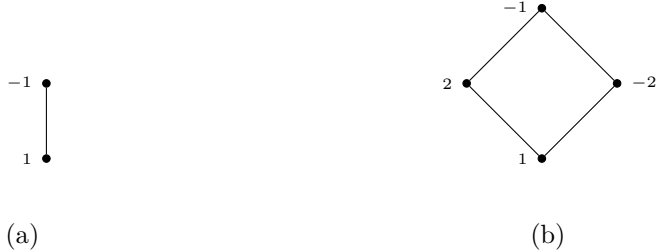
\begin{figure}[htbp]
    \begin{center}
     \begin{subfigure}[b]{.45\textwidth}
      \begin{center}
          \begin{tikzpicture}
              \node[vertex] (1) at (0,0) [label=left:$1$] {};
              \node[vertex] (-1) at (0,1) [label=left:$-1$] {};
              \draw (1)--(-1);
          \end{tikzpicture}
      \end{center}
      \caption{}
  \end{subfigure}
        \begin{subfigure}[b]{.45\textwidth}
            \begin{center}
    \begin{tikzpicture}
    \node[vertex] (1) at (0,0) [label=left:$1$] {};
	\node[vertex] (2) at (-1,1) [label=left:$2$] {};
	\node[vertex] (-2) at (1,1) [label=right:$-2$] {};
	\node[vertex] (-1) at (0,2) [label=left:$-1$] {};
	\draw (1)--(2)--(-1)--(-2)--(1);
    \end{tikzpicture}
\end{center}
\caption{}
\end{subfigure}
\end{center}
\caption{The posets from which the initial complete intersections are built}
\label{fig:baseposets}
\end{figure}
\end{thm}

\begin{proof}

The proof proceeds by induction, with base cases of $n=1$ and $n=2$. Up to
isomorphism, there is only one signed poset for $n=1$, namely $P_1=\{+e_1\}$
(see Figure~\ref{fig:baseposets}a) and
its toric ideal is trivial, so it is certainly an initial complete
intersection. For $n=2$, there are four signed posets which are initial complete
intersections: those obtained from $P_1$ by hanging, duplication of a hanger
and type B hanging, and the signed poset in Figure~\ref{fig:baseposets}b, whose
toric ideal is $(U_{12}U_{1\m2}-U_1^2)$, which is principal, so the signed
poset is an initial complete intersection.

Suppose $P$ is a signed poset on $n \geq 3$ and $\Ghat(P)$ is connected and
$P$ is an initial complete intersection. Consider $a \in \Ghat(P)$ and let $I(a)
= \{i \in \pm [n]: i < a\}$. From Lemma~\ref{lem:notypeanotinitci}, one knows
that one may assume $a < -a$. There are three cases.
\begin{enumerate}[label=(\alph*)]
    \item \textbf{Suppose that $a$ is not minimal and for all $a' \in \Ghat(P)$ not
comparable to $a$ with $I_<(a')$ isotropic, $I_<(a) \cap I_<(a') = \emptyset$.} Then
define two induced subposets of $P$: $P_{< a}$, the induced subposet on $i$
such that $\pm i < a$, and $P \smallsetminus P_{<a}$, the induced subposet on
$i$ such that $\pm i \not < a$.
Furthermore, $P$ is obtained by hanging $P_{<a}$ below $P\smallsetminus
P_{<a}$.
\item \textbf{Suppose there exists $a' \in \Ghat(P)$ not comparable to $a$ with $I_<(a')$
isotropic and $I_<(a) \cap I_<(a') \ne \emptyset$.}

\noindent\textbf{Claim:} $a' < -a'$.

Suppose not. Then both $I_<(a')$ and $I_<(-a')$ are isotropic and intersect
nontrivially. Then since $a < -a$, it must be that $a \ne -a'$, meaning
$I_<(a')$ intersects both $I_<(-a')$ and $I_<(a)$ nontrivially, contradicting
that $P$ had $S/\initI$ a complete intersection.

As in type A (see~\citet[Theorem 10.6]{FerayReiner2012}), one decomposes $P$ into four induced subposets:
\[
   P = \widehat{P} \sqcup P_{<a,a'} \sqcup (P_{<a} \smallsetminus P_{<a,a'})
\sqcup (P_{<a'} \smallsetminus P_{<a,a'}),
\]
where $P_{<a,a'}$ is the induced subposet on $i$ such that $\pm i < a$ and $\pm
i < a'$ in $\Ghat(P)$, $(P_{<a}\smallsetminus P_{<a,a'})$ is the induced
subposet on $i$ such that $\pm i < a$ and $\pm i \not < a'$, similarly for
$(P_{<a'}\smallsetminus P_{<a,a'})$ and 
$\widehat{P} = P\smallsetminus (P_{<a}\sqcup P_{<a'})$.

Observe that any element less than either $a$ or $a'$ is less than its
negative. The signed posets $P_{<a} \smallsetminus P_{<a'}$ and $P_{<a'}\smallsetminus P_{<a}$
may be empty, but $P_{<a,a'}$ is not. 

One constructs a signed poset $Q$ using hanging, duplication and type B hanging and
then shows that $Q = P$. There are two cases of the construction.
\begin{itemize}
\item Suppose there is $b \in \widehat{P}$  with $a < b$, $a' \not < b$ and $b
\not < -b$. Construct $Q$ as follows:
\begin{enumerate}[label=(\arabic*)]
    \item Start with $\widehat{P}\smallsetminus \{a',b\}$.
\item Hang $P_{<a,a'}$ below $a$ in $\widehat{P}\smallsetminus \{a',b\}$.
\item Duplicate $a$.
\item Hang $P_{<a}\smallsetminus P_{<a'}$ and $P_{<a'}\smallsetminus P_{<a}$
below $a$ and $a'$, respectively.
\item Type B hang $b$ from $c$, where $c \lessdot b$ in $P$.
\end{enumerate}
\item Suppose no such $b$ exists. Construct $Q$ as follows:
\begin{enumerate}[label=(\arabic*)]
    \item Start with $\widehat{P}\smallsetminus \{a',b\}$.
\item Hang $P_{<a,a'}$ below $a$ in $\widehat{P}\smallsetminus \{a',b\}$.
\item Duplicate $a$.
\item Hang $P_{<a}\smallsetminus P_{<a'}$ and $P_{<a'}\smallsetminus P_{<a}$
below $a$ and $a'$, respectively.
\end{enumerate}
\end{itemize}
In each case, each of the induced subposets used
($P_{<a,a'},\widehat{P}\smallsetminus\{a'\},\widehat{P}\smallsetminus\{a',b\},P_
{ <a } \smallsetminus P_{<a'},P_{<a'}\smallsetminus P_{<a},$) is an initial
complete intersection courtesy of Lemma~\ref{lem:inducedbad}, since they are
induced subposets of an initial complete intersection.
Consequently, $Q$ is an initial complete intersection. 

Next, one needs to show that $P$ and $Q$ are really the same poset. Certainly
the restrictions of $P$ and $Q$ to $P_{<a,a'}$, $P_{<a}\smallsetminus
P_{<a,a'}$ and $P_{<a'} \smallsetminus P_{<a,a'}$ are the same. It remains to
check that when restricted to the vertices of $\widehat{P}$, $P$ and $Q$ are
the same. Suppose $c > a$ and $c \not > a'$ in $\widehat{P}$. Consider two
cases.
\begin{enumerate}[label=(\alph{enumi}.\roman*)]
\item \textbf{Suppose $c < -c$.} Recall that there must be at least one element, call it
    $\ell$, such that $\ell < a$ and $\ell < a'$. Moreover, $\ell < -\ell$.
    This means that the induced subposet on $a,b,\ell$ is isomorphic to
    Figure~\ref{fig:excludedposets}(j).
\item \textbf{Suppose $c \not < -c$.} The existence of such a $c$ means $Q$ was
constructed using the first construction. If $c$ is not the $b$ from the
construction, one has that $a \in I_\leq(c) \cap I_\leq(b)$. However,
$I_\leq(c)$ and $I_\leq(-c)$ also intersect nontrivially, a contradiction,
unless $b = -c$. If $b= -c$, recall that there must be some $\ell \in
P_{<a,a'}$. Taking the induced subposet of $a,a',b,\ell$ gives the signed poset shown in
Figure~\ref{fig:constructionpf}.
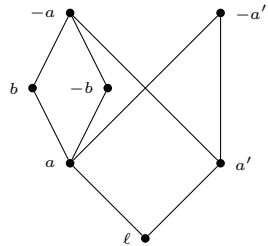
\begin{figure}[htbp]
\begin{center}
    \begin{tikzpicture}
        \node[vertex] (a) at (0,0) [label=left:$a$] {};
		\node[vertex] (b) at (-.5,1) [label=left:$b$] {};
		\node[vertex] (-b) at (.5,1) [label=left:$-b$] {};
		\node[vertex] (-a) at (0,2) [label=left:$-a$] {};
		\node[vertex] (l) at (1,-1) [label=left:$\ell$] {};
		\node[vertex] (a') at (2,0) [label=right:$a'$] {};
		\node[vertex] (-a') at (2,2) [label=right:$-a'$] {};
		\draw (l)--(a)--(b)--(-a)--(-b)--(a);
		\draw (l)--(a')--(-a')--(a);
		\draw (a')--(-a);
    \end{tikzpicture}
    \caption{An induced subposet when $c > a, c\not > a', c < -c$ in case
        (b.ii) of the proof
        of Theorem~\ref{thm:construction}}
    \label{fig:constructionpf}
\end{center}
\end{figure}
However, this poset is not an initial complete intersection, since
$\{b,a,\ell\}$ intersects both $\{-b,a,\ell\}$ and $\{\ell,a'\}$ nontrivially, a contradiction,
courtesy of Lemma~\ref{lem:inducedbad}.
\end{enumerate}
It remains to check that given $x,y$ in two different pieces of the
decomposition, they share the same relation in $P$ and in $Q$. The cases break
down as they did in type A.
\begin{enumerate}[label=(\alph{enumi}.\Roman*)]
\item \textbf{Suppose $x \in P_{<a} \smallsetminus P_{<a'}$ and $y \in
P_{a'}\smallsetminus P_{<a}$.} Then $x$ and $y$ must be incomparable in both $P$
and $Q$.
\item \textbf{Suppose $x \in P_{<a}\smallsetminus P_{<a'}$ (or $P_{<a'}\smallsetminus
P_{<a}$) and $y \in P_{<a,a'}$.} In this case, $x$ and $y$ are incomparable in
$Q$ by construction. If $x<_P y$, then $x \in P_{<a,a'}$, a contradiction. If
$y <_P x$, then the induced subposet of $a,x,y$ is isomorphic to the signed
poset in Figure~\ref{fig:excludedposets}(j),
 contradiction. Thus, one has that $x$ and $y$
must also be incomparable in $P$.
\item \textbf{Suppose $x \in P_{<a}\smallsetminus P_{<a'}$ and $y \in
        \widehat{P}$.}
Neither $y \leq_Q x$ nor $y \leq_P x$ is possible. Observe that $x \leq_Q y$ if
and only if $a \leq_P y$ and $a \leq_P y$ implies $x \leq_P y$. Therefore, one
must show that if $a \not\leq_P y$ then $x \not \leq_P y$. Suppose not, and $a
\not\leq_P y$, but $x \leq_P y$. One knows that there exists $\ell \in
P_{<a,a'}$. Suppose $\ell \not\leq_P y$. If $\ell < \ell'$, then the subposet
of $P$ induced by $a,a',x,y,\ell$ is isomorphic to the signed poset in
Figure~\ref{fig:excludedposets}(n)
a contradiction. If $\ell$ and $-\ell$ are not comparable, the subposet of $P$
induced by $a,y,x$ is isomorphic to the signed poset in
Figure~\ref{fig:excludedposets}(j) or (k), a contradiction in both cases.

Thus, one must have that $\ell \leq_P y$. There are two cases:
\begin{itemize}
\item $y \not \geq_P a'$. Then the subposet of $P$ induced by $a,a',y,\ell$ is
    isomorphic to the signed poset in Figure~\ref{fig:excludedposets}(m), a
    contradiction.

\item $y \geq_P a'$. Then the subposet of $P$ induced by $a,y,\ell$ is
    isomorphic to the signed poset in Figure~\ref{fig:excludedposets}(j), a
    contradiction.
\end{itemize}
\item \textbf{Suppose $x \in P_{<a'}\smallsetminus P_{< a}$ and $y \in
        \widehat{P}$.}
This case is the same as the previous case with the roles of $a$ and $a'$
exchanged.
\item \textbf{Suppose $x \in P_{<a,a'}$ and $y \in \widehat{P}$.} Begin by observing
that $y \leq_Q x$ and $y \leq_P x$ are impossible. One has that $x \leq_Q y$ if
and only if $a \leq_P y$ or $a'\leq_P y$. Both $a\leq_P y$ and $a'\leq_P y$
imply $x \leq_P y$, so one needs to show that if neither $a \not\leq_P y$ nor
$a' \not\leq_P y$, then $x \not\leq_P y$. Suppose not. If $y < -y$, the
subposet induced by $x,y,a,a'$ is isomorphic to the forbidden signed poset in
Figure~\ref{fig:excludedposets}(m).

 If $y$ and $-y$ are not comparable, then the
subposet induced by $x,y,a$ is isomorphic to the signed poset in either
Figure~\ref{fig:excludedposets}(j) or (k), a contradiction.
\end{enumerate}

\item \textbf{Suppose that $a$ \emph{is} minimal and for all $a' \in \Ghat(P)$ not
comparable to $a$ with $I_<(a')$ isotropic, $I_<(a) \cap I_<(a') = \emptyset$.} One
may assume that no $a$ exists that falls into either of the earlier two cases.
Then every $a \in \Ghat(P)$ with $a < -a$ is a minimal element of $\Ghat(P)$.
One can then divide $[n]$ into $a_1,\ldots,a_k$ and $b_1,\ldots,b_j$, where
$a_i < -a_i$ and $b_\ell$ and $-b_\ell$ are not comparable. Since $P$ is an
initial complete intersection, no $a_i$ lies below more than one $b_\ell$. There
are two cases.
\begin{itemize}
\item Suppose $a_1$ and $a_2$ are covered by all the same elements (i.e.\
    $-a_1,\ldots,-a_k$ and possibly $b_\ell$ for some $\ell$). Let $P' =
    P\smallsetminus\{a_1\}$. Then $a_2$ is a hanger in $P'$ and $P$ is obtained
    from $P'$ by duplicating $a_2$.
\item Suppose there is no such $a_1,a_2$. Then, there must be an $a_i$ such
    that $b_\ell > a_i$ and $b_\ell$ covers no other $a_s$. Let $P' =
    P\smallsetminus \{b_\ell\}$. Then $P$ is obtained from $P'$ by a type B
    hanging of $b_\ell$ from $a_i$.
\end{itemize}
\end{enumerate}
\end{proof}

\subsection{Computing Rational Functions}\label{subsec:wtrationalfunctions}

Having characterized the initial complete intersections, attention now turns to
calculating various rational function identities. One knows from
Proposition~\ref{prop:gbinitmingens} that the $U_JU_K$ as $\{J,K\}$ runs over
$\Pi(P)$ form a minimal generating set for $\initI$. When $P$ is an initial
complete intersection, one then has
\begin{equation}\label{eq:hilbseries}
\Hilb(\Rpwt,\bm{x}) = \Hilb(\Swt{P}/\initI,\bm{x})= \dfrac{\prod_{\{J,K\} \in
\Pi(P)} (1-\bm{x}^J\bm{x}^K)}{\prod_{J \in \Jconn(P)} (1-\bm{x}^J)}.
\end{equation}
One can then apply Proposition~\ref{prop:sandsigmaci} to obtain $\Phi_P$.
\begin{cor}
Suppose $P$ is an initial complete intersection. Then
\[
\Phi_P(\bm{x}) = \dfrac{\prod_{\{J,K\} \in \Pi(P)} \langle
\bm{x},\chi_J+\chi_K\rangle}{\prod_{J \in \Jconn(P)} \langle
\bm{x},\chi_J\rangle}.
\]
\end{cor}

\begin{figure}
		\begin{center}
				\begin{tikzpicture}
						\node[vertex] (1) at (0,0) [label=left:$1$] {};
						\node[vertex] (3) at (-2,0) [label=left:$3$] {};
						\node[vertex] (2) at (-1,1) [label=left:$2$] {};
						\node[vertex] (-2) at (1,1) [label=left:$-2$]{};
						\node[vertex] (-1) at (0,2) [label=right:$-1$]{};
						\node[vertex] (-3) at (2,2) [label=right:$-3$]{};
						\draw (1)--(2)--(-1)--(-2)--(1);
                        \draw (3) edge[bend left] (-1);
                        \draw (1) edge[bend right] (-3);
                        \draw (3)--(-3);
				\end{tikzpicture}
		\end{center}
		\caption{A signed poset}
		\label{fig:phiex}
\end{figure}

For example, consider the signed poset in Figure~\ref{fig:phiex}. Then
$\Swt{P}=k[U_1,U_3,U_{12},U_{1\m2}]$ and $\initI = (U_{12}U_{1\m2})$. Then
\[
\Hilb(\Swt{P}/\initI, \bm{x}) =
\dfrac{1-x_1^2}{(1-x_1)(1-x_3)(1-x_1x_2)(1-x_1x_2^{-1})},
\]
and applying Proposition~\ref{prop:sandsigmaci} gives
\[
\Phi_P(\bm{x}) = \dfrac{2x_1}{x_1x_3(x_1+x_2)(x_1-x_2)}.
\]
On the other hand, Table~\ref{tab:linearext} gives the linear extensions of
$P$, their descent sets and their major index.
\begin{table}[htbp]
    \centering
    \begin{tabular}{ccc}
    $w$ & $\Des(w)$ & $\maj(w)$ \\
$\begin{pmatrix} 1& 2& 3 \\ 1 & 2 & 3 \end{pmatrix}$ & $\emptyset$ & $0$
\\[1pc]
$\begin{pmatrix*}[r] 1 & 2& 3 \\ 1 & -2 & 3 \end{pmatrix*}$ & $\{2\}$ & $2$
\\[1pc]
$\begin{pmatrix*}[r] 1 & 2 & 3 \\ 1& 3 & 2 \end{pmatrix*}$ & $\{2\}$ & $2$
\\[1pc]
$\begin{pmatrix*}[r] 1 & 2 & 3 \\ 1 &3 & -2 \end{pmatrix*}$ & $\{3\}$ & $3$
\\[1pc]
$\begin{pmatrix} 1& 2& 3 \\ 3 & 1 & 2 \end{pmatrix}$ & $\{1\}$ & $1$ \\[1pc]
$\begin{pmatrix*}[r] 1 & 2 & 3 \\ 3 & 1 & -2 \end{pmatrix*}$ & $\{1,3\}$ & $4$
    \end{tabular}
    \caption{Linear extensions for the signed poset of Figure~\ref{fig:phiex},
        their descents and major index.}
    \label{tab:linearext}
\end{table}
Then
\begin{multline*}
		\Phi_P(\bm{x}) = \sum_{w \in \cL(P)}
		w\left(\dfrac{1}{x_1(x_1+x_2)\cdots(x_1+\cdots+x_n)}\right) \\
		= \dfrac{1}{x_1(x_1+x_2)(x_1+x_2+x_3)} +
		\dfrac{1}{x_1(x_1-x_2)(x_1-x_2+x_3)} +
		\dfrac{1}{x_1(x_1+x_3)(x_1+x_2+x_3)} \\
		+ \dfrac{1}{x_1(x_1+x_3)(x_1+x_3-x_2)} +
		\dfrac{1}{x_3(x_1+x_3)(x_1+x_2+x_3)} +
		\dfrac{1}{x_3(x_1+x_3)(x_1-x_2+x_3)} \\
		= \dfrac{2}{(x_1-x_2)(x_1+x_2)x_3} =
		\dfrac{2x_1}{x_1x_3(x_1-x_2)(x_1+x_2)}.
\end{multline*}

Altering the grading so $\deg U_J = |J|$ gives the following identity (c.f.\
Proposition~\ref{prop:qmajppartitions}).

\begin{cor}\label{cor:qsum}
Suppose the signed poset $P$ is an initial complete intersection. Then
\[
\sum_{w \in \cL(P)} q^{\maj(w)} = [n]!_q \dfrac{\prod_{\{J,K\} \in \Pi(P)}
    [|J|+|K|]_q}{\prod_{J \in \Jconn(P)} [|J|]_q}.
\]
\end{cor}

\begin{proof}
Collapsing the $\bZ^n$-grading to the $\bN$-grading where $\deg U_J = |J|$,
\eqref{eq:hilbseries} transforms into
\[
    \Hilb(\Swt{P}/\initI,q) = \sum_{f \in \cA(P)} q^{|f|} = \dfrac{\prod_{\{J,K\} \in \Pi(P)} 1-q^{|J|+|K|}}{\prod_{J \in\Jconn(P)}
        1-q^{|J|}}.
\]
Using Proposition~\ref{prop:qmajppartitions} one then has
\[
    \sum_{w \in \cL(P)} q^{\maj(w)} = [n]!_q \dfrac{\prod_{\{J,K\} \in \Pi(P)} 1-q^{|J|+|K|}}{\prod_{J \in\Jconn(P)}
        1-q^{|J|}} = [n]!_q\dfrac{\prod_{\{J,K\} \in \Pi(P)}
        [|J|+|K|]_q}{\prod_{J \in \Jconn(P)} [|J|]_q},
\]
as claimed.
\end{proof}

Looking at the signed poset from Figure~\ref{fig:phiex} once again, one has
\[
    \sum_{w \in \cL(P)} q^{\maj(w)} = [3]!_q
    \dfrac{[2+2]_q}{[1]_q[1]_q[2]_q[2]_q} = \dfrac{[4]_q[3]_q}{[2]_q} = 1 + q +
    2q^2 + q^3 + q^4,
\]
matching the previous tabulation (see Table~\ref{tab:linearext}).

Lastly, one can alter the grading of $\Swt{P}/\initI$ a third time, taking $\deg U_J = 1$ to obtain
the following.

\begin{cor}
Suppose $P$ is a signed poset which is an initial complete intersection. Then
\[
\sum_{f \in \cA(P)} t^{\nu(f)} = \dfrac{(1-t^2)^{|\Pi(P)|}}{(1-t)^{|\Jconn(P)|}},
\]
where $\nu(f)$ is the number of ideals in the unique expression of $f$ as a sum
of nontrivially intersecting connected ideals.
\end{cor}

\chapter{Unfinished Business}\label{ch:unfinished}

\section{Two triangulations of the weight cone}

As was the case for \feray and Reiner in type A, the ideal $\initI$ suggests
a triangulation of the weight cone $\Kpwt$. They explained (see~\cite[\S
11]{FerayReiner2012})  
\begin{itemize}
    \item that $\Kpwt$ is a maximal cone in the normal fan of the
graphic zonotope associated to the graph underlying the Hasse diagram,
    \item that the normal fan of the graphic zonotope is refined by the normal fan of the
graph associahedron of the same graph and,
    \item lastly, that $\initI$ is the
Stanley-Reisner ideal for the simplicial complex describing the triangulation
of the weight cone by the the normal fan of the graph associahedron.
\end{itemize}

This section explains these two triangulations of the weight cone, the first
indexed by linear extensions, the second by sets of pairwise trivially
intersecting connected ideals, explains
Zaslavsky's signed graph analogue of the graphic zonotope and how it gives a
triangulation of $\Kpwt$, analogous to type A, but leaves open the problem of
finding the correct definition of signed graph associahedron.

\begin{defn}
A \emph{triangulation} of the cone $K$ is a collection
$T=\{\sigma_1,\ldots,\sigma_k\}$ of simplicial cones such that
\begin{itemize}
    \item $\bigcup \sigma_i = K$;
    \item if $\sigma \in T$, then every face of $\sigma$ is in $T$;
    \item for any $\sigma_i,\sigma_j \in T$,$\sigma_i \cap \sigma_j$ is a common face of $\sigma_i$ and
        $\sigma_j$
\end{itemize}
\end{defn}

The first triangulation of $\Kpwt$ to be discussed is one suggested by
Proposition~\ref{prop:fundthmppartitions}, that where the maximal cones are
unions of the
$w\Phi^+$-partitions for $w \in \cL(P)$, called the \emph{$P$-partition
    triangulation}. We next describe a second triangulation called the
\emph{trivially intersecting ideals triangulation}. 

\begin{prop}\label{prop:nontrivtriangulation}
    Suppose $P\subset \Bn$ is a signed poset. The weight cone $\Kpwt$ is
    triangulated by cones $\{\sigma_A \colon \mathop{\mathrm{span}} A\}$, where $A$ runs over all
    sets of connected ideals of $P$, the elements of which pairwise intersect
    trivially.
\end{prop}

Proposition~\ref{prop:nontrivtriangulation} is an immediate consequence
of~\citet[Theorem 8.3]{Sturmfels1996a}. Moreover, as mentioned in
Section~\ref{subsec:characterisinginitci}, since $\initI$ is
square-free, it is the Stanley-Reisner ideal of the complex on $\Jconn(P)$
whose faces are given by sets of ideals which pairwise intersect trivially.

Figure~\ref{fig:bigtriangulationexample} shows the two different
triangulations: the $P$-partition triangulation in
Figure~\ref{fig:linexttriangulation}, where the maximal cones are indexed by the
linear extensions of $P$, and the non-intersecting ideals triangulation of
Proposition~\ref{prop:nontrivtriangulation} in
Figure~\ref{fig:idealtriangulation}, where the maximal cones are indexed by
signed posets (more on these signed posets later).

\begin{figure}
    \begin{center}
\tdplotsetmaincoords{70}{100}
\begin{subfigure}[b]{.45\textwidth}
    \begin{center}
        \begin{tikzpicture}[scale=2]
            \node[circle,draw] (1) at (0,0) {$1$};
            \node[circle,draw] (2) at (1,0) {$2$};
            \node[circle,draw] (3) at (.5,.5) {$3$};
            \draw (1.10)--node[very near start,above] {$+$} node[very near end,
            above] {$-$} (2.170);
            \draw (1.350)--node[very near start,below] {$+$} node[very near
            end, below] {$+$} (2.190);
            \draw (3) edge[loop right] node[midway,right] {$+$} (3);
        \end{tikzpicture}
        \caption{Hasse diagram}
        \label{fig:trianghasse}
    \end{center}
\end{subfigure}
\begin{subfigure}[b]{.45\textwidth}
    \begin{center}
        \begin{tikzpicture}[tdplot_main_coords]
            \node[vertex] (m201) at (-2,0,1) [label=right:$\m201$] {};
            \node[vertex] (m20m1) at (-2,0,-1)[label=right:$\m20\m1$] {};
            \node[vertex] (0m2m1) at (0,-2,-1)[label=left:$0\m2\m1$] {};
            \node[vertex] (0m21) at (0,-2,1) [label=left:$0\m21$] {};
            \node[vertex] (02m1) at (0,2,-1) [label=right:$02\m1$] {};
            \node[vertex] (021) at (0,2,1) [label=right:$021$] {};
            \node[vertex] (20m1) at (2,0,-1) [label=left:$20\m1$] {};
            \node[vertex] (201) at (2,0,1) [label=left:$201$] {};
            \draw (m201)--(021)--(02m1)--(20m1)--(0m2m1)--(0m21)--(201);
            \draw[dashed] (0m2m1)--(m20m1)--(m201);
            \draw (201)--(021);
            \draw (0m21)--(m201);
            \draw (20m1)--(201);
            \draw[dashed] (m20m1)--(02m1);
        \end{tikzpicture}
    \end{center}
    \caption{The acyclotope $\cZ[\Sigma_P]$}
    \label{fig:acyclotope}
\end{subfigure}
\vspace{1 pc}\\
\begin{subfigure}[b]{.45\textwidth}
    \begin{center}
        \includegraphics[width=\textwidth]{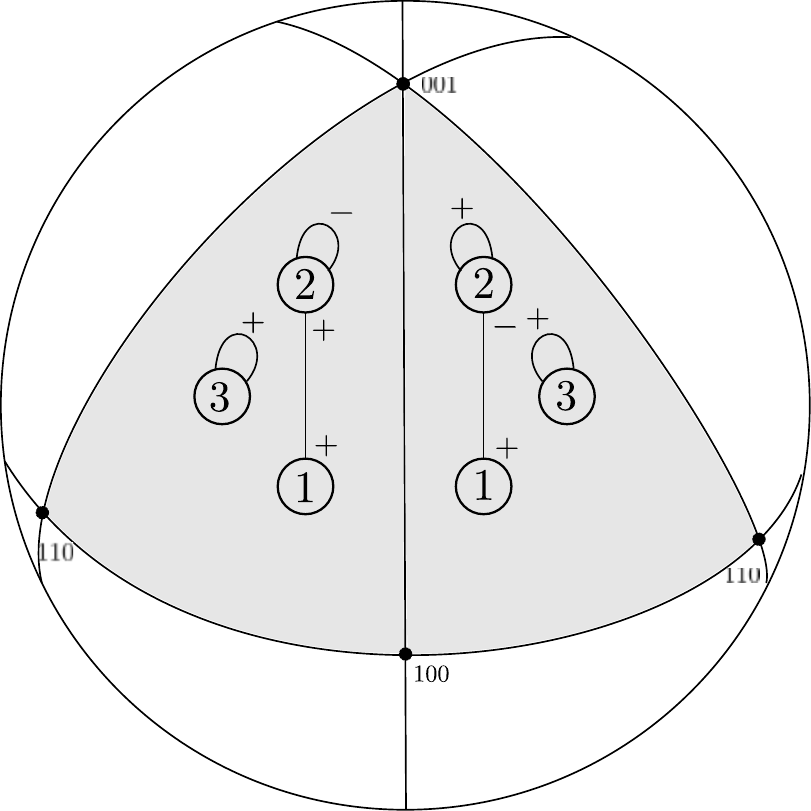}
\end{center}
\caption{$\Kpwt$ triangulated by sets of pairwise trivially intersecting ideals}
\label{fig:idealtriangulation}
\end{subfigure}
\begin{subfigure}[b]{.45\textwidth}
    \begin{center}
        \includegraphics[width=\textwidth]{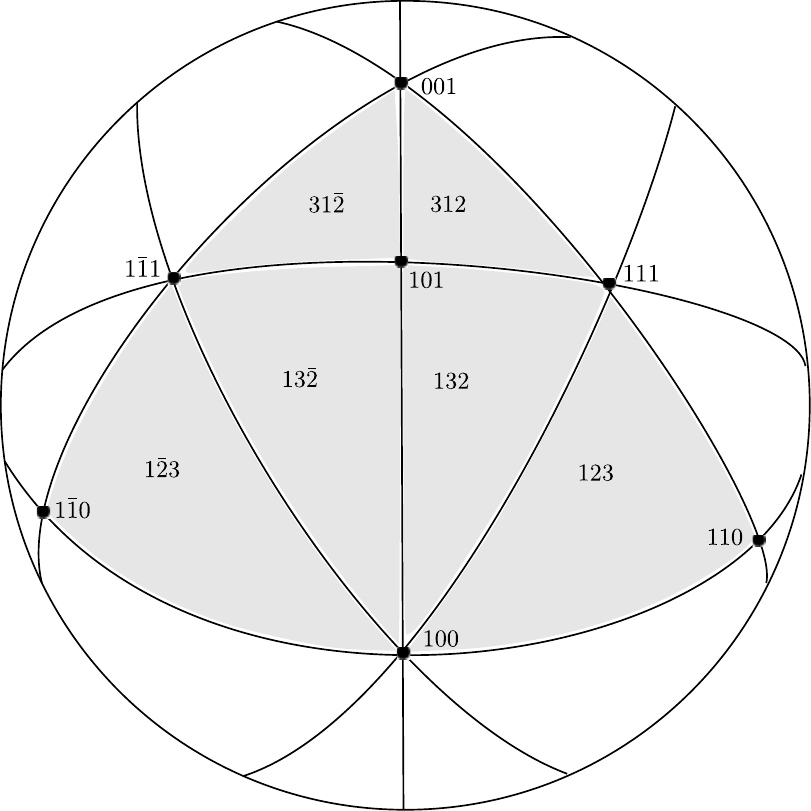}
\end{center}
\caption{$\Kpwt$ triangulated by $\cN(\cZ[\pm K_n^B])$ \\ \ }
\label{fig:linexttriangulation}
\end{subfigure}

\caption{The two different triangulations of $\Kpwt$}
\label{fig:bigtriangulationexample}
\end{center}
\end{figure}

In type A, the analogous triangulation is explained by \feray and Reiner in terms of
the graphic zonotope.

\begin{defn}\label{def:zonotope}
A \emph{zonotope} is a polytope that is the Minkowski sum of line segments. In
particular, if $G$ is a graph, then
\[
    \cZ[G] = \sum_{e=(i,j) \in G} [-(e_i-e_j),e_i-e_j],
\]
is called the \emph{graphic zonotope}.
\end{defn}

\feray and Reiner note that the maximal cones of the normal fan $\cN(\cZ[G])$
are indexed by acyclic orientations of $G$. In particular, if $P$ is a poset and $G$
is the graph underlying its Hasse diagram, there is an orientation, call it
$\omega$, of $G$ such that the corresponding maximal cone of the normal fan,
$\cN_\omega$, is, in fact, $\Kpwt$.

\begin{defn}\label{def:graphicalbuildingset}
Let $G$ be a graph. Its \emph{graphical building set} $\cB(G)$ is the
collection of sets of vertices $J$ where the vertex-induced subgraph $G|_J$ is
connected.
\end{defn}

The graphical building set is used to define the graph associahedron
of~\citet*{CarrDevadoss2006}.

\begin{defn}\label{def:graphassociahedron}
    Suppose $G$ is a graph. Its \emph{graph associahedron} is
    \[
        \cP_G = \sum_{J \in \cB(G)} \conv\{e_j \colon j \in J\}.
    \]
\end{defn}

\feray and Reiner show the following.
\begin{prop}[{\cite[Proposition 11.7]{FerayReiner2012}}]
Suppose $P$ is a poset on $[n]$, $G$ is the graph underlying its Hasse diagram
and $w$ is the orientation of $G$ giving the Hasse diagram of $P$. Then the
simplicial complex $\Delta_P$ having $\initI$ as its Stanley-Reisner ideal
describes the triangulation of the cone $\cN_w$ in the fan $\cN(\cZ[G])$ by
cones of the normal fan $\cN(\cP_G)$.
\end{prop}

The maximal cones of this triangulation are indexed by $\cB(G)$-forests,
forests $F$ in which every principal ideal $F_{\leq i}$ is a connected ideal of
$P$ and, whenever $i$ and $j$ are incomparable in $F$, then the ideal $F_{\leq i}$ and
$F_{\leq j}$ of $P$ is disconnected.

One would hope for a similar understanding of the triangulation of
Proposition~\ref{prop:nontrivtriangulation}. In the signed poset case, one can
avail oneself of the analogue of the graphic zonotope---the acyclotope of
Zaslavsky.

\begin{defn}\label{def:acyclotope}
Suppose $\Sigma$ is a signed graph and $\tau$ is a bidirection orienting
$\Sigma$. Then \emph{acyclotope} of $\Sigma$, written $\cZ[\Sigma]$ is the
polytope
\[
    \cZ[\Sigma] = \sum_{e \in E} [-x_\tau(e),x_\tau(e)],
\]
where $x_\tau(e)$ is the column vector associated to $e$ in the incidence
matrix of $\Sigma$.
\end{defn}

Figure~\ref{fig:acyclotope} gives an example of the acyclotope of the signed
poset in Figure~\ref{fig:trianghasse}.

\begin{prop}\label{prop:braidarrangement}
The hyperplane arrangement $H(\Sigma)$ whose fan corresponds to the normal fan of the graphic
zonotope is:
\[
    \begin{array}{ll}
x_i = \sigma(e) x_j & \text{for an edge $e=(i,j)$} \\
x_i = 0 & \text{for a loop or half edge $e=(i,i)$, $e=(i,-)$}
    \end{array}
\]
The regions of $H(\Sigma)$ correspond to various orientations of $\Sigma$.
\end{prop}

\begin{defn}
A \emph{cycle} of an oriented signed graph is a matroid circuit such that there
is no vertex $v$ such that all $\tau(v,e)$ coincide as $e$ runs over the edges
of the circuit incident to $v$. An orientation that contains no cycles will be
said to be \emph{acyclic}.
\end{defn}

Zaslavsky generalized a result of Greene to the signed graph/acyclotope case in
the following.

\begin{thm}[{\citet[Theorem 4.4]{Zaslavsky1991}}]
Suppose $\Sigma$ is a signed graph. There is a one-to-one correspondence
between the regions of $H[\Sigma]$ and acyclic orientations of $\Sigma$.
\end{thm}

As an example, consider the signed poset, $P$, shown in
Figure~\ref{fig:zonotopeposet}. Figure~\ref{fig:sigmap} shows $\Sigma_P$, the
signed graph underlying the Hasse diagram of $P$
(Figure~\ref{fig:zonotopehasse}). As $\Sigma_P$ has a single unbalanced cycle,
so every orientation will be acyclic. 
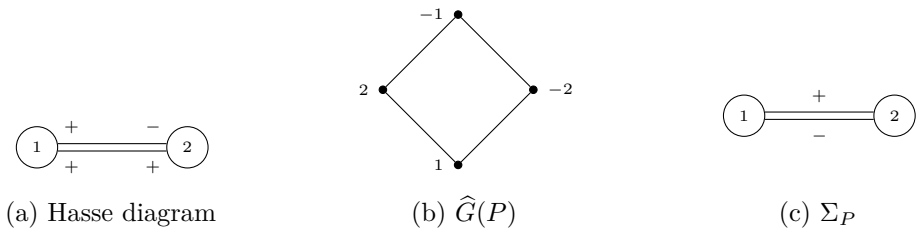
\begin{figure}[htbp]
    \begin{center}
    \begin{subfigure}[b]{.30\textwidth}
        \begin{center}
            \begin{tikzpicture}[scale=2]
                \node[circle,draw] (1) at (0,0) {$1$};
                \node[circle,draw] (2) at (1,0) {$2$};
                \draw (1.10)-- node[very near start, above] {$+$} node[very
                near end, above]{$-$}(2.170);
                \draw (1.350)--node[very near start,below] {$+$} node[very near
                end,below] {$+$}(2.190);
            \end{tikzpicture}
            \caption{Hasse diagram}
            \label{fig:zonotopehasse}
        \end{center}
    \end{subfigure}
    \begin{subfigure}[b]{.30\textwidth}
        \begin{center}
            \begin{tikzpicture}
                \node[vertex] (1) at (0,0) [label=left:$1$] {};
                \node[vertex] (2) at (-1,1) [label=left:$2$] {};
                \node[vertex] (-2) at (1,1) [label=right:$-2$] {};
                \node[vertex] (-1) at (0,2) [label=left:$-1$] {};
                \draw (1)--(2)--(-1)--(-2)--(1);
            \end{tikzpicture}
            \caption{$\Ghat(P)$}
        \end{center}
    \end{subfigure}
    \begin{subfigure}[b]{.30\textwidth}
\begin{center}
            \begin{tikzpicture}[scale=2]
                \node[circle,draw] (1) at (0,0) {$1$};
                \node[circle,draw] (2) at (1,0) {$2$};
                \draw (1.10)-- node[midway,above]{$+$}(2.170);
                \draw (1.350)--node[midway,below] {$-$} (2.190);
            \end{tikzpicture}
        \end{center}
        \caption{$\Sigma_P$}
        \label{fig:sigmap}
    \end{subfigure}
\end{center}
\caption{$P = \{+e_1-e_2,+e_1+e_2,+e_1\}$}
\label{fig:zonotopeposet}
\end{figure}
\begin{figure}[htbp]
    \begin{center}
        \begin{subfigure}[b]{.45\textwidth}
            \begin{center}
            \begin{tikzpicture}
                        \filldraw[draw=black,fill=gray!60](0,0)--(1,-1)--(2,0)--(1,1)--cycle;
                            \node[vertex] at (0,0) [label=left:${(0,0)}$] {};
            \node[vertex] (1m1) at (1,-1) [label=below:${(1,-1})$] {};
            \node[vertex] (10) at (2,0) [label=right:${(2,0)}$] {};
            \node[vertex] (11) at (1,1) [label=above:${(1,1)}$] {};
            \end{tikzpicture}
        \end{center}
        \caption{the zonotope $\cZ[\Sigma]$}
        \end{subfigure}
        \begin{subfigure}[b]{.45\textwidth}
            \begin{center}
                \begin{tikzpicture}
                    \draw[gray!60] (-2,0)--(2,0);
                    \draw[gray!60] (0,-2)--(0,2);
                    \draw (-2,-2)--(2,2);
                    \draw (-2,2)--(2,-2);
                    \node[anchor=south west] at (2,2) {$x_1=x_2$}; 
                    \node[anchor=north west] at (2,-2) {$x_1=-x_2$};
                    \node[circle,draw] (1a) at (-1,1.75) {$1$};
                    \node[circle,draw] (2a) at (1,1.75) {$2$};
                    \draw (1a.10)--node[very near start,above] {$-$} node[very
                    near end,above] {$+$} (2a.170);
                    \draw (1a.350)--node[very near start,below] {$+$} node[very
                    near end,below] {$+$} (2a.190);
                    \node[circle,draw] (1b) at (-1,-1.75) {$1$};
                    \node[circle,draw] (2b) at (1,-1.75) {$2$};
                    \draw (1b.10)--node[very near start,above] {$+$} node[very
                    near end,above] {$+$} (2b.170);
                    \draw (1b.350)--node[very near start,below] {$-$} node[very
                    near end,below] {$-$} (2b.190);
                    \node[circle,draw] (1c) at (-1.75,1) {$1$};
                    \node[circle,draw] (2c) at (-1.75,-1) {$2$};
                    \draw (1c.260)--node[very near start,left] {$+$} node[very
                    near end,left]{$+$} (2c.100);
                    \draw (1c.280)--node[very near start,right] {$-$} node[very
                    near end,right] {$-$} (2c.80);
                    \node[circle,draw] (1d) at (1.75,1) {$1$};
                    \node[circle,draw] (2d) at (1.75,-1) {$2$};
                    \draw (1d.260)--node[very near start,left] {$+$} node[very
                    near end,left]{$-$} (2d.100);
                    \draw (1d.280)--node[very near start,right] {$+$} node[very
                    near end,right] {$+$} (2d.80);
                \end{tikzpicture}
                \caption{$\cN(\cZ[\Sigma])$ with maximal cones labeled by
                    acyclic orientations of $\Sigma$}
            \end{center}
        \end{subfigure}
    \end{center}
    \caption{The zonotop $\cZ[\Sigma]$ and its Newton polytope for $\Sigma$
        from Figure~\ref{fig:zonotopeposet}}
    \label{fig:zonotope}
\end{figure}
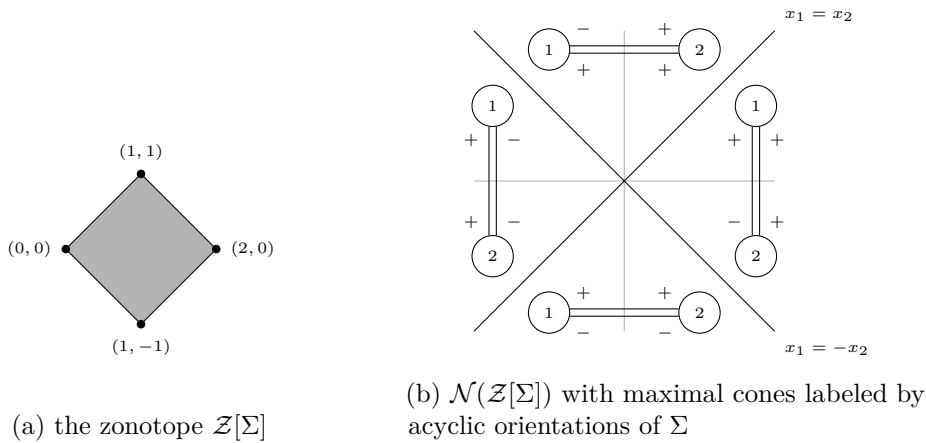
Figure~\ref{fig:zonotope} shows the
acyclotope $\cZ[\Sigma_P]$ and its normal fan $\cN(\cZ[\Sigma_P])$. Recall from
Proposition~\ref{prop:fundthmppartitions} that the $P$-partitions of a signed poset are the
disjoint union of the $w\Phi^+$-partitions for each $w \in \cL(P)$. This
corresponds to the triangulation of $\cN(\cZ[\Sigma_P])$ by the normal fan of
the acyclotope of the complete graph, as in type A. There are a number of
possible choices for a complete signed graph, but taking the lead from the type
A braid arrangement, there is a clear choice.

\begin{defn}
Let $\pm K_n^B$ be the signed graph whose vertices are $[n]$ and whose edges are:
\begin{itemize}
    \item an edge $\{i,j\}$ signed $+$ for all pairs $i,j \in [n]$,
    \item an edge $\{i,j\}$ signed $-$ for all pairs $i,j \in [n]$
    \item a half edge at $i$ signed $+$ for each $i \in [n]$.
\end{itemize}
Call $\pm K_n^B$ a \emph{complete signed graph}.
\end{defn}
One could opt to replace the half edges with loops, but in that case an
orientation of a complete signed graph would, strictly speaking, correspond to $\Cn$ roots rather
than $\Bn$, though it would not alter the normal fan of the zonotope.
Figure~\ref{fig:k2} shows $\pm K_2^B$, $\cZ(\pm K_2^B)$ and
$\cN(\cZ(\pm(K_2^B))$.
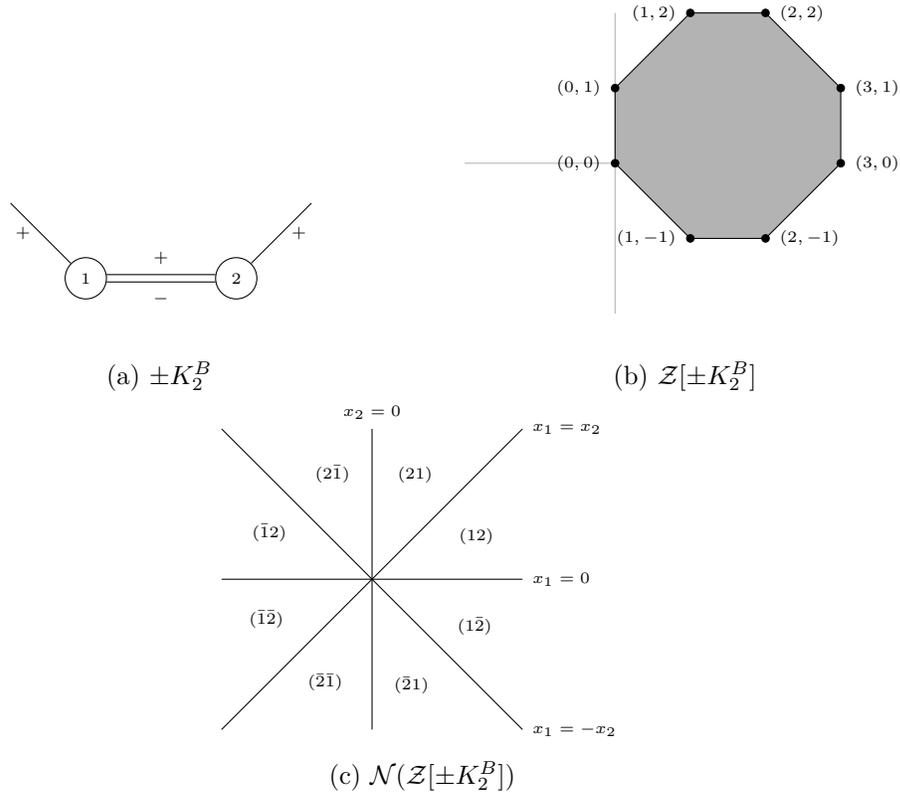
\begin{figure}
    \begin{center}
        \begin{subfigure}[b]{.45\textwidth}
            \begin{center}
                \begin{tikzpicture}
                    \node[circle,draw] (1) at (0,0){$1$};
                    \node[circle,draw] (2) at (2,0){$2$};
                    \draw (1.10)--node[midway,above] {$+$}(2.170);
                    \draw (1.350) --node[midway, below] {$-$} (2.190);
                    \draw (1.135)--node[midway,left]{$+$} (-1,1);
                    \draw (2.45)--node[midway,right]{$+$} (3,1);
                \end{tikzpicture}
            \end{center}
            \caption{$\pm K_2^B$}
        \end{subfigure}
        \begin{subfigure}[b]{.45\textwidth}
            \begin{center}
                \begin{tikzpicture}
                    \draw[gray!60] (-2,0)--(2,0);
                    \draw[gray!60] (0,-2)--(0,2);
                    \filldraw[draw=black,fill=gray!60](0,0)
                    --(0,1)--(1,2)--(2,2)--(3,1)--(3,0)--(2,-1)--(1,-1)--cycle;
                    \node[vertex] at (0,0) [label=left:${(0,0)}$] {};
                    \node[vertex] at (0,1)[label=left:${(0,1)}$] {};
                    \node[vertex] at (1,2) [label=left:${(1,2)}$] {};
                    \node[vertex] at (2,2) [label=right:${(2,2)}$] {};
                    \node[vertex] at (3,1) [label=right:${(3,1)}$] {};
                    \node[vertex] at (3,0) [label=right:${(3,0)}$] {};
                    \node[vertex] at (2,-1) [label=right:${(2,-1)}$] {};
                    \node[vertex] at (1,-1) [label=left:${(1,-1)}$] {};
                \end{tikzpicture}
            \end{center}
            \caption{$\cZ[\pm K_2^B]$}
        \end{subfigure} \\
        \begin{subfigure}[b]{\textwidth}
            \begin{center}
                \begin{tikzpicture}
                    \draw (2,0)--(-2,0);
                    \draw (0,2)--(0,-2);
                    \draw (2,2)--(-2,-2);
                    \draw (2,-2)--(-2,2);
                    \node[anchor=west] at (2,0) {$x_1=0$};
                    \node[anchor=south] at (0,2) {$x_2=0$};
                    \node[anchor=west] at (2,2) {$x_1=x_2$};
                    \node[anchor=west] at (2,-2) {$x_1=-x_2$};
                    \draw (22.5:1.5cm) node{$(12)$};
                    \draw (67.5:1.5cm) node{$(21)$};
                    \draw (110.5:1.5cm) node{$(2\m1)$};
                    \draw (155.5:1.5cm) node{$(\m12)$};
                    \draw (200.5:1.5cm) node{$(\m1\m2)$};
                    \draw (245.5:1.5cm) node{$(\m2\m1)$};
                    \draw (290.5:1.5cm) node{$(\m2 1)$};
                    \draw (335.5:1.5cm) node{$(1\m2)$};
                \end{tikzpicture}
                \caption{$\cN(\cZ[\pm K_2^B])$}
            \end{center}
        \end{subfigure}
    \end{center}
    \caption{The complete signed graph $\pm K_2^B$, its acyclotope and the
        normal fan of the acyclotope}
    \label{fig:k2}
\end{figure}

One sees that the example $\Kpwt$ is triangulated by the cones spanned by
$(12)\Phi_2^+$ and $(21)\Phi_2^+$, which are, in fact, the linear extensions of
$P$.

\begin{prop}
Suppose $P \subset \Bn$ is a signed poset with $\Sigma$ the signed graph
underlying its Hasse diagram. Then $\cN(\cZ[\Sigma])$ is refined by
$\cN(\cZ[\pm K_n^B])$ and, if $\tau$ is the orientation of $\Sigma$
corresponding to $P$, then $\cN_\tau$ is triangulated by the cones of
$\cN(\cZ[\pm K_n^B])$ corresponding to the linear extensions of $P$.
\end{prop}

\begin{proof}
First, it is straightforward to see that $\cN(\cZ[\Sigma])$ is refined by
$\cN(\cZ[\pm K_n^B])$. The hyperplanes defining $\cN(\cZ[\Sigma])$ are a subset
of the hyperplanes defining $\cN(\cZ[\pm K_n^B])$.

Now, suppose $\tau$ is the orientation of $\Sigma$ corresponding to $P$. By
construction, the maximal cones of $\cN(\cZ[\pm K_n^B])$ are the
$w\Phi^+$-partition cones for the elements of the Weyl group, i.e.\ the signed
permutations. Thus, if the maximal cone corresponding to $w$ lies in
$\cN_\tau$, by definition, $w$ will be a linear extension of $P$.
\end{proof}

As was noted above, one can index the maximal cones of the triangulation from
Proposition~\ref{prop:nontrivtriangulation} by certain signed posets, namely
those whose weight cones are the maximal cones of the triangulation. By
construction, these signed posets have simplicial and unimodular weight cones.
They should be the signed analogue of type A's $\cB(G)$-forests.

\begin{quest}\label{finalquest}
What is the appropriated analogue of Carr and Devadoss's graph associahedron
for signed graphs? Does its normal fan give the triangulation of $\Kpwt$ given
in Proposition~\ref{prop:nontrivtriangulation}?
\end{quest}

At the moment, it appears an appropriate signed graph associahedron can be
obtained by shaving the $n$-cube at faces corresponding to ``signed tubes'' of
the signed graph, with a proof proceeding as that of Carr and Devadoss.
However, all the details have yet to be written down.


\section{The type C weight cone}

Thus far, consideration of the weight cone, ideals and $P$-partitions has been
restricted to those signed posets $P \subset \Bn$. However, though one can read
the ideals from either $\Gbhat(P)$ or $\Gchat(\Pc)$, understanding one does not immediately lead to understanding the other, as the
next two examples illustrate.

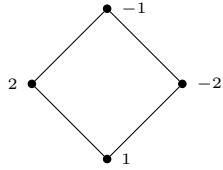
\begin{figure}[htbp]
    \begin{center}
        \begin{tikzpicture}
            \node[vertex] (1) at (0,0) [label=right:$1$] {};
            \node[vertex] (2) at (-1,1) [label=left:$2$] {};
            \node[vertex] (-2) at (1,1) [label=right:$-2$] {};
            \node[vertex] (-1) at (0,2) [label=right:$-1$] {};
            \draw (1)--(2)--(-1)--(-2)--(1);
        \end{tikzpicture}
    \end{center}
    \caption{$P=\{e_1-e_2,e_1+e_2,e_1\}$}
    \label{fig:bvscex1}
\end{figure}

Consider the signed posets $P = \{e_1-e_2,e_1+e_2,e_1\}$ and $\Pc
=\{e_1-e_2,e_1+e_2,2e_2\}$. Figure~\ref{fig:bvscex1} illustrates $\Ghat(P)$.
The connected ideals of $P$ and $\Pc$ are shown in Table~\ref{tab:idealbc}.
Recall that ideals live in the coweight lattice, so the maximal ideals of $P$
and $\Pc$ do not coincide.
\begin{table}[htbp]
    \begin{center}
    \begin{tabular}{ccc}
ideal in $\Ghat(P)$ & ideal in $P$ & ideal in $\Pc$ \\
$\{1\}$ & $(1,0)$ & $(1,0)$ \\
$\{1,2\}$ & $(1,1)$ & $(\frac{1}{2},\frac{1}{2})$ \\
$\{1,-2\}$ & $(1,-1) $ & $(\frac{1}{2},-\frac{1}{2})$
    \end{tabular}
\end{center}
\caption{Type $B$ and $C$ ideals for $\Ghat(P)$ from Figure~\ref{fig:bvscex1}}
\label{tab:idealbc}
\end{table}
Inspecting the ideals reveals that all three connected ideals are required to
generate the semigroup $\Kpwt \cap \Lcwt_B$. However, only
$(\frac{1}{2},\frac{1}{2})$ and $(\frac{1}{2},-\frac{1}{2})$ are required to
generate $\Kpwt \cap \Lcwt_C$. Consequently, one sees $\Kpwt$ is unimodular
with respect to $\Lcwt_C$ but \emph{not} with respect to $\Lcwt_B$. In both
cases, one has the triangulation into cones defined by collections of pairwise
trivially intersecting ideals, even though $\Kpwt$ is already simplicial and
unimodular when viewed in type C.

Alone this is not enough to conclude that the arguments from
Chapter~\ref{ch:weightcone} do not go through almost immediately in type C.
Consider the signed poset in Figure~\ref{fig:bvscex2}.
\begin{figure}[htbp]
    \begin{center}
        \begin{tikzpicture}
            \node[vertex] (1) at (0,0) [label=left:$1$] {};
            \node[vertex] (3) at (-2,0) [label=left:$3$] {};
            \node[vertex] (2) at (-1,1) [label=left:$2$] {};
            \node[vertex] (-2) at (1,1) [label=right:$-2$] {};
            \node[vertex] (-1) at (0,2) [label=left:$-1$] {};
            \node[vertex] (-3) at (2,2) [label=right:$-3$] {};
            \draw (3)--(2)--(-1)--(-2)--(1)--(2);
            \draw (-2)--(-3);
        \end{tikzpicture}
        \caption{$\Ghat(\Pc)$ for $\Pc =
            \{e_1-e_2,e_1+e_2,e_1+e_3,e_3-e_2,2e_1\}$}
        \label{fig:bvscex2}
    \end{center}
\end{figure}
The connected ideals are $\{1\},\{3\},\{1,-2\},\{1,2,3\},\{1,-2,-3\}$. The
toric ideal is the kernel of $\phi \colon \Swt{\Pc} \to k[x_1^{\pm 1},x_2^{\pm
    1},x_3^{\pm 1}]$ defined by
\begin{align*}
    \phi(U_1) &= x_1^2 \\
    \phi(U_3) &= x_3^2 \\
    \phi(U_{1\m2}) &= x_1^2x_2^{-2} \\
    \phi(U_{123}) &= x_1x_2x_3 \\
    \phi(U_{1\m2\m3}) &= x_1x_2^{-1}x_3^{-1}.
\end{align*}
Using Macaulay2, one sees that
\[
    \ker \phi=(U_{1\m2}U_{123}^2 - U_1^2U_3, U_{123}U_{1\m2\m3}-U_1),
\]
meaning that while the generators of the toric ideal are still indexed by pairs of
nontrivially intersecting connected ideals, one needs to tweak the definition
of $\syz(U_J,U_K)$ somewhat. Proposition~\ref{prop:gbinitmingens} used that the
leading term of $\syz(U_J,U_K)$ in type B is quadratic, which is clearly not
the case here, though it is likely that a similar result holds and many of the
results of Chapter~\ref{ch:weightcone} could be pushed through into type C.

\section{On characterizing the $\Rprt$ complete
    intersections}\label{subsec:rtconecharacterisingci}
In this section, attention returns to the root cone and its semigroup (refer to
Chapter~\ref{ch:rootcone} for previous discussion).
\citet*[Theorem 7.7]{Boussicault2009} had shown that $\Psi_P$ factored for posets they called
``gluings of diamonds along chains'', of which strongly planar posets were a
subset.  A poset was said to be a gluing of diamonds along chains if it could
be decomposed into a collection of diamonds by means of disconnecting chains. A
diamond was a cycle with a unique maximum and minimum. A disconnecting chain is
a chain in the poset that partitions the vertices into three groups: the chain
itself, and two other sets such that the paths between the two sets much pass
through a vertex of the chain.
For example, consider the poset in Figure~\ref{fig:morrisex}.
\begin{figure}[htbp]
		\begin{center}
		\begin{tikzpicture}
				\foreach \x in {0,1,2,3,4,5}
				{
				\pgfmathtruncatemacro{\label}{\x+1}
				\node[vertex] (\label) at (0,\x) [label=right:$\label$] {};
				}
				\node[vertex] (8) at (1,0) [label=right:$7$] {};
				\node[vertex] (9) at (1,1) [label=right:$8$] {};
				\node[vertex] (10) at (-1,1.5) [label=left:$9$] {};
				\node[vertex] (11) at (-1,2.5) [label=left:$10$] {};
				\node[vertex] (12) at (1,3.5) [label=right:$11$] {};
				\draw (1)--(6);
				\draw (1)--(9)--(8)--(2);
				\draw (1)--(10)--(4);
				\draw (2)--(11)--(5);
				\draw (3)--(12)--(6);
		\end{tikzpicture}
\end{center}
\caption{A poset which has $\Rprt$ a complete
intersection but is not gluing of diamonds along chains}
\label{fig:morrisex}
\end{figure}
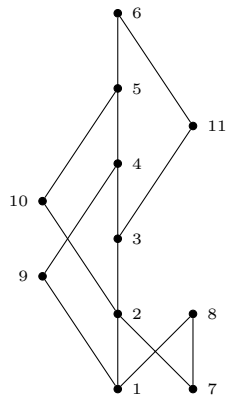
Certainly, it can be broken apart via disconnecting chains in the posets shown
in Figure~\ref{fig:brokenup}.
\begin{figure}[htbp]
    \begin{center}
        \begin{tabular}{cc}
            \begin{subfigure}[b]{.4\textwidth}
                \begin{center}
            \begin{tikzpicture}
                \node[vertex] (2) at (0,0) [label=right:$2$] {};
                \node[vertex] (3) at (0,1) [label=right:$3$] {};
                \node[vertex] (4) at (0,2) [label=right:$4$] {};
                \node[vertex] (5) at (0,3) [label=right:$5$] {};
                \node[vertex] (10) at (-1,1.5) [label=left:$10$] {};
                \draw (2)--(5)--(10)--(2);            
            \end{tikzpicture}
            \end{center}\caption{}
        \end{subfigure}
                & 
                \begin{subfigure}[b]{.4\textwidth}
                    \begin{center}
            \begin{tikzpicture}
                \node[vertex] (2) at (0,0) [label=right:$3$] {};
                \node[vertex] (3) at (0,1) [label=right:$4$] {};
                \node[vertex] (4) at (0,2) [label=right:$5$] {};
                \node[vertex] (5) at (0,3) [label=right:$6$] {};
                \node[vertex] (10) at (1,1.5) [label=left:$7$] {};
                \draw (2)--(5)--(10)--(2); 
            \end{tikzpicture}
                    \end{center}\caption{}
        \end{subfigure}  
            \\
            \begin{subfigure}[b]{.4\textwidth}
                \begin{center}
            \begin{tikzpicture}
                \node[vertex] (2) at (0,0) [label=right:$1$] {};
                \node[vertex] (3) at (0,1) [label=right:$2$] {};
                \node[vertex] (4) at (0,2) [label=right:$3$] {};
                \node[vertex] (5) at (0,3) [label=right:$4$] {};
                \node[vertex] (10) at (-1,1.5) [label=left:$9$] {};
                \draw (2)--(5)--(10)--(2);    
            \end{tikzpicture} 
        \end{center}\caption{}
    \end{subfigure} &
    \begin{subfigure}[b]{.4\textwidth}
        \begin{center}
            \begin{tikzpicture}
                \node[vertex] (1) at (0,0) [label=left:$1$] {};
                \node[vertex] (2) at (0,1) [label=left:$2$] {};
                \node[vertex] (7) at (1,0) [label=right:$7$] {};
                \node[vertex] (8) at (1,1) [label=right:$8$] {};
                \draw (1)--(2)--(7)--(8)--(1);
            \end{tikzpicture}
            \caption{}
            \label{fig:nondiamond}
        \end{center}
    \end{subfigure}
        \end{tabular}
    \end{center}
    \caption{The poset of Figure~\ref{fig:morrisex} broken into unicycle
        components}
    \label{fig:brokenup}
\end{figure}
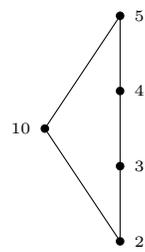
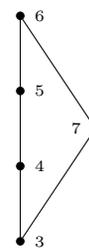
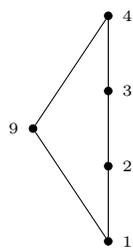
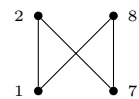
The poset in Figure~\ref{fig:nondiamond} is not a diamond, so is not
covered by Boussicault and \feray's result. However, the fact the poset of
Figure~\ref{fig:morrisex} is a complete intersection is explained by
\citet*[Theorem 8.6]{BoussicaultFerayLascouxReiner2012}, who give an algebraic
explanation for the factorization via opening/closing notches---locating Boussicault and
\feray's disconnecting chains can serve as a guide for which notches to open. 

An immediate consequence of~\cite[Theorem
8.6]{BoussicaultFerayLascouxReiner2012} is that a poset
which can be broken apart into unicyclic components (those having one or no
cycles) by opening notches must
have $\Rprt$
a complete intersection. It turns out that this suffices to characterize the
posets for which $\Rprt$ is a complete intersection, which was conjectured by the
author of this thesis and V.\ Reiner and proved
by~\citet{Morris2013}.

\begin{thm}[{\cite[Theorem 3]{Morris2013}}]
Suppose $P$ is a poset. Its root cone semigroup ring $\Rprt$ is a complete
intersection if and only if the Hasse diagram of $P$ can be obtained from
unicyclic posets by repeated gluings along chains.
\end{thm}

One would hope for a similar result in the signed poset case. Like in the poset
case, it is clear that strong planarity is not a necessary condition for
$\Rprt$ to be a complete intersection. After all, consider the signed poset in
Figure~\ref{fig:stronglyplanarnotnecc}. Its toric ideal is principal and there
is no notch that can be opened, so
$\Rprt$ must be a complete intersection, though the poset is clearly not
strongly planar.
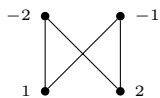
\begin{figure}[htbp]
    \begin{center}
        \begin{tikzpicture}
            \node[vertex] (1) at (0,0)[label=left:$1$] {};
            \node[vertex] (2) at (1,0) [label=right:$2$] {};
            \node[vertex] (-1) at (1,1) [label=right:$-1$] {};
            \node[vertex] (-2) at (0,1) [label=left:$-2$] {};
            \draw (1)--(-1)--(2)--(-2)--(1);
        \end{tikzpicture}
    \end{center}
    \caption{A signed poset which is not strongly planar, but has $\Rprt$ a
        complete intersection}
    \label{fig:stronglyplanarnotnecc}
\end{figure}

Unicyclic posets, of course, have principal root cone toric ideals.
Furthermore, cycles in the Hasse diagram of the poset correspond to circuits in
its matroid. This turns out to be the key fact to identifying the correct
signed analogue of ``unicyclic''.

\begin{conj}
		Suppose $P \subset \Bn$ (\resp $\Pc \subset \Cn$) is a signed poset.
		$\Rprt$ (\resp $\Rrt{\Pc}$) is a complete intersection if and only if
		$\Ghat(P)$ can be broken into biconnected components $P_1,\ldots,P_k$,
		each of which has a Hasse diagram with at most one circuit, by
		opening a series of signed notches.
\end{conj}

\begin{spacing}{1}
\emergencystretch=1em
\printbibliography[heading=bibintoc]
\nocite{*}
\end{spacing}
\end{document}